\documentclass[11pt,reqno]{amsart}
\usepackage{amsmath,amsfonts,amsthm,amssymb}
\usepackage[utf8]{inputenc}
\usepackage[T2A]{fontenc}
\usepackage[english]{babel}
\usepackage{mathrsfs}

\newcommand{\esssup}{\mathop{\mathrm{ess}\:\mathrm{sup}}\limits} 

\def\wh{\widehat}
\def\wt{\widetilde}
\def\R{\mathbb R}
\def\AC{\mathbb C}

\def\N{\mathbb N}
\def\A{\mathcal A}

\def\O{\mathcal O}

\def\u{\mathbf u}
\def\e{\mathbf e}

\def\x{\mathbf x}
\def\v{\mathbf v}

\def\k{\mathbf k}
\def\q{\mathbf q}

\def\D{\mathbf D}
\def\FF{\mathbf F}

\def\1{\mathbf 1}

\def\H{\mathfrak H}

\def\O{\mathcal O}
\def\NN{\mathfrak N}
\def\bxi{\boldsymbol \xi}

\def\1{\bold 1}

\def\Ran{\mathrm{Ran}\,}
\def\rank{\mathrm{rank}\,}

\def\eps{\varepsilon}
\def\Dom{\mathrm{Dom}\,}
\def\Ker{\mathrm{Ker}\,}

\def\clos{\mathrm{clos}\,}

\def\le{\leqslant}

\def\ge{\geqslant}

\theoremstyle{theorem}
\newtheorem{theorem}{Theorem}[section]
\newtheorem{proposition}[theorem]{Proposition}
\newtheorem{lemma}[theorem]{Lemma}
\newtheorem{condition}[theorem]{Condition}
\newtheorem{corollary}[theorem]{Corollary}

\theoremstyle{definition}

\newtheorem{remark}[theorem]{Remark}

\numberwithin{equation}{section}

\usepackage[height = 23.5cm, a4paper, hmargin={2.5cm , 2.5cm}]{geometry}

\begin{document}

\dedicatory{To the anniversary  of Nina Nikolaevna Ural'tseva}

\title[Homogenization of hyperbolic equations]
{Homogenization of hyperbolic equations\\ with periodic coefficients in $\R^d$: \\
sharpness of the results}

\author{M.~A.~Dorodnyi, T.~A.~Suslina}

\address{St.-Petersburg State University
\\
Universitetskaya nab. 7/9
\\
St.~Petersburg, 199034, Russia}

\email{mdorodni@yandex.ru}

\email{t.suslina@spbu.ru}


\thanks{Supported by Russian Science Foundation (project 17-11-01069).}

\begin{abstract}
In $L_2(\R^d;\AC^n)$,  a selfadjoint strongly elliptic  second order differential operator $\A_\eps$ is considered. 
It is assumed that the coefficients of the operator $\A_\eps$ are periodic and depend on  $\x/\eps$, 
where  \hbox{$\eps >0$}~is a small parameter. 
We find approximations for the operators $\cos ( \A_\eps^{1/2}\tau)$ and $\A_\eps^{-1/2}\sin ( \A_\eps^{1/2}\tau)$ in the norm of operators acting from the Sobolev space $H^s(\R^d)$ to  $L_2(\R^d)$ (with suitable  $s$).
We also find approximation with corrector for the operator $\A_\eps^{-1/2}\sin ( \A_\eps^{1/2}\tau)$  in the $(H^s \to H^1)$-norm. The question about the sharpness of the results with respect to the type of the operator norm and with respect to the dependence of estimates on $\tau$ is studied.  
The results are applied to study the behavior of the solutions of the Cauchy problem for the hyperbolic equation 
$\partial_\tau^2 \u_\eps = - \A_\eps \u_\eps + \FF$.
\end{abstract}

\maketitle
\tableofcontents

\section*{Introduction}

The paper concerns homogenization theory for periodic differential operators (DOs).
An extensive literature is devoted to homogenization problems; first of all, we mention the books 
\cite{BeLP, BaPa, ZhKO}.  For homogenization problems in $\R^d$, one of the methods   is the spectral approach based on the Floquet--Bloch theory; see, e.~g., \cite[Chapter 4]{BeLP}, \cite[Chapter 2]{ZhKO},
\cite{Se,Zh,COrVa}.

\subsection{The class of operators}
We consider selfadjoint second order DOs acting in $L_2(\R^d;\AC^n)$ and admitting a factorization of the form
\begin{equation}
\label{0.1}
\mathcal{A} = f(\x)^* b(\D)^* g(\x)b(\D) f(\x).
\end{equation}
Here  $b(\D)= \sum_{l=1}^d b_l D_l$
is the first order $(m \times n)$-matrix DO such that 
$m \ge n$ and the symbol $b(\boldsymbol{\xi})$ has maximal rank. 
The matrix-valued functions $g(\x)$ (of size $m\times m$) and  $f(\x)$ (of size $n\times n$) 
are periodic with respect to some lattice $\Gamma$; 
$g(\x)$ is positive definite and bounded; $f, f^{-1} \in L_\infty$. 
It is convenient to start with the study of the simpler class of operators given by  
\begin{equation}
\label{0.2}
\wh{\mathcal{A}} =  b(\D)^* g(\x)b(\D).
\end{equation}
Many operators of mathematical physics can be written in the form \eqref{0.1} or \eqref{0.2};
see \cite{BSu1} and \cite[Chapter 4]{BSu3}. The simplest example is the acoustics operator 
$\wh{\mathcal{A}} =  
- \operatorname{div} g(\x) \nabla = \D^* g(\x) \D.$

Now we introduce the small parameter $\eps>0$. For any $\Gamma$-periodic function $\varphi(\x)$,  denote $\varphi^\eps(\x) := \varphi(\eps^{-1} \x)$. Consider the operators
\begin{align}
\label{0.3}
&\mathcal{A}_\eps = f^\eps(\x)^* b(\D)^* g^\eps(\x)b(\D) f^\eps(\x),
\\
\label{0.4}
&\wh{\mathcal{A}}_\eps = b(\D)^* g^\eps(\x)b(\D).
\end{align}

\subsection{Operator error estimates for elliptic and parabolic problems in~$\R^d$}

In a series of papers \cite{BSu1,BSu2,BSu3,BSu4} by Birman and Suslina, an operator-theoretic  (spectral) approach to homogenization problems in $\R^d$ was suggested and developed. This approach was based on the scaling transformation, the Floquet--Bloch theory, and the analytic perturbation theory. 

Let us discuss the results for the simpler operator  \eqref{0.4}. In \cite{BSu1}, it was proved that
\begin{equation}
\label{0.5}
\| (\wh{\mathcal{A}}_\eps +I)^{-1} - (\wh{\mathcal{A}}^0 +I)^{-1} \|_{L_2(\R^d) \to L_2(\R^d)} \le C \eps.
\end{equation}
Here $\wh{\mathcal{A}}^0\! =\! b(\D)^* g^0 b(\D)$ is the effective operator with the constant~effective matrix~$g^0$. Approximations for the resolvent $(\wh{\mathcal{A}}_\eps\! +\!I)^{-\!1}$ in the~${(L_2 \!\!\to\!\! L_2)}$\nobreakdash-norm with the error term $O(\eps^2)$ and in the \hbox{$(L_2 \!\!\to\!\! H^1)$}-norm with the error term $O(\eps)$
(with correctors taken into account) were obtained in  \cite{BSu2, BSu3} and \cite{BSu4}, respectively.
 
The operator-theoretic approach was applied to parabolic problems in \cite{Su1, Su2, Su3, V, VSu1, VSu2}.
In \cite{Su1, Su2}, it was proved that
\begin{equation}
\label{0.6}
\| e^{- \tau \wh{\mathcal{A}}_\eps} - e^{-\tau \wh{\mathcal{A}}^0} \|_{L_2(\R^d) \to L_2(\R^d)} 
\le C \eps (\tau + \eps^2)^{-1/2}, \quad \tau >0.
\end{equation}
Approximations for the exponential $e^{-\tau \wh{\mathcal{A}}_\eps}$ in the $(L_2 \to L_2)$-norm with the error $O(\eps^2)$ and in the \hbox{$(L_2 \to H^1)$}-norm with the error  $O(\eps)$
(with  correctors taken into account) were obtained in \cite{V} and \cite{Su3}, respectively.
Even more accurate approximations for the resolvent and the semigroup of the operator $\wh{\mathcal{A}}_\eps$
were found in  \cite{VSu1, VSu2}. 

The operator-theoretic approach was applied also to the more general class of operators  
$\wh{\mathcal B}_\eps$ with the principal part $\wh{\mathcal A}_\eps$ and the lower order terms: 
the resolvent of this operator was studied in \cite{Su44, Su55} and the semigropup in  \cite{M100, M200}.

Estimates of the form \eqref{0.5}, \eqref{0.6} are called  \textit{operator error estimates} in homogenization theory. They are order-sharp. A different approach to operator error estimates (the so called shift method) was suggested by Zhikov and Pastukhova; see \cite{Zh2, ZhPas1, ZhPas2} and also the survey \cite{ZhPas3}.

\subsection{Operator error estimates for the nonstationary Schrödinger-type equations and hyperbolic equations\label{sec0.3}}
 The situation with homogenization of the nonstationary Schrödinger-type equations and hyperbolic equations  differs from the case of the elliptic and parabolic problems. 
 The operator-theoretic approach was applied to the nonstationary problems in  \cite{BSu5}.
Again, let us dwell on the results for the operator \eqref{0.4}.
In operator terms, we are talking about approximation of the operators
$e^{-i \tau \wh{\mathcal{A}}_\eps}$ and $\cos (\tau \wh{\mathcal{A}}_\eps^{1/2})$ (where $\tau \in \R$) 
for small $\eps$. It turned out that it is impossible to approximate these operators in the 
$(L_2 \! \to\! L_2)$-norm, and therefore we have to change the type of norm.  In \cite{BSu5}, it was proved that
\begin{align}
\label{0.7}
&\| e^{- i \tau \wh{\mathcal{A}}_\eps} - e^{-i \tau \wh{\mathcal{A}}^0} \|_{H^3(\R^d) \to L_2(\R^d)} 
\le C (1+|\tau|)\eps, 
\\
\label{0.8}
&\| \cos(\tau \wh{\mathcal{A}}_\eps^{1/2}) - \cos ( \tau (\wh{\mathcal{A}}^0)^{1/2} ) \|_{H^2(\R^d) \to L_2(\R^d)} 
\le C (1+|\tau|)\eps. 
\end{align}
 Recently Meshkova  \cite{M, M2} has obtained a similar result for the operator $\wh{\mathcal{A}}_\eps^{-1/2} \sin(\tau \wh{\mathcal{A}}_\eps^{1/2})$, together with approximation in the ``energy'' norm:
 \begin{equation}
\label{0.9}
\| \wh{\mathcal{A}}_\eps^{-1/2}\! \sin(\tau \wh{\mathcal{A}}_\eps^{1/2}) \!-\! 
(\wh{\mathcal{A}}^0)^{-1/2} \!\sin ( \tau (\wh{\mathcal{A}}^0)^{1/2} ) \|_{H^1(\R^d) \to L_2(\R^d)} 
\!\le\! C (1\!+\!|\tau|)\eps,
\end{equation}
\begin{equation}
\label{0.10}
\| \wh{\mathcal{A}}_\eps^{-1/2} \sin(\tau \wh{\mathcal{A}}_\eps^{1/2}) - 
(\wh{\mathcal{A}}^0)^{-1/2} \sin ( \tau (\wh{\mathcal{A}}^0)^{1/2} )  - \eps K(\eps)\|_{H^2(\R^d) \to H^1(\R^d)} 
\le C (1+|\tau|)\eps. 
\end{equation}
Here $K(\eps)$ is an appropriate corrector. (It is impossible to prove analogs of estimate  \eqref{0.10}
for the operators $e^{-i \tau \wh{\mathcal{A}}_\eps}$ and $\cos(\tau \wh{\mathcal{A}}_\eps^{1/2})$.)

To explain the method, let us discuss the proof of estimate \eqref{0.8}. 
Denote $\mathcal{H}_0 := - \Delta$.
Clearly, estimate \eqref{0.8} is equivalent to the inequality
\begin{equation}
\label{0.10a}
\bigl\| \bigl( \cos( \tau \wh{\mathcal{A}}_\eps^{1/2}) - 
 \cos ( \tau (\wh{\mathcal{A}}^0)^{1/2} ) \bigr) ({\mathcal H}_0 +I)^{-1} \bigr\|_{L_2(\R^d) \to L_2(\R^d)} 
\le C (1+|\tau|)\eps. 
\end{equation}
By the scaling transformation,  \eqref{0.10a} is equivalent to the estimate
\begin{equation}
\label{0.11}
\bigl\|  \bigl( \cos( \eps^{-1}\tau \wh{\mathcal{A}}^{1/2}) - 
 \cos (\eps^{-1} \tau (\wh{\mathcal{A}}^0)^{1/2} )  \bigr) \eps^2 ({\mathcal H}_0 + \eps^2 I)^{-1} \bigr \|_{L_2(\R^d) \to L_2(\R^d)} 
\le C (1+|\tau|)\eps. 
\end{equation}

Next, by the Floquet--Bloch theory, the operator $\wh{\mathcal{A}}$ expands  in the direct integral 
of the operators $\wh{\mathcal{A}}(\k)$ acting in $L_2(\Omega;\AC^n)$ (where $\Omega$ is the cell of the lattice  $\Gamma$) and given by the expression $b(\D+\k)^* g(\x) b(\D+\k)$ with periodic boundary conditions.
The operator $\wh{\mathcal{A}}(\k)$ has discrete spectrum. The operator family  
$\wh{\mathcal{A}}(\k)$ is studied by methods of the analytic perturbation theory  (with respect to the onedimensional parameter $t=|\k|$). It is possible to obtain the analog of inequality \eqref{0.11} for the operators $\wh{\mathcal{A}}(\k)$ with the constant that does not depend on $\k$. This yields estimate  \eqref{0.11}.

The operator exponential was further studied in \cite{Su4} and \cite{D}.
In \cite{Su4}, it was shown that estimate \eqref{0.7} is sharp with respect to the type of the operator norm: 
some conditions on the operator were found under which the estimate  
$\| e^{- i \tau \wh{\mathcal{A}}_\eps} - e^{-i \tau \wh{\mathcal{A}}^0} \|_{H^s \to L_2} \le C(\tau) \eps$ 
does not hold if  $s<3$. In \cite{D},  it was proved that estimate \eqref{0.7} is sharp with respect to the dependence on  $\tau$ (for large $|\tau|$): the factor  
$(1+ |\tau|)$ in the right-hand side cannot be replaced by $(1+|\tau|)^{\alpha}$ with $\alpha<1$.
On the other hand, in \cite{Su4}, it was shown that,  under some additional conditions, 
the result can be improved with respect to the type of the operator norm: $H^3$ can be replaced by  $H^2$.
Finally, in \cite{D}, it was proved that, under the same conditions, the result can be improved in another sense: the factor  $(1+|\tau|)$ can be replaced by $(1+|\tau|)^{1/2}$. As a result, under some additional conditions  (that are automatically satisfied for the acoustics operator) it was proved that
\begin{equation*}
\| e^{- i \tau \wh{\mathcal{A}}_\eps} - e^{-i \tau \wh{\mathcal{A}}^0} \|_{H^2(\R^d) \to L_2(\R^d)} 
\le C (1+|\tau|)^{1/2}\eps. 
\end{equation*}

The hyperbolic problems were studied in \cite{DSu} (see also \cite{DSu1}). It was shown that estimates  \eqref{0.8}, \eqref{0.9} are sharp with respect to the type of the operator norm, but under some additional assumptions the results can be improved: $H^2$ can be replaced by $H^{3/2}$ in \eqref{0.8}, and $H^1$~can be replaced 
by~$H^{1/2}$ in~\eqref{0.9}.

The nonstationary problems were also investigated for more general class of operators
$\wh{\mathcal B}_\eps$ (with the lower order terms):  the exponential $e^{- i \tau \wh{\mathcal{B}}_\eps}$ was studied in~\cite{D2}, and the hyperbolic problems were studied in~\cite{M3} where  a different approach based on modification of the Trotter--Kato theorem was suggested.

\vspace{-2mm}

\subsection{Main results}
In the present paper, we continue to study the behavior of the operators $\cos(\tau \wh{\mathcal{A}}_\eps^{1/2})$ and $\wh{\mathcal{A}}_\eps^{-1/2} \sin(\tau \wh{\mathcal{A}}_\eps^{1/2})$ for small $\eps$.
On one hand, we confirm the sharpness of estimates \eqref{0.8}--\eqref{0.10}:
we find a condition on the operator under which these estimates cannot be improved neither regarding the type of operator norm, nor regarding the dependence on  $\tau$.
This condition is formulated in the spectral terms. 

Consider the operator family $\wh{\mathcal A}(\k)$ and put
$$
\k = t \boldsymbol{\theta},\quad t=|\k|,\quad \boldsymbol{\theta} \in \mathbb{S}^{d-1}.
$$
  This family depends on the parameter $t$ analytically. 
 For $t=0$ the number $\lambda_0=0$ is the $n$-multiple eigenvalue of the ``unperturbed'' operator~$\wh{\mathcal A}(0)$.   Then for small $t$, there exist real-analytic branches of the eigenvalues $\lambda_l(t,\boldsymbol{\theta})$ ($l=1,\dots,n$) of the operator  $\wh{\mathcal A}(\k)$.  For small $t$, 
 we have the following convergent power series expansions
  \begin{equation*}
 \lambda_l(t,\boldsymbol{\theta}) = \gamma_l(\boldsymbol{\theta}) t^2 + 
 \mu_l(\boldsymbol{\theta}) t^3 + \nu_l(\boldsymbol{\theta}) t^4 + \dots, \quad l=1,\dots,n,
 \end{equation*}
 where $\gamma_l(\boldsymbol{\theta}) >0$ and $\mu_l(\boldsymbol{\theta}), \nu_l(\boldsymbol{\theta}) \in \R$.
 If $\mu_l(\boldsymbol{\theta}_0)\ne 0$ for some $l$ and some  
 $\boldsymbol{\theta}_0 \in \mathbb{S}^{d-1}$, then estimates  \eqref{0.8}--\eqref{0.10} cannot be improved.

 On the other hand, under some additional assumptions, we improve the results and obtain the following estimates:
 \begin{align}
\label{0.13}
& \| \cos(\tau \wh{\mathcal{A}}_\eps^{1/2}) - \cos ( \tau (\wh{\mathcal{A}}^0)^{1/2} ) \|_{H^{3/2}(\R^d) \to L_2(\R^d)} 
\le C (1+|\tau|)^{1/2} \eps,
\\
\label{0.14}
& \| \wh{\mathcal{A}}_\eps^{-1/2} \sin(\tau \wh{\mathcal{A}}_\eps^{1/2}) - 
(\wh{\mathcal{A}}^0)^{-1/2} \sin ( \tau (\wh{\mathcal{A}}^0)^{1/2} ) \|_{H^{1/2}(\R^d) \to L_2(\R^d)} 
\le C (1+|\tau|)^{1/2}\eps,
\\
\label{0.15}
& \| \wh{\mathcal{A}}_\eps^{-1/2} \sin(\tau \wh{\mathcal{A}}_\eps^{1/2}) - 
(\wh{\mathcal{A}}^0)^{-1/2} \sin ( \tau (\wh{\mathcal{A}}^0)^{1/2} ) - \eps K(\eps) \|_{H^{3/2}(\R^d) \to H^1(\R^d)} 
 \le C (1+|\tau|)^{1/2} \eps. 
\end{align}
For $n=1$, a sufficient condition that ensures estimates \eqref{0.13}--\eqref{0.15} is that
 $\mu(\boldsymbol{\theta})= \mu_1(\boldsymbol{\theta}) =0$ for any 
$\boldsymbol{\theta} \in \mathbb{S}^{d-1}$.
In particular, this condition is satisfied for the operator $\wh{\mathcal{A}}_\eps = \D^* g^\eps(\x)\D$ if
$g(\x)$ is a symmetric matrix with real entries.
For $n\ge 2$,  in addition to the condition that all the coefficients
 $\mu_l(\boldsymbol{\theta})$ are equal to zero, we impose one more condition in terms 
 of the coefficients $\gamma_l(\boldsymbol{\theta})$. The simplest version of this condition is that the
 different branches $\gamma_l(\boldsymbol{\theta})$ do not intersect each other. 

Next, we show that estimates  \eqref{0.13}--\eqref{0.15} are also sharp: if all the coefficients
$\mu_l(\boldsymbol{\theta})$ are equal to zero, but
$\nu_j(\boldsymbol{\theta}_0) \ne 0$ (for some $j$ and some $\boldsymbol{\theta}_0$), then estimates
 \eqref{0.13}--\eqref{0.15} cannot be improved neither regarding the norm type, 
 nor regarding the dependence on $\tau$.

Using interpolation, we also obtain estimates in the $(H^s\! \to\! L_2)$ or \hbox{$(H^s \!\to \!H^1)$}-norms.  
For instance, in the general case, the operator from \eqref{0.8} satisfies estimate of order $O((1\!+\!|\tau|)^{s/2}\eps^{s/2})$ in the $(H^s \!\to\! L_2)$-norm  with $0\!\le\! s \!\le\! 2$. 

We obtain qualified error estimates for small $\eps$ and large $\tau$: in the general case, it is possibe to consider  
$\tau =O(\eps^{-\alpha})$ with $0<\alpha<1$, while in the case of improvement it is possible to consider 
$\tau =O(\eps^{-\alpha})$ with ${0<\alpha<2}$.

 For more general operator  \eqref{0.3}, we obtain analogs of the results described above for the operators $\cos(\tau {\mathcal{A}}_\eps^{1/2})$ and ${\mathcal{A}}_\eps^{-1/2} \sin(\tau {\mathcal{A}}_\eps^{1/2})$ sandwiched between appropriate factors (for instance, for
 $f^\eps \cos(\tau {\mathcal{A}}_\eps^{1/2}) (f^\eps)^{-1}$). 

The results formulated in the operator terms are applied to homogenization of the solutions of the Cauchy problem for hyperbolic equations. In particular, we consider the acoustics equation and the elasticity system.

\vspace{-2mm}

\subsection{Method} The results are obtained by further development of the operator-theoretic approach. 
We follow the plan outlined above in Subsection \ref{sec0.3}. Our considerations are based on  the abstract operator-theoretic scheme. A  family of operators $A(t) = X(t)^*X(t)$, $t \in \R$, acting in some Hilbert space $\H$ is studied. Here $X(t)= X_0 + tX_1$. (The family $A(t)$ models the operator family $\A(\k) = \A(t \boldsymbol{\theta})$, but in the abstract statement the parameter $\boldsymbol{\theta}$ is absent.) It is assumed that the point $\lambda_0 =0$ is an isolated eigenvalue of multiplicity $n$ for the operator $A(0)$. Then for $|t| \le t_0$ the perturbed operator $A(t)$ has exactly $n$ eigenvalues on the interval  
$[0,\delta]$ ($\delta$ and $t_0$ are controlled explicitly). These eigenvalues and the corresponding eigenvectors are real-analytic functions of $t$. The coefficients of the corresponding power series expansions are called the  \textit{threshold characteristics} of the operator $A(t)$. We distinguish the finite rank operator $S$ (the so called \textit{spectral germ} of the family $A(t)$) acting in the subspace $\NN = \operatorname{Ker} A(0)$. 
The spectral germ  carries information about the threshold characteristics of principal order.

In terms of the spectral germ, we find appropriate approximations for the operators 
$\cos(\eps^{-1} \tau A(t)^{1/2})$ and $A(t)^{-1/2}\sin(\eps^{-1} \tau  A(t)^{1/2})$. 
Application of these abstract results leads to  the required estimates for DOs. 
  However,  at this step there is an additional difficulty. It concerns improvement of the results 
  under the assumption that all the coefficients
$\mu_l(\boldsymbol{\theta})$ are equal to zero. In the general case,  
it is impossible to make constructions uniform with respect to the parameter
 $\boldsymbol{\theta}$ and we are forced to impose additional conditions 
(assuming that the different branches $\gamma_l(\boldsymbol{\theta})$ do not intersect). 

\subsection{Plan of the paper} The paper consists of three chapters. Chapter~1 \hbox{(\S\S1--6)} contains necessary  abstract operator-theoretic material;  here main results in abstract terms are obtained. 
In Chapter 2 (\S\S7--14), periodic DOs of the form \eqref{0.1}, \eqref{0.2} are studied.
In \S 7, the class of operators is introduced and the direct integral expansion is described; 
the corresponding operator family ${\mathcal A}(\k)$ is included in the framework of the abstract scheme. In \S 8, the effective characteristics for the operator $\wh{\mathcal A}$ are described. In \S 9,  approximations for the operator-valued functions of  $\wh{\mathcal A}(\k)$ are deduced from the abstract theorems,
in \S 10, the sharpness of these results is confirmed. The effective characteristics of the operator \eqref{0.1} are described in \S 11.
Approximations  for the operator-valued functions of  ${\mathcal A}(\k)$ are found in  \S 12,  and the sharpness of these results is discussed in \S 13. Finally, in~\S 14, using the direct integral expansion, we deduce  approximations for the operator-valued functions of the operators \eqref{0.1} and \eqref{0.2}.
Chapter~3 (\S\S 15--18) is devoted to homogenization problems. In \S 15, with the help of the scaling transformation, we deduce  main results of the paper (approximations for the operator-valued functions of 
$\wh{\A}_\eps$ and $\A_\eps$) from the results of Chapter~2. In \S 16,  the results  are applied to study the solutions of the Cauchy problem for hyperbolic equations. 
\S\S 17, 18 are devoted to applications of the general results to the particular equations of mathematical physics.

\subsection{Notation} Let $\H$ and $\H_*$ be complex separable Hilbert spaces. The symbols
 $(\,\cdot\,,\,\cdot\,)_\H$ and $\|\,\cdot\,\|_\H$ stand for the inner product and the norm in $\H$, respectively; the symbol $\|\,\cdot\,\|_{\H \to \H_*}$ denotes the norm of a bounded operator from $\H$ to $\H_*$.
 Sometimes we omit the indices. By $I = I_\H$ we denote the identity operator in  $\H$. 
If $A: \H \to \H_*$ is a linear operator, then $\operatorname{Dom} A$ and $\operatorname{Ker} A$ denote its domain and its kernel, respectively. 
If $P$ is the orthogonal projection of the space  $\H$ onto $\NN$, then $P^\perp$ is the orthogonal projection onto $\NN^\perp:= \H \ominus \NN$.

The symbols $\langle \,\cdot\,, \,\cdot\, \rangle$ and $|\,\cdot\,|$ stand for the inner product and the norm in  $\AC^n$; ${\mathbf 1}_n$ is the unit $(n \times n)$-matrix. If $a$ is an $(m\times n)$-matrix, then the symbol~$|a|$ denotes the norm of the matrix $a$ viewed as a linear operator from $\AC^n$ to $\AC^m$.
Next, we denote $\x = (x_1,\dots, x_d) \in \R^d$, $i D_j = \partial_j = \partial / \partial x_j$,
$j=1,\dots,d$; $\D = - i \nabla = (D_1,\dots, D_d)$.
The classes $L_p$ (where $1 \le p \le \infty$) and the Sobolev classes  (of order $s \ge 0$) of $\AC^n$-valued functions in a domain  $\O \subset \R^d$ are denoted by $L_p(\O;\AC^n)$ and 
$H^s(\O;\AC^n)$, respectively. Sometimes we write simply $L_p(\O)$, $H^s(\O)$.

Different constants in estimates are denoted by $C$, $\mathcal C$, $\mathrm C$, $\mathfrak C$, and $c$ (probably, with indices and marks).

\subsection{Acknowledgements} M.~A.~Dorodnyi is  a Young Russian Mathematics award winner and would like to thank its sponsors and jury. T.~A.~Suslina is grateful to  Mittag-Leffler Institute (Stockholm, Sweden). The work was partially completed during the participation of T.~A.~Suslina in the Research Program ``Spectral Methods in Mathematical Physics'' in February and March 2019.

\section*{Chapter 1. Abstract operator-theoretic scheme}

\section{Quadratic operator pencils}
\label{abstr_section_1}

The material of this section is borrowed from  \cite{BSu1, BSu2, VSu1, Su4, D}.

\subsection{The operators $X(t)$ and $A(t)$}  
\label{abstr_X_A_section}

Let  $\mathfrak{H}$ and $\mathfrak{H}_{*}$~be  complex separable Hilbert spaces. Suppose that   $X_{0}: \mathfrak{H} \to \mathfrak{H}_{*}$~is a densely defined and closed operator, and $X_{1} : \mathfrak{H} \to \mathfrak{H}_{*}$~is a bounded operator.  
Then the operator $X(t) = X_0 + t X_1$,  $t \in \mathbb{R}$, is closed on  $\Dom X_0$. Consider the family of selfadjoint operators $A(t) = X(t)^*  X(t)$ in $\mathfrak{H}$. The operator $A(t)$ is generated by the closed quadratic form $\| X(t) u \|^{2}_{\mathfrak{H}_*}$, $u \in \Dom X_0$. Denote $A_0 := A(0)$;
$\mathfrak{N} := \Ker  A_0 = \Ker X_0$; $\mathfrak{N}_{*} := \Ker X^*_0$.

\textit{It is assumed that the point $\lambda_0 = 0$~is an isolated point of the spectrum of  $A_0$ and $0 < n := \dim \mathfrak{N} < \infty$, $n \le n_* := \dim  \mathfrak{N}_* \le \infty$}.

Let $d^0$ be the distance from the point $\lambda_0 = 0$ to the rest of the spectrum of $A_0$. 
By $P$ and $P_*$ we denote the orthogonal projections of  $\mathfrak{H}$ onto~$\mathfrak{N}$ and of   $\mathfrak{H}_*$ onto  $\mathfrak{N}_*$, respectively.
Let $F(t;[a, b])$ be the spectral projection of the operator $A(t)$ for the interval $[a,b]$. We put
$$
\mathfrak{F} (t;[a,b]) := F(t;[a, b]) \mathfrak{H}.
$$
 Fix a number $\delta > 0$ such that  $8 \delta < d^0$.  Next, we choose a number $t_0 > 0$ so that
 \begin{equation}
\label{abstr_t0_fixation}
t_0 \le \delta^{1/2} \|X_1\|^{-1}.
\end{equation}
As was shown in~\cite[Chapter~1,~(1.3)]{BSu1}, for  $|t| \le t_0$ we have $F(t; [0,\delta]) = F(t;[0, 3 \delta])$ and 
$\rank F(t; [0,\delta]) = n$. We shall write $F(t)$ instead of $F(t; [0,\delta])$.

\subsection{The operators $Z$, $R$, and $S$}
\label{abstr_Z_R_S_op_section}
According to~\cite[Chapter~1, \S1]{BSu1} and~\cite[\S1]{BSu2}, we introduce the operators appearing in the considerations of the perturbation theory.

Let $\omega \in \mathfrak{N}$ and let $\phi = \phi(\omega) \in \Dom X_0 \cap \mathfrak{N}^{\perp}$ be a (weak) solution of the equation $X^*_0 (X_0 \phi + X_1 \omega) = 0$.
Define the operator $Z : \mathfrak{H} \to \mathfrak{H}$ by the relation $Zu = \phi (P u)$, $u \in \mathfrak{H}$. 
Note that $PZ=0$, whence $Z^* P=0$. We have
\begin{equation}\label{1.2a}
\| X_0 Z \| \le \|X_1\|, \quad  \| Z \| \le (8 \delta)^{-1/2}\|X_1\|.
\end{equation} 
Next, we define the operator $R : \mathfrak{N}  \to \mathfrak{N}_*$ by the formula $R := X_0 Z + X_1$.
Then  $R= P_*X_1 |_{\mathfrak{N}}$.

The operator $S :=  R^* R : \mathfrak{N} \to \mathfrak{N}$ is called the 
\emph{spectral germ} of the family $A(t)$ at $t=0$. We have  $S = P X^*_1 P_* X_1 |_{\mathfrak{N}}$. 
The spectral germ is called \emph{nondegenerate} if $\Ker S = \{0\}$. Note that
\begin{equation}
\label{abstr_Z_R_S_est}
\| R \| \le \| X_1 \|, \quad   \| S \| \le \| X_1 \|^2.
\end{equation}

\subsection{The operators $Z_2$ and $R_2$}
\label{abstr_Z2_R2_section}
We introduce the operators $Z_2$ and $R_2$ (see \cite[\S 1]{VSu1}). 
Let $\omega \in \mathfrak{N}$, and let $\psi= \psi(\omega) \in \Dom X_0 \cap \mathfrak{N}^{\perp}$ be a  (weak) solution of the equation
$X^*_0 (X_0 \psi + X_1 Z \omega) = - P^\perp X_1^* R \omega$.
Obviously, the solvability condition is satisfied. We define the operator  
$Z_2: \mathfrak{H} \to \mathfrak{H}$ by the relation $Z_2 u = \psi(P u)$, $u \in \mathfrak{H}$.
Finally, we introduce the operator $R_2: \mathfrak{N} \to \mathfrak{H}_*$ by the formula 
$R_2:= X_0 Z_2 + X_1 Z$.

\subsection{The analytic branches of eigenvalues and eigenvectors of the operator $A(t)$} 
According to the general analytic perturbation theory (see~\cite{Ka}), for $|t| \le t_0$ there exist real-analytic functions $\lambda_l (t)$ (the branches of the eigenvalues) and real-analytic  $\mathfrak{H}$-valued functions $\varphi_l (t)$ (the branches of the eigenvectors) such that 
\begin{equation*}
A(t) \varphi_l(t) = \lambda_l (t) \varphi_l(t), \quad l = 1, \ldots, n,\quad |t| \le t_0,
\end{equation*}
and the set $\varphi_l (t), \; l = 1, \ldots, n,$ forms an \emph{orthonormal basis} in $\mathfrak{F}(t;[0,\delta])$. 
For \emph{sufficiently small} $t_*$ (where $0 < t_* \le t_0$) and $|t| \le t_*$ we have the following convergent power series expansions:
\begin{align}
\label{abstr_A(t)_eigenvalues_series}
\lambda_l(t) &= \gamma_l t^2 + \mu_l t^3 + \nu_l t^4 + \ldots, \quad \gamma_l \ge 0, \quad \; \mu_l, \nu_l \in \mathbb{R}, \quad   l = 1, \ldots, n, \\ 
\label{abstr_A(t)_eigenvectors_series}
\varphi_l (t) &= \omega_l + t \psi_l^{(1)} + \ldots, \quad	l = 1, \ldots, n.
\end{align}
The elements $\omega_l =  \varphi_l (0), \, l = 1, \ldots, n,$ form an orthonormal basis in the subspace $\mathfrak{N}$. In \cite[Chapter~1, \S1]{BSu1} and  \cite[\S1]{BSu2}, it was shown that 
$\widetilde{\omega}_l := \psi_l^{(1)} - Z \omega_l \in \mathfrak{N}$, $l = 1, \ldots, n$,
\begin{equation}
\label{abstr_S_eigenvectors}
S \omega_l = \gamma_l \omega_l , \quad l = 1, \ldots, n.
\end{equation}
Thus,  \emph{the numbers $\gamma_l$ and the elements $\omega_l$ defined 
by~\emph{(\ref{abstr_A(t)_eigenvalues_series})} and~\emph{(\ref{abstr_A(t)_eigenvectors_series})} are eigenvalues and eigenvectors of the germ $S$}. We have
\begin{equation}
\label{abstr_SP_repr_gamma_omega}
P = \sum_{l=1}^{n} (\,\cdot\,, \omega_l) \omega_l, \quad 
SP = \sum_{l=1}^{n} \gamma_l (\,\cdot\,, \omega_l) \omega_l.
\end{equation}

\subsection{Threshold approximations}
We need approximations for the spectral projection $F(t)$ and the operator $A(t) F(t)$ on the interval  $[0,t_0]$. 
The following statement was obtained in~\cite[Chapter~1, Theorems~4.1~and~4.3]{BSu1}.  
Below by $\beta_j$ we denote \textit{absolute constants} \textit{assuming that} $\beta_j \ge 1$.

\begin{proposition}[see~\cite{BSu1}] Under the assumptions of Subsection~\emph{\ref{abstr_X_A_section}}, we have
	\begin{align}
	\label{abstr_F(t)_threshold_1}
	\| F(t) - P \| &\le C_1 |t|, \quad |t| \le t_0, \\
	\label{abstr_A(t)_threshold_1}
	\| A(t)F(t) - t^2 SP \| &\le C_2 |t|^3, \quad |t| \le t_0.
	\end{align}
	The number $t_0$ is subject to~\eqref{abstr_t0_fixation} and the constants $C_1,$ $C_2$ are given by
	\begin{equation}
	\label{abstr_C1_C2}
	C_1 = \beta_1 \delta^{-1/2} \| X_1 \|, \quad   C_2 = \beta_2 \delta^{-1/2}\| X_1 \|^3.
	\end{equation}
\end{proposition}

We also need  more accurate approximations; see~\cite[\S 2~and~\S 4]{BSu2}.

\begin{proposition}[see~\cite{BSu2}] Under the assumptions of Subsection~\emph{\ref{abstr_X_A_section}}, we have 	\begin{align}
	\label{abstr_F(t)_threshold_2}
	F(t) &= P + tF_1 + F_2(t), & &\| F_2(t) \| \le  C_3 t^2,& |t| &\le t_0,
	\\ 
	\nonumber
	A(t) F(t) &= t^2 SP + t^3 K + \Psi (t), & &\| \Psi (t) \| \le C_4 t^4, & |t| &\le t_0, 
	\end{align}
	where $C_3 = \beta_3 \delta^{-1} \| X_1 \|^2$  and $C_4 = \beta_4 \delta^{-1}\| X_1 \|^4$.
	The operator $K$ can be represented as $K = K_0 + N = K_0 + N_0 + N_*,$
where $K_0$ takes $\mathfrak{N}$ to $\mathfrak{N}^{\perp}$ and $\mathfrak{N}^{\perp}$ to $\mathfrak{N},$ 
and  $N = N_0 + N_*$ takes $\mathfrak{N}$ into itself and takes  $\mathfrak{N}^{\perp}$ to $\{ 0 \}$. In terms of the coefficients of the power series expansions, we have
\begin{gather}
\notag
F_1 = \sum_{l=1}^{n} \left( (\,\cdot\,, Z \omega_l) \omega_l + (\,\cdot\,, \omega_l) Z \omega_l \right) , 
\quad 
K_0 = \sum_{l=1}^{n} \gamma_l \left( (\,\cdot\,, Z \omega_l) \omega_l + (\,\cdot\,, \omega_l) Z \omega_l \right) , \\
\label{abstr_N_0_N_*}
N_0 = \sum_{l=1}^{n} \mu_l (\,\cdot\,, \omega_l) \omega_l, \quad N_* = \sum_{l=1}^{n} \gamma_l \left( (\,\cdot\,, \widetilde{\omega}_l) \omega_l + (\,\cdot\,, \omega_l) \widetilde{\omega}_l\right) .
\end{gather}
In the invariant terms, 
\begin{align}
\label{abstr_F1_K0_N_invar}
F_1 = ZP + PZ^*, \quad & K_0 = Z S P + S P Z^*, 
\\
\label{abstr_N}
N = Z^*X_1^* R P &+ (RP)^* X_1 Z.
\end{align}
\end{proposition}

\begin{remark}
	\label{abstr_N_remark}
	 In the basis $\{\omega_l\}_{l=1}^n$, the operators $N$, $N_0$, and $N_*$ (restricted to $\mathfrak{N}$)
	 are given by the matrices of size $n \times n$. The operator $N_0$ is diagonal:
		\begin{equation}
		\label{abstr_N_0_matrix_elem}
		(N_0 \omega_j, \omega_k ) = \mu_j \delta_{jk}, \quad  j,\, k = 1, \dots,n.
		\end{equation}
		The matrix entries of the operator $N_*$ are given by 
		\begin{equation*}
		(N_* \omega_j, \omega_k) = \gamma_k (\omega_j, \widetilde{\omega}_k) + \gamma_j (\widetilde{\omega}_j, \omega_k ) = ( \gamma_j - \gamma_k)(\widetilde{\omega}_j, \omega_k ), \quad j, k = 1,\ldots, n.
		\end{equation*}        
		Here we  have taken into account that  (see~\emph{\cite[(1.18)]{BSu2}})
		\begin{equation}
		\label{abstr_omega_tilde_omega_rel}
		(\widetilde{\omega}_j, \omega_k ) + (\omega_j, \widetilde{\omega}_k) = 0, \quad j, k = 1,\ldots, n.
		\end{equation}
		It is seen that the diagonal entries of $N_*$ are equal to zero: 
		\hbox{$(N_* \omega_j, \omega_j) = 0$}, $j = 1, \ldots , n$. Moreover,
		$(N_* \omega_j, \omega_k) = 0$ if $\gamma_j = \gamma_k$.
\end{remark}

\subsection{Nondegeneracy condition}
\label{abstr_nondegenerated_section}
Below we impose the following additional condition~(cf.~\cite[Chapter~1, Subsection~5.1]{BSu1}).

\begin{condition}
\label{cond_A}
For some $c_* > 0$ we have 
\begin{equation}
\label{abstr_A(t)_nondegenerated}
A(t) \ge c_* t^2 I, \quad  |t| \le t_0.
\end{equation}
\end{condition}

From~(\ref{abstr_A(t)_nondegenerated}) it follows that  $\lambda_l (t) \ge c_* t^2, \; l = 1, \ldots, n$, for $|t| \le t_0$. By~(\ref{abstr_A(t)_eigenvalues_series}), this implies that
$\gamma_l \ge c_* > 0$, $l= 1, \ldots, n$.
Thus, the germ is nondegenerate (see \eqref{abstr_S_eigenvectors}):
\begin{equation}
\label{abstr_S_nondegenerated}
S \ge c_* I_{\mathfrak{N}}.
\end{equation}

\subsection{Division of the eigenvalues of the operator $A(t)$ into clusters}
\label{abstr_cluster_section}

The material of this subsection is borrowed from~\cite[\S2]{Su4}. It is meaningful for  $n \ge 2$.

Suppose that Condition~\ref{cond_A} is satisfied. 
 Now it is convenient to change the notation, tracing the multiplicities of the eigenvalues of the germ~$S$. 
 Let $p$ be the number of different eigenvalues of the germ. We enumerate these eigenvalues in the increasing order and denote them by $\gamma_j^\circ$, $j=1,\dots,p$. Their multiplicities are denoted by
 $k_1, \ldots, k_p$ (obviously, $k_1 + \dots + k_p = n$). 
  The eigenspaces are denoted by $\mathfrak{N}_j = \Ker(S - \gamma^{\circ}_j I_\mathfrak{N})$, $j = 1 ,\ldots, p$. Then $\mathfrak{N} = \sum_{j=1}^{p} \oplus \mathfrak{N}_j$.
Let $P_j$~be the orthogonal projection of $\mathfrak{H}$ onto $\mathfrak{N}_j$. Then 
$P = \sum_{j=1}^{p} P_j$, and $P_j P_l = 0$ for $j \ne l$.
Correspondingly, we change the notation for the eigenvectors of the germ (those that are  \textquotedblleft embryos\textquotedblright  \ in~(\ref{abstr_A(t)_eigenvectors_series})) dividing them in $p$ parts, so that  $\omega^{(j)}_1, \ldots, \omega^{(j)}_{k_j}$ correspond to the eigenvalue $\gamma^{\circ}_j$ and form an orthonormal basis in $\mathfrak{N}_j$.

\begin{remark}
 According to Remark~\ref{abstr_N_remark},
 	$P_j N_* P_j = 0$ and  $P_l N_0 P_j  = 0$ for $l \ne j$.
	This implies the  invariant representations for the operators $N_0$ and $N_*$:
	\begin{equation}
	\label{abstr_N_invar_repers}
	N_0 = \sum_{j=1}^{p} P_j N P_j, \quad N_* = \sum_{\substack{1 \le j,\, l \le p: \, j \ne l}} P_j N P_l. 
	\end{equation}
\end{remark}

 For each pair of indices $(j, l), 1 \le j,l \le p, j \ne l$, we denote 
\begin{equation}
\label{abstr_c_circ_jl}
c^{\circ}_{jl} := \min \{c_*, n^{-1} |\gamma^{\circ}_l - \gamma^{\circ}_j|\}.
\end{equation}
Clearly, there exists a number  $i_0 = i_0 (j,l)$, where $j \le i_0 \le l-1$ for $j < l$ and $l \le i_0 \le j-1$ for $l < j$, such that $\gamma^{\circ}_{i_0 + 1} -  \gamma^{\circ}_{i_0} \ge c^{\circ}_{jl}$. 
We choose a number $t^{00}_{jl} \le t_0$ satisfying the inequality
\begin{equation}
\label{abstr_t00_jl}
t^{00}_{jl} \le (4C_2)^{-1} c^{\circ}_{jl} = (4 \beta_2)^{-1} \delta^{1/2} \|X_1\|^{-3 } c^{\circ}_{jl}.
\end{equation}
Denote  $\Delta_{jl}^{(1)}:=[\gamma^{\circ}_1 - c^{\circ}_{jl}/4, \gamma^{\circ}_{i_0} + c^{\circ}_{jl}/4]$ and \hbox{$\Delta_{jl}^{(2)}:=[\gamma^{\circ}_{i_0+1} - c^{\circ}_{jl}/4, \gamma^{\circ}_p + c^{\circ}_{jl}/4]$}. 
  The spectral projections of the operator $A(t)$ corresponding to the intervals $t^2 \Delta_{jl}^{(1)}$ and $t^2 \Delta_{jl}^{(2)}$ are denoted by  $F^{(1)}_{jl} (t)$ and $F^{(2)}_{jl} (t)$, respectively. 
In \cite[\S 2]{Su4}, it was shown that $F(t) = F^{(1)}_{jl} (t) + F^{(2)}_{jl} (t)$
for $|t| \le  t^{00}_{jl}$ and the following statement was proved.

\begin{proposition}[see~\cite{Su4}]
	For $|t| \le t^{00}_{jl}$ we have 
	\begin{equation*}
	\begin{split}
	\| F^{(1)}_{jl} (t) - (P_1 + \cdots + P_{i_0}) \| &\le C_{5,jl} |t|, \\
	\| F^{(2)}_{jl} (t) - (P_{i_0+1} + \cdots + P_p) \| &\le C_{5,jl} |t|.
	\end{split}
	\end{equation*}
	The constant $C_{5,jl}$ is given by 
	$C_{5,jl} = \beta_5 \delta^{-1/2} \|X_1\|^5 (c^{\circ}_{jl})^{-2}$.
\end{proposition}

\subsection{The coefficients $\nu_l$\label{sec_nu_l}}
For definiteness, suppose that  enumeration in 
\eqref{abstr_A(t)_eigenvalues_series}, \eqref{abstr_A(t)_eigenvectors_series} is such that
$\gamma_1 \le \dots \le \gamma_n$. 
The coefficients  $\nu_l$ and the vectors $\omega_l$, $l=1,\dots,n$, in the expansions
\eqref{abstr_A(t)_eigenvalues_series}, \eqref{abstr_A(t)_eigenvectors_series} are eigenvalues and eigenvectors of some problem; see \cite[Subsection~1.8]{D}.
We need to describe this problem in the case where $\mu_l=0$, $l=1,\dots,n,$ i.~e., $N_0=0$.

\begin{proposition}[see~\cite{D}]\label{Prop_nu_1}
Let $N_1^0 := Z_2^* X_1^* RP + (RP)^* X_1 Z_2 + R_2^* R_2 P$. 
Suppose that $N_0=0$.
Let $\gamma_1^\circ, \dots, \gamma_p^\circ$ be the different eigenvalues of the operator 
$S,$ and let $k_1, \dots, k_p$ be their multiplicities. Suppose that $P_q$ is the orthogonal projection onto the subspace 
$\mathfrak{N}_q  = \operatorname{Ker} (S - \gamma_q^\circ I_{\mathfrak{N}}),$ $q=1,\dots,p$.
We introduce the operators $\mathcal{N}^{(q)},$ $q=1,\dots,p,$ as follows\emph{:} 
the operator $\mathcal{N}^{(q)}$ acts in  $\mathfrak{N}_q$ and is given by   
$$
\mathcal{N}^{(q)} := P_q \Bigl( N_1^0 - \frac{1}{2} Z^* Z SP - \frac{1}{2} SP Z^* Z  \Bigr) \Bigl\vert_{\mathfrak{N}_q}
+ \sum_{j=1,\dots,p: j\ne q} (\gamma_q^\circ - \gamma_j^\circ)^{-1} P_q N P_j N \bigr\vert_{\mathfrak{N}_q}.
$$
Denote $i(q)= k_1 + \dots + k_{q-1} +1$.
Let $\nu_l$ be the coefficients of $t^4$ in the expansions  \eqref{abstr_A(t)_eigenvalues_series},
and let $\omega_l$ be the embryos from  \eqref{abstr_A(t)_eigenvectors_series}, $l=1,\dots,n$.
Then  
$$
\mathcal{N}^{(q)} \omega_l = \nu_l \omega_l, \quad l= i(q), i(q)+1,\dots, i(q) + k_q -1.
$$
\end{proposition}

\section{Approximation for the operators $\cos ( \tau A(t)^{1/2}) P$ and $A(t)^{-1/2} \sin (\tau A(t)^{1/2}) P$}

\subsection{Approximation in the operator norm in $\H$} 
Denote
\begin{align}
\label{2.00}
	\mathcal{J}(t,\tau) := &  e^{-i\tau A(t)^{1/2}} P -  e^{-i\tau (t^2S)^{1/2}} P,
\\
\label{2.1}
	\mathcal{E}(t,\tau) := & A(t)^{-1/2} e^{-i\tau A(t)^{1/2}} P - (t^2S)^{-1/2} e^{-i\tau (t^2S)^{1/2}} P. 
	\end{align}
We need estimates of the operators \eqref{2.00} and \eqref{2.1} established (with the help of the threshold approximations) in~\cite[Subsection 2.3]{BSu5}, \cite[Subsection~2.1]{M} and~\cite[(2.34), (2.49), (2.53), (2.54)]{DSu}.

\begin{proposition}[see~\cite{BSu5,M}]
	\label{prop2.1}
	For $\tau \in \mathbb{R}$ we have  
	\begin{align*}
	\| \mathcal{J}(t,\tau) \|& \le 2 C_1 |t| + C_6 |\tau| t^2, \quad |t| \le t_0,
	\\
	\| \mathcal{E}(t,\tau) \|& \le C_7 + C_8 |\tau| |t|, \quad 0 < |t| \le t_0.	
	\end{align*}
	The number $t_0$ is subject to condition \eqref{abstr_t0_fixation}. The constant $C_1$ is defined by  \eqref{abstr_C1_C2}, and
	 $C_6= \beta_6 \delta^{-1/2}  \|X_1\|^2 \bigl( 1 + c_*^{-1/2} \|X_1\| \bigr)$.
	The constants $C_7$ and $C_8$ are given by
	\begin{equation*}
C_7 = \beta_7 \delta^{-1/2} c_*^{-1/2} \|X_1\| \left( 1 + c_*^{-1} \|X_1\|^2\right) , \quad C_8 =  c_*^{-1/2} C_6.
	\end{equation*}
\end{proposition}

\begin{proposition}[see~\cite{DSu}]
	\label{abstr_exp_enched_est_thrm1}
	Suppose that the operator $N$ defined by~\eqref{abstr_N} is equal to zero{\rm :} $N=0$. 
	Then for $\tau \in \mathbb{R}$ we have
	\begin{align*}
	\| \mathcal{J}(t,\tau)  \| & \le 2 C_1 |t| + C_9 |\tau| |t|^3, \quad |t| \le t_0,
	\\
	\| \mathcal{E}(t,\tau)  \| & \le C_7 + C_{10} |\tau| t^2, \quad 0 < |t| \le t_0.
	\end{align*}
The number $t_0$ is subject to condition \eqref{abstr_t0_fixation}. The constants $C_9$ and  $C_{10}$ are given by  
	\begin{align*}
	C_{9}& = \beta_{9} \delta^{-1}  \|X_1\|^3 \big( 1 + c_*^{-1/2} \|X_1\| + c_*^{-3/2} \|X_1\|^3 + 
	c_*^{-5/2} \|X_1\|^5\big), 
	\\
	C_{10}& = c_*^{-1/2} C_9.
	\end{align*}
\end{proposition}

\begin{proposition}[see~\cite{DSu}]
	\label{abstr_exp_enched_est_thrm2}
Denote
	$$\mathcal{Z} := \{ (j, l) \colon 1 \le j, l \le p, \; j \ne l, \; P_j N P_l \ne 0\},$$
	\begin{equation}
	\label{abstr_c^circ}
	c^\circ : = \min_{(j,l) \in \mathcal{Z}} c^\circ_{jl},
	\end{equation}
where the numbers $c^\circ_{jl}$ are defined by \eqref{abstr_c_circ_jl}. 
Suppose that the number $t^{00} \le t_0$ is such that
	\begin{equation}
	\label{abstr_t^00}
	t^{00} \le (4 \beta_2)^{-1} \delta^{1/2} \|X_1\|^{-3} c^\circ.
	\end{equation}
	Suppose that the operator $N_0$ defined by \eqref{abstr_N_invar_repers} is equal to zero{\rm :} \hbox{$N_0=0$}.
	 Then for $\tau \in \mathbb{R}$ we have 
	\begin{align*}
	\| \mathcal{J}(t,\tau)  \| & \le C_{11} |t| + C_{12} |\tau| |t|^3, \quad  |t| \le t^{00},
	\\
	\| \mathcal{E}(t,\tau)  \| & \le C_{13} + C_{14} |\tau| t^2, \quad 0 < |t| \le t^{00}.
	\end{align*}
	The constants $C_{11},$ $C_{12},$ $C_{13}$, and $C_{14}$ are given by 
	\begin{align*}
	C_{11} &= \beta_{11} \delta^{-1/2}  \|X_1\| \bigl( 1 +  n^2 c_*^{-1/2} \|X_1\|^3 (c^\circ)^{-1}\bigr),
	\\
	 C_{12} &= \beta_{12} \delta^{-1} \|X_1\|^3 \bigl( 1 + c_*^{-1/2} \|X_1\| + c_*^{-3/2} \|X_1\|^3 + c_*^{-5/2} \|X_1\|^5 \bigr) + \beta_{12} \delta^{-1} c_*^{-1/2}  \|X_1\|^8 n^2 (c^\circ)^{-2},
	 \\
	C_{13} &= \beta_{13} \delta^{-1/2} c_*^{-1/2} \|X_1\| \bigl( 1 + c_*^{-1} \|X_1\|^2 + n^2 c_*^{-1/2} \|X_1\|^3 (c^\circ)^{-1}\bigr), 
	\\
	C_{14}& = c_*^{-1/2} C_{12}.
	\end{align*}
\end{proposition}

Propositions  \ref{prop2.1}--\ref{abstr_exp_enched_est_thrm2} directly imply approximations for the operators
$\cos ( \tau A(t)^{1/2}) P$ and $A(t)^{-1/2} \sin (\tau A(t)^{1/2}) P$.  Denote 
\begin{align}
\label{3.01}
\mathcal{J}_1(t, \tau) &:=    \cos (\tau A(t)^{1/2})  {P}  
-   \cos ( \tau (t^2 {S})^{1/2} )  {P},
\\
\label{3.02}
\mathcal{J}_2(t, \tau) &:=    A(t)^{-1/2} \sin (\tau A(t)^{1/2})  {P}  
-   (t^2 {S})^{-1/2} \sin ( \tau (t^2 {S})^{1/2} ) {P}.
\end{align}

\begin{theorem}[see~\cite{BSu5,M}]
\label{th2.1}
	For $\tau \in \mathbb{R}$ and $|t| \le t_0$ we have 
	\begin{align}
	\label{2.5}
	& \bigl\|  \mathcal{J}_1(t, \tau)
	  \bigr\| \le 2 C_{1} |t| + C_{6}  |\tau| t^2, 
	\\
	\label{2.6}
	& \bigl\| \mathcal{J}_2(t, \tau)  \bigr\| \le C_{7}  + C_{8} 
	 |\tau| |t|.
	\end{align}
\end{theorem}

\begin{theorem}[see~\cite{DSu}]
Suppose that the operator $N$ defined by~\eqref{abstr_N} is equal to zero{\rm :} $N = 0$. 
Then for $\tau \in \mathbb{R}$ and $|t| \le t_0$ we have 
	\begin{align}
	\label{2.7}
	& \bigl\| \mathcal{J}_1(t, \tau) \bigr\| \le 2 C_{1} |t| + C_{9}  |\tau| |t|^3, 
	\\
	\label{2.8}
	& \bigl\| \mathcal{J}_2(t, \tau) \bigr\| \le C_{7}  + C_{10} |\tau| t^2.
	\end{align} 
\end{theorem}

\begin{theorem}[see~\cite{DSu}]
\label{th2.7}
	Suppose that the operator $N_0$ defined by~\eqref{abstr_N_invar_repers} is equal to zero{\rm :} $N_0 = 0$. Then for  $\varepsilon > 0,$ $\tau \in \mathbb{R}$, and $|t| \le t^{00}$ we have 
	\begin{align*}
	& \bigl\| \mathcal{J}_1(t, \tau) \bigr\| \le  C_{11} |t| + C_{12}  |\tau| |t|^3, 
	\\
	& \bigl\| \mathcal{J}_2(t, \tau) \bigr\| \le C_{13}  + 
	C_{14} |\tau| t^2.
	\end{align*} 
\end{theorem}

\subsection{Approximation of the operator $A(t)^{-1/2} \sin(\tau A(t)^{1/2})$ in the 
\hbox{\textquotedblleft energy\textquotedblright} norm}

We obtain approximation for the operator  $A(t)^{-1/2} \sin(\tau A(t)^{1/2})$ in the 
\hbox{\textquotedblleft energy\textquotedblright} norm. We need two estimates, the first one  follows from  
\eqref{abstr_t0_fixation}, \eqref{abstr_Z_R_S_est}, and  \eqref{abstr_A(t)_threshold_1}, and the second one was proved in~\cite[(2.23)]{BSu4}:
\begin{align}
\label{abstr_A(t)F(t)_sqrt_est}
\| A (t)^{1/2} F(t) \| &\le C_{15} |t|, \quad |t| \le  t_0; & C_{15} &= (1 + \beta_2)^{1/2} \|X_1\|, 
\\
\label{abstr_sqrtA(t)F2(t)_est}
\| A(t)^{1/2}F_2(t) \| &\le C_{16} t^2, \quad |t| \le  t_0; & C_{16} &= \beta_{16} \delta^{-1/2} \|X_1\|^2.
\end{align}

By~(\ref{abstr_F(t)_threshold_1}), for $\tau \in \R$ we have 
\begin{equation}\label{3.5}
\| A(t)^{1/2} ( A(t)^{-1/2}e^{-i\tau A(t)^{1/2}} P 
- A(t)^{-1/2}e^{-i\tau A(t)^{1/2}} F(t)  P) \| \le C_1 |t|,  \quad  |t| \le t_0.
\end{equation}
Next,
\begin{equation}\label{3.6}
\begin{split}
A(t)^{1/2} F(t)\bigl(A(t)^{-1/2} e^{-i\tau A(t)^{1/2}} P&
 -  (t^2S)^{-1/2} e^{-i\tau (t^2S)^{1/2}} P\bigr)
 = A(t)^{1/2} F(t) \mathcal{E}(t,\tau) P,
\end{split}
\end{equation}
where the operator $\mathcal{E}(t,\tau)$ is defined by  \eqref{2.1}.
The right-hand side is estimated with the help of~(\ref{abstr_A(t)F(t)_sqrt_est}) and Proposition~\ref{prop2.1}
(if the additional assumptions are satisfied, we apply Propositions~\ref{abstr_exp_enched_est_thrm1} and \ref{abstr_exp_enched_est_thrm2}). For $\tau \in \R$ we obtain:
{\allowdisplaybreaks
\begin{align}
\label{3.8}
\| A(t)^{1/2} \!F(t) \mathcal{E}(t,\tau) P \| &\!\le\! C_{15} |t|(C_7\! +\! C_8 |\tau| |t| ),\quad|t|\!\le\! t_0;
\\
\label{3.9}
\| A(t)^{1/2}\! F(t) \mathcal{E}(t,\tau) P \| &\!\le\! C_{15} |t|(C_7 \!+\! C_{10} |\tau| t^2 ),\quad|t| \!\le\! t_0,  \text{ if }  N\!=\!0;
\\
\label{3.10}
\| A(t)^{1/2}\! F(t) \mathcal{E}(t,\tau) P \| &\!\le\! C_{15} |t|(C_{13}\! +\! C_{14} |\tau| t^2),\quad|t|\! \le\! t^{00},   \text{ if }  N_0\!=\!0.
\end{align}
}

By~(\ref{abstr_F(t)_threshold_2}), (\ref{abstr_F1_K0_N_invar}), and the identity $Z^* P = 0$, we have  
\begin{equation}\label{3.11}
A(t)^{1/2} F(t)  (t^2S)^{-1/2} e^{-i\tau (t^2S)^{1/2}} P  
= A(t)^{1/2} (I + tZ + F_2(t)) (t^2S)^{-1/2} e^{-i\tau (t^2S)^{1/2}} P.
\end{equation}
Using~(\ref{abstr_S_nondegenerated}) and~(\ref{abstr_sqrtA(t)F2(t)_est}), we obtain 
\begin{equation}\label{3.12}
\| A(t)^{1/2} F_2(t)  (t^2S)^{-1/2} e^{-i\tau (t^2S)^{1/2} } P \| \le c_*^{-1/2} C_{16} |t|, \quad \tau \in \mathbb{R}, \  |t| \le t_0.
\end{equation}

As a result, relations \eqref{3.5}--\eqref{3.12} imply the following results.

\begin{theorem}[see~\cite{M}]
	\label{abstr_sin_thrm_wo_eps}
	Let  
	\begin{equation}
	\label{3.13}
	\Sigma(t,\tau)\!:= \! \big(A(t)^{-1/2} \sin(\tau A(t)^{1/2})\! -\! (I\!+\!tZ)(t^2 S)^{-1/2} \sin(\tau (t^2 S)^{1/2})\big)P.
	\end{equation}
	For $\tau \in \mathbb{R}$ and $|t| \le t_0$ we have 
	\begin{equation}
	\label{abstr_sin_est_wo_eps}
	\| A(t)^{1/2} \Sigma(t,\tau) \| \le C_{17} |t| + C_{18}  |\tau| t^2.
	\end{equation}
	The constants $C_{17}$ and $C_{18}$ are given by  
	$$
	C_{17} = C_1 + C_7 C_{15} +c_*^{-1/2}C_{16},\quad C_{18} = C_8 C_{15}.
		$$
\end{theorem}

\begin{theorem}
	\label{abstr_sin_enchcd_thrm_wo_eps_1}
	Suppose that the assumptions of Theorem {\rm \ref{abstr_sin_thrm_wo_eps}} are satisfied. 
	Suppose that the operator~$\!N$ defined by~\eqref{abstr_N} is equal to zero{\rm:} $N\! = \!0$. 
	Then for~$\tau \!\in \!\mathbb{R}$~and~$|t| \!\le\! t_0$ we have
	\begin{equation}
	\label{abstr_sin_enchcd_est_wo_eps_1}
	\bigl\| A(t)^{1/2} \Sigma(t,\tau)  \bigr\|  \le C_{17} |t| + C_{19} |\tau| |t|^3,  \quad  C_{19} = C_{10} C_{15}.
	\end{equation}
\end{theorem}

\begin{theorem}
	\label{abstr_sin_enchcd_thrm_wo_eps_2}
	\!Suppose that the assumptions of Theorem {\rm\ref{abstr_sin_thrm_wo_eps}} are satisfied.
	\!Suppose that the operator~$\!N_0$ defined by~\eqref{abstr_N_invar_repers} is equal to zero{\rm :} $N_0\! =\! 0$. Then for~$\tau\! \in \!\mathbb{R}$~and~$|t| \!\le\! t^{00}$ we have 
	\begin{equation*}
	\bigl\| A(t)^{1/2} \Sigma(t,\tau)  \bigr\| \le C_{20} |t| + C_{21} |\tau| |t|^3.
	\end{equation*}
The constants $C_{20}$ and $C_{21}$ are given by 
          $$
	C_{20} = C_1 + C_{13} C_{15} +c_*^{-1/2}C_{16},\quad C_{21} = C_{14} C_{15}.
	$$
\end{theorem}

Theorem~\ref{abstr_sin_thrm_wo_eps} was known earlier  (see~\cite[Proposition~2.2]{M}).

\section{Approximation for the operators $\cos(\varepsilon^{-1} \tau A(t)^{1/2}) P$ 
and $A(t)^{-1/2} \sin(\varepsilon^{-1} \tau A(t)^{1/2}) P$}
\label{abstr_aprox_thrm_section}

\subsection{Approximation in the operator norm in $\H$}

Now we introduce the parameter $\varepsilon > 0$. We study the behavior of the operators 
$\cos(\varepsilon^{-1} \tau A(t)^{1/2}) P$ and $A(t)^{-1/2} \sin(\varepsilon^{-1} \tau A(t)^{1/2}) P$ for small
 $\varepsilon$. It is convenient to multiply these operators by the 
 \textquotedblleft smoothing factor\textquotedblright \ $\varepsilon^s (t^2 + \varepsilon^2)^{-s/2}P$, where $s > 0$. (This term is explained by the fact that  in applications to DOs  such multiplication  turns into  smoothing.) Our goal is to obtain approximations for the smoothed operator $\cos(\varepsilon^{-1} \tau A(t)^{1/2}) P$ with error  of order 
$O (\varepsilon)$ and for the smoothed operator $A(t)^{-1/2} \sin(\varepsilon^{-1} \tau A(t)^{1/2}) P$ with error of order  $O (1)$ for minimal possible~$s$.

\begin{theorem}[see~\cite{BSu5,M}]
	\label{th3.1}
	Suppose that the operators $\mathcal{J}_1(t,\tau)$ and $\mathcal{J}_2(t,\tau)$ are defined by  \eqref{3.01}\textup, \eqref{3.02}.
	For $\varepsilon > 0, \tau \in \mathbb{R}$, and $|t| \le t_0$ we have 
	\begin{align}
	\label{2.5a}
	& \bigl\| \mathcal{J}_1(t, \eps^{-1}\tau)
	 \bigr\| \varepsilon^2 (t^2 + \varepsilon^2)^{-1} \le (C_{1} + C_{6}|\tau|) \varepsilon,
	\\
	\label{2.6a}
	& \bigl\| \mathcal{J}_2(t,\eps^{-1} \tau)
	 \bigr\| \varepsilon (t^2 + \varepsilon^2)^{-1/2} \le C_{7} + C_{8} |\tau|.
	\end{align}
\end{theorem}

Theorem \ref{th3.1} directly follows from estimates \eqref{2.5} and \eqref{2.6} with~$\tau$ replaced by $\varepsilon^{-1} \tau$. Earlier, estimate \eqref{2.5a} was obtained in  
\cite[Theorem 2.7]{BSu5}, and estimate \eqref{2.6a} was proved in  \cite[Theorem 2.3]{M}.

This result can be improved under some additional assumptions.

\begin{theorem}
	\label{th3.2}
	Suppose that the operator $N$ defined by~\eqref{abstr_N} is equal to zero{\rm :} ${N = 0}$. 
	Then for $\varepsilon > 0,$ $\tau \in \mathbb{R}$, and $|t| \le t_0$ we have 
	\begin{align}
	\label{2.7a}
	& \bigl\|  \mathcal{J}_1(t, \eps^{-1}\tau)
	 \bigr\| \varepsilon^{3/2} (t^2 + \varepsilon^2)^{-3/4} \le (2 C_{1} + C_{9}' |\tau|^{1/2}) \varepsilon,
	\\
	\label{2.8a}
	& \bigl\| \mathcal{J}_2(t, \eps^{-1}\tau)
	 \bigr\| \varepsilon^{1/2} (t^2 + \varepsilon^2)^{-1/4} \le C_{7} + C_{10}' |\tau|^{1/2}.
	\end{align}
Here $C_9' = \max \{C_9; 2\}$ and $C_{10}' = \max \{C_{10}; 2 c_*^{-1/2}\}$.
\end{theorem}

\begin{proof}
For $\tau=0$ estimates \eqref{2.7a} and \eqref{2.8a} are obvious. Suppose that $\tau \ne 0$.
If $|t| \ge \eps^{1/3} |\tau|^{-1/3}$, then $\varepsilon^{3/2} (t^2 + \varepsilon^2)^{-3/4} \le \eps |\tau|^{1/2}$,  
whence the left-hand side of  \eqref{2.7a} does not exceed $2 \eps |\tau|^{1/2}$. 

 Now, assume  that  $|t| \le t_0$ and $|t| < \eps^{1/3} |\tau|^{-1/3}$. 
 We apply inequality~\eqref{2.7} with $\tau$ replaced by $\varepsilon^{-1} \tau$:
  \begin{align*}
  \bigl\| \mathcal{J}_1(t, \eps^{-1}\tau)
 	 \bigr\| \varepsilon^{3/2} (t^2 + \varepsilon^2)^{-3/4} 
	  	&\le (2 C_{1} |t| + C_{9} \eps^{-1}
	 |\tau| |t|^3) \varepsilon^{3/2} (t^2 + \varepsilon^2)^{-3/4}
	 \\
	 &\le 2 C_1 \eps + C_9 |\tau| \eps^{1/2} |t|^{3/2} \le
	 2 C_1 \eps + C_9 |\tau|^{1/2} \eps.
\end{align*}
As a result, we arrive at  \eqref{2.7a}.

Similarly, if $|t| \ge \eps^{1/3} |\tau|^{-1/3}$, then $|t|^{-1} \varepsilon^{1/2} (t^2 + \varepsilon^2)^{-1/4} \le 
 |\tau|^{1/2}$. Therefore, by \eqref{abstr_A(t)_nondegenerated} and \eqref{abstr_S_nondegenerated}, 
 the left-hand side of \eqref{2.8a} does not exceed  $2 c_*^{-1/2} |\tau|^{1/2}$.

For  $|t| \le t_0$ and $|t| < \eps^{1/3} |\tau|^{-1/3}$, by 
 \eqref{2.8} with  $\tau$ replaced by $\varepsilon^{-1} \tau$, we have 
  \begin{align*}
  \bigl\| \mathcal{J}_2(t,\tau) \bigr\| \varepsilon^{1/2} (t^2 + \varepsilon^2)^{-1/4} 
	  	&\le (C_{7}  + C_{10} \eps^{-1}|\tau| t^2) \varepsilon^{1/2} (t^2 + \varepsilon^2)^{-1/4}
	 \\
	 &\le C_7  + C_{10} \eps^{-1/2} |\tau|  |t|^{3/2} \le
	 C_7 + C_{10} |\tau|^{1/2}.
\end{align*}
As a result, we obtain estimate \eqref{2.8a}.
\end{proof}

Similarly, Theorem \ref{th2.7} implies the following result.

\begin{theorem}
	\label{th3.3}
	Suppose that the operator $N_0$ defined by~\eqref{abstr_N_invar_repers} is equal to zero{\rm :} $N_0 = 0$. Then for $\varepsilon > 0,$ $\tau \in \mathbb{R}$, and $|t| \le t^{00}$ we have
	\begin{align*}
	& \bigl\| \mathcal{J}_1(t, \eps^{-1}\tau)
	 \bigr \| \varepsilon^{3/2} (t^2 + \varepsilon^2)^{-3/4} \le (C_{11} + C_{12}' |\tau|^{1/2}) \varepsilon;
	\\
	& \bigl\| \mathcal{J}_2(t, \eps^{-1}\tau) 
	 \bigr \| \varepsilon^{1/2} (t^2 + \varepsilon^2)^{-1/4} \le C_{13} + C_{14}' |\tau|^{1/2},
	\end{align*}
	where $C_{12}' = \max \{ C_{12};2\}$ and $C_{14}' = \max \{C_{14}; 2 c_*^{-1/2}\}$.
\end{theorem}

\subsection{Approximation of the operator $A(t)^{-1/2} \sin(\varepsilon^{-1} \tau A(t)^{1/2}) P$ in the 
\hbox{\textquotedblleft energy\textquotedblright}  norm}

We apply Theorem~\ref{abstr_sin_thrm_wo_eps}. By~(\ref{abstr_sin_est_wo_eps}) (with $\tau$ replaced by  $\varepsilon^{-1} \tau$), for $|t| \le t_0$ we have 
\begin{equation*}
\begin{aligned}
&\| A(t)^{1/2}  \Sigma(t, \eps^{-1}\tau)   \| \varepsilon^2 (t^2 + \varepsilon^2)^{-1}  
\\
&\ \ \le (C_{17}|t| + C_{18} \varepsilon^{-1}  |\tau| t^2) \varepsilon^2 (t^2 + \varepsilon^2)^{-1} \le (C_{17} + C_{18}|\tau|) \varepsilon.
\end{aligned}
\end{equation*}
We arrive at the following result which was earlier proved in~\cite[Theorem~2.4]{M}.

\begin{theorem}[see~\cite{M}]
	\label{abstr_sin_general_thrm}
	Suppose that the operator $\Sigma(t,\tau)$ is defined by  \eqref{3.13}.
	For  $\varepsilon > 0, \tau \in \mathbb{R}$, and $|t| \le t_0$ we have 
	\begin{equation*}
	\bigl \| A(t)^{1/2} \Sigma(t,\eps^{-1}\tau) \bigr \| \varepsilon^2 (t^2 + \varepsilon^2)^{-1} \le (C_{17} + C_{18}|\tau|) \varepsilon.
	\end{equation*}
\end{theorem}

Theorem~\ref{abstr_sin_enchcd_thrm_wo_eps_1} allows us to improve the result of Theorem~\ref{abstr_sin_general_thrm} in the case where $N = 0$.

\begin{theorem}
	\label{th3.5}
	\!Suppose that the assumptions of Theorem~\emph{\ref{abstr_sin_general_thrm}} are satisfied. 
	\!Suppose that the operator~$N$ defined by~\eqref{abstr_N} is equal to zero{\rm :} $N \!=\! 0$. Then for $\varepsilon \!> \!0, \tau\! \in\! \mathbb{R}$,~and~$|t|\! \le\! t_0$ we have 
	\begin{equation} 
	\label{abstr_sin_enchanced_est_1}    
	\bigl\| A(t)^{1/2}  \Sigma(t, \eps^{-1}\tau) \bigr\| \varepsilon^{3/2} (t^2 + \varepsilon^2)^{-3/4} \le  (C_{17} + C_{19}' |\tau|^{1/2} ) \varepsilon. 
	\end{equation}    
	Here $C'_{19} =  \max\{1+ (2+ 8^{-1/2}) \| X_1\| c_*^{-1/2}, C_{19}\}$.
\end{theorem}

\begin{proof}
It suffices to assume that  $\tau \ne 0$. Note that for
 $|t| \ge {\varepsilon}^{1/3} |\tau|^{-1/3}$ we have  $\varepsilon^{3/2} (t^2 + \varepsilon^2)^{-3/4} \le \varepsilon |\tau|^{1/2}$.  By \eqref{abstr_t0_fixation} and \eqref{1.2a}, 
$$
\|   A(t)^{1/2} (P + tZ P)\| = \| (X_0 + t X_1) (P+ tZP) \| \le (2+ 8^{-1/2}) \| X_1\| |t|, \quad |t| \le t_0.
$$
Combining this with \eqref{abstr_S_nondegenerated} and \eqref{3.13}, we see that  
the norm $\| A(t)^{1/2}  \Sigma(t, \eps^{-1}\tau)\|$ does not exceed the constant 
$\widetilde{C}_{19}= 1+ (2+ 8^{-1/2}) \| X_1\| c_*^{-1/2}$ for $|t| \le t_0$,  whence the left-hand 
side of~(\ref{abstr_sin_enchanced_est_1}) does not exceed
 $\widetilde{C}_{19} \varepsilon |\tau|^{1/2}$ for $|t|\le t_0$ and $|t| \ge \eps^{1/3} |\tau|^{-1/3}$.

By~(\ref{abstr_sin_enchcd_est_wo_eps_1}) with  $\tau$ replaced by  $\varepsilon^{-1} \tau$, for $|t| \le t_0$ 
and $|t| < {\varepsilon}^{1/3} |\tau|^{-1/3}$ we obtain 
\begin{align*}
\bigl\| A(t)^{1/2} &\Sigma(t, \eps^{-1}\tau)  \bigr\| \varepsilon^{3/2} (t^2 + \varepsilon^2)^{-3/4} 
\\
& \le \left( C_{17}|t| + C_{19} \varepsilon^{-1} | \tau | |t|^3 \right) \varepsilon^{3/2} (t^2 + \varepsilon^2)^{-3/4} 
\\
&\le C_{17} \varepsilon + C_{19} |\tau| \varepsilon^{1/2} |t|^{3/2} \le  (C_{17}  + C_{19} |\tau|^{1/2}) \varepsilon.
\end{align*}
As a result, we arrive at  estimate  \eqref{abstr_sin_enchanced_est_1} with the constant  
$C_{19}' = \max \{ C_{19}; \wt{C}_{19}\}$.
\end{proof}

Similarly, Theorem~\ref{abstr_sin_enchcd_thrm_wo_eps_2} implies the following result.

\begin{theorem}
	\label{abstr_sin_enchanced_thrm_2}
	\!Suppose that the assumptions of Theorem~\emph{\ref{abstr_sin_general_thrm}} are satisfied. 
	\!Suppose that the operator~$\!N_0$ defined by~\eqref{abstr_N_invar_repers} is equal to zero{\rm:} $N_0\! =\! 0$. \!Then for  $\varepsilon \!>\! 0, \tau\! \in\! \mathbb{R}$, and~$|t| \!\le \!t^{00}$ we have
	\begin{equation*} 
	\bigl\| A(t)^{1/2}  \Sigma(t, \eps^{-1} \tau)  \bigr\| \varepsilon^{3/2} (t^2 + \varepsilon^2)^{-3/4} \le  (C_{20} + C_{21}' |\tau|^{1/2} )  \varepsilon. 
	\end{equation*}    
	Here $C'_{21} =  \max\{1+ (2+ 8^{-1/2}) \| X_1\| c_*^{-1/2}, C_{21}\}$.
\end{theorem}

\begin{remark}
	\label{abstr_constants_remark}
We  have tracked how the constants in  the estimates depend on the parameters of  the problem. The constants  
	$C_1$, $C_{6}$,  $C_{7}$, $C_8$ from Theorem~\ref{th3.1};
	$C_9'$, $C_{10}'$ from Theorem \ref{th3.2}; $C_{17}$, $C_{18}$ from 
	Theorem~\ref{abstr_sin_general_thrm}; $C_{19}'$ from Theorem  
	\ref{th3.5}  are estimated by polynomials with {\rm (}absolute{\rm )} positive coefficients of the parameters $\delta^{-1/2}$, $c_*^{-1/2}$, $\|X_1\|$. The constants $C_{11}$,  $C_{12}'$,
	$C_{13}$, $C_{14}'$ from Theorem \ref{th3.3};
	$C_{20}$,  $C'_{21}$ from Theorem~\ref{abstr_sin_enchanced_thrm_2} are controlled by polynomials with positive coefficients of the same parameters, and also of  $(c^{\circ})^{-1}$ and $n$.
\end{remark}

\section{Sharpness of the results of \S \ref{abstr_aprox_thrm_section}}

\subsection{Sharpness of the results regarding the smoothing factor}
The following statement obtained in \cite[Theorem~3.5]{DSu} confirms that Theorem \ref{th3.1}  is sharp in the general case.

\begin{theorem}[see~\cite{DSu}]
	\label{th3.8}
	Suppose that the operators $\mathcal{J}_1(t,\tau)$ and $\mathcal{J}_2(t,\tau)$ are defined by  \eqref{3.01} and  \eqref{3.02}. Suppose that $N_0 \ne 0$. 
	
	\noindent $1^\circ$. Let $\tau \ne 0$ and $0 \le s < 2$. Then there does not exist a constant 
	\hbox{$C(\tau) > 0$} such that  the inequality
	\begin{equation}
	\label{**.1}
	\bigl\| \mathcal{J}_1(t,\eps^{-1}\tau) 
	\bigr \|  \varepsilon^{s} (t^2 + \varepsilon^2)^{-s/2} \le  C(\tau) \varepsilon
	\end{equation}
holds for all sufficiently small $|t|$ and $\varepsilon > 0$.
	
	\noindent $2^\circ$. Let $\tau \ne 0$ and $0 \le r < 1$. Then there does not exist a constant 
	\hbox{$C(\tau) > 0$} such that the inequality 
	\begin{equation}
	\label{**.2}
	\bigl  \| \mathcal{J}_2(t,\eps^{-1}\tau) 
	 \bigr  \|  \varepsilon^{r} (t^2 + \varepsilon^2)^{-r/2} \le  C(\tau)
	\end{equation}
holds for all sufficiently small $|t|$ and $\varepsilon > 0$.
\end{theorem}

Next, we confirm the sharpness of Theorems \ref{th3.2} and \ref{th3.3}. 

\begin{theorem}
	\label{th3.9}
	Let $N_0 = 0$ and $\mathcal{N}^{(q)} \ne 0$ for some $q\in \{1,\dots,p\}$.
	 
	\noindent $1^\circ$. Let $\tau \ne 0$ and $0 \le s < 3/2$. Then there does not exist a constant $C(\tau) > 0$
	such that \eqref{**.1}  holds for all sufficiently small $|t|$ and $\varepsilon > 0$.

	\noindent $2^\circ$. Let $\tau \ne 0$ and $0 \le r < 1/2$.  Then there does not exist a constant $C(\tau) > 0$ such that  \eqref{**.2}
	holds for all sufficiently small $|t|$ and $\varepsilon > 0$.
\end{theorem}

\begin{proof}
Let us check statement $1^\circ$.
It suffices to assume that  \hbox{$1\le s < 3/2$}. Since $F(t)^\perp P = (P - F(t))P$,
 then~(\ref{abstr_F(t)_threshold_1}) implies that  
\begin{equation}\label{**.3}
\begin{split}
\| \cos(\eps^{-1} \tau A(t)^{1/2}) F(t)^\perp P  \| \eps (t^2 + \eps^2)^{-1/2}
\le C_1 |t| \eps (t^2 + \eps^2)^{-1/2} \le C_1 \eps, \quad
   |t| \le t_0.
\end{split}
\end{equation}

	Next, for $ |t| \le t_0$ we have 
	\begin{equation}
	\label{**.4}
	\cos(\varepsilon^{-1} \tau A(t)^{1/2})F (t) = \sum_{l=1}^{n} \cos \bigl( \varepsilon^{-1} \tau \sqrt{\lambda_l(t)}\bigr) (\,\cdot\,,\varphi_l(t))\varphi_l(t).
	\end{equation}
	From the convergence of series~(\ref{abstr_A(t)_eigenvectors_series}) it follows that 
	\begin{equation}
	\label{**.5}
	\| \varphi_l(t) - \omega_l \| \le c_1 |t|, \quad |t| \le t_*, \quad l = 1,\ldots,n.
	\end{equation}

	We prove by contradiction. Suppose that, for some \hbox{$0 \ne \tau \in \mathbb{R}$} and 
	\hbox{$1 \le s < 3/2$},  inequality~\eqref{**.1} holds for all sufficiently small $|t|$ and~$\varepsilon$. 
	By \eqref{abstr_SP_repr_gamma_omega} and \eqref{**.3}--\eqref{**.5}, this is equivalent to existence of a constant $\widetilde{C}(\tau) > 0$ such that the  inequality
	\begin{equation}
	\label{**.6}
	\Bigl \|  \sum_{l=1}^n \left( \cos \bigl(\varepsilon^{-1} \tau \sqrt{\lambda_l(t)} \bigr)   - 
	\cos \bigl(\varepsilon^{-1} \tau |t| \sqrt{\gamma_l} \bigr) \right) (\,\cdot\,, \omega_l) \omega_l  \Bigr\|
	  \varepsilon^{s} (t^2 + \varepsilon^2)^{-s/2}
		\le \widetilde{C} (\tau) \varepsilon
	\end{equation}
holds for all sufficiently small $|t|$ and $\varepsilon$.
	
	According to \eqref{abstr_N_0_matrix_elem} and Proposition \ref{Prop_nu_1}, the conditions $N_0 = 0$ and $\mathcal{N}^{(q)} \ne 0$ for some $q\in \{1,\dots,p\}$ mean that in expansions \eqref{abstr_A(t)_eigenvalues_series} we have $\mu_l=0$ for any $l=1,\dots,n$ and $\nu_j \ne 0$ at least for one $j$. Then 
	$$
	\lambda_j(t) = \gamma_j t^2 + \nu_j t^4 + O(|t|^5), \quad |t|\le t_*.
	$$
	Hence, decreasing $t_*$ if necessary, we have 
	\begin{equation}
	\label{**.7}
	\sqrt{\lambda_j(t)} = \sqrt{\gamma_j} |t| \Bigl(1 + \frac{\nu_j}{2\gamma_j} t^2 + O(|t|^3) \Bigr), \quad |t|\le t_*.
	\end{equation}
	We apply the operator under the norm sign in \eqref{**.6} to the element $\omega_j$. Then
	\begin{equation}
	\label{**.8}
	\Bigl|   \cos \Bigl(\varepsilon^{-1} \tau \sqrt{\lambda_j(t)} \Bigr)   - 
	\cos \left( \varepsilon^{-1} \tau |t| \sqrt{\gamma_j} \right)   \Bigr|
	  \varepsilon^{s} (t^2 + \varepsilon^2)^{-s/2} \le \widetilde{C} (\tau) \varepsilon
	\end{equation}
	for all sufficiently small $|t|$ and $\varepsilon$.
	Next, we put  
	\begin{equation}
	\label{**.9}
	t = t(\varepsilon) = (2 \pi)^{1/3} \gamma_j^{1/6} | \nu_j \tau |^{-1/3}  \varepsilon^{1/3} = c \varepsilon^{1/3}.
	\end{equation}
	Then  
	$$
	\cos(\varepsilon^{-1} \tau t(\varepsilon) \sqrt{\gamma_j}) = \cos(\alpha_j \varepsilon^{-2/3}),
	$$
	where
	$\alpha_j := (\operatorname{sgn} \tau) (2\pi)^{1/3} \gamma_j^{2/3} |\tau|^{2/3} |\nu_j|^{-1/3}$.
	 Assuming that  $\varepsilon$ (and then also $t(\eps)$) is sufficiently small and taking~\eqref{**.7} into account, we have 
	 $\varepsilon^{-1} \tau \sqrt{\lambda_j(t(\eps))}  = \alpha_j \varepsilon^{-2/3} + 
\pi \operatorname{sgn} (\tau \nu_j) + O(\varepsilon^{1/3})$, whence 
	 $\cos \left(\varepsilon^{-1} \tau \sqrt{\lambda_j(t(\eps))} \right) =  - \cos\left(\alpha_j \varepsilon^{-2/3} + O(\varepsilon^{1/3})\right)$.
	Thus, from~\eqref{**.8}  it follows that the expression 
	$$
	\bigl| \cos \bigl(\alpha_j \varepsilon^{-2/3} + O(\varepsilon^{1/3})\bigr) + \cos\bigl(\alpha_j \varepsilon^{-2/3}
	\bigr) \bigr| \varepsilon^{2 s/3-1} (c^2 + \varepsilon^{4/3})^{-s/2}$$ 
	is uniformly bounded for small $\varepsilon > 0$. But this is not true if  $s < 3/2$.	(It suffices to consider the sequence $\varepsilon_k = \alpha_j^{3/2} (2\pi k)^{-3/2}$, $k \in \mathbb{N}$.) We arrive at a contradiction. This completes the proof of statement $1^\circ$.

	Statement $2^\circ$ is checked similarly. By~(\ref{abstr_F(t)_threshold_1}) and  
	\eqref{abstr_A(t)_nondegenerated},
	 \begin{equation}
	\label{**.10}
	\bigl\|  A(t)^{-1/2} \sin(\varepsilon^{-1} \tau A(t)^{1/2})  F(t)^\perp P \bigr \|
	  \le c_*^{-1/2}{C}_1, \quad |t|\le t_0.
	  \end{equation}
	
	Next, for $ |t| \le t_0$ we have 
	\begin{equation}
	\label{**.11}
	A(t)^{-1/2} \sin(\varepsilon^{-1} \tau A(t)^{1/2})F (t) = \sum_{l=1}^{n} \frac{\sin\bigl( \varepsilon^{-1} \tau \sqrt{\lambda_l(t)}\bigr) }{\sqrt{\lambda_l(t)}}(\,\cdot\,,\varphi_l(t))\varphi_l(t).
	\end{equation}
	
	Suppose that, for some  $\tau \ne 0$ and $0\le r < 1/2$, inequality \eqref{**.2} holds for all sufficiently small  $|t|$ and $\varepsilon$. Combining this with  \eqref{abstr_SP_repr_gamma_omega}, 
	\eqref{abstr_A(t)_nondegenerated}, \eqref{**.5}, \eqref{**.10}, and \eqref{**.11}, we see that there exists a constant $\check{C}(\tau)$ such that the inequality
	\begin{equation}
	\label{**.12}
	\Bigl\| \sum_{l=1}^{n} \Big( \frac{\sin(\varepsilon^{-1} \tau \sqrt{\mathstrut \lambda_l(t)})}{\sqrt{\mathstrut \lambda_l(t)}} - 
	\frac{\sin(\varepsilon^{-1} \tau |t| \sqrt{\mathstrut \gamma_l})}{|t| \sqrt{\mathstrut \gamma_l}} \Big)
	(\,\cdot\,,\omega_l)\omega_l  \Bigr \| 
	\varepsilon^{r} (t^2 + \varepsilon^2)^{-r/2}  \le \check{C} (\tau)
	\end{equation}
	holds for all sufficiently small $|t|$ and $\varepsilon$.
	
	Applying the operator under the norm sign in  \eqref{**.12} to the element $\omega_j$, we conclude that
	\begin{equation*}
	\Big| \frac{\sin(\varepsilon^{-1} \tau \sqrt{\mathstrut \lambda_j(t)})}{\sqrt{\mathstrut \lambda_j(t)}} - \frac{\sin(\varepsilon^{-1} \tau |t| \sqrt{\mathstrut \gamma_j})}{|t| \sqrt{\mathstrut \gamma_j}} \Big| \varepsilon^{r}  (t^2 + \varepsilon^2)^{-r/2} \le \check{C} (\tau)
	\end{equation*}
	 for all sufficiently small $|t|$ and $\varepsilon$.
	 Substituting $t=t(\eps) = c \eps^{1/3}$ as in \eqref{**.9} and using \eqref{**.7}, we see that the expression
	 $$
	 \left| \bigl( 1+ O(\eps^{2/3})\bigr) \sin \bigl(\alpha_j \varepsilon^{-2/3} + O(\varepsilon^{1/3}) \bigr)   
	 + \sin\bigl(\alpha_j \varepsilon^{-2/3} \bigr) \right| \varepsilon^{(2 r -1)/3} (c^2 + \varepsilon^{4/3})^{-r/2}
	 $$  
	is uniformly bounded for small $\varepsilon > 0$. But this is not true if $r < 1/2$.	(It suffices to consider the sequence $\varepsilon_k = \alpha_j^{3/2} (2\pi k + \pi/2)^{-3/2}$, $k \in \mathbb{N}$.) This contradiction completes the proof of statement~$2^\circ$.
\end{proof}

Now, we show that the result of Theorem~\ref{abstr_sin_general_thrm} cannot be improved in the general situation. 

\begin{theorem}
	\label{abstr_s<2_general_thrm}
		Suppose that the operator $\Sigma(t,\tau)$ is defined by \eqref{3.13}.
	Suppose that $N_0 \ne 0$. Let $\tau \!\ne\! 0$ and \hbox{$0\! \le\! s \!<\! 2$}. Then there does not exist a constant ${C(\tau)\! > \!0}$ such that  the inequality 
	\begin{equation}
	\label{abstr_s<2_est_imp}
	\bigl \| A(t)^{1/2} \Sigma(t, \eps^{-1} \tau)  \bigr \|  \varepsilon^{s} (t^2 + \varepsilon^2)^{-s/2} \le  C(\tau) \varepsilon
	\end{equation}
holds for all sufficiently small $|t|$ and $\varepsilon > 0$.
\end{theorem}

\begin{proof}
	We prove by contradiction. Suppose that, for some $0 \ne \tau \in \mathbb{R}$ and \hbox{$1 \le s < 2$}, inequality~(\ref{abstr_s<2_est_imp}) holds for all sufficiently small $|t|$ and $\varepsilon$. 
	Taking~(\ref{abstr_A(t)_nondegenerated}) into account, we see that there exists a constant $\widetilde{C}(\tau) > 0$ such that 
	\begin{equation*}
	\bigl \|   \Sigma(t, \eps^{-1} \tau)\bigr \|
	  |t|\varepsilon^{s} (t^2 + \varepsilon^2)^{-s/2} \le \widetilde{C} (\tau) \varepsilon
	\end{equation*}
	for all sufficiently small $|t|$ and $\varepsilon$.
	Since $|t|\varepsilon^{s} (t^2 + \varepsilon^2)^{-s/2} \le \eps$ and the operators $A(t)^{-1/2}(P-F(t))$ and $tZ (t^2S)^{-1/2}$ are uniformly bounded  (by \eqref{1.2a}, \eqref{abstr_F(t)_threshold_1}, \eqref{abstr_A(t)_nondegenerated}, \eqref{abstr_S_nondegenerated}), then  for some constant
	$\widehat{C}(\tau)>0$ we have 
	 \begin{equation}
	\label{abstr_s<2_est_c}
	\bigl \|  A(t)^{-1/2} \sin(\varepsilon^{-1} \tau A(t)^{1/2})  F(t)  - (t^2 S)^{-1/2} \sin(\varepsilon^{-1} \tau (t^2 S)^{1/2}P) P \bigr\|
	  |t|\varepsilon^{s} (t^2 + \varepsilon^2)^{-s/2} 
	  \le \widehat{C} (\tau) \varepsilon
	\end{equation}
	for all sufficiently small $|t|$ and $\varepsilon$. 
	Next, from \eqref{abstr_SP_repr_gamma_omega}, \eqref{abstr_A(t)_nondegenerated}, \eqref{**.5}, \eqref{**.11}, and  (\ref{abstr_s<2_est_c}) 
	it follows that there exists a constant $\check{C} (\tau)$ such that 
	\begin{equation}
	\label{abstr_s<2_est_b}
	\Bigl\| \sum_{l=1}^{n} \left( \frac{\sin\bigl(\varepsilon^{-1} \tau \sqrt{\mathstrut \lambda_l(t)}\bigr)}{\sqrt{\mathstrut \lambda_l(t)}} - 
	\frac{\sin\bigl(\varepsilon^{-1} \tau |t| \sqrt{\mathstrut \gamma_l}\bigr)}{|t| \sqrt{\mathstrut \gamma_l}} \right) (\,\cdot\,,\omega_l)\omega_l  \Bigr \|
	\varepsilon^{s} |t| (t^2 + \varepsilon^2)^{-s/2} \le \check{C} (\tau) \varepsilon
	\end{equation}
	for all sufficiently small $|t|$ and $\eps$.
	
	According to \eqref{abstr_N_0_matrix_elem}, the condition $N_0 \ne 0$ means that $\mu_j \ne 0$ at least for one $j$. Then $\lambda_j(t) = \gamma_j t^2 + \mu_j t^3 + O(t^4)$ for $|t|\le t_*$. Decreasing $t_*$ if necessary, we have 
	\begin{equation}
	\label{**.7a}
	\sqrt{\lambda_j(t)} = \sqrt{\gamma_j} |t|\Bigl( 1 + \frac{\mu_j}{2\gamma_j} t + O(t^2)\Bigr), \quad |t|\le t_*.
	\end{equation}

	Applying the operator under the norm sign in~(\ref{abstr_s<2_est_b}) to the element $\omega_j$, we obtain 
	\begin{equation}
	\label{abstr_s<2_est_cc}
	\Big| \frac{\sin\bigl(\varepsilon^{-1} \tau \sqrt{\mathstrut \lambda_j(t)}\bigr)}{\sqrt{\mathstrut \lambda_j(t)}} - \frac{\sin\bigl(\varepsilon^{-1} \tau |t| \sqrt{\mathstrut \gamma_j}\bigr)}{|t| \sqrt{\mathstrut \gamma_j}} \Big| \varepsilon^{s} |t| (t^2 + \varepsilon^2)^{-s/2} \le \check{C} (\tau) \varepsilon
	\end{equation}
	for all sufficiently small $|t|$ and $\varepsilon$. 
	We put 
	\begin{equation*}
	t = \widetilde{t}(\varepsilon) = (2 \pi)^{1/2} \gamma_j^{1/4}  |{\mu_j \tau}|^{-1/2} \varepsilon^{1/2} = \widetilde{c} \varepsilon^{1/2}.
	\end{equation*}
	Then  
	$$
	\sin(\varepsilon^{-1} \tau \widetilde{t}(\varepsilon) \sqrt{\gamma_j}) = \sin(\widetilde{\alpha}_j \varepsilon^{-1/2}),
	$$
	where
	$\widetilde{\alpha}_j := (\operatorname{sgn} \tau) (2\pi)^{1/2} \gamma_j^{3/4} |\tau|^{1/2} |\mu_j|^{-1/2}$. Assuming that  $\varepsilon$ is sufficiently small and using~(\ref{**.7a}), we have 
	$\varepsilon^{-1} \tau \sqrt{\lambda_j(\widetilde{t} (\eps))} = \widetilde{\alpha}_j \varepsilon^{-1/2} + \pi \operatorname{sgn} (\tau \mu_j) + O(\varepsilon^{1/2})$, whence
$\sin\Bigl(\varepsilon^{-1} \tau \sqrt{\lambda_j(\widetilde{t} (\eps))} \Bigr) 
= - \sin\left(\widetilde{\alpha}_j \varepsilon^{-1/2} + O(\varepsilon^{1/2})\right).$
	Thus, from~\eqref{abstr_s<2_est_cc} it follows that the expression
	$$
	\Big| \bigl( 1+ O(\eps^{1/2})\bigr) \sin\bigl(\widetilde{\alpha}_j \varepsilon^{-1/2} + O(\varepsilon^{1/2}) \bigr) 
	+ \sin\bigl( \widetilde{\alpha}_j \varepsilon^{-1/2} \bigr) \Big| \varepsilon^{s/2-1} (\widetilde{c}^2 + \varepsilon)^{-s/2}
	$$ 
	is uniformly bounded for small $\varepsilon > 0$. But this is not true if  $s < 2$.	(It suffices to consider the sequence $\varepsilon_k = \widetilde{\alpha}_j^2 (\pi/2 + 2\pi k)^{-2}$, $k \in \mathbb{N}$.) This contradiction completes the proof.
\end{proof}

Finally, we confirm that Theorems  \ref{th3.5} and \ref{abstr_sin_enchanced_thrm_2} are sharp.

\begin{theorem}
	\label{th3.12}
		Suppose that the operator $\Sigma(t,\tau)$ is defined by \eqref{3.13}.
	Suppose that $N_0 = 0$ and $\mathcal{N}^{(q)} \ne 0$ for some $q \in \{1,\dots,p\}$. Let $\tau \ne 0$ and \hbox{$0 \le s < 3/2$}.  Then there does not exist a constant $C(\tau) > 0$ such that estimate 
\eqref{abstr_s<2_est_imp}	 holds for all sufficiently small $|t|$ and $\varepsilon > 0$.
\end{theorem}

\begin{proof}
 As in the proof of Theorem \ref{abstr_s<2_general_thrm}, 
  supposing the opposite, we see that inequality  \eqref{abstr_s<2_est_b} holds for some $\tau\ne 0$ and $1\le s<3/2$. Under our assumptions,  $\mu_l=0$, $l=1,\dots,n$, and $\nu_j \ne 0$ for some $j$.
 Then $\sqrt{\lambda_j(t)}$ satisfies \eqref{**.7}. 
 Applying the operator under the norm sign in \eqref{abstr_s<2_est_b} to the element $\omega_j$, we obtain inequality  \eqref{abstr_s<2_est_cc}. Next, substituting $t=t(\eps) = c \eps^{1/3}$ as in  
 \eqref{**.9}, we conclude that the expression 
  $$
	\big| \big( 1+ O(\eps^{2/3})\big) \sin\bigl({\alpha}_j \varepsilon^{-2/3} + O(\varepsilon^{1/3})\bigr) 
	+ \sin\bigl({\alpha}_j \varepsilon^{-2/3} \bigr) \big| \varepsilon^{2s/3-1} (c^2 + \varepsilon^{4/3})^{-s/2}
	$$ 
	is uniformly bounded for small $\varepsilon > 0$. But this is not true if $s < 3/2$.	
	This contradiction completes the proof.
\end{proof}

\subsection{Sharpness of the results with respect to time}
Now we prove the following statement  confirming that Theorem \ref{th3.1} is sharp regarding the 
dependence on  $\tau$ (for large $|\tau|$).

\begin{theorem}\label{th3.13}
Let $N_0 \ne 0$. 

	\noindent $1^\circ$. Let $s \ge 2$.
There does not exist a positive function $C(\tau)$ such that 
$\lim_{\tau \to \infty} C(\tau)/ |\tau| =0$ and estimate
 \eqref{**.1} holds for all $\tau \in \R$ and sufficiently small $|t|$ and $\eps$.  

\noindent $2^\circ$. Let $r \ge 1$. There does not exist a positive function
 $C(\tau)$ such that 
$\lim_{\tau \to \infty} C(\tau)/ |\tau| =0$ end estimate \eqref{**.2} holds for all
 $\tau \in \R$ and sufficiently small $|t|$ and $\eps$.
\end{theorem}

\begin{proof}
Let us check statement  $1^\circ$.
We prove by contradiction. Suppose that for some 
$s \ge 2$ there exists a positive function $C(\tau)$ such that 
$\lim_{\tau \to \infty} C(\tau)/ |\tau| =0$ and estimate \eqref{**.1} holds for  all
 $\tau \in \R$ and sufficiently small $|t|$ and $\eps$.  
	By \eqref{abstr_SP_repr_gamma_omega} and \eqref{**.3}--\eqref{**.5}, this is equivalent to the existence
	of a function $\widetilde{C}(\tau) > 0$ such that $\lim_{\tau \to \infty} \widetilde{C}(\tau)/ |\tau| =0$ and 
	the inequality 
\begin{equation}
	\label{3.31}
	\big\| \! \sum_{l=1}^n ( \cos (\varepsilon^{-\!1} \tau\! \sqrt{\lambda_l(t)} )  \! - \!\cos (\varepsilon^{-\!1} \tau |t|\! \sqrt{\gamma_l}) ) (\,\cdot\,, \omega_l) \omega_l \big\|
	  \varepsilon^{s} (t^2\!\! +\! \varepsilon^2)^{-\!s/2}\!\! \le\! \widetilde{C} (\tau) \varepsilon
	\end{equation}
	holds for all $\tau \in \R$ and sufficiently small $|t|$ and $\varepsilon$.
	
According to \eqref{abstr_N_0_matrix_elem}, the condition $N_0 \ne 0$ means that  $\mu_j \ne 0$ at least for one $j$. Then  \eqref{**.7a} is valid.
Applying the operator under the norm sign in \eqref{3.31} to the element $\omega_j$, we obtain 
\begin{equation}
	\label{3.32}
	\big|   \cos \big(\varepsilon^{-1} \tau \sqrt{\lambda_j(t)}\big)   - 
	\cos \big(\varepsilon^{-1} \tau |t| \sqrt{\gamma_j} \big) \big|
	  \varepsilon^{s} (t^2 + \varepsilon^2)^{-s/2} \le \widetilde{C} (\tau) \varepsilon
	\end{equation}
	for all $\tau \in \R$ and sufficiently small $|t|$ and $\varepsilon$.
	 Rewrite \eqref{3.32} in the form 
\begin{equation}
	\label{3.33}
	2\Big|   \sin \Big(  \frac{\tau}{2\eps} \Bigl(\sqrt{\lambda_j(t)}\! + \!|t| \sqrt{\gamma_j} \Bigr)    \Big)     
 \sin \Big(  \frac{\tau}{2\eps}  \Bigl( \sqrt{\lambda_j(t)}\! -\! |t| \sqrt{\gamma_j} \Bigr)  \Big)\Big|
	 \frac{ \varepsilon^{s}}{ (t^2\! +\! \varepsilon^2)^{s/2}} 
	\le \widetilde{C} (\tau) \varepsilon.
	\end{equation}
	
	Using \eqref{**.7a}, assume that $t_*$ is so small that 
	\begin{equation}
	\label{3.34} 
	\frac{1}{4} |\mu_j| \gamma_j^{-1/2} t^2 \leqslant \Bigl| \sqrt{\lambda_j(t)} - |t| \sqrt{\gamma_j} \Bigr| \le 
	\frac{3}{4} |\mu_j| \gamma_j^{-1/2} t^2, \quad |t| \le t_*.
	\end{equation}
Let $\tau \ne 0$, and suppose that $\eps \le \eps_* |\tau|$, $\eps_*= (4\pi)^{-1} \gamma_j^{-1/2} |\mu_j| t_*^2$.	
We put  
\begin{equation}
	\label{3.35} 
	t_\flat= t_{\flat}(\eps, \tau) = c_\flat  |\tau|^{-1/2} \eps^{1/2}, \quad c_\flat = \sqrt{\pi/2} \gamma_j^{1/4} |\mu_j|^{-1/2}.
	\end{equation}
Then $t_\flat \le t_*/2$ and, by  \eqref{3.34},
\begin{equation}
	\label{3.36} 
	\Big| \frac{\tau}{2\eps}  \Bigl( \sqrt{\lambda_j(t_\flat)} - t_\flat \sqrt{\gamma_j} \Bigr) \Big| \le \frac{3\pi}{16} < \frac{\pi}{4}.
	\end{equation}
We apply the estimate $|\sin y| \ge \frac{2}{\pi} |y|$ for $|y| \le \pi/2$.
Then, by \eqref{3.34}, 
\begin{equation}
	\label{3.37}
	\Big| \sin \Big(  \frac{\tau}{2\eps}  \Bigl( \sqrt{\lambda_j(t_\flat)} - t_\flat \sqrt{\gamma_j} \Bigr)  \Big)\Big|
	\ge \frac{|\tau|}{\pi \eps}  \Bigl| \sqrt{\lambda_j(t_\flat)} - t_\flat \sqrt{\gamma_j} \Bigr|
		\ge  \frac{|\tau|}{ 4 \pi \eps} |\mu_j| \gamma_j^{-1/2} t_\flat^2 = \frac{1}{8}. 
	\end{equation}
Now,  \eqref{3.33} and \eqref{3.37} imply that 
$$
\frac{1}{4}\Big|   \sin \Big(  \frac{\tau}{2\eps} \Bigl(\sqrt{\lambda_j(t_\flat)} + t_\flat \sqrt{\gamma_j} \Bigr)    \Big)     \Big|   
	  \varepsilon^{s} (t_\flat^2 + \varepsilon^2)^{-s/2} \le \widetilde{C} (\tau) \varepsilon,
$$
which is equivalent to the inequality 
\begin{equation}
	\label{3.38}
\frac{1}{4}\Big|   \sin \Big(  \frac{\tau}{2\eps} \Bigl(\sqrt{\lambda_j(t_\flat)} + t_\flat \sqrt{\gamma_j} \Bigr)  \Big)     \Big|   
	(\varepsilon |\tau|)^{s/2-1}   (c_\flat^2 + \varepsilon |\tau|)^{-s/2} \le \frac{\widetilde{C} (\tau)}{|\tau|}.
\end{equation}
By \eqref{3.36}, the argument of the  sine in  \eqref{3.38} differs from  $\eps^{-1} \tau t_\flat \sqrt{\gamma_j} =(\operatorname{sgn} \tau) \sqrt{\gamma_j} c_\flat |\tau|^{1/2} \eps^{-1/2}$ by no more than $\pi/4$. We put
$$
\eps_k = \gamma_j c_\flat^2 |\tau| (2\pi k + \pi/2)^{-2},
$$
assuming that $k\in \N$ is sufficiently large so that 
$\eps_k \le \eps_* |\tau|$.  Let $t_k = t_\flat(\eps_k,\tau)$.  Then $\eps_k^{-1} \tau t_k \sqrt{\gamma_j} = (\operatorname{sgn} \tau) (2\pi k +\pi/2)$, whence   
$$
\Big| \sin \Big(  \frac{\tau}{2\eps_k} \Bigl(\sqrt{\lambda_j(t_k)} + t_k \sqrt{\gamma_j} \Bigr)  \Big)   
 \Big|   \ge 1/\sqrt{2}.
$$
Now,  \eqref{3.38} with $\eps = \eps_k$ implies that
\begin{equation*}
\frac{1}{4\sqrt{2} c_\flat^2}
	\Big(\frac{\gamma_j  \tau^2}{(2\pi k + \pi/2)^{2}}\Big)^{s/2 -1}   \Big(1 + \frac{\gamma_j  \tau^2}{(2\pi k + \pi/2)^{2}} \Big)^{-s/2} \le \frac{\widetilde{C} (\tau)}{|\tau|}
\end{equation*}
for all sufficiently large $k$. According to our assumption, the right-hand side tends to zero as $\tau \to \infty$.
 Putting $\tau = \tau_k = 2\pi k + \pi/2$ and tending $k$ to infinity, we arrive at a contradiction. 

Statement $2^\circ$ is checked similarly. We prove by contradiction. Suppose that for some  $r \ge 1$ there exists a positive function  $C(\tau)$ such that 
$\lim_{\tau \to \infty} C(\tau)/ |\tau| =0$ and estimate \eqref{**.2} holds for all 
 $\tau \in \R$ and sufficiently small $|t|$ and $\eps$.
 Similarly to  the proof of inequality~\eqref{3.32}, this implies that  
\begin{equation}
	\label{3.40}
\Big|   \frac{\sin \bigl(\varepsilon^{-1} \tau \sqrt{\lambda_j(t)} \bigr)}{\sqrt{\lambda_j(t)}}   - 
	\frac{\sin \bigl(\varepsilon^{-1} \tau |t| \sqrt{\gamma_j} \bigr)}{|t| \sqrt{\gamma_j}}\Big|
	  \varepsilon^r (t^2 + \varepsilon^2)^{-r/2} \le \widetilde{C} (\tau),
	\end{equation}
and  $\lim_{\tau \to \infty} \widetilde{C} (\tau) / |\tau| =0$.
By \eqref{**.7a},  the quantity 
$$
|(\lambda_j(t))^{-1/2} - |t|^{-1} \gamma_j^{-1/2}|
$$
is uniformly bounded for $|t| \le t_*$. Therefore,  \eqref{3.40} implies that 
\begin{equation}
	\label{3.41}
\Big|   \sin \Big(\varepsilon^{-1} \tau \sqrt{\lambda_j(t)} \Big)    - 
	\sin \Bigl(\varepsilon^{-1} \tau |t| \sqrt{\gamma_j} \Bigr) \Big| |t|^{-1}
	  \varepsilon^r (t^2 + \varepsilon^2)^{-r/2} \le \widehat{C} (\tau),
	\end{equation}
and $\lim_{\tau \to \infty} \widehat{C} (\tau) / |\tau| =0$. Rewrite \eqref{3.41} in the form 
\begin{equation}
	\label{3.42}
2\Big|\cos \Big(\frac{\tau}{2\eps} \Bigl(\sqrt{\lambda_j(t)} + |t| \sqrt{\gamma_j} \Bigr)\Big)\sin \Big(  \frac{\tau}{2\eps}  \Bigl( \sqrt{\lambda_j(t)} - |t| \sqrt{\gamma_j} \Bigr)\Big)\Big|
	\frac{  \varepsilon^r} { |t|(t^2 + \varepsilon^2)^{r/2}} \le \widehat{C} (\tau).
	\end{equation}
As above, we assume that  \eqref{3.34} is satisfied and $\eps \le \eps_* |\tau|$. 
Let $t_\flat$ be given by  \eqref{3.35}. Then  \eqref{3.37} is valid.
 As a result,  \eqref{3.42} implies that
 $$
\frac{1}{4}\Big|\cos\Big(  \frac{\tau}{2\eps} \Bigl(\sqrt{\lambda_j(t_\flat)} + t_\flat \sqrt{\gamma_j} \Bigr)\Big)\Big|   
	  t_\flat^{-1}\varepsilon^r (t_\flat^2 + \varepsilon^2)^{-r/2} \le \widehat{C} (\tau),
$$
which is equivalent to the inequality
\begin{equation}
	\label{3.43}
\frac{1}{4}\Big|   \cos \Big(  \frac{\tau}{2\eps} \Bigl(\sqrt{\lambda_j(t_\flat)} + t_\flat \sqrt{\gamma_j} \Bigr)\Big)\Big|   
	  \frac{ (\varepsilon |\tau|)^{(r-1)/2}  }{c_\flat (c_\flat^2 + \varepsilon |\tau|)^{r/2}}  
	  \le \frac{\widehat{C} (\tau)}{|\tau|}.
\end{equation}
By \eqref{3.36}, the argument of cosine in \eqref{3.43} differs from $\eps^{-1} \tau t_\flat \sqrt{\gamma_j} =(\operatorname{sgn} \tau) \sqrt{\gamma_j} c_\flat |\tau|^{1/2} \eps^{-1/2}$ by no more than $\pi/4$.
We put $\widetilde{\eps}_k = \gamma_j c_\flat^2 |\tau| (2\pi k)^{-2}$, assuming that $k\in \N$ is
 sufficiently large so that 
$\widetilde{\eps}_k \le \eps_* |\tau|$.  Let $\widetilde{t}_k = t_\flat(\eps_k,\tau)$. Then  $\widetilde{\eps}_k^{-1} \tau \widetilde{t}_k \sqrt{\gamma_j} = (\operatorname{sgn} \tau) 2\pi k$. Therefore,  
$$
\Big| \cos\Big(  \frac{\tau}{2\widetilde{\eps}_k}\Big(\sqrt{\lambda_j(\widetilde{t}_k)} +\widetilde{t}_k \sqrt{\gamma_j}\Big)\Big)\Big|   \ge 1/\sqrt{2}.
$$
Now,  \eqref{3.43} with $\eps = \widetilde{\eps}_k$ yields the inequality 
\begin{equation*}
\frac{1}{4\sqrt{2} c_\flat^2}\Big(\frac{\gamma_j  \tau^2}{(2\pi k)^{2}} \Big)^{(r-1)/2}\Big(1 + \frac{\gamma_j  \tau^2}{(2\pi k)^{2}}\Big)^{-r/2} \le \frac{\widehat{C} (\tau)}{|\tau|}
\end{equation*}
for all sufficiently large $k$. According to our assumption, the right-hand side tends to zero as $\tau \to \infty$.
 Putting $\tau = \widetilde{\tau}_k = 2\pi k$ and tending $k$ to infinity, we arrive at a contradiction. 
\end{proof}

Now, we confirm the sharpness of  Theorem  \ref{abstr_sin_general_thrm}  regarding  the dependence on\,~$\!\tau\!$.

\begin{theorem}\label{th4.6}
Suppose that the operator $\Sigma(t,\tau)$ is defined by \eqref{3.13}.
Let $N_0 \ne 0$, and let $s \ge 2$. 
Then there does not exist a positive function $C(\tau)$ such that 
$\lim_{\tau \to \infty} C(\tau)/ |\tau| =0$ and estimate  \eqref{abstr_s<2_est_imp} holds for all
$\tau \in \R$ and sufficiently small $|t|$ and  $\eps$.
\end{theorem}

\begin{proof}
We prove by contradiction. Suppose that for some $s \ge 2$ there exists a positive function  
$C(\tau)$ such that 
$\lim_{\tau \to \infty} C(\tau)/ |\tau| =0$ and estimate  \eqref{abstr_s<2_est_imp} holds for all $\tau \in \R$ and sufficiently small $|t|$ and $\eps$. Together with \eqref{abstr_A(t)_nondegenerated} this implies  that
\begin{equation}
\label{3.46}
\bigl \|  \Sigma(t, \eps^{-1} \tau)\bigr  \| |t|\eps^s (t^2 + \eps^2)^{-s/2} \le \widetilde{C}(\tau) \eps,
\end{equation}
and  $\lim_{\tau \to \infty} \widetilde{C}(\tau)/ |\tau| =0$.
Since $|t|\eps^s (t^2 + \eps^2)^{-s/2} \le \eps$ and the norm
$\| tZ (t^2 S)^{-1/2} \sin \bigl( \eps^{-1} \tau (t^2 S)^{1/2} \bigr) P\|$ is uniformly bounded, 
from \eqref{3.46}, \eqref{3.02}, and \eqref{3.13} it follows that
\begin{equation}
\label{3.47}
\bigl  \|  {\mathcal J}_2(t,  \eps^{-1} \tau)\bigr  \| |t|\eps^s (t^2 + \eps^2)^{-s/2} \le \widehat{C}(\tau) \eps,
\end{equation}
and $\lim_{\tau \to \infty} \widehat{C}(\tau)/ |\tau| =0$.
The condition $N_0 \ne 0$ means that $\mu_j \ne 0$ for some $j$.
Similarly to   \eqref{3.40}, from \eqref{3.47} we obtain  
\begin{equation}
	\label{3.48}
\Big|   \frac{\sin \bigl(\varepsilon^{-1} \tau \sqrt{\lambda_j(t)} \bigr)}{\sqrt{\lambda_j(t)}}   - 
	\frac{\sin \left(\varepsilon^{-1} \tau |t| \sqrt{\gamma_j} \right)}{|t| \sqrt{\gamma_j}} \Big|
	  |t|\varepsilon^s (t^2 + \varepsilon^2)^{-s/2} \le \check{C} (\tau) \eps,
	\end{equation}
and $\lim_{\tau \to \infty} \check{C} (\tau) / |\tau| =0$.
By \eqref{**.7a}, the quantity   
$$
|(\lambda_j(t))^{-1/2} - |t|^{-1} \gamma_j^{-1/2}|
$$
is uniformly bounded for $|t| \le t_*$. Therefore,  \eqref{3.48} implies that 
\begin{equation*}
\Big|   \sin \Big(\varepsilon^{-1} \tau \sqrt{\lambda_j(t)} \Big)    - 
	\sin \Bigl(\varepsilon^{-1} \tau |t| \sqrt{\gamma_j} \Bigr)\Big| 
	  \varepsilon^s (t^2 + \varepsilon^2)^{-s/2} \le \check{C}' (\tau) \eps,
\end{equation*}
and $\lim_{\tau \to \infty} \check{C}' (\tau) / |\tau| =0$. 
Similarly to  the proof of Theorem~\ref{th3.13}, substituting $t=t_\flat$ (see \eqref{3.35}), we deduce that  
\begin{equation}
	\label{3.51}
\frac{1}{4}\Big|   \cos\Big(  \frac{\tau}{2\eps} \Bigl(\sqrt{\lambda_j(t_\flat)} + t_\flat \sqrt{\gamma_j} \Bigr)\Big)\Big|    \frac{   (\varepsilon |\tau|)^{s/2-1} }{ (c_\flat^2 + \varepsilon |\tau|)^{s/2}} \le \frac{\check{C}' (\tau)}{|\tau|}.
	\end{equation}
Now,  \eqref{3.51} with $\eps = \widetilde{\eps}_k= \gamma_j c_\flat^2 |\tau| (2\pi k)^{-2}$ yields  the inequality
\begin{equation*}
	\frac{1}{4\sqrt{2} c_\flat^2}\Big(\frac{\gamma_j  \tau^2}{(2\pi k)^{2}}\Big)^{s/2-1}\Big(1 + \frac{\gamma_j  \tau^2}{(2\pi k)^{2}}\Big)^{-s/2} \le \frac{\check{C}' (\tau)}{|\tau|}
\end{equation*}
for all sufficiently large $k$. Here the right-hand side tends to zero as  $\tau \to \infty$.
 Putting $\tau = \widetilde{\tau}_k = 2\pi k$ and tending $k$ to infinity, we arrive at a contradiction. 
\end{proof}

Next, we confirm that Theorems \ref{th3.2} and \ref{th3.3} are sharp  regarding the dependence on~$\tau$.

\begin{theorem}
	\label{th4.*2}
	Suppose that  $N_0 = 0$ and $\mathcal{N}^{(q)} \ne 0$ for some $q\in \{1,\dots,p\}$.
	 
	\noindent $1^\circ$.  Let $s \ge 3/2$. There does not exist a positive function $C(\tau)$ such that 
	$\lim_{\tau\to \infty} C(\tau)/ |\tau|^{1/2} = 0$ and estimate  \eqref{**.1} holds for all
	 $\tau \in \R$ and sufficiently small $|t|$ and $\varepsilon$.
	
	\noindent $2^\circ$. Let $r\ge 1/2$.
	There does not exist a positive function $C(\tau)$ such that 
	$\lim_{\tau \to \infty} C(\tau)/ |\tau|^{1/2} = 0$ 
	and estimate \eqref{**.2} holds for all $\tau \in \R$ and sufficiently small $|t|$ and $\varepsilon$.
	\end{theorem}

\begin{proof}
The conditions $N_0 = 0$ and $\mathcal{N}^{(q)} \ne 0$ for some $q\in \{1,\dots,p\}$ 
mean that  $\mu_l=0$ for $l=1,\dots,n$, and  $\nu_j \ne 0$ at least for one $j$. 
Then expansion  \eqref{**.7} is valid.

Let us check statement $1^\circ$. We prove by contradiction. 
Similarly to the proof of Theorem \ref{th3.9}, we suppose the opposite and obtain
\begin{equation}
	\label{*4.3}
	\Big|   \cos \Bigl(\varepsilon^{-1} \tau \sqrt{\lambda_j(t)} \Bigr)   - 
	\cos \Bigl(\varepsilon^{-1} \tau |t| \sqrt{\gamma_j} \Bigr) \Big|
	  \varepsilon^{s} (t^2 + \varepsilon^2)^{-s/2} \le \widetilde{C} (\tau) \varepsilon
	\end{equation}
	for some  $s \ge 3/2$, and $\lim_{\tau\to \infty} \widetilde{C}(\tau)/ |\tau|^{1/2} = 0$.
Rewrite  \eqref{*4.3} as follows:
\begin{equation}
	\label{*4.4}
	2\Big|   \sin\Big(\frac{\tau}{2\varepsilon} \Big(\sqrt{\lambda_j(t)} + |t| \sqrt{\gamma_j} \Big)\Big)
	   \sin\Big(\frac{\tau}{2\varepsilon} \Big(\sqrt{\lambda_j(t)} - |t| \sqrt{\gamma_j} \Big)\Big)\Big|
	 \frac{ \varepsilon^{s}} {(t^2 + \varepsilon^2)^{s/2}} \le \widetilde{C} (\tau) \varepsilon.
	\end{equation}
	
	Using \eqref{**.7}, we assume that  $t_*$ is so small that 
\begin{equation}
	\label{*4.5} 
	\frac{1}{4} |\nu_j| \gamma_j^{-1/2} |t|^3 \leqslant \Bigl| \sqrt{\lambda_j(t)} - |t| \sqrt{\gamma_j} \Bigr| \le \frac{3}{4} |\nu_j| \gamma_j^{-1/2} |t|^3, \quad |t| \le t_*.
	\end{equation}
Let $\tau \ne 0$, and let $\eps \le \eps_\dag |\tau|$, $\eps_\dag = (4\pi)^{-1} \gamma_j^{-1/2} |\nu_j| t_*^3$.	
We put  
\begin{equation}
	\label{*4.6} 
	t_\dag = t_{\dag}(\eps, \tau) = c_\dag  |\tau|^{-1/3} \eps^{1/3}, \quad c_\dag = (\pi/2)^{1/3} 
	\gamma_j^{1/6} |\nu_j|^{-1/3}.
	\end{equation}
Then $t_\dag \le t_*/2$, and, by \eqref{*4.5},
\begin{equation}
	\label{*4.7} 
	\Big| \frac{\tau}{2\eps}  \Big( \sqrt{\lambda_j(t_\dag)} - t_\dag \sqrt{\gamma_j} \Big)\Big| \le \frac{3\pi}{16} < \frac{\pi}{4}.
	\end{equation}
We apply the estimate $|\sin y| \ge \frac{2}{\pi} |y|$ for $|y| \le \pi/2$.
Then, by \eqref{*4.5}, 
\begin{equation}
	\label{*4.8}
	\Big| \sin \Big(  \frac{\tau}{2\eps}  \Bigl( \sqrt{\lambda_j(t_\dag)} - t_\dag \sqrt{\gamma_j} \Bigr)\Big)\Big|
	\ge \frac{|\tau|}{\pi \eps}  \Big| \sqrt{\lambda_j(t_\dag)} - t_\dag \sqrt{\gamma_j}\Big|
	 \ge  \frac{|\tau|}{ 4 \pi \eps} |\nu_j| \gamma_j^{-1/2} t_\dag^3 = \frac{1}{8}. 
	\end{equation}
Now, \eqref{*4.4} and \eqref{*4.8} imply that 
$$
\frac{1}{4}\Big|\sin\Big(  \frac{\tau}{2\eps} \Big(\sqrt{\lambda_j(t_\dag)} + t_\dag \sqrt{\gamma_j} \Big)\Big) \Big|   
	  \varepsilon^{s} (t_\dag^2 + \varepsilon^2)^{-s/2} \le \widetilde{C} (\tau) \varepsilon,
$$
which is equivalent to the inequality
\begin{equation}
	\label{*4.9}
\frac{1}{4}\Big|\sin\Big(\frac{\tau}{2\eps} \Big(\sqrt{\lambda_j(t_\dag)} + t_\dag \sqrt{\gamma_j} \Big)\Big)\Big|\frac{(\varepsilon |\tau|^{1/2})^{2s/3 -1}}{ (c_\dag^2 + \varepsilon^{4/3} |\tau|^{2/3})^{s/2}} \le \frac{\widetilde{C} (\tau)}{|\tau|^{1/2}}.
\end{equation}
By \eqref{*4.7}, the argument of sine in \eqref{*4.9} differs from  
$$
\eps^{-1} \tau t_\dag \sqrt{\gamma_j} =(\operatorname{sgn} \tau) \sqrt{\gamma_j} c_\dag |\tau|^{2/3} \eps^{-2/3}
$$
 by no more than $\pi/4$.
We put $\widehat{\eps}_k = {\gamma_j^{3/4} c_\dag^{3/2} |\tau|}{ (2\pi k + \pi/2)^{-3/2}}$, assuming that  $k\in \N$ is sufficiently large so that $\widehat{\eps}_k \le \eps_\dag |\tau|$.  
Let $\widehat{t}_k = t_\dag(\widehat{\eps}_k,\tau)$.
Then $\widehat{\eps}_k^{-1} \tau \widehat{t}_k \sqrt{\gamma_j} = (\operatorname{sgn} \tau) (2\pi k +\pi/2)$,
whence  
$$
\Big| \sin\Big(  \frac{\tau}{2 \widehat{\eps}_k} \Big(\sqrt{\lambda_j(\widehat{t}_k)} + \widehat{t}_k \sqrt{\gamma_j} 
\Big)\Big)\Big|   \ge 1/\sqrt{2}.
$$
Now,  from \eqref{*4.9} with  $\eps = \widehat{\eps}_k$ it follows that
\begin{equation*}
\frac{1}{4\sqrt{2} c_\dag^{3/2}}\Big( \frac{\gamma_j  \tau^2}{(2\pi k + \pi/2)^{2}}\Big)^{s/2 -3/4}\Big(1 + \frac{\gamma_j  \tau^2}{(2\pi k + \pi/2)^{2}}\Big)^{-s/2} \le \frac{\widetilde{C} (\tau)}{|\tau|^{1/2}}
\end{equation*}
for all sufficiently large $k$. According to our assumption, the right-hand side tends to zero as $\tau \to \infty$.
 Putting $\tau = {\tau}_k = 2\pi k + \pi/2$ and tending $k$ to infinity, we arrive at a contradiction.

Statement $2^\circ$ is checked similarly. 
Suppose the opposite. Then for some $r \ge 1/2$ we obtain the inequality 
\begin{equation}
	\label{*4.11}
\bigg|  \frac{ \sin \left(\varepsilon^{-1} \tau \sqrt{\lambda_j(t)} \right)}{\sqrt{\lambda_j(t)}}   - 
	\frac{\sin \left(\varepsilon^{-1} \tau |t| \sqrt{\gamma_j} \right)}{|t| \sqrt{\gamma_j}}\bigg|
	  \varepsilon^{r} (t^2 + \varepsilon^2)^{-r/2} \le \widetilde{C} (\tau),
	\end{equation}
	and $\lim_{\tau \to \infty} \widetilde{C}(\tau)/ |\tau|^{1/2} = 0$.
	By \eqref{**.7}, the quantity 
	$$
	\bigl| \lambda_j(t)^{-1/2} - |t|^{-1} \gamma_j^{-1/2}\bigr|
	$$
	is uniformly bounded for  $|t| \le t_*$. Therefore,  \eqref{*4.11} implies that 
\begin{equation*}
	\Big|  \sin \Bigl(\varepsilon^{-1} \tau \sqrt{\lambda_j(t)} \Bigr)     -\sin\Big(\varepsilon^{-1} \tau |t| \sqrt{\gamma_j}\Big)\Big| |t|^{-1}\varepsilon^{r} (t^2 + \varepsilon^2)^{-r/2} \le \widehat{C} (\tau),
	\end{equation*}
	and $\lim_{\tau \to \infty} \widehat{C}(\tau)/ |\tau|^{1/2} = 0$.
	Similarly to the proof of statement~$1^\circ$, assuming that 	
	 $\eps \le \eps_\dag |\tau|$ and substituting $t=t_\dag$ (see \eqref{*4.6}), we arrive at 
\begin{equation*}
	\frac{1}{4}\Big|\cos\Big(\frac{\tau}{2\varepsilon} \Big(\sqrt{\lambda_j(t_\dag)} + t_\dag \sqrt{\gamma_j} \Big)\Big)\Big| 
	t_\dag^{-1}  \varepsilon^{r} (t_\dag^2 + \varepsilon^2)^{-r/2} \le \widehat{C} (\tau),
	\end{equation*}
which is equivalent to 
\begin{equation}
	\label{*4.15}
\frac{1}{4c_\dag}
\Big|   \cos \Big(\frac{\tau}{2\varepsilon} \Bigl(\sqrt{\lambda_j(t_\dag)} + t_\dag \sqrt{\gamma_j} \Bigr) \Big) \Big| 
  (\varepsilon |\tau|^{1/2})^{(2r-1)/3}
	 (c_\dag^2 + \varepsilon^{4/3} |\tau|^{2/3})^{-r/2} \le 
	\frac{\widehat{C} (\tau)}{|\tau|^{1/2}}.
	\end{equation}
By \eqref{*4.7}, the argument of cosine in \eqref{*4.15} differs from  
$$
\eps^{-1} \tau t_\dag \sqrt{\gamma_j} =(\operatorname{sgn} \tau) \sqrt{\gamma_j} c_\dag |\tau|^{2/3} \eps^{-2/3}
$$
 by no more than $\pi/4$. We put  
$$
\check{\eps}_k = \gamma_j^{3/4} c_\dag^{3/2} |\tau| (2\pi k)^{-3/2},
$$
 assuming that $k\in \N$ is sufficiently large. Let $\check{t}_k = t_\dag (\check{\eps}_k, \tau)$.
Then $\check{\eps}_k^{-1} \tau \check{t}_k \sqrt{\gamma_j} = (\operatorname{sgn} \tau) 2\pi k$,
whence  
$$
\Bigl| \cos \Bigl(  \frac{\tau}{2 \check{\eps}_k} \Bigl(\sqrt{\lambda_j(\check{t}_k)} + \check{t}_k \sqrt{\gamma_j} \Bigr)  \Bigr)     \Bigr|   \ge 1/\sqrt{2}.$$
Now, from \eqref{*4.15} with $\eps = \check{\eps}_k$ it follows that 
\begin{equation*}
	\frac{1}{4\sqrt{2} c_\dag^{3/2}}  \Big(\frac{\gamma_j  \tau^2}{(2\pi k)^{2}} \Big)^{r/2 - 1/4}
	  \Big(1 + \frac{\gamma_j  \tau^2}{(2\pi k)^{2}} \Big)^{-r/2} \le \frac{\widehat{C} (\tau)}{|\tau|^{1/2}}
\end{equation*}
 for all sufficiently large $k$. According to our assumption, the right-hand side tends to zero as $\tau \to \infty$.
 Putting $\tau = \widetilde{\tau}_k = 2\pi k$ and tending  $k$ to infinity, we arrive at a contradiction. 
\end{proof}

Finally, we confirm the sharpness of Theorems \ref{th3.5} and \ref{abstr_sin_enchanced_thrm_2} regarding  the dependence on  $\tau$.

\begin{theorem}
	\label{th4.final}
		Suppose that the operator $\Sigma(t,\tau)$ is defined by \eqref{3.13}.
	Let \hbox{$N_0 = 0$} and  $\mathcal{N}^{(q)} \ne 0$ for some $q \in \{1,\dots,p\}$. 
	Let $s \ge 3/2$. There does not exist a positive function $C(\tau)$ such that 
	$\lim_{\tau \to \infty} C(\tau)/ |\tau|^{1/2} = 0$ and estimate
	\eqref{abstr_s<2_est_imp} holds for all  $\tau \in \R$ and sufficiently small $|t|$ and $\varepsilon > 0$.
	\end{theorem}

\begin{proof}
 Under our assumptions, $\mu_l=0$ for all $l=1,\dots,n,$ and  $\nu_j \ne 0$ at least for one $j$. 
Then expansion  \eqref{**.7} is satisfied.

Suppose the opposite. Then, similarly to  \eqref{3.46}--\eqref{3.48}, we see that for some 
 $s \ge 3/2$  the inequality
\begin{equation}
	\label{*4.18}
\Big|   \frac{\sin \left(\varepsilon^{-1} \tau \sqrt{\lambda_j(t)} \right)}{\sqrt{\lambda_j(t)}}    - 
	\frac{\sin \left(\varepsilon^{-1} \tau |t| \sqrt{\gamma_j} \right)}{|t| \sqrt{\gamma_j}} \Big| |t|
	  \varepsilon^{s} (t^2 + \varepsilon^2)^{-s/2} \le \widetilde{C} (\tau) \eps
	\end{equation}
holds and $\lim_{\tau \to \infty} \widetilde{C}(\tau)/ |\tau|^{1/2} = 0$.
	By \eqref{**.7}, the quantity  	
	$$
	\bigl| \lambda_j(t)^{-1/2} - |t|^{-1} \gamma_j^{-1/2}\bigr|
	$$
	is uniformly bounded for  $|t| \le t_*$, whence  \eqref{*4.18} implies that 
\begin{equation}
	\label{*4.19}
	\Big|  \sin \Bigl(\varepsilon^{-1} \tau \sqrt{\lambda_j(t)} \Bigr)     - 
	\sin \left(\varepsilon^{-1} \tau |t| \sqrt{\gamma_j} \right)  \Big| 
	  \varepsilon^{s} (t^2 + \varepsilon^2)^{-s/2} \le \widehat{C} (\tau) \eps,
	\end{equation}
	and $\lim_{\tau \to \infty} \widehat{C}(\tau)/ |\tau|^{1/2} = 0$.
	Similarly to  \eqref{*4.4}--\eqref{*4.9}, from  \eqref{*4.19} we deduce that
\begin{equation*}
\frac{1}{4}\Big|\cos\Big(\frac{\tau}{2\eps} \Big(\sqrt{\lambda_j(t_\dag)} + t_\dag \sqrt{\gamma_j} \Big)\Big)\Big|   
	(\varepsilon |\tau|^{1/2})^{2s/3 -1}  (c_\dag^2 + \varepsilon^{4/3} |\tau|^{2/3})^{-s/2} \le \frac{\widehat{C} (\tau)}{|\tau|^{1/2}}.
\end{equation*}
For $\eps=  \check{\eps}_k = \gamma_j^{3/4} c_\dag^{3/2} |\tau| (2\pi k)^{-3/2}$
this yields the inequality 
\begin{equation*}
	\frac{1}{4\sqrt{2} c_\dag^{3/2}}\Big( \frac{\gamma_j  \tau^2}{(2\pi k)^{2}}\Big)^{s/2 - 3/4}\Big(1 + \frac{\gamma_j  \tau^2}{(2\pi k)^{2}} \Big)^{-s/2} \le \frac{\widehat{C} (\tau)}{|\tau|^{1/2}}
\end{equation*}
for all sufficiently large $k$. By our assumption, the right-hand side tends to zero as $\tau \to \infty$.
 Putting $\tau = \widetilde{\tau}_k = 2\pi k$ and tending $k$ to infinity, we arrive at a contradiction. 
\end{proof}


\section{Operator of the form $A(t) = M^* {\widehat{A}} (t) M$. Approximation of the sandwiched operators
 $\cos (\tau A(t)^{1/2})$ and  $A(t)^{-1/2} \sin(\tau A(t)^{1/2})$}
\label{abstr_sandwiched_section}

\subsection{The operator family of the form $A(t) = M^* \widehat{A} (t) M$}
\label{abstr_A_and_Ahat_section}
Along with the space $\mathfrak{H}$, we consider yet another separable Hilbert space  $\widehat{\mathfrak{H}}$. Let $\widehat{X} (t) = \widehat{X}_0 + t \widehat{X}_1 \colon \widehat{\mathfrak{H}} \to \mathfrak{H}_* $~be the family of operators of the same form as  $X(t)$. Suppose that $\widehat{X} (t)$ satisfies the assumptions of Subsection~\ref{abstr_X_A_section}. Let $M \colon \mathfrak{H} \to \widehat{\mathfrak{H}}$~be an isomorphism. Assume that $M \Dom X_0 = \Dom \widehat{X}_0$,  $X(t) = \widehat{X} (t) M$,
and then also $X_0 = \widehat{X}_0 M$, $X_1 = \widehat{X}_1 M$. In $\widehat{\mathfrak{H}}$,  we introduce the family of selfadjoint operators $\widehat{A} (t) = \widehat{X} (t)^* \widehat{X} (t)$. Then, obviously,
\begin{equation}
\label{abstr_A_and_Ahat}
A(t) = M^* \widehat{A} (t) M.
\end{equation} 
 In what follows, all the objects corresponding to the family  $\widehat{A}(t)$ are marked by \textquotedblleft$\, \widehat{\phantom{\_}} \,$\textquotedblright. Note that $\widehat{\mathfrak{N}} = M \mathfrak{N}$
and $\widehat{\mathfrak{N}}_* =  \mathfrak{N}_*$.
 In the space $\widehat{\mathfrak{H}}$, we consider the positive definite operator $Q := (M M^*)^{-1}$.
Let $Q_{\widehat{\mathfrak{N}}}$~be the block of the operator $Q$ in  
$\widehat{\mathfrak{N}}$, i.~e., $Q_{\widehat{\mathfrak{N}}} = \widehat{P} Q|_{\widehat{\mathfrak{N}}}$.
Obviously, $Q_{\widehat{\mathfrak{N}}}$~is an isomorphism in $\widehat{\mathfrak{N}}$.

As was shown in \cite[Proposition~1.2]{Su2}, the orthogonal projection $P$ of $\mathfrak{H}$ onto $\mathfrak{N}$ 
and the orthogonal projection $\widehat{P}$ of $\widehat{\mathfrak{H}}$ onto $\widehat{\mathfrak{N}}$ 
satisfy the following relation:
\begin{equation}
\label{abstr_P_and_P_hat_relation}
P = M^{-1} (Q_{\widehat{\mathfrak{N}}})^{-1} \widehat{P} (M^*)^{-1}.  
\end{equation}
Let $\widehat{S} \colon \widehat{\mathfrak{N}} \to \widehat{\mathfrak{N}}$~be the spectral germ of the family $\widehat{A} (t)$ at $t = 0$,
and let  $S$~be the germ of the family $A (t)$. In \cite[Chapter~1, Subsection~1.5] {BSu1}, it was proved that
\begin{equation}
\label{abstr_S_and_S_hat_relation}
S = P M^* \widehat{S} M |_\mathfrak{N}.
\end{equation}

Assume that $A(t)$ satisfies Condition \ref{cond_A}.
Then the germ $S$ (as well as $\wh{S}$) is nondegenerate.

\subsection{The operators $\widehat{Z}_Q$ and  $\widehat{N}_Q$}
\label{abstr_hatZ_Q_and_hatN_Q_section}
We introduce the operator $\widehat{Z}_Q$ acting in  $\widehat{\mathfrak{H}}$ and taking an element  $\widehat{u} \in \widehat{\mathfrak{H}}$ into the weak solution  $\widehat{\phi}_Q \in \operatorname{Dom} \wh{X}_0$ of 
the problem 
$\widehat{X}^*_0 (\widehat{X}_0 \widehat{\phi}_Q + \widehat{X}_1 \widehat{\omega}) = 0$, 
$Q \widehat{\phi}_Q \perp \widehat{\mathfrak{N}}$,
where $\widehat{\omega} = \widehat{P} \widehat{u}$. As was shown in~\cite[\S6]{BSu2}, the operator $Z$ for the family  $A(t)$ and the operator $\widehat{Z}_Q$ satisfy the following relation: 
\begin{equation}
\label{abstr_Z_and_hatZ_Q_relat}
\widehat{Z}_Q =M Z M^{-1} \widehat{P}.
\end{equation}
Next, we put  
\begin{equation}
\label{abstr_hatN_Q}
\widehat{N}_Q := \widehat{Z}_Q^* \widehat{X}_1^* \widehat{R} \widehat{P} + (\widehat{R} \widehat{P})^* \widehat{X}_1  \widehat{Z}_Q.
\end{equation}
According to~\cite[\S6]{BSu2}, the operator $N$ for the family $A(t)$ and the operator~(\ref{abstr_hatN_Q}) introduced above satisfy the following relation:
\begin{equation}
\label{abstr_N_and_hatN_Q_relat}
\widehat{N}_Q = \widehat{P} (M^*)^{-1} N M^{-1} \widehat{P}.
\end{equation}
Recall that $N = N_0 + N_*$, and define the operators
\begin{equation}
\label{abstr_N0*_and_hatN0*_Q_relat}
\widehat{N}_{0,Q} = \widehat{P} (M^*)^{-1} N_0 M^{-1} \widehat{P}, \quad \widehat{N}_{*,Q} = \widehat{P} (M^*)^{-1} N_* M^{-1} \widehat{P}.
\end{equation}
Then $\widehat{N}_Q = \widehat{N}_{0,Q} + \widehat{N}_{*,Q}$.
The following lemma was proved in~\cite[Lemma~5.1]{Su4}.

\begin{lemma}[see~\cite{Su4}]
	\label{abstr_N_and_Nhat_lemma}
	Suppose that the assumptions of Subsection~\emph{\ref{abstr_A_and_Ahat_section}} are satisfied. Suppose that the operators $N$ and $N_0$ are defined by~\eqref{abstr_N} and \eqref{abstr_N_invar_repers} and the operators $\widehat{N}_Q$ and $\widehat{N}_{0,Q}$ are defined by~\eqref{abstr_hatN_Q} and \eqref{abstr_N0*_and_hatN0*_Q_relat}. Then the condition  $N = 0$ is equivalent to 
	the relation $\widehat{N}_Q = 0$. The condition $N_0 = 0$ is equivalent to the relation $\widehat{N}_{0,Q} = 0$.
\end{lemma}

\subsection{The operators $\widehat{Z}_{2,Q}$,  $\widehat{R}_{2,Q}$, and $\widehat{N}^0_{1,Q}$\label{sec5.3}}

Let $\widehat{\omega} \in \widehat{\mathfrak{N}}$ and let   
$\widehat{\psi}_Q = \widehat{\psi}_Q(\widehat{\omega}) \in \Dom \widehat{X}_0$ be a (weak) solution of the problem
\begin{equation*}
\widehat{X}^*_0 (\widehat{X}_0 \widehat{\psi}_Q + \widehat{X}_1 \widehat{Z}_Q \widehat{\omega}) = 
- \widehat{X}^*_1 \widehat{R} \widehat{\omega} + 
Q Q_{\widehat{\mathfrak{N}}}^{-1} \widehat{P} \widehat{X}^*_1 \widehat{R} \widehat{\omega}, \quad Q \widehat{\psi}_Q \perp \widehat{\mathfrak{N}}.
\end{equation*}
Clearly, the right-hand side of this equation belongs to  $\widehat{\mathfrak{N}}^\perp = \operatorname{Ran} 
\widehat{X}_0^*$, and so the solvability condition is satisfied. We define the operator  
$\widehat{Z}_{2,Q} : \widehat{\H} \to \widehat{\H}$ by the relation 
$\widehat{Z}_{2,Q} \widehat{u} = \widehat{\psi}_Q(\widehat{P} \widehat{u})$,
$\widehat{u} \in \widehat{\mathfrak H}$.
 Next,  define the operator $\widehat{R}_{2,Q}: \widehat{\mathfrak N} \to \H_*$ by the relation
$\widehat{R}_{2,Q} := \widehat{X}_0 \widehat{Z}_{2,Q} + \widehat{X}_1 \widehat{Z}_{Q}$.
We put
\begin{equation}
\label{abstr_hatN0_1Q}
\widehat{N}^0_{1,Q} = \widehat{Z}^*_{2,Q}  \widehat{X}_1^*    \widehat{R}  \widehat{P}
+ (\widehat{R}  \widehat{P})^*  \widehat{X}_1 \widehat{Z}_{2,Q} + \widehat{R}_{2,Q}^* \widehat{R}_{2,Q} \wh{P}.
\end{equation}

In~\cite[Subsection~6.3]{VSu1}, it was proved that
\begin{equation*}
\begin{aligned}
\widehat{Z}_{2,Q} = M Z_2 M^{-1} \widehat{P}, \quad \widehat{R}_{2,Q} = R_2 M^{-1}\vert_{\widehat{\mathfrak N}},
\quad
\widehat{N}^0_{1,Q} = \widehat{P} (M^*)^{-1} N_1^0 M^{-1} \widehat{P}.
\end{aligned}
\end{equation*}

\subsection{Relationship between the operators and the coefficients of the power series expansions}
Now, we describe relationship between the coefficients of the power series expansions~(\ref{abstr_A(t)_eigenvalues_series}), (\ref{abstr_A(t)_eigenvectors_series}) and the operators $\widehat{S}$ and $Q_{\widehat{\mathfrak{N}}}$. (See~\cite[Subsections~1.6,~1.7]{BSu3}.) We put $\zeta_l := M \omega_l \in \widehat{\mathfrak{N}}, \, l = 1, \ldots, n$. Then from~(\ref{abstr_S_eigenvectors}) and~(\ref{abstr_P_and_P_hat_relation}), (\ref{abstr_S_and_S_hat_relation}) it follows that 
\begin{equation}
\label{abstr_hatS_gener_spec_problem}
\widehat{S} \zeta_l  = \gamma_l Q_{\widehat{\mathfrak{N}}} \zeta_l, \quad l = 1, \ldots, n. 
\end{equation}
The set $\zeta_1, \ldots, \zeta_n$ forms a basis in $\widehat{\mathfrak{N}}$ orthonormal with the weight $Q_{\widehat{\mathfrak{N}}}$:
\begin{equation}
\label{abstr_sndwchd_zeta_basis}
(Q_{\widehat{\mathfrak{N}}} \zeta_l, \zeta_j) = \delta_{lj}, \qquad l,j = 1,\ldots,n.
\end{equation}

The operators $\widehat{N}_{0,Q}$ and  $\widehat{N}_{*,Q}$ can be described in terms of the coefficients of the power series expansions~(\ref{abstr_A(t)_eigenvalues_series}) and~(\ref{abstr_A(t)_eigenvectors_series}); cf.~(\ref{abstr_N_0_N_*}). We put $\widetilde{\zeta}_l :=  M \widetilde{\omega}_l \in \widehat{\mathfrak{N}}, \; l = 1, \ldots,n $. Then
\begin{equation}
\label{abstr_hatN_0Q_N_*Q}
\begin{split}
\widehat{N}_{0,Q}& = \sum_{k=1}^{n} \mu_k (\,\cdot\,, Q_{\widehat{\mathfrak{N}}} \zeta_k) Q_{\widehat{\mathfrak{N}}} \zeta_k, \\
\widehat{N}_{*,Q} &= \sum_{k=1}^{n} \gamma_k \left( (\,\cdot\,, Q_{\widehat{\mathfrak{N}}} \widetilde{\zeta}_k) Q_{\widehat{\mathfrak{N}}} \zeta_k + (\,\cdot\,, Q_{\widehat{\mathfrak{N}}} \zeta_k) Q_{\widehat{\mathfrak{N}}} \widetilde{\zeta}_k \right). 
\end{split}
\end{equation}

\begin{remark}
	By~\eqref{abstr_sndwchd_zeta_basis} and~\eqref{abstr_hatN_0Q_N_*Q}, we have
	\begin{align*}
	(\widehat{N}_{0,Q} \zeta_j, \zeta_l)& = \mu_l \delta_{jl}, && j,l=1,\ldots,n,\\
	(\widehat{N}_{*,Q} \zeta_j, \zeta_l) &= \gamma_l (\zeta_j, Q_{\widehat{\mathfrak{N}}} \widetilde{\zeta}_l) + \gamma_j (Q_{\widehat{\mathfrak{N}}} \widetilde{\zeta}_j, \zeta_l), && j,l=1,\ldots,n,    
	\end{align*}
Relations~\eqref{abstr_omega_tilde_omega_rel} imply that 
$(Q_{\widehat{\mathfrak{N}}} \widetilde{\zeta}_j, \zeta_l) + (\zeta_j, Q_{\widehat{\mathfrak{N}}} \widetilde{\zeta}_l) = 0$, $j, l = 1, \ldots,n$.
It follows that $(\widehat{N}_{*,Q} \zeta_j, \zeta_l) = 0$ if $\gamma_j = \gamma_l$.
\end{remark}

Now, we return to the notation of~Subsection~\ref{abstr_cluster_section}. Recall that the different eigenvalues 
of the germ $S$ are denoted by  $\gamma^{\circ}_j$, $j = 1,\ldots,p$, and the corresponding eigenspaces by $\mathfrak{N}_j$. The vectors $\omega^{(j)}_i$, $i = 1,\ldots, k_j,$ form an orthonormal basis in $\mathfrak{N}_j$. Then the same numbers $\gamma^{\circ}_j$, $j = 1,\ldots,p$,~are different eigenvalues of the 
problem~(\ref{abstr_hatS_gener_spec_problem}), and
$M \mathfrak{N}_j = \operatorname{Ker} (\wh{S} - \gamma_j^\circ Q_{\widehat{\mathfrak{N}}}) =: \wh{\mathfrak{N}}_{j,Q}$~are the corresponding eigenspaces. The vectors $\zeta^{(j)}_i = M\omega^{(j)}_i, \, i = 1,\ldots, k_j,$ form a basis in $\wh{\mathfrak{N}}_{j,Q}$ orthonormal with the weight $Q_{\widehat{\mathfrak{N}}}$. 
By $\mathcal{P}_j$ we denote the \textquotedblleft skew\textquotedblright \ projection onto $\wh{\mathfrak{N}}_{j,Q}$ which is orthogonal with respect to the inner product  $(Q_{\widehat{\mathfrak{N}}} \,\cdot\,, \,\cdot\,)$, i.~e.,
$\mathcal{P}_j = \sum_{i=1}^{k_j} (\,\cdot\,, Q_{\widehat{\mathfrak{N}}} \zeta^{(j)}_i) \zeta^{(j)}_i$, 
$j = 1, \ldots, p$.
Clearly, we have $\mathcal{P}_j =M P_j M^{-1} \widehat{P}$.
Using~(\ref{abstr_N_invar_repers}), (\ref{abstr_N_and_hatN_Q_relat}), and~(\ref{abstr_N0*_and_hatN0*_Q_relat}), it is easy to obtain the invariant representations
\begin{equation}
\label{abstr_hatN_0Q_N_*Q_invar_repr}
\widehat{N}_{0,Q} = \sum_{j=1}^{p} \mathcal{P}_j^* \widehat{N}_Q \mathcal{P}_j, \quad \widehat{N}_{*,Q} = \sum_{\substack{1 \le l,j \le p: \, l \ne j}} \mathcal{P}_l^* \widehat{N}_Q \mathcal{P}_j.
\end{equation}

\subsection{The coefficients $\nu_l$}

The coefficients $\nu_l$ from expansions \eqref{abstr_A(t)_eigenvalues_series} and the vectors  
$\zeta_l = M \omega_l$, $l=1,\dots,n$, are the eigenvalues and the eigenvectors of some problem; see 
\cite[Subsection~3.4]{D}.
We need to describe this problem in the case where $\mu_l=0$, $l=1,\dots,n,$ i.~e., $\wh{N}_{0,Q}=0$.

\begin{proposition}[see~\cite{D}]
\label{Prop_nu}
Let $\wh{N}_{0,Q}=0$. Suppose that the operator  
$\wh{N}_{1,Q}^0$ is defined by \eqref{abstr_hatN0_1Q}.
Let $\gamma_1^\circ, \dots, \gamma_p^\circ$ be the different eigenvalues of problem \eqref{abstr_hatS_gener_spec_problem}, and let $k_1, \dots, k_p$ be their multiplicities. Let  
$\wh{\mathfrak{N}}_{q,Q}  = \operatorname{Ker} (\wh{S} - \gamma_q^\circ Q_{\widehat{\mathfrak{N}}})$, and let 
$\wh{P}_{q,Q}$ be the orthogonal projection of the space $\wh{\H}$ onto  
$\wh{\mathfrak{N}}_{q,Q}$, $q=1,\dots,p$. 
We introduce the operators $\wh{\mathcal{N}}_Q^{(q)},$ $q=1,\dots,p$\emph{:} the operator $\wh{\mathcal{N}}_Q^{(q)}$ acts in $\wh{\mathfrak{N}}_{q,Q}$ and is given by the expression  
\begin{equation*}
\begin{aligned}
\wh{\mathcal{N}}_Q^{(q)} 
:= \wh{P}_{q,Q} \Big( \wh{N}_{1,Q}^0 
- \frac{1}{2} \wh{Z}_Q^* Q \wh{Z}_Q Q^{-1} \wh{S} \wh{P} - 
\frac{1}{2} \wh{S} \wh{P} Q^{-1} \wh{Z}_Q^* Q \wh{Z}_Q \Big) \Big\vert_{\wh{\mathfrak{N}}_{q,Q}}
\\
+ \sum_{j=1,\dots,p: j\ne q} (\gamma_q^\circ - \gamma_j^\circ)^{-1} 
\wh{P}_{q,Q}  \wh{N}_Q   \wh{P}_{j,Q} Q^{-1} \wh{P}_{ j,Q}   \wh{N}_Q 
\Big\vert_{\wh{\mathfrak{N}}_{q,Q}}.
\end{aligned}
\end{equation*}
Denote $i(q)= k_1 + \dots + k_{q-1} +1$.
Let $\nu_l$ be the coefficients of  $t^4$ in expansions  \eqref{abstr_A(t)_eigenvalues_series}, and let 
 $\omega_l$ be the embryos from expansions \eqref{abstr_A(t)_eigenvectors_series}. Let
$\zeta_l = M \omega_l$, $l=1,\dots,n$.
Denote ${Q}_{\wh{\mathfrak{N}}_{q,Q}} = \wh{P}_{q,Q} Q \vert_{\wh{\mathfrak{N}}_{q,Q}}$.
Then  
\begin{equation*}
\widehat{\mathcal{N}}^{(q)}_Q  \zeta_l = \nu_l {Q}_{\wh{\mathfrak{N}}_{q,Q} } \zeta_l, \quad l= i(q), i(q)+1,\dots, i(q) + k_q -1.
\end{equation*}
\end{proposition}

\subsection{Approximation of the sandwiched operators  $\cos(\eps^{-1}\tau A(t)^{1/2})$ and 
$A(t)^{-1/2} \sin( \eps^{-1}\tau A(t)^{1/2})$}
In this section, we find approximations of the operators $\cos(\eps^{-1}\tau A(t)^{1/2})$ and
$A(t)^{-1/2} \sin(\eps^{-1} \tau A(t)^{1/2})$ for the family~(\ref{abstr_A_and_Ahat}) in terms of the germ  $\widehat{S}$ of the operator $\widehat{A}(t)$ and the isomorphism $M$. It turns out that it is convenient to border the operators under consideration by appropriate factors.

Denote $M_0 := (Q_{\widehat{\mathfrak{N}}})^{-1/2}$.  We have 
\begin{align}
\label{abstr_sandwiched_cos_S_relation}
M \cos(\tau (t^2 S)^{1/2} ) P M^* &= M_0 \cos (\tau (t^2 M_0 \widehat{S} M_0)^{1/2}) M_0 \widehat{P},
\\
\label{abstr_sandwiched_sin_S_relation}
M (t^2 S)^{-1/2} 
\sin(\tau (t^2 S)^{1/2}) P M^* 
&= M_0 (t^2 M_0 \widehat{S} M_0)^{-1/2} \sin (\tau (t^2 M_0 \widehat{S} M_0)^{1/2}) M_0 \widehat{P},
\\
\label{abstr_sandwiched_sin_S_relation2}
M (t^2 S)^{-1/2} 
\sin(\tau (t^2 S)^{1/2} ) M^{-1} \wh{P} 
&= 
M_0 (t^2 M_0 \widehat{S} M_0)^{-1/2} \sin (\tau (t^2 M_0 \widehat{S} M_0)^{1/2}) M_0^{-1} \widehat{P}.
\end{align}
Relation \eqref{abstr_sandwiched_cos_S_relation} was checked in~\cite[Proposition~3.3]{BSu5}, and
 \eqref{abstr_sandwiched_sin_S_relation} follows from  \eqref{abstr_sandwiched_cos_S_relation}
 with the help of integration in $\tau$. Finally, relation \eqref{abstr_sandwiched_sin_S_relation2}
 is deduced from  \eqref{abstr_sandwiched_sin_S_relation} by multiplying by $M_0^{-2}\wh{P}= 
 Q_{\wh{\mathfrak{N}}} \wh{P}$ from the right and taking \eqref{abstr_P_and_P_hat_relation} into account.

 We introduce the notation 
 \allowdisplaybreaks{
 \begin{align}
\label{abstr_J1_def}
J_1(t, \tau) :=&  M  \cos (\tau A(t)^{1/2}) M^{-1} \widehat{P}  
-  M_0  \cos ( \tau (t^2 M_0 \widehat{S} M_0)^{1/2} ) M_0^{-1} \widehat{P},
\\
\label{abstr_J2_def}
\begin{split}
J_2(t, \tau) :=&  M  A(t)^{-1/2} \sin (\tau A(t)^{1/2}) M^{-1} \widehat{P}  
\\ 
&-   M_0 (t^2 M_0 \widehat{S} M_0)^{-1/2} \sin ( \tau (t^2 M_0 \widehat{S} M_0)^{1/2} ) M_0^{-1} \widehat{P},
\end{split}
\\
\label{abstr_tildeJ2_def}
\begin{split}
\wt{J}_3(t, \tau) :=&  M  A(t)^{-1/2} \sin (\tau A(t)^{1/2}) P M^{*}   
\\
&-   M_0 (t^2 M_0 \widehat{S} M_0)^{-1/2} \sin ( \tau (t^2 M_0 \widehat{S} M_0)^{1/2} ) M_0 \widehat{P},
\end{split}
\\
\label{abstr_J3_def}
\begin{split}
{J}_3(t, \tau) :=&  M  A(t)^{-1/2} \sin (\tau A(t)^{1/2}) M^* \wh{P}    
\\
&-   M_0 (t^2 M_0 \widehat{S} M_0)^{-1/2} \sin ( \tau (t^2 M_0 \widehat{S} M_0)^{1/2} ) M_0 \widehat{P}.
\end{split}
\end{align}
}

\begin{lemma}
	\label{abstr_cos_sin_sandwiched_est_lemma}
	Suppose that  ${\mathcal J}_1(t,\tau)$ and ${\mathcal J}_2(t,\tau)$ are defined by
 \eqref{3.01}, \eqref{3.02}. Under the assumptions of Subsection~\emph{\ref{abstr_A_and_Ahat_section}} we have 
	\begin{align}
	\label{abstr_J1_est1}
	& \| J_1(t, \tau)\|  \le \|M\| \|M^{-1}\| \| {\mathcal J}_1(t,\tau)  \|,
	\\
	\label{abstr_J2_est1}
	& \| J_2(t, \tau)\|  \le \|M\| \|M^{-1}\| \| {\mathcal J}_2(t,\tau) \|,
	\\
	\label{abstr_tildeJ2_est1}
	& \| \wt{J}_3(t, \tau)\|  \le \|M\|^2  \| {\mathcal J}_2(t,\tau) \|,
	\\
	\label{abstr_J1_est2}
	& \| {\mathcal J}_1(t,\tau)  \| \le \|M\|^2 \|M^{-1}\|^2 \| J_1(t, \tau)\|,
	\\
         \label{abstr_J2_est2}
	& \| {\mathcal J}_2(t,\tau) \|\le \|M\|^2 \|M^{-1}\|^2 \| J_2(t, \tau)\|,
	\\
          \label{abstr_tildeJ2_est2}
	& \| {\mathcal J}_2(t,\tau) \|\le \|M^{-1}\|^2 \| \wt{J}_3(t, \tau)\|.
	\end{align}
\end{lemma}

\begin{proof}
Inequalities \eqref{abstr_J1_est1}, \eqref{abstr_tildeJ2_est1}, \eqref{abstr_J1_est2}, and \eqref{abstr_tildeJ2_est2}
were proved in \cite[Lemma 4.2]{DSu}.

By \eqref{abstr_sandwiched_sin_S_relation2} and \eqref{abstr_J2_def}, 
\begin{equation}
\label{J2_relation}
    J_2(t,\tau) = M   {\mathcal J}_2(t,\tau)  M^{-1} \wh{P}.
\end{equation}
This implies inequality  \eqref{abstr_J2_est1}. 
Conversely, it is obvious that 
$$
\begin{aligned}
\| {\mathcal J}_2(t,\tau)  \|
\le \|M^{-1}\|^2 \| M {\mathcal J}_2(t,\tau) PM^* \|. 
\end{aligned}
$$
Using the relation $PM^* = M^{-1} Q_{\wh{\mathfrak N}}^{-1}\wh{P}$
(see \eqref{abstr_P_and_P_hat_relation}) and \eqref{J2_relation}, we rewrite the right-hand side as  
 $\|M^{-1}\|^2 \| J_2(t,\tau) Q_{\wh{\mathfrak N}}^{-1}\wh{P}\|$. Together with the inequality 
$\|Q_{\wh{\mathfrak N}}^{-1}\wh{P}\| \le \|M\|^2$ (which follows from the relation $Q_{\wh{\mathfrak N}}^{-1}\wh{P} 
=M P M^*$) this implies~\eqref{abstr_J2_est2}.
\end{proof}

By  \eqref{abstr_P_and_P_hat_relation}, 
$PM^* = PM^* \wh{P}$. From \eqref{abstr_tildeJ2_def} and \eqref{abstr_J3_def} it follows that  
$$
J_3(t,\tau) - \wt{J}_3(t,\tau) = M A(t)^{-1/2} \sin ( \tau A(t)^{1/2}) (I-P) M^* \wh{P}.
$$
 Applying \eqref{abstr_F(t)_threshold_1} and \eqref{abstr_A(t)_nondegenerated}, we obtain  
\begin{equation}
\label{J3-J3tilde}
\| J_3(t,\tau) - \wt{J}_3(t,\tau) \| \le \|M\|^2 \bigl( \delta^{-1/2} + C_1 c_*^{-1/2}\bigr) =: \wt{C},
\quad \tau \in \R, \ |t|\le t_0.
\end{equation}

Using inequalities \eqref{abstr_J1_est1}--\eqref{abstr_tildeJ2_est1},  \eqref{J3-J3tilde} and applying  Lemma \ref{abstr_N_and_Nhat_lemma}, we deduce the following three theorems from Theorems \ref{th3.1},
\ref{th3.2}, and \ref{th3.3}.
In formulations, we use the notation \eqref{abstr_J1_def}, \eqref{abstr_J2_def}, and \eqref{abstr_J3_def}. 

\begin{theorem}[see~\cite{BSu5,M,DSu}]
\label{th5.5}
	Under the assumptions of Subsection~\emph{\ref{abstr_A_and_Ahat_section}}, for 
	$\tau \in \mathbb{R}$, $\varepsilon > 0$, and $|t| \le t_0$ we have
	\begin{align}
	\label{abstr_cos_sandwiched_general_est}
	&\|  J_1(t, \varepsilon^{-1} \tau) \|  \varepsilon^2 (t^2 + \varepsilon^2)^{-1} 
	\le \|M\| \|M^{-1}\| (C_{1}  + C_{6} |\tau| ) \varepsilon,
	\\
	\label{abstr_sin_sandwiched_general_est1}
	&\|  J_2(t, \varepsilon^{-1} \tau) \|  \varepsilon (t^2 + \varepsilon^2)^{-1/2} 
	\le \|M\| \|M^{-1}\| (C_{7}  + C_{8} |\tau| ),
	\\
	\label{abstr_sin_sandwiched_general_est2}
	&\|  {J}_3(t, \varepsilon^{-1} \tau) \|  \varepsilon (t^2 + \varepsilon^2)^{-1/2} 
	\le \|M\|^2 (C_{7}  + C_{8} |\tau| ) + \wt{C}.
	\end{align}
\end{theorem}

Earlier, estimate \eqref{abstr_cos_sandwiched_general_est} was obtained in  \cite[Theorem 3.4]{BSu5}, estimate \eqref{abstr_sin_sandwiched_general_est1} in \cite[Theorem 3.3]{M}, and estimate
\eqref{abstr_sin_sandwiched_general_est2} in \cite[Theorem 4.3]{DSu}.

\begin{theorem}
\label{th5.6}
	Suppose that the operator $\wh{N}_Q$ defined by \eqref{abstr_hatN_Q} is equal to zero\emph{:}  $\wh{N}_Q=0$. Then for $\tau \in \mathbb{R},$ $\varepsilon > 0$, and $|t| \le t_0$ we have 
	\begin{align}
	\label{abstr_cos_sandwiched_improved_est}
	&\|  J_1(t, \varepsilon^{-1} \tau) \|  \varepsilon^{3/2} (t^2 + \varepsilon^2)^{-3/4} 
	\le \|M\| \|M^{-1}\| ( 2 C_{1}  + C_{9}' |\tau|^{1/2} ) \varepsilon,
	\\
	\label{abstr_sin_sandwiched_improved_est1}
	&\|  J_2(t, \varepsilon^{-1} \tau) \|  \varepsilon^{1/2} (t^2 + \varepsilon^2)^{-1/4} 
	\le \|M\| \|M^{-1}\| (C_{7}  + C_{10}' |\tau|^{1/2} ),
	\\
	\label{abstr_sin_sandwiched_improved_est2}
	&\|  {J}_3(t, \varepsilon^{-1} \tau) \|  \varepsilon^{1/2} (t^2 + \varepsilon^2)^{-1/4} 
	\le \|M\|^2 (C_{7}  + C_{10}' |\tau|^{1/2} ) + \wt{C}.
	\end{align}
\end{theorem}

\begin{theorem}
\label{th5.7}
	Suppose that the operator $\wh{N}_{0,Q}$ defined by  \eqref{abstr_hatN_0Q_N_*Q_invar_repr} is equal 
	to zero\emph{:}  $\wh{N}_{0,Q}=0$. Then for
	 $\tau \in \mathbb{R},$ $\varepsilon > 0$, and $|t| \le t^{00}$ we have 
	\begin{align*}
	&\|  J_1(t, \varepsilon^{-1} \tau) \|  \varepsilon^{3/2} (t^2 + \varepsilon^2)^{-3/4} 
	\le \|M\| \|M^{-1}\| ( C_{11}  + C_{12}' |\tau| ^{1/2}) \varepsilon,
	\\
	&\|  J_2(t, \varepsilon^{-1} \tau) \|  \varepsilon^{1/2} (t^2 + \varepsilon^2)^{-1/4} 
	\le \|M\| \|M^{-1}\| (C_{13}  + C_{14}' |\tau|^{1/2} ),
	\\
	&\|  {J}_3(t, \varepsilon^{-1} \tau) \|  \varepsilon^{1/2} (t^2 + \varepsilon^2)^{-1/4} 
	\le \|M\|^2 (C_{13}  + C_{14}' |\tau|^{1/2} ) + \wt{C}.
	\end{align*}
\end{theorem}

\subsection{Approximation in the ``energy'' norm for the sandwiched operator
${A(t)^{-1/2} \sin(\eps^{-1}\tau A(t)^{1/2})}$}
Denote 
\begin{align}
\label{abstr_J_def}
\begin{split}
J(t, \tau) :=&  M  A(t)^{-1/2} \sin (\tau A(t)^{1/2}) M^{-1} \widehat{P}  
\\
&-  (I + t\widehat{Z}_Q) M_0 (t^2 M_0 \widehat{S} M_0)^{-1/2} \sin ( \tau (t^2 M_0 \widehat{S} M_0)^{1/2} ) M_0^{-1} \widehat{P}.
\end{split}
\end{align}

\begin{lemma}
	\label{abstr_sin_sandwiched_est_lemma}
	Let $\Sigma(t,\tau)$ be the operator \eqref{3.13} and let $J(t,\tau)$ be the operator \eqref{abstr_J_def}. 
	Under the assumptions of Subsection~\emph{\ref{abstr_A_and_Ahat_section}}, we have 
	\begin{align}
	\label{abstr_sin_sandwiched_est_1}
	\| \wh{A}(t)^{1/2} J(t, \tau)\| & \le \|M^{-1}\| \|A(t)^{1/2} \Sigma(t, \tau) \|,
	\\
	\label{abstr_sin_sandwiched_est_3}
	\|A(t)^{1/2} \Sigma(t, \tau) \| & \le \|M\|^2 \| M^{-1}\| \| \wh{A}(t)^{1/2} J(t, \tau)\|.
		\end{align}
\end{lemma}

\begin{proof}
	From~(\ref{abstr_Z_and_hatZ_Q_relat}) and (\ref{abstr_sandwiched_sin_S_relation2}) it follows that  
	\begin{equation}
	\label{abstr_sin_sandwiched_est_proof_2}
	J(t,\tau) = M \Sigma(t,\tau) M^{-1} \wh{P}.
	\end{equation}
	Relations~(\ref{abstr_A_and_Ahat}) and~(\ref{abstr_sin_sandwiched_est_proof_2}) imply~(\ref{abstr_sin_sandwiched_est_1}). Conversely, it is obvious that 
	\begin{equation*}
\|A(t)^{1/2} \Sigma(t,\tau)  \| \le  \| M^{-1} \| \| A(t)^{1/2}   \Sigma(t,\tau)  P  M^* \|.
	\end{equation*}
Combining the relation $PM^* = M^{-1} Q^{-1}_{\wh{\mathfrak N}}\wh{P}$  and \eqref{abstr_A_and_Ahat},
 \eqref{abstr_sin_sandwiched_est_proof_2}, we represent the right-hand side 
in the form $\| M^{-1} \| \| \wh{A}(t)^{1/2} J(t,\tau) Q^{-1}_{\wh{\mathfrak N}}\wh{P} \|$. 
Together with the inequality $\| Q^{-1}_{\wh{\mathfrak N}}\wh{P}\| \le \| M \|^2$, this 
implies~(\ref{abstr_sin_sandwiched_est_3}).
\end{proof}

Applying inequality~(\ref{abstr_sin_sandwiched_est_1}) and using Lemma~\ref{abstr_N_and_Nhat_lemma}, 
from Theorems~\ref{abstr_sin_general_thrm}, \ref{th3.5}, \ref{abstr_sin_enchanced_thrm_2} we deduce the following results. 

\begin{theorem}[see~\cite{M}]
	\label{abstr_sin_sandwiched_general_thrm}
	Suppose that $J(t,\tau)$ is defined by  \eqref{abstr_J_def}.
	Under the assumptions of Subsection~\emph{\ref{abstr_A_and_Ahat_section}},
	for $\tau \in \mathbb{R},$ $\varepsilon > 0$, and $|t| \le t_0$ we have 
	\begin{equation*}
	\begin{split}
	\| \wh{A}(t)^{1/2}  J(t, \varepsilon^{-1} \tau) \|  \varepsilon^2 (t^2 + \varepsilon^2)^{-1} &\le \|M^{-1}\| (C_{17}  + C_{18} |\tau| ) \varepsilon.
	\end{split}
	\end{equation*}
\end{theorem}

Theorem~\ref{abstr_sin_sandwiched_general_thrm} was known earlier (see~\cite[Theorem 3.3]{M}).

\begin{theorem}
	\label{abstr_sin_sandwiched_ench_thrm_1}
	Suppose that the operator $\widehat{N}_Q$ defined by~\eqref{abstr_N_and_hatN_Q_relat} is equal 
	to zero{\rm :}  \hbox{$\widehat{N}_Q = 0$}. Then for $\tau \in \mathbb{R},$ $\varepsilon > 0$, and 
	$|t| \le t_0$ we have 
	\begin{equation*}
	\begin{split}
	\| \wh{A}(t)^{1/2} J(t, \varepsilon^{-1} \tau) \| \varepsilon^{3/2} (t^2 + \varepsilon^2)^{-3/4} &\le \|M^{-1}\| (C_{17}  + C_{19}' |\tau|^{1/2} ) \varepsilon.
	\end{split}
	\end{equation*}
\end{theorem}

\begin{theorem}
	\label{abstr_sin_sandwiched_ench_thrm_2}
	Suppose that the operator $\widehat{N}_{0,Q}$ defined by~\eqref{abstr_hatN_0Q_N_*Q_invar_repr} is equal to zero{\rm :} $\widehat{N}_{0,Q}=0$. Then for $\tau \in \mathbb{R},$ $\eps >0$, and $|t| \le t^{00}$ we have 	\begin{equation*}
	\begin{split}
	\| \wh{A}(t)^{1/2} J(t, \varepsilon^{-1} \tau) \| \varepsilon^{3/2} (t^2 + \varepsilon^2)^{-3/4} &\le \|M^{-1}\| (C_{20}  + C_{21}' |\tau|^{1/2} ) \varepsilon.
	\end{split}
	\end{equation*}
\end{theorem}

\section{Sharpness of the results of \S \ref{abstr_sandwiched_section}}

\subsection{Sharpness of the results regarding  the smoothing factor}
The following theorem confirms that  Theorems \ref{th5.5} and~\ref{abstr_sin_sandwiched_general_thrm} are sharp 
in the general case.

\begin{theorem}
	\label{th6.1}
	Suppose that the assumptions of Subsection~\emph{\ref{abstr_A_and_Ahat_section}} are satisfied.
	Let $\wh{N}_{0,Q} \ne 0$. 
	
	\noindent $1^\circ$. Let $\tau \ne 0$ and  $0 \le s < 2$. Then there does not exist a constant  
	\hbox{$C(\tau) > 0$} such that the estimate 
	\begin{equation}
	\label{6*.1}
	\left\| J_1(t, \eps^{-1} \tau) \right\|  \varepsilon^{s} (t^2 + \varepsilon^2)^{-s/2} \le  C(\tau) \varepsilon
	\end{equation}
	holds for all sufficiently small $|t|$ and $\varepsilon > 0$.

	\noindent $2^\circ$. Let $\tau \ne 0$ and $0 \le r < 1$.  Then there does not exist a constant 
	\hbox{$C(\tau) > 0$} such that the estimate
	\begin{equation}
	\label{6*.2}
	\left\| J_2(t, \eps^{-1} \tau)  \right\|  \varepsilon^{r} (t^2 + \varepsilon^2)^{-r/2} \le  C(\tau)
	\end{equation}
	holds for all sufficiently small $|t|$ and $\varepsilon > 0$.

	\noindent $3^\circ$. Let $\tau \ne 0$ and $0 \le r < 1$.  Then there does not exist a constant  
	\hbox{$C(\tau) > 0$} such that the estimate
	\begin{equation}
	\label{6*.3}
	\left\| {J}_3(t, \eps^{-1} \tau)  \right\|  \varepsilon^{r} (t^2 + \varepsilon^2)^{-r/2} \le  C(\tau)
	\end{equation}
	holds for all sufficiently small $|t|$ and $\varepsilon > 0$.

	\noindent $4^\circ$.
	 Let $\tau \ne 0$ and $0 \le s < 2$.  Then there does not exist a constant  
	\hbox{$C(\tau) > 0$} such that the estimate
		\begin{equation}
		\label{abstr_sndwchd_s<2_est_imp1}
		\bigl \| \wh{A}(t)^{1/2} J (t, \varepsilon^{-1} \tau) \bigr \| \varepsilon^{s} (t^2 + \varepsilon^2)^{-s/2} \le  C(\tau) \varepsilon
		\end{equation}
		holds for all sufficiently small $|t|$ and $\varepsilon > 0$.
		\end{theorem}

\begin{proof}
Statements $1^\circ$ and $3^\circ$ were proved in  \cite[Theorem 4.6]{DSu}.

Let us prove statement $2^\circ$. By Lemma \ref{abstr_N_and_Nhat_lemma}, the condition $\wh{N}_{0,Q} \ne 0$ is equivalent to the condition $N_0\ne 0$. We suppose the opposite. Then,  using inequality \eqref{abstr_J2_est2}, 
we see that \eqref{**.2} is satisfied for some $0 \le r <1$. But this contradicts statement~$2^\circ$ of Theorem \ref{th3.8}.

Let us check statement $4^\circ$. 	Suppose the opposite. Then, using~\eqref{abstr_sin_sandwiched_est_3},
we arrive at inquality~\eqref{abstr_s<2_est_imp} with some $0 \le s < 2$. 
 But this contradicts the statement of Theorem~\ref{abstr_s<2_general_thrm}. 
\end{proof}

Next, we confirm that Theorems  \ref{th5.6}, \ref{th5.7}, \ref{abstr_sin_sandwiched_ench_thrm_1}, and  \ref{abstr_sin_sandwiched_ench_thrm_2} are sharp. (We omit the results for~$J_2$, because they will not be used in the study of DOs.)

\begin{theorem}
	\label{th6.2}
	Suppose that the assumptions of Subsection~\emph{\ref{abstr_A_and_Ahat_section}} are satisfied.
	Let $\wh{N}_{0,Q}\!=\! 0$ and  $\wh{\mathcal N}_Q^{(q)}\!\ne\! 0$ for some $q$ \emph{(}i.~e.,  
	$\nu_l \!\ne \!0$ for some~$l$\emph{)}. 
	
	\noindent $1^\circ$. Let $\tau \!\ne\! 0$ and $0\! \le \!s \!<\! 3/2$. Then there does not exist a constant ${C(\tau)\! >\! 0}$ such that estimate \eqref{6*.1} holds for all sufficiently small $|t|$ and $\varepsilon > 0$.

	\noindent $2^\circ$. Let $\tau \!\ne\! 0$ and $0 \!\le\! r \!< \!1/2$.  Then there does not exist a constant ${C(\tau)\! > \!0}$ such that estimate \eqref{6*.3} holds for all sufficiently small $|t|$ and $\varepsilon > 0$.
	
	\noindent $3^\circ$. Let $\tau \!\ne\! 0$ and $0 \!\le\! s\! <\! 3/2$. Then there does not exist a constant $C(\tau)\! >\! 0$ such that estimate \eqref{abstr_sndwchd_s<2_est_imp1} holds for all sufficiently small $|t|$ and  
	$\varepsilon >0$.	
\end{theorem}

\begin{proof}
By Lemma  \ref{abstr_N_and_Nhat_lemma}, the condition $\wh{N}_{0,Q}= 0$ is equivalent to the condition $N_0= 0$.
Next, according to Proposition \ref{Prop_nu}, the condition  $\wh{\mathcal N}_Q^{(q)}\ne 0$ for some $q$ means that  
$\nu_l \ne 0$ for some $l \in \{ i(q),\dots,i(q) + k_q -1\}$. By Proposition \ref{Prop_nu_1}, it follows that  
${\mathcal N}^{(q)}\ne 0$. Thus, the assumptions of Theorems \ref{th3.9} and \ref{th3.12} are satisfied.

Let us prove statement $1^\circ$. Assuming the opposite and using inequality  \eqref{abstr_J1_est2}, we see that \eqref{**.1} is satisfied for some  $0 \le s< 3/2$. But this contradicts statement~$1^\circ$ of Theorem \ref{th3.9}. 
	
	Statement $2^\circ$ is checked with the help of  \eqref{abstr_tildeJ2_est2}, \eqref{J3-J3tilde}, and statement $2^\circ$ of Theorem \ref{th3.9}. Statement $3^\circ$ follows from   
	\eqref{abstr_sin_sandwiched_est_3} and Theorem \ref{th3.12}.
	\end{proof}

\subsection{Sharpness of the results with respect to time}

 Using Lemma~\ref{abstr_N_and_Nhat_lemma} and relations \eqref{abstr_J1_est2}--\eqref{abstr_tildeJ2_est2}, \eqref{J3-J3tilde}, \eqref{abstr_sin_sandwiched_est_3}, we deduce the following result 
from Theorems \ref{th3.13} and~\ref{th4.6}. This result confirms that Theorems~\ref{th5.5} 
and~\ref{abstr_sin_sandwiched_general_thrm} are sharp.

\begin{theorem}\label{th6.5}
Suppose that $\wh{N}_{0,Q} \ne 0$. 

	\noindent $1^\circ$. Let $s \ge 2$. There does not exist a positive function $C(\tau)$ 
	such that $\lim_{\tau \to \infty} C(\tau)/ |\tau| =0$ and   estimate \eqref{6*.1} holds for all
 $\tau \in \R$ and sufficiently small  $|t|$ and $\eps$.  

\noindent $2^\circ$. Let $r \ge 1$. There does not exist a positive function $C(\tau)$ such that 
$\lim_{\tau \to \infty} C(\tau)/ |\tau| =0$ and estimate \eqref{6*.2} holds for all
  $\tau \in \R$ and sufficiently small  $|t|$ and $\eps$.

\noindent $3^\circ$. Let $r \ge 1$. There does not exist a positive function $C(\tau)$ such that 
$\lim_{\tau \to \infty} C(\tau)/ |\tau| =0$ and estimate \eqref{6*.3} holds for all $\tau \in \R$ and sufficiently small  $|t|$ and $\eps$.

\noindent $4^\circ$. Let $s \ge 2$. There does not exist a positive function $C(\tau)$ such that 
$\lim_{\tau \to \infty} C(\tau)/ |\tau| =0$ and estimate  \eqref{abstr_sndwchd_s<2_est_imp1} holds for all $\tau \in \R$ and sufficiently small $|t|$ and $\eps$.
\end{theorem}

Similarly to the proof of Theorem \ref{th6.2}, from Theorems~\ref{th4.*2} and \ref{th4.final} we deduce the following result which demonstrates that Theorems \ref{th5.6},  \ref{th5.7}, \ref{abstr_sin_sandwiched_ench_thrm_1}, 
and~\ref{abstr_sin_sandwiched_ench_thrm_2} are sharp.

\begin{theorem}
	\label{th6.7}
	Suppose that  $\wh{N}_{0,Q} = 0$ and $\wh{\mathcal{N}}_Q^{(q)} \ne 0$ for some $q\in \{1,\dots,p\}$.
	 
	\noindent $1^\circ$.  Let $s \ge 3/2$. There does not exist a positive function $C(\tau)$ such that  
	$\lim_{\tau \to \infty} C(\tau)/ |\tau|^{1/2} = 0$ and estimate  \eqref{6*.1} holds for all $\tau \in \R$ and sufficiently small  $|t|$ and $\varepsilon > 0$.

	\noindent $2^\circ$. Let $r \ge 1/2$. There does not exist a positive function $C(\tau)$ such that 
	$\lim_{\tau \to \infty} C(\tau)/ |\tau|^{1/2} = 0$ and estimate  \eqref{6*.3} holds for all $\tau \in \R$ and sufficiently small $|t|$ and $\varepsilon > 0$.

\noindent $3^\circ$. Let $s \ge 3/2$. There does not exist a positive function $C(\tau)$ such that 
	$\lim_{\tau \to \infty} C(\tau)/ |\tau|^{1/2} = 0$ and estimate \eqref{abstr_sndwchd_s<2_est_imp1} holds for all $\tau \in \R$ and sufficiently small $|t|$ and $\varepsilon > 0$.
\end{theorem}

\section*{Chapter 2.  Periodic differential operators in $L_2(\mathbb{R}^d; \mathbb{C}^n)$}

\section{The class of differential operators in  $L_2(\mathbb{R}^d; \mathbb{C}^n)$}

\subsection{Lattices. Fourier series}

Let $\Gamma$~be a lattice in $\mathbb{R}^d$ generated by the basis $\mathbf{a}_1, \ldots , \mathbf{a}_d$, i.~e., 
$\Gamma = \bigl\{ \mathbf{a} \in \mathbb{R}^d \colon \mathbf{a} = \sum_{j=1}^{d} n_j \mathbf{a}_j, \; n_j \in \mathbb{Z} \bigr\}$, and let  $\Omega$~be the elementary cell of this lattice:
$$
\Omega := \bigl\{ \mathbf{x} \in \mathbb{R}^d \colon \mathbf{x} = \sum_{j=1}^{d} \xi_j \mathbf{a}_j, \; 0 < \xi_j < 1 \bigr\}.
$$ 
The basis $\mathbf{b}_1, \ldots , \mathbf{b}_d$ dual to the basis $\mathbf{a}_1, \ldots , \mathbf{a}_d$ 
is defined by the relations $\left< \mathbf{b}_l, \mathbf{a}_j \right> = 2 \pi \delta_{lj}$. This basis generates a 
lattice~$\widetilde \Gamma$ \emph{dual} to the lattice $\Gamma$.
By  $\widetilde \Omega$ we denote the central \emph{Brillouin zone} of the lattice $ \widetilde \Gamma$:
\begin{equation}
\label{Brillouin_zone}
\widetilde \Omega = \bigl\{ \mathbf{k} \in \mathbb{R}^d : | \mathbf{k} | < | \mathbf{k} - \mathbf{b} |, \; 0 \ne \mathbf{b} \in \widetilde \Gamma \bigr\}.
\end{equation}
Denote $| \Omega | = \operatorname{meas} \Omega$, $| \widetilde \Omega | = \operatorname{meas} \widetilde \Omega$, and note that  $| \Omega |  | \widetilde \Omega | = (2 \pi)^d$. Let $r_0$~be the radius of the ball  \emph{inscribed} in $\clos \widetilde \Omega$, and let $r_1 := \max_{\mathbf{k} \in \partial\widetilde{\Omega}} |\mathbf{k}|$. Note that 
\begin{equation}
\label{r_0}
2 r_0 = \min|\mathbf{b}|,  \quad 0 \ne \mathbf{b} \in \widetilde \Gamma.
\end{equation}
The following discrete Fourier transformation  is associated with the lattice $ \Gamma $:
\begin{equation}
\label{fourier}
\mathbf{v}(\mathbf{x}) = | \Omega |^{-1/2} \sum_{\mathbf{b} \in \widetilde \Gamma} \widehat{\mathbf{v}}_{\mathbf{b}} \exp (i \left<\mathbf{b}, \mathbf{x} \right>), \quad \mathbf{x} \in \Omega.
\end{equation}
This transform is a unitary mapping of $l_2 (\widetilde \Gamma; \mathbb{C}^n) $ onto 
$L_2 (\Omega; \mathbb{C}^n)$:
\begin{equation}
\label{fourier_unitary}
\int\limits_{\Omega} |\mathbf{v}(\mathbf{x})|^2 d \mathbf{x} = \sum_{\mathbf{b} \in \widetilde \Gamma} | \widehat{\mathbf{v}}_{\mathbf{b}} |^2.
\end{equation}

\textit{Let $\widetilde H^1(\Omega; \mathbb{C}^n)$ be the subspace of functions from $ H^1(\Omega; \mathbb{C}^n)$ whose $\Gamma$-periodic extension to $\mathbb{R}^d$ belongs to  
$H^1_{\mathrm{loc}}(\mathbb{R}^d; \mathbb{C}^n)$}. We have 
\begin{equation}
\label{D_and_fourier}
\int\limits_{\Omega} |(\mathbf{D} + \mathbf{k}) \mathbf{v}|^2\, d\mathbf{x} = \sum_{\mathbf{b} \in \widetilde{\Gamma}} |\mathbf{b} + \mathbf{k} |^2 |\widehat{\mathbf{v}}_{\mathbf{b}}|^2, \quad \mathbf{v} \in \widetilde{H}^1(\Omega; \mathbb{C}^n), \; \mathbf{k} \in \mathbb{R}^d,
\end{equation} 
and convergence of the series in the right-hand side of~(\ref{D_and_fourier}) is equivalent to the relation 
${\mathbf{v} \in \widetilde{H}^1(\Omega; \mathbb{C}^n)}$. From~(\ref{Brillouin_zone}),~(\ref{fourier_unitary}), and~(\ref{D_and_fourier}) it follows that 
\begin{equation}
\label{AO}
\int\limits_{\Omega}\!\! |(\mathbf{D}\! +\! \mathbf{k}) \mathbf{v}|^2 d\mathbf{x}  \ge \!\sum_{\mathbf{b} \in \widetilde{\Gamma}}\! | \mathbf{k} |^2 |\widehat{\mathbf{v}}_{\mathbf{b}}|^2\! = | \mathbf{k} |^2\!\! \int\limits_{\Omega}\! |\mathbf{v}|^2 d\mathbf{x}, \quad \mathbf{v} \!\in \!\widetilde{H}^1(\Omega; \mathbb{C}^n), \; \mathbf{k} \!\in\! \widetilde{\Omega}.
\end{equation}

\subsection{The Gelfand transformation}
First, we define the Gelfand transform $\mathcal{U}$ for functions of the Schwartz class 
 $\mathbf{v} \in \mathcal{S}(\mathbb{R}^d; \mathbb{C}^n)$ by the formula:
\begin{equation*}
\widetilde{\mathbf{v}} ( \mathbf{k}, \mathbf{x}) = (\mathcal{U} \- \mathbf{v}) (\mathbf{k}, \mathbf{x}) = | \widetilde \Omega |^{-1/2} \sum_{\mathbf{a} \in \Gamma} e^{- i \left< \mathbf{k}, \mathbf{x} + \mathbf{a} \right> } \mathbf{v} ( \mathbf{x} + \mathbf{a}),  \quad \mathbf{x} \in \Omega, \ \mathbf{k} \in \widetilde \Omega.
\end{equation*}
We have $\| \widetilde{\mathbf{v}} \|_{L_2(\wt{\Omega} \times \Omega)} = \| \v \|_{L_2(\R^d)}$,
and $\mathcal{U}$ extends by continuity up to unitary mapping 
\begin{equation*}
\mathcal{U} : L_2 (\mathbb{R}^d; \mathbb{C}^n) \to \int\limits_{\widetilde \Omega} \oplus  L_2 (\Omega; \mathbb{C}^n) d \mathbf{k} =:  \mathcal{H}.
\end{equation*}

\subsection{Factorized second order operators  $\mathcal{A}$}
\label{A_oper_subsect}

Let $b (\mathbf{D}) = \sum_{l=1}^d b_l D_l$, where $b_l$ are constant  ($ m \times n $)-matrices (in general, with complex entries).
\emph{Suppose that $m \ge n$}. Consider the symbol $b(\boldsymbol{\xi})= \sum_{l=1}^d b_l \xi_l$ and \emph{suppose that} $\rank b( \boldsymbol{\xi} ) = n$, $0 \ne  \boldsymbol{\xi} \in \mathbb{R}^d $. This condition is equivalent to the inequalities 
\begin{equation}
\label{rank_alpha_ineq}
\alpha_0 \mathbf{1}_n \le b( \boldsymbol{\theta} )^* b( \boldsymbol{\theta} ) \le \alpha_1 \mathbf{1}_n, \quad  \boldsymbol{\theta} \in \mathbb{S}^{d-1}, \quad 0 < \alpha_0 \le \alpha_1 < \infty,
\end{equation}
with some $ \alpha_0, \alpha_1 > 0 $.
Note that  \eqref{rank_alpha_ineq} implies the following estimates for the norms of the matrices $b_l$:
\begin{equation}
\label{5.7a} 
|b_l | \le \alpha_1^{1/2}, \quad l=1,\dots,d.
\end{equation}

Suppose that $f(\mathbf{x}), \; \mathbf{x} \in \mathbb{R}^d$,~is a $\Gamma$-periodic   ($n \times n$)-matrix-valued function  and $h(\mathbf{x}), \; \mathbf{x} \in \mathbb{R}^d$,~is a  $\Gamma$-periodic  ($m \times m$)-matrix-valued function. Assume that
\begin{equation}
\label{h_f_L_inf}
f, f^{-1} \in L_{\infty} (\mathbb{R}^d); \quad h, h^{-1} \in L_{\infty} (\mathbb{R}^d).
\end{equation}
Let 
$$
\mathcal{X}:  L_2 (\mathbb{R}^d ; \mathbb{C}^n) \to  L_2 (\mathbb{R}^d ; \mathbb{C}^m)
$$ 
be a closed operator given by the expression $\mathcal{X} = h b( \mathbf{D} ) f$ on the domain  
$$
\Dom \mathcal{X} = \{ \mathbf{u} \in L_2 (\mathbb{R}^d ; \mathbb{C}^n) \colon f \mathbf{u} \in 
H^1  (\mathbb{R}^d ; \mathbb{C}^n) \}.
$$
A selfadjoint operator $\mathcal{A} = \mathcal{X}^* \mathcal{X}$ in $L_2 (\mathbb{R}^d ; \mathbb{C}^n)$ is generated by the closed quadratic form $\mathfrak{a}[\mathbf{u}, \mathbf{u}] = \| \mathcal{X} \mathbf{u} \|^2_{L_2(\mathbb{R}^d)}, \; \mathbf{u} \in \Dom \mathcal{X}$. Formally,
\begin{equation}
\label{A}
\mathcal{A} = f (\mathbf{x})^* b( \mathbf{D} )^* g( \mathbf{x} )  b( \mathbf{D} ) f(\mathbf{x}),
\end{equation} 
where $ g( \mathbf{x} ) = h( \mathbf{x} )^*  h( \mathbf{x} ) $. Using the Fourier transform 
and~(\ref{rank_alpha_ineq}),~(\ref{h_f_L_inf}), it is easy to check that  
\begin{equation}
\label{a_form_ineq}
\alpha_0 \| g^{-1} \|_{L_{\infty}}^{-1} \| \mathbf{D} (f \mathbf{u}) \|_{L_2}^2 \! \le \mathfrak{a}[\mathbf{u}, \mathbf{u}]\le \alpha_1 \| g \|_{L_{\infty}} \| \mathbf{D} (f \mathbf{u}) \|_{L_2}^2, \quad \mathbf{u} \!\in\! \Dom \mathcal{X}.
\end{equation}

\subsection{The operators $\mathcal{A}(\mathbf{k})$}
Let $\mathbf{k} \in \mathbb{R}^d$. We put 
\begin{equation}
\label{Spaces_H}
\mathfrak{H} = L_2 (\Omega; \mathbb{C}^n), \quad \mathfrak{H}_* = L_2 (\Omega; \mathbb{C}^m),
\end{equation}
and consider the closed operator $\mathcal{X} (\mathbf{k}) \colon \mathfrak{H} \!\to \!\mathfrak{H}_*$
given by $\mathcal{X} (\mathbf{k})\! =\! hb(\mathbf{D}\! + \!\mathbf{k})f$ on the domain
${\Dom \mathcal{X} (\mathbf{k}) \!=\! \big\{ \mathbf{u} \in \mathfrak{H} \colon   f \mathbf{u} \in \widetilde{H}^1 (\Omega; \mathbb{C}^n)\big\} \!=:\!  \mathfrak{d}}$.
A selfadjoint operator 
$\mathcal{A} (\mathbf{k}) =\mathcal{X} (\mathbf{k})^* \mathcal{X} (\mathbf{k})$ in  $ \mathfrak{H}$
is generated by the quadratic form 
$\mathfrak{a}(\mathbf{k})[\mathbf{u}, \mathbf{u}] = \| \mathcal{X}(\mathbf{k}) \mathbf{u} \|_{\mathfrak{H}_*}^2$, $\mathbf{u} \in \mathfrak{d}$.
Using expansion of a function $\v=f\mathbf{u}$ in the Fourier series~(\ref{fourier}) and conditions~(\ref{rank_alpha_ineq}),~(\ref{h_f_L_inf}), it is easy to check that  
\begin{equation}
\label{a(k)_form_est}
\alpha_0 \|g^{-1} \|_{L_\infty}^{-1} \|(\mathbf{D} + \mathbf{k}) f \mathbf{u} \|_{L_2 (\Omega)}^2 
\le \mathfrak{a}(\mathbf{k})[\mathbf{u}, \mathbf{u}]
\le \alpha_1 \|g \|_{L_\infty} \|(\mathbf{D} + \mathbf{k}) f \mathbf{u} \|_{L_2 (\Omega)}^2, \  \mathbf{u} \in \mathfrak{d}.
\end{equation}

From~(\ref{AO}) and the lower estimate~(\ref{a(k)_form_est}) it follows that 
\begin{equation}
\label{A(k)_nondegenerated_and_c_*}	
\mathcal{A} (\mathbf{k}) \ge c_* |\mathbf{k}|^2 I, \qquad \mathbf{k} \in \widetilde{\Omega}, \; c_* = \alpha_0\|f^{-1} \|_{L_\infty}^{-2} \|g^{-1} \|_{L_\infty}^{-1} .
\end{equation}

We put $\mathfrak{N} :=  \Ker \mathcal{A} (0) = \Ker \mathcal{X} (0)$.
Relations~(\ref{a(k)_form_est}) with $\mathbf{k} = 0$ show that 
\begin{equation}
\label{Ker2}
\mathfrak{N} = \left\lbrace \mathbf{u} \in L_2 (\Omega; \mathbb{C}^n) \colon f \mathbf{u} = \mathbf{c} \in \mathbb{C}^n \right\rbrace, \quad \dim \mathfrak{N} = n. 
\end{equation}

As follows from~(\ref{r_0}) and~(\ref{D_and_fourier}) with $\mathbf{k} = 0$, 
a function  $\mathbf{v} \in \widetilde{H}^1 (\Omega; \mathbb{C}^n)$ such that 
$\int_{\Omega} \mathbf{v} \, d \mathbf{x} = 0$ (i.~e., $\widehat{\mathbf{v}}_0 = 0$) satisfies 
\begin{equation}
\label{AP}
\| \mathbf{D} \mathbf{v} \|_{L_2 (\Omega)}^2 \ge 4 r_0^2 \| \mathbf{v} \|_{L_2 (\Omega)}^2, \quad \mathbf{v} \in \widetilde{H}^1 (\Omega; \mathbb{C}^n), \; \int\limits_{\Omega} \mathbf{v} \, d \mathbf{x} = 0.
\end{equation} 
From~(\ref{AP}) and  the lower estimate~(\ref{a(k)_form_est}) with $\mathbf{k} = 0$ it follows that the distance  $d^0$ from the point $\lambda_0 = 0$ to the rest of the spectrum of the operator $\mathcal{A}(0)$ satisfies the estimate 
\begin{equation}
\label{d0_est}
d^0 \ge 4 c_* r_0^2.
\end{equation}

Denote by $E_j(\mathbf{k})$, $j \in \mathbb{N}$, the consecutive (counting multiplicities) eigenvalues of the operator $\mathcal{A}(\mathbf{k})$ (the band functions).
The band functions $E_j(\mathbf{k})$ are continuous and $\widetilde{\Gamma}$-periodic.
According to \eqref{A(k)_nondegenerated_and_c_*}, we have 
$E_j(\mathbf{k}) \ge c_* | \mathbf{k} |^2$, $j=1,\dots,n$.
As was shown in~\cite[Chapter~2, Subsection~2.2] {BSu1}, $E_{n+1}(\mathbf{k}) \ge c_* r_0^2$.

\subsection{The direct integral for the operator $\mathcal{A}$}
Under the Gelfand transform, the operator $\mathcal{A}$ expands in the direct integral:
\begin{equation}
\label{decompose}
\mathcal{U} \mathcal{A}  \mathcal{U}^{-1} = \int\limits_{\widetilde \Omega} \oplus \mathcal{A} (\mathbf{k}) \, d \mathbf{k}.
\end{equation}
This means the following. Let $\mathbf{v} \in \Dom \mathcal{X}$, then
$\widetilde{\mathbf{v}}(\mathbf{k}, \,\cdot\,) \in \mathfrak{d}$ for a.e. $\mathbf{k} \in \widetilde \Omega$ and  
\begin{equation}
\label{AB}
\mathfrak{a}[\mathbf{v}, \mathbf{v}] = \int\limits_{\widetilde{\Omega}} \mathfrak{a}(\mathbf{k}) [\widetilde{\mathbf{v}}(\mathbf{k}, \,\cdot\,), \widetilde{\mathbf{v}}(\mathbf{k}, \,\cdot\,)] \, d \mathbf{k} .
\end{equation}
Conversely, if $\widetilde{\mathbf{v}} \in \mathcal{H}$ satisfies 
$\widetilde{\mathbf{v}}(\mathbf{k}, \,\cdot\,) \in \mathfrak{d}$ for a.e. $\mathbf{k} \in \widetilde \Omega$
and the integral in~(\ref{AB}) is finite, then $\mathbf{v} \in \Dom \mathcal{X}$ and~(\ref{AB}) is valid.

From~(\ref{decompose}) it follows that the spectrum of the operator $\mathcal{A}$ coincides with the union of the intervals  (bands) $\Ran E_j,\; j \in\mathbb{N}$. 
Herewith,  the first $n$ spectral bands of the operator $\mathcal{A}$ overlap and have common bottom $\lambda_0 = 0$, while  the $(n + 1)$-th band is separated from zero.

\subsection{Incorporation of the operators $\mathcal{A} (\mathbf{k})$ in the abstract scheme}

If $ d > 1 $, then the operators  $\mathcal{A} (\mathbf{k})$ depend on the multidimensional parameter $\mathbf{k}$. According to \cite[Chapter~2]{BSu1}, we introduce the one-dimensional parameter $t = | \mathbf{k}|$. 
We rely on the scheme of Chapter~1. Now all constructions will depend on the parameter 
$\boldsymbol{\theta} = \mathbf{k} / | \mathbf{k}| \in \mathbb{S}^{d-1}$, and we have to make estimates 
uniform  in~$\boldsymbol{\theta}$. The spaces  $\mathfrak{H}$ and $\mathfrak{H}_*$ are defined 
by~(\ref{Spaces_H}). We put $X(t) = X(t, \boldsymbol{\theta}) =:  \mathcal{X}(t \boldsymbol{\theta})$. Then $X(t, \boldsymbol{\theta}) = X_0 + t  X_1 (\boldsymbol{\theta})$, where $X_0 = h(\mathbf{x}) b (\mathbf{D}) f(\mathbf{x}), \; \Dom X_0 = \mathfrak{d}$, and $ X_1 (\boldsymbol{\theta})$ is a bounded operator of multiplication by the matrix $h(\mathbf{x}) b(\boldsymbol{\theta}) f(\mathbf{x})$. Next, we put $A(t) = A(t, \boldsymbol{\theta}) =:  \mathcal{A}(t \boldsymbol{\theta})$.  The kernel $\mathfrak{N} = \Ker X_0 = \Ker \mathcal{A} (0)$ is described by \eqref{Ker2}, $\dim \mathfrak{N} = n$. The number $d^0$ satisfies estimate~(\ref{d0_est}). As was shown in  \cite[Chapter~2,~\S3]{BSu1}, the condition $n \le n_* = \dim \Ker X^*_0$ is also satisfied. Moreover, either $n_* = n$ (if $m = n$), or  $n_* = \infty$ (if $m > n$). Thus, all the assumptions of the abstract scheme are satisfied.

According to Subsection~\ref{abstr_X_A_section}, we should fix a number $\delta >0$ such that $\delta < d^0/8$. Using~(\ref{A(k)_nondegenerated_and_c_*}) and~(\ref{d0_est}), we put 
\begin{equation}
\label{delta_fixation}
\delta = \frac{1}{4} c_* r^2_0 = \frac{1}{4} \alpha_0\|f^{-1} \|_{L_\infty}^{-2} \|g^{-1} \|_{L_\infty}^{-1} r^2_0.
\end{equation}  
Note that, by~(\ref{rank_alpha_ineq}) and~(\ref{h_f_L_inf}), we have
\begin{equation}
\label{X_1_estimate}
\| X_1 (\boldsymbol{\theta}) \| \le  \alpha^{1/2}_1 \| h \|_{L_{\infty}} \| f \|_{L_{\infty}}, \quad \boldsymbol{\theta} \in \mathbb{S}^{d-1}.
\end{equation}

We choose  $t_0$ (see~(\ref{abstr_t0_fixation})) as follows: 
\begin{equation}
\label{t0_fixation}
t_0 = \delta^{1/2} \alpha_1^{-1/2} \| h \|_{L_{\infty}}^{-1} \|f\|_{L_{\infty}}^{-1}
= \frac{r_0}{2} \alpha_0^{1/2} \alpha_1^{-1/2} \left( \| h \|_{L_{\infty}} \| h^{-1} \|_{L_{\infty}} \|f \|_{L_\infty} \|f^{-1} \|_{L_\infty} \right)^{-1}.
\end{equation}
Note that $t_0 \le r_0/2$. Hence, the ball $|\mathbf{k}| \le t_0$ lies entirely in $\widetilde{\Omega}$. 
It is important that  $c_*$, $\delta$, $t_0$ (see~(\ref{A(k)_nondegenerated_and_c_*}), (\ref{delta_fixation}), (\ref{t0_fixation})) do not depend on  $\boldsymbol{\theta}$.

By~\eqref{A(k)_nondegenerated_and_c_*}, Condition~\ref{cond_A} is satisfied. The germ  $S(\boldsymbol{\theta})$ of the operator $A(t, \boldsymbol{\theta})$ is nondegenerate uniformly in  $\boldsymbol{\theta}$ (cf.~\eqref{abstr_S_nondegenerated}): 
\begin{equation}
\label{S_nondegenerated}
S(\boldsymbol{\theta}) \ge c_* I_{\mathfrak{N}}, \quad \boldsymbol{\theta} \in \mathbb{S}^{d-1}.
\end{equation}

\section{The effective characteristics of the operator $\widehat{\mathcal{A}} = b(\mathbf{D})^* g(\mathbf{x}) b(\mathbf{D})$}

\subsection{The operator $A(t, \boldsymbol{\theta})$ in the case where $f = \mathbf{1}_n$}
A special role is played by the operator  $\mathcal A$ with $f = \mathbf{1}_n$. In this case, we agree to mark all objects by hat  \textquotedblleft$\widehat{\phantom{\_}} $\textquotedblright. Then for the operator
\begin{equation}
\label{hatA}
\widehat{\mathcal{A}} = b(\mathbf{D})^* g(\mathbf{x}) b(\mathbf{D})
\end{equation}
the family  
\begin{equation}
\label{hatA(k)}
\widehat{\mathcal{A}} (\mathbf{k}) = b(\mathbf{D}+\mathbf{k})^* g(\mathbf{x}) b(\mathbf{D}+\mathbf{k})
\end{equation}
is denoted by $\widehat{A} (t, \boldsymbol{\theta})$. The kernel~(\ref{Ker2}) takes the form
\begin{equation}
\label{Ker3}
\widehat{\mathfrak{N}} = \left\lbrace \mathbf{u} \in L_2 (\Omega; \mathbb{C}^n) \colon \mathbf{u} = \mathbf{c} \in \mathbb{C}^n \right\rbrace, 
\end{equation}
i.~e., $\widehat{\mathfrak{N}}$ consists of constant vector-valued functions. The orthogonal projection  
$\widehat{P}$ of the space  $L_2 (\Omega; \mathbb{C}^n)$ onto the subspace (\ref{Ker3}) is the operator of averaging over the cell:
\begin{equation}
\label{Phat_projector}
\widehat{P} \mathbf{u} = |\Omega|^{-1} \int\limits_{\Omega} \mathbf{u} (\mathbf{x}) \, d\mathbf{x}.
\end{equation}

In the case where $f = \mathbf{1}_n$, the constants (\ref{A(k)_nondegenerated_and_c_*}), (\ref{delta_fixation}), and~(\ref{t0_fixation}) take the form 
\begin{align}
\label{hatc_*}
&\widehat{c}_*  = \alpha_0 \|g^{-1} \|_{L_\infty}^{-1}, \\
\label{hatdelta_fixation}
&\widehat{\delta} =  \frac{1}{4} \alpha_0 \|g^{-1} \|_{L_\infty}^{-1} r^2_0,\\
\label{hatt0_fixation}
&\widehat{t}_0  = \frac{r_0}{2} \alpha_0^{1/2} \alpha_1^{-1/2} \left( \| g \|_{L_{\infty}} \| g^{-1} \|_{L_{\infty}} \right)^{-1/2}.
\end{align}

The inequality (\ref{X_1_estimate}) turns into 
\begin{equation}
\label{hatX_1_estmate}
\| \widehat{X}_1 (\boldsymbol{\theta}) \| \le \alpha_1^{1/2} \| g \|_{L_{\infty}}^{1/2}.
\end{equation}

\subsection{The operators $\widehat{Z}(\boldsymbol{\theta})$, $\widehat{R}(\boldsymbol{\theta})$, and $\widehat{S}(\boldsymbol{\theta})$\label{sec8.2}}
\label{germ_and_eff_g0_section}

Now the operators $\widehat{Z}(\boldsymbol{\theta})$, $\widehat{R}(\boldsymbol{\theta})$, and $\widehat{S}(\boldsymbol{\theta})$ for the family $\widehat{A}(t, \boldsymbol{\theta})$ (in abstract terms, defined in Subsection~\ref{abstr_Z_R_S_op_section}) depend on $\boldsymbol{\theta}$.
They were found in \cite[Subsection 4.1]{BSu3} and~\cite[Chapter~3,~\S1]{BSu1}.

Let $\Lambda \in \widetilde{H}^1 (\Omega)$~be a periodic ($n \times m$)-matrix-valued function satisfying the equation 
\begin{equation}
\label{equation_for_Lambda}
b(\mathbf{D})^* g(\mathbf{x}) (b(\mathbf{D}) \Lambda (\mathbf{x}) + \mathbf{1}_m) = 0, \quad \int\limits_{\Omega} \Lambda (\mathbf{x}) \, d \mathbf{x} = 0.
\end{equation}
Then the operators $\widehat{Z}(\boldsymbol{\theta}): \mathfrak{H} \to \mathfrak{H}$ and $\widehat{R}(\boldsymbol{\theta}): \wh{\mathfrak{N}} \to \mathfrak{N}_*$ are represented as 
\begin{equation}
\label{hat_Z_theta}
\widehat{Z}(\boldsymbol{\theta}) = [ \Lambda] b(\boldsymbol{\theta}) \wh{P},
\quad 
\widehat{R}(\boldsymbol{\theta}) = [ h (b(\D) \Lambda + \1_m)] b(\boldsymbol{\theta}).
\end{equation}
Here and in what follows, square brackets denote the operator of multiplication by a function.
The spectral germ $\widehat{S} (\boldsymbol{\theta})= \widehat{R}(\boldsymbol{\theta})^* \widehat{R}(\boldsymbol{\theta})$ of the family $\widehat{A}(t, \boldsymbol{\theta})$ acting in $\widehat{\mathfrak{N}}$
is given by 
$\widehat{S} (\boldsymbol{\theta}) = b(\boldsymbol{\theta})^* g^0 b(\boldsymbol{\theta})$, 
where  $g^0$~is the so called effective matrix. 
The \textit{effective matrix} $g^0$ is defined in terms of the matrix $\Lambda (\mathbf{x})$:
\begin{gather}
\label{g_tilde}
\widetilde{g} (\mathbf{x}) :=  g(\mathbf{x})( b(\mathbf{D}) \Lambda (\mathbf{x}) + \mathbf{1}_m), \\
\label{g0}
g^0 = | \Omega |^{-1} \int\limits_{\Omega} \widetilde{g} (\mathbf{x}) \, d \mathbf{x}.
\end{gather}
It turns out that the matrix $g^0$ is positive definite.

Using~(\ref{equation_for_Lambda}), it is easy to check that
\begin{align}
\label{6.11a}
\| g^{1/2} b(\mathbf{D}) \Lambda \|_{L_2(\Omega)} &\le |\Omega|^{1/2} \|g\|_{L_\infty}^{1/2},
\\
\label{Lambda_est}
\| \Lambda \|_{L_2(\Omega)} &\le | \Omega |^{1/2} M_1, \quad M_1 := (2 r_0)^{-1} \alpha_0^{-1/2} \|g\|_{L_\infty}^{1/2} \|g^{-1}\|_{L_\infty}^{1/2},
\\
\label{D_Lambda_est}
\| \mathbf{D} \Lambda \|_{L_2(\Omega)} &\le | \Omega |^{1/2} M_2, \quad M_2 := \alpha_0^{-1/2} \|g\|_{L_\infty}^{1/2} \|g^{-1}\|_{L_\infty}^{1/2}.
\end{align}

\subsection{The effective operator\label{sec8.3}}

Consider the symbol 
\begin{equation}
\label{effective_oper_symb}
\widehat{S} (\mathbf{k}) :=  t^2 \widehat{S} (\boldsymbol{\theta}) = b(\mathbf{k})^* g^0 b(\mathbf{k}), \quad \mathbf{k} \in \mathbb{R}^{d}.
\end{equation}
Note that  
\begin{equation*}
\widehat{S} (\mathbf{k})  \ge \wh{c}_* |\k|^2 \1_n,  \quad \mathbf{k} \in \mathbb{R}^{d},
\end{equation*}
which folllows from \eqref{S_nondegenerated} (with $f=\1_n$). 
Expression~(\ref{effective_oper_symb}) is the symbol of the DO 
\begin{equation}
\label{hatA0}
\widehat{\mathcal{A}}^0 = b(\mathbf{D})^* g^0 b(\mathbf{D}),
\end{equation}
acting in  $L_2(\mathbb{R}^d; \mathbb{C}^n)$ and called the \emph{effective operator} for the operator 
$\widehat{\mathcal{A}}$.

Let $\widehat{\mathcal{A}}^0 (\mathbf{k})$~be the operator family in $L_2(\Omega; \mathbb{C}^n)$ corresponding to the effective operator~(\ref{hatA0}). Then $\widehat{\mathcal{A}}^0 (\mathbf{k}) = b(\mathbf{D} + \mathbf{k})^* g^0 b(\mathbf{D} + \mathbf{k})$ with periodic boundary conditions. Together with~(\ref{Phat_projector}) and~(\ref{effective_oper_symb}) this implies that
\begin{equation}
\label{hatS_P=hatA^0_P}
\widehat{S} (\mathbf{k}) \widehat{P} = \widehat{\mathcal{A}}^0 (\mathbf{k}) \widehat{P}.
\end{equation}

\subsection{The properties of the effective matrix}

The following properties of the matrix $g^0$ were checked in~\cite[Chapter~3, Theorem~1.5]{BSu1}.

\begin{proposition}[see~\cite{BSu1}]
	The effective matrix satisfies the following estimates 
	\begin{equation}
	\label{Voigt_Reuss}
	\underline{g} \le g^0 \le \overline{g},
	\end{equation}
	where  $\overline{g} := | \Omega |^{-1} \int_{\Omega} g (\mathbf{x}) \, d \mathbf{x}$
	and $\underline{g} := ( | \Omega |^{-1} \int_{\Omega} g (\mathbf{x})^{-1} \, d \mathbf{x})^{-1}$.
	If $m = n$, then $g^0 = \underline{g}$.
\end{proposition}

Estimates~(\ref{Voigt_Reuss}) are known in homogenization theory for particular DOs as the Voigt--Reuss bracketing. Note that estimates~(\ref{Voigt_Reuss}) imply that 
\begin{equation}
\label{g^0_est}
| g^0 | \le \|g\|_{L_\infty}, \quad | (g^0)^{-1} | \le \|g^{-1}\|_{L_\infty}.
\end{equation}

Now, we distinguish conditions under which one of the inequalities in~(\ref{Voigt_Reuss}) becomes an identity; see  \cite[Chapter~3, Propositions~1.6, 1.7]{BSu1}.

\begin{proposition}[see~\cite{BSu1}]
	The identity  $g^0 = \overline{g}$ is equivalent to the relations 
	\begin{equation}
	\label{g0=overline_g_relat}
	b(\mathbf{D})^* \mathbf{g}_k (\mathbf{x}) = 0, \quad k = 1, \ldots, m,
	\end{equation}
	where $\mathbf{g}_k (\mathbf{x}), \; k = 1, \ldots,m,$~are the columns of the matrix $g (\mathbf{x})$.
\end{proposition}

\begin{proposition}[see~\cite{BSu1}]
	The identity $g^0 = \underline{g}$ is equivalent to the representations
	\begin{equation}
	\label{g0=underline_g_relat}
	\mathbf{l}_k (\mathbf{x}) \!= \!\mathbf{l}^0_k\! +\! b(\mathbf{D}) \mathbf{w}_k(\mathbf{x}), \quad \mathbf{l}^0_k\! \in \!\mathbb{C}^m, \quad \mathbf{w}_k \in \widetilde{H}^1 (\Omega; \mathbb{C}^n), \quad k \!=\! 1, \ldots,m,
	\end{equation}
	where $\mathbf{l}_k (\mathbf{x}), \; k = 1, \ldots,m,$~are the columns of the matrix $g (\mathbf{x})^{-1}$.
\end{proposition}

\subsection{Analytic branches of the eigenvalues and eigenvectors}

The analytic (in $t$) branches of the eigenvalues $\widehat{\lambda}_l (t, \boldsymbol{\theta})$ and the analytic branches of the eigenvectors $\widehat{\varphi}_l (t, \boldsymbol{\theta})$ of the operator $\widehat{A} (t, \boldsymbol{\theta})$ admit the power series expansions of the form~(\ref{abstr_A(t)_eigenvalues_series}),~(\ref{abstr_A(t)_eigenvectors_series}) with coefficients depending on $\boldsymbol{\theta}$ (we do not control the interval of convergence $t = |\mathbf{k}| \le t_* (\boldsymbol{\theta})$):
\begin{align}
\label{hatA_eigenvalues_series}
\widehat{\lambda}_l (t, \boldsymbol{\theta}) &= \widehat{\gamma}_l (\boldsymbol{\theta}) t^2 + 
\widehat{\mu}_l (\boldsymbol{\theta}) t^3 + \widehat{\nu}_l (\boldsymbol{\theta}) t^4+ \ldots, & &l = 1, \ldots, n, \\
\label{hatA_eigenvectors_series}
\widehat{\varphi}_l (t, \boldsymbol{\theta}) &= \widehat{\omega}_l (\boldsymbol{\theta}) + t \widehat{\psi}^{(1)}_l (\boldsymbol{\theta}) + \ldots, & &l = 1, \ldots, n.
\end{align}
According to~(\ref{abstr_S_eigenvectors}), the numbers $\widehat{\gamma}_l (\boldsymbol{\theta})$ and the elements $\widehat{\omega}_l (\boldsymbol{\theta})$ are eigenvalues and eigenvectors of the germ:
\begin{equation*}
b(\boldsymbol{\theta})^* g^0 b(\boldsymbol{\theta}) \widehat{\omega}_l (\boldsymbol{\theta}) = \widehat{\gamma}_l (\boldsymbol{\theta}) \widehat{\omega}_l (\boldsymbol{\theta}), \quad l = 1, \ldots, n.
\end{equation*}

\subsection{The operator $\widehat{N} (\boldsymbol{\theta})$}

 As was shown in~\cite[\S4]{BSu3}, the operator $N$ (see~\eqref{abstr_N}) for the family $\widehat{A} (t, \boldsymbol{\theta})$  takes the form 
\begin{equation}
\label{N(theta)}
\widehat{N} (\boldsymbol{\theta}) = b(\boldsymbol{\theta})^* L(\boldsymbol{\theta}) b(\boldsymbol{\theta}) \widehat{P},
\end{equation}
where $L (\boldsymbol{\theta})$~is the ($m \times m$)-matrix-valued function given by 
\begin{equation}
\label{L(theta)}
L (\boldsymbol{\theta}) = | \Omega |^{-1} \int\limits_{\Omega} (\Lambda (\mathbf{x})^* b(\boldsymbol{\theta})^* \widetilde{g}(\mathbf{x}) + \widetilde{g}(\mathbf{x})^* b(\boldsymbol{\theta}) \Lambda (\mathbf{x}) ) \, d \mathbf{x}.
\end{equation}
Here $\Lambda (\mathbf{x})$~is the $\Gamma$-periodic solution of problem~(\ref{equation_for_Lambda}) and $\widetilde{g}(\mathbf{x})$~is the matrix-valued function~(\ref{g_tilde}).

In~\cite[\S4]{BSu3}, some conditions ensuring that~\hbox{$\widehat{N} (\boldsymbol{\theta}) \equiv 0$} are given.

\begin{proposition}[see~\cite{BSu3}]
	\label{N=0_proposit}
	Suppose that at least one of the following assumptions is satisfied\emph{:}
	
\noindent		$1^\circ$. The operator $\widehat{\mathcal{A}}$ is given by  $\widehat{\mathcal{A}} = \mathbf{D}^* g(\mathbf{x}) \mathbf{D}$, where $g(\mathbf{x})$~is a symmetric matrix with real entries.

		\noindent $2^\circ$. Relations~\emph{(\ref{g0=overline_g_relat})} are satisfied, i.~e., $g^0 = \overline{g}$.
		
\noindent		$3^\circ$.  Relations~\emph{(\ref{g0=underline_g_relat})} are satisfied, i.~e., 
 $g^0 = \underline{g}$. 

\noindent	Then $\widehat{N} (\boldsymbol{\theta}) = 0$ for all $\boldsymbol{\theta} \in \mathbb{S}^{d-1}$.
\end{proposition}

On the other hand, in~\cite[Subsections~10.4,~13.2,~14.6]{BSu3} there are examples of the operators  
$\widehat{\mathcal{A}}$ for which the operator $\widehat{N} (\boldsymbol{\theta})$ is not equal to zero. 
See also~\cite[Example~8.7]{Su4}, \cite[Subsection~14.3]{DSu}. 
Recall (see~Remark~\ref{abstr_N_remark}) that $\widehat{N} (\boldsymbol{\theta}) = \widehat{N}_0 (\boldsymbol{\theta}) + \widehat{N}_* (\boldsymbol{\theta})$, where the operator $\widehat{N}_0 (\boldsymbol{\theta})$ is diagonal in the basis  $\{ \widehat{\omega}_l (\boldsymbol{\theta})\}_{l=1}^n$ and the operator $\widehat{N}_* (\boldsymbol{\theta})$ has zero diagonal entries. We have 
\begin{equation*}
(\widehat{N} (\boldsymbol{\theta}) \widehat{\omega}_l (\boldsymbol{\theta}), \widehat{\omega}_l (\boldsymbol{\theta}))_{L_2 (\Omega)} = (\widehat{N}_0 (\boldsymbol{\theta}) \widehat{\omega}_l (\boldsymbol{\theta}), \widehat{\omega}_l (\boldsymbol{\theta}))_{L_2 (\Omega)} = \widehat{\mu}_l (\boldsymbol{\theta}), \quad l=1, \ldots, n.
\end{equation*}

The following statement was proved  in~\cite[Subsection~4.3]{BSu3}. 

\begin{proposition}
	\label{N_0=0_proposit}
	Suppose that the matrices $b(\boldsymbol{\theta})$ and $g (\mathbf{x})$~have real entries. Suppose that 
	the vectors $\widehat{\omega}_l (\boldsymbol{\theta}),$ $l = 1, \ldots, n,$  in expansions~\emph{(\ref{hatA_eigenvectors_series})} can be chosen real. Then 
	$\widehat{\mu}_l (\boldsymbol{\theta}) = 0$,  $l=1, \ldots, n,$ i.~e., $\widehat{N}_0 (\boldsymbol{\theta}) = 0$.
\end{proposition}

In the \textquotedblleft real \textquotedblright \ case under onsideration,  the germ $\widehat{S} (\boldsymbol{\theta})$ is a symmetric matrix with real entries. Clearly, in the case of the simple eigenvalue $\widehat{\gamma}_j (\boldsymbol{\theta})$ of the germ, the embryo  $\widehat{\omega}_j (\boldsymbol{\theta})$ is determined uniquely up to a phase factor, and it  can always be chosen real. We arrive at the following corollary.

\begin{corollary}
	\label{S_spec_simple_coroll}
	Suppose that the matrices $b(\boldsymbol{\theta})$ and $g (\mathbf{x})$~have real entries. Suppose that the spectrum of the germ $\widehat{S} (\boldsymbol{\theta})$ is simple. Then 
	$\widehat{N}_0 (\boldsymbol{\theta}) = 0$.
\end{corollary}

\subsection{The operators $\widehat{Z}_2 (\boldsymbol{\theta})$, $\widehat{R}_2 (\boldsymbol{\theta})$, and
$\widehat{N}_1^0 (\boldsymbol{\theta})$}

We describe the operators $Z_2$, $R_2$, and $N_1^0$ (in abstract terms they were defined in Subsections \ref{abstr_Z2_R2_section} and \ref{sec_nu_l}) for the family $\wh{A}(t,\boldsymbol{\theta})$. Let $\Lambda_l^{(2)}(\x)$
be the $\Gamma$-periodic solution of the problem 
$$
b(\D)^* g(\x) \bigl(b(\D) \Lambda_l^{(2)}(\x) + b_l \Lambda(\x)\bigr)
= b_l^* \bigl( g^0 - \widetilde{g}(\x) \bigr),
\quad \int\limits_\Omega \Lambda_l^{(2)}(\x) \,d\x =0. 
$$
We put  
$
\Lambda^{(2)}(\x; \boldsymbol{\theta}) := \sum_{l=1}^d \Lambda_l^{(2)}(\x) \theta_l.
$
As was checked  in \cite[Subsection 6.3]{VSu2}, 
$$
\widehat{Z}_2 (\boldsymbol{\theta}) = \Lambda^{(2)}(\x; \boldsymbol{\theta}) b(\boldsymbol{\theta}) \widehat{P},
\quad 
\widehat{R}_2 (\boldsymbol{\theta}) = h(\x) \bigl( b(\D)\Lambda^{(2)}(\x; \boldsymbol{\theta}) 
+ b(\boldsymbol{\theta}) \Lambda(\x) \bigr) b(\boldsymbol{\theta}).
$$
Finally, in \cite[Subsection 6.4]{VSu2} it was shown that
\begin{align}
\label{N_10_theta}
\widehat{N}_1^0 (\boldsymbol{\theta}) = b(\boldsymbol{\theta})^* L_2(\boldsymbol{\theta}) b(\boldsymbol{\theta})
\widehat{P},
\end{align}
\begin{equation}
\label{L2_theta}
\begin{split}
&L_2(\boldsymbol{\theta})\! = \!|\Omega|^{-1}\!\! \int\limits_\Omega\!\! ( \Lambda^{(2)}(\x; \boldsymbol{\theta})^* b(\boldsymbol{\theta})^*
\widetilde{g}(\x)
\!+ \!\widetilde{g}(\x)^* b(\boldsymbol{\theta}) \Lambda^{(2)}(\x; \boldsymbol{\theta})) d\x
\\
&\!\!+\!  |\Omega|^{-\!1}\!\!\! \int\limits_\Omega \!\!( b(\D) \Lambda^{(2)}(\x; \boldsymbol{\theta})\! +\! b(\boldsymbol{\theta}) \Lambda(\x))^*\!\!
{g}(\x) \!( b(\D) \Lambda^{(2)}(\x; \boldsymbol{\theta})\! + \!b(\boldsymbol{\theta}) \Lambda(\x))
 d\x.
\end{split}
\end{equation}

\subsection{Multiplicities of the eigenvalues of the germ}
\label{eigenval_multipl_section}
In this subsection, we assume that  \hbox{$n \ge 2$}. We pass to the notation adopted in~Subsection~\ref{abstr_cluster_section}.
In general, the number $p(\boldsymbol{\theta})$ of the different eigenvalues  $\widehat{\gamma}^{\circ}_1 (\boldsymbol{\theta}), \ldots, \widehat{\gamma}^{\circ}_{p(\boldsymbol{\theta})} (\boldsymbol{\theta})$ of the spectral germ $\widehat{S}(\boldsymbol{\theta})$ and their multiplicities $k_1 (\boldsymbol{\theta}), \ldots, k_{p(\boldsymbol{\theta})} (\boldsymbol{\theta})$ depend on the parameter $\boldsymbol{\theta} \in \mathbb{S}^{d-1}$. For each fixed $\boldsymbol{\theta}$, let $\widehat{P}_j (\boldsymbol{\theta})$ be the orthogonal projection of $L_2 (\Omega; \mathbb{C}^n)$ onto the eigenspace  
$\wh{\mathfrak N}_j(\boldsymbol{\theta})$ of the germ $\widehat{S}(\boldsymbol{\theta})$ corresponding to the eigenvalue $\widehat{\gamma}_j^{\circ} (\boldsymbol{\theta})$. We have the following invariant representations for the operators  $\widehat{N}_0 (\boldsymbol{\theta})$  and $\widehat{N}_* (\boldsymbol{\theta})$:
\begin{gather}
\label{N0_invar_repr}
\widehat{N}_0 (\boldsymbol{\theta}) = \sum_{j=1}^{p(\boldsymbol{\theta})} \widehat{P}_j (\boldsymbol{\theta}) \widehat{N} (\boldsymbol{\theta}) \widehat{P}_j (\boldsymbol{\theta}), \\
\label{N*_invar_repr}
\widehat{N}_* (\boldsymbol{\theta}) = \sum_{\substack{1 \le j, l \le p(\boldsymbol{\theta}):\, j \ne l}} \widehat{P}_j (\boldsymbol{\theta}) \widehat{N} (\boldsymbol{\theta}) \widehat{P}_l (\boldsymbol{\theta}).
\end{gather}

\subsection{The coefficients $\wh{\nu}_l(\boldsymbol{\theta})$}
Applying Proposition~\ref{Prop_nu_1}, we arrive at the following statement.

\begin{proposition}\label{Prop_nu1_theta}
Let $\wh{N}_0(\boldsymbol{\theta})=0$. 
Suppose that $\wh{\gamma}_1^\circ(\boldsymbol{\theta}), \dots, \wh{\gamma}_{p(\boldsymbol{\theta})}^\circ(\boldsymbol{\theta})$ are the different eigenvalues of the operator 
$\wh{S}(\boldsymbol{\theta})$ and  $k_1(\boldsymbol{\theta}), \dots, k_{p(\boldsymbol{\theta})}(\boldsymbol{\theta})$ are their multiplicities. Let $\wh{P}_q(\boldsymbol{\theta})$ be the orthogonal projection of the space $L_2(\Omega;\AC^n)$ onto the subspace  
$\wh{\mathfrak{N}}_q (\boldsymbol{\theta})  = \operatorname{Ker} (\wh{S}(\boldsymbol{\theta}) 
- \wh{\gamma}_q^\circ(\boldsymbol{\theta}) I_{\wh{\mathfrak{N}}}),$ $q=1,\dots,p(\boldsymbol{\theta})$.
Let  $\wh{Z}(\boldsymbol{\theta})$ and
$\wh{N}_1^0(\boldsymbol{\theta})$ be the operators defined by \eqref{hat_Z_theta} and \eqref{N_10_theta}\textup,  \eqref{L2_theta}\textup, respectively.
We introduce the operators $\wh{\mathcal{N}}^{(q)}(\boldsymbol{\theta}),$ $q=1,\dots,p(\boldsymbol{\theta})$\emph{:}  the operator $\wh{\mathcal{N}}^{(q)}(\boldsymbol{\theta})$  acts in  
$\wh{\mathfrak{N}}_q(\boldsymbol{\theta})$ and is given by the expression   
\begin{equation}
\label{8.32aa}
\begin{split}
&\wh{\mathcal{N}}^{(q)}( \boldsymbol{\theta})
 \\
&:= \wh{P}_q(\boldsymbol{\theta}) 
\Big( \wh{N}_1^0(\boldsymbol{\theta}) - \frac{1}{2} \wh{Z}(\boldsymbol{\theta})^* \wh{Z}(\boldsymbol{\theta}) 
\wh{S}(\boldsymbol{\theta}) \wh{P} - \frac{1}{2} \wh{S}(\boldsymbol{\theta}) \wh{P} 
\wh{Z}(\boldsymbol{\theta})^* \wh{Z}(\boldsymbol{\theta})  \Big) \Big\vert_{\wh{\mathfrak{N}}_q(\boldsymbol{\theta})}
\\
&\qquad+ \sum_{j=1,\dots,p(\boldsymbol{\theta}): j\ne q} (\wh{\gamma}_q^\circ(\boldsymbol{\theta}) -
 \wh{\gamma}_j^\circ(\boldsymbol{\theta}))^{-1} 
\wh{P}_q(\boldsymbol{\theta}) \wh{N}(\boldsymbol{\theta}) \wh{P}_j(\boldsymbol{\theta}) 
\wh{N}(\boldsymbol{\theta}) \Big \vert_{\wh{\mathfrak{N}}_q(\boldsymbol{\theta})}.
\end{split}
\end{equation}
Denote $i(q,\boldsymbol{\theta})= k_1(\boldsymbol{\theta}) + \dots + k_{q-1}(\boldsymbol{\theta}) +1$.
Let $\wh{\nu}_l(\boldsymbol{\theta})$ be the coefficients of  $t^4$ in expansions 
\eqref{hatA_eigenvalues_series}\textup, and let  $\wh{\omega}_l(\boldsymbol{\theta})$ be the embryos from
\eqref{hatA_eigenvectors_series}\textup, $l=1,\dots,n$.
Then  
$$
\wh{\mathcal{N}}^{(q)}(\boldsymbol{\theta}) \wh{\omega}_l(\boldsymbol{\theta}) 
= \wh{\nu}_l(\boldsymbol{\theta}) \wh{\omega}_l(\boldsymbol{\theta}), \quad 
l=i(q, \boldsymbol{\theta}), i(q,\boldsymbol{\theta})+1,\dots, i(q,\boldsymbol{\theta}) + k_q(\boldsymbol{\theta}) -1.
$$
\end{proposition}

\section{Approximation for the operators  
$\cos(\varepsilon^{-1} \tau \widehat{\mathcal{A}}(\mathbf{k})^{1/2})$ and 
$\widehat{\mathcal{A}}(\mathbf{k})^{-1/2} \sin(\varepsilon^{-1} \tau \widehat{\mathcal{A}}(\mathbf{k})^{1/2})$
\label{sec9}}

\subsection{Approximation in the operator norm in   $L_2(\Omega;\AC^n)$. The general case}

Consider the operator $\mathcal{H}_0 = -\Delta$ in $L_2 (\mathbb{R}^d; \AC^n)$. In the direct integral expansion, the operator $\mathcal{H}_0$ is associated with the family of operators $\mathcal{H}_0 (\mathbf{k})$ acting in  $L_2 (\Omega; \AC^n)$. The operator $\mathcal{H}_0 (\mathbf{k})$ is given by the differential expression $| \mathbf{D} + \mathbf{k} |^2$ with periodic boundary conditions.  Denote 
\begin{equation}
\label{R(k, epsilon)}
\mathcal{R}(\mathbf{k}, \varepsilon) := \varepsilon^2 (\mathcal{H}_0 (\mathbf{k}) + \varepsilon^2 I)^{-1}.
\end{equation}
Obviously, 
\begin{equation}
\label{R_P}
\mathcal{R}(\mathbf{k}, \varepsilon)^{s/2}\widehat{P} = \varepsilon^s (t^2 + \varepsilon^2)^{-s/2} \widehat{P}, \qquad s > 0.
\end{equation}
Note that for  $|\k| > \wh{t}_0$ we have 
\begin{equation}
\label{R_P_2}
\| \mathcal{R}(\mathbf{k}, \varepsilon)^{s/2}\widehat{P} \|_{L_2(\Omega) \to L_2(\Omega)}
\le  (\wh{t}_0)^{-s} \varepsilon^s, \quad \eps > 0,\ \k \in \wt{\Omega}, \ |\k| > \wh{t}_0.
\end{equation}
Next, using the discrete Fourier transform, we obtain  
\begin{equation}
\label{R(k,eps)(I-P)_est}
\begin{split}
\| \mathcal{R}(\mathbf{k}, \varepsilon)^{s/2} (I - \widehat{P}) \|_{L_2(\Omega) \to L_2 (\Omega) }  \le \sup_{0 \ne \mathbf{b} \in \widetilde{\Gamma}} \varepsilon^s (|\mathbf{b} + \mathbf{k}|^2 + \varepsilon^2)^{-s/2} \le r_0^{-s} \varepsilon^s,
\\
\varepsilon > 0, \quad \mathbf{k} \in \widetilde{\Omega}.
\end{split}
\end{equation}

Denote 
\begin{align}
\label{J1_hat}
\widehat{J}_1(\mathbf{k}, \tau ) :=& \cos \bigl( \tau \wh{\mathcal A}(\k)^{1/2}\bigr) - 
\cos \bigl(\tau \wh{\mathcal A}^0(\k)^{1/2}\bigr),
\\
\label{J2_hat}
\wh{J}_2(\mathbf{k}, \tau ) :=& \wh{\mathcal A}(\k)^{-1/2} \sin \bigl(\tau \wh{\mathcal A}(\k)^{1/2}\bigr) - 
\wh{\mathcal A}^0(\k)^{-1/2} \sin \bigl(\tau \wh{\mathcal A}^0(\k)^{1/2}\bigr).
\end{align}

We apply theorems from~\S\ref{abstr_aprox_thrm_section} to the operator $\widehat{A}(t, \boldsymbol{\theta}) = \widehat{\mathcal{A}}(\mathbf{k})$. According to Remark~\ref{abstr_constants_remark}, we  can track the dependence  of the constants in estimates on the problem data. Note that $\widehat{c}_*$, $\widehat{\delta}$, and $\widehat{t}_0$ do not depend on $\boldsymbol{\theta}$ (see (\ref{hatc_*})--(\ref{hatt0_fixation})). According to~(\ref{hatX_1_estmate}), the norm $\| \widehat{X}_1 (\boldsymbol{\theta}) \|$ can be replaced by $\alpha_1^{1/2} \| g \|_{L_{\infty}}^{1/2}$. Therefore, the constants from Theorem~\ref{th3.1}
(applied to the operator $\widehat{\mathcal{A}} (\mathbf{k})$) will not depend on $\boldsymbol{\theta}$. 
They will depend only on $\alpha_0$, $\alpha_1$, $\|g\|_{L_\infty}$, $\|g^{-1}\|_{L_\infty}$, and $r_0$.

\begin{theorem}[see~\cite{BSu5,M}] 
\label{th9.1}
Suppose that  $\wh{J}_1(\mathbf{k}, \tau )$ and $\wh{J}_2(\mathbf{k}, \tau )$ are the operators 
defined by 
\eqref{J1_hat}\textup, \eqref{J2_hat}. Then for $\tau \in \R,$ $\eps >0,$ and $\k \in \wt{\Omega}$ we have
\begin{align}
\label{9.7}
 \|  \wh{J}_1(\mathbf{k}, \eps^{-1} \tau )  \mathcal{R}(\k,\eps) \|_{L_2(\Omega) \to L_2(\Omega)}
&\le \wh{\mathcal C}_1 (1+ |\tau|) \eps,
\\
\label{9.8}
 \|  \wh{J}_2(\mathbf{k}, \eps^{-1} \tau ) \mathcal{R}(\k,\eps)^{1/2} \|_{L_2(\Omega) \to L_2(\Omega)}
&\le \wh{\mathcal C}_2 (1+ |\tau|).
\end{align}
The constants $\wh{\mathcal C}_1$ and $\wh{\mathcal C}_2$ depend only on $\alpha_0,$ $\alpha_1,$ $\|g\|_{L_\infty},$  $\|g^{-1}\|_{L_\infty}$, and $r_0$.
\end{theorem}

Theorem \ref{th9.1} is deduced from Theorem~\ref{th3.1} and relations  \eqref{R_P}--\eqref{R(k,eps)(I-P)_est}.
We should also take into account the following obvious estimates:
\begin{equation} 
\label{9.9}
\|  \wh{J}_1(\mathbf{k}, \eps^{-1} \tau ) \|_{L_2(\Omega) \to L_2(\Omega)} \le 2,
\quad 
\|  \wh{J}_2(\mathbf{k}, \eps^{-1} \tau ) \|_{L_2(\Omega) \to L_2(\Omega)}
\le 2 \eps^{-1}|\tau|.
\end{equation}
Earlier, estimate \eqref{9.7} was obtained in \cite[Theorem 7.2]{BSu5}, and inequality \eqref{9.8} was proved in 
 \cite[Subsection 7.4]{M}.

Below (for interpolation purposes in Chapter  3) we shall also need the following statement.

\begin{proposition}
\label{prop9.1a}
Under the assumptions of Theorem \emph{\ref{th9.1}}, for $\tau \in \R,$ $\eps >0$, and $\k \in \wt{\Omega}$ 
the operator \eqref{J2_hat} satisfies the following estimate:
\begin{align}
\label{9.9a}
\|  \wh{J}_2(\mathbf{k}, \eps^{-1} \tau ) \|_{L_2(\Omega) \to L_2(\Omega)}
\le \wh{\mathcal C}_2' (1+ \eps^{-1/2}|\tau|^{1/2}).
\end{align}
The constant $\wh{\mathcal C}_2'$ depends only on $\alpha_0,$ $\alpha_1,$ $\|g\|_{L_\infty},$  
$\|g^{-1}\|_{L_\infty},$ and $r_0$.
\end{proposition}

\begin{proof}
From  \eqref{2.6} (with $\tau$ replaced  by $\eps^{-1} \tau$) it follows that 
\begin{equation}
\label{9.9b}
\|  \wh{J}_2(\mathbf{k}, \eps^{-\!1} \tau ) \wh{P} \|_{L_2(\Omega) \to L_2(\Omega)}
\!\!\le\! \wh{\mathcal C}_2^{(1)} (1\!+ \!\eps^{-1}|\tau| |\k|), \quad \tau\! \in \!\R,\ \eps\!>\!0, \ |\k|\! \le \!\wh{t}_0.
\end{equation}

Next, for $|\k| \le \wh{t}_0$ the norms of the operators  $\wh{\mathcal A}(\k)^{-1/2} (I - \wh{P})$ and
\hbox{$\wh{\mathcal A}^0(\k)^{-1/2} (I - \wh{P})$} are uniformly bounded (see \eqref{abstr_F(t)_threshold_1},  \eqref{A(k)_nondegenerated_and_c_*}), whence 
\begin{equation}
\label{9.9d}
\|  \wh{J}_2(\mathbf{k}, \eps^{-1} \tau ) (I-\wh{P}) \|_{L_2(\Omega) \to L_2(\Omega)}
\le \wh{\mathcal C}_2^{(2)}, \quad \tau \in \R,\ \eps>0, \ |\k| \le \wh{t}_0.
\end{equation}

If $\eps |\tau|^{-1} > \wh{t}_0^{\,2}$, then  estimate  \eqref{9.9a} follows directly from the second inequality in \eqref{9.9}. Assume that  $\eps |\tau|^{-1} \le \wh{t}_0^{\,2}$. Then, by \eqref{9.9b}, 
\begin{equation*}
\|  \wh{J}_2(\mathbf{k}, \eps^{-1} \tau ) \wh{P} \|_{L_2(\Omega) \to L_2(\Omega)}
\le \wh{\mathcal C}_2^{(1)} (1+ \eps^{-1/2}|\tau|^{1/2}), \quad  |\k| \le \eps^{1/2} |\tau|^{-1/2}.
\end{equation*}
Together with  \eqref{9.9d} this implies estimate   \eqref{9.9a}  for $|\k| \le \eps^{1/2} |\tau|^{-1/2}$.
 
The required estimate for $|\k| > \eps^{1/2} |\tau|^{-1/2}$ follows from \eqref{A(k)_nondegenerated_and_c_*} 
(for $\wh{\mathcal A}(\k)$ and $\wh{\mathcal A}^0(\k)$):
\begin{equation*}
\|  \wh{J}_2(\mathbf{k}, \eps^{-1} \tau )  \|_{L_2(\Omega) \to L_2(\Omega)}
\le 2 \wh{c}_*^{-1/2} |\k|^{-1} 
\le 2 \wh{c}_*^{-1/2} \eps^{-1/2} |\tau|^{1/2}, \quad 
|\k| > \eps^{1/2} |\tau|^{-1/2}.\qedhere
\end{equation*}
\end{proof}

\subsection{Approximation in the operator norm in  $L_2(\Omega;\AC^n)$. The case where
 $\wh{N}(\boldsymbol{\theta}) =0$}

Now we improve the result of Theorem \ref{th9.1} under the additional assumptions. We impose the following condition.

\begin{condition}
\label{cond_B}
	 Let $\widehat{N}(\boldsymbol{\theta})$ be the operator  defined by~\emph{(\ref{N(theta)})}. Suppose that 
	 \hbox{$\widehat{N}(\boldsymbol{\theta}) = 0$} for all $\boldsymbol{\theta} \in \mathbb{S}^{d-1}$. 
\end{condition}

\begin{theorem} 
\label{th9.2}
Let $\wh{J}_1(\mathbf{k}, \tau )$ and $\wh{J}_2(\mathbf{k}, \tau )$ be the operators defined by  \eqref{J1_hat}\textup, \eqref{J2_hat}. Suppose that Condition \emph{\ref{cond_B}} is satisfied.
Then for $\tau \in \mathbb{R},$ $\varepsilon > 0,$ and $\mathbf{k} \in \widetilde{\Omega}$  we have 
\begin{align}
\label{9.7a}
& \|  \wh{J}_1(\mathbf{k}, \eps^{-1} \tau ) \mathcal{R}(\k,\eps)^{3/4} \|_{L_2(\Omega) \to L_2(\Omega)}
\le \wh{\mathcal C}_3 (1+ |\tau|)^{1/2} \eps,
\\
\label{9.8a}
& \|  \wh{J}_2(\mathbf{k}, \eps^{-1} \tau ) \mathcal{R}(\k,\eps)^{1/4}\|_{L_2(\Omega) \to L_2(\Omega)}
\le \wh{\mathcal C}_4 (1+ |\tau|)^{1/2}.
\end{align}
The constants  $\wh{\mathcal C}_3$ and $\wh{\mathcal C}_4$ depend only on  $\alpha_0,$ $\alpha_1,$ $\|g\|_{L_\infty},$  $\|g^{-1}\|_{L_\infty}$, and $r_0$.
\end{theorem}

\begin{proof}
We start with the proof of inequality \eqref{9.7a}.
Applying \eqref{2.7a} and taking   \eqref{hatS_P=hatA^0_P}  and \eqref{R_P} into account, we have
\begin{equation}
\label{9.7b}
\|  \wh{J}_1(\mathbf{k}, \eps^{-1} \tau ) \mathcal{R}(\k,\eps)^{3/4} \wh{P}\|_{L_2(\Omega) \to L_2(\Omega)}
\le \wh{\mathcal C}^\circ_3 (1+ |\tau|)^{1/2} \eps,
\quad  \tau \in \R,\ \eps>0, \quad |\k| \le \wh{t}_0.
\end{equation}
 From \eqref{R_P_2} with $s=1$ and the first estimate in \eqref{9.9} we see that the left-hand side in \eqref{9.7b}
does not exceed $2 (\wh{t}_0)^{-1} \eps$ for  $|\k| > \wh{t}_0$.
Finally, by \eqref{R(k,eps)(I-P)_est} with $s=1$ and the first estimate in \eqref{9.9}, the quantity 
$\|  \wh{J}_1(\mathbf{k}, \eps^{-1} \tau ) \mathcal{R}(\k,\eps)^{3/4} (I-\wh{P})\|_{L_2(\Omega) \to L_2(\Omega)}$ 
does not exceed $2 r_0^{-1}\eps$ for all $\k \in \wt{\Omega}$.
 As a result, we arrive at  \eqref{9.7a}.

We proceed to the proof of estimate \eqref{9.8a}.
By \eqref{2.8a}, \eqref{hatS_P=hatA^0_P}, and  \eqref{R_P},
\begin{equation*}
\|  \wh{J}_2(\mathbf{k}, \eps^{-1} \tau ) \mathcal{R}(\k,\eps)^{1/4} \wh{P}\|_{L_2(\Omega) \to L_2(\Omega)}
\le \wh{\mathcal C}^\circ_4 (1+ |\tau|)^{1/2},\quad \tau \in \R,\ \eps>0, \ |\k| \le \wh{t}_0.
\end{equation*}
From  \eqref{9.9d} it follows that the quantity  
$\|  \wh{J}_2(\mathbf{k}, \eps^{-1} \tau ) \mathcal{R}(\k,\eps)^{1/4} (I-\wh{P}) 
\|$ is bounded by the constant $\wh{\mathcal C}_2^{(2)}$ for $\tau \in \R$, $\eps >0$, and $|\k| \le \wh{t}_0$.
Finally, for $\k \in \wt{\Omega}$ and $|\k| > \wh{t}_0$ the left-hand side of  \eqref{9.8a} does not exceed  
$2 \wh{c}_*^{-1/2} (\wh{t}_0)^{-1}$ due to estimate \eqref{A(k)_nondegenerated_and_c_*} 
(for the operators $\wh{\mathcal A}(\k)$ and  $\wh{\mathcal A}^0(\k)$).
As a result, we obtain  \eqref{9.8a}.
\end{proof}

We shall also need the following statement.

\begin{proposition}
\label{prop9.2a}
Under the assumptions of Theorem \emph{\ref{th9.2}}, for $\tau \in \R,$ $\eps >0,$ and $\k \in \wt{\Omega}$ 
we have
\begin{align}
\label{9.9xx}
\|  \wh{J}_2(\mathbf{k}, \eps^{-1} \tau ) \|_{L_2(\Omega) \to L_2(\Omega)}
\le \wh{\mathcal C}_4' (1+ \eps^{-1/3}|\tau|^{1/3}).
\end{align}
  The constant $\wh{\mathcal C}_4'$ depends only on  $\alpha_0,$ $\alpha_1,$ $\|g\|_{L_\infty},$  $\|g^{-1}\|_{L_\infty},$ and $r_0$.
\end{proposition}

\begin{proof}
 From \eqref{2.8} (with $\tau$ replaced by $\eps^{-1} \tau$) it follows that 
\begin{equation}
\label{9.9bb}
\|  \wh{J}_2(\mathbf{k}, \eps^{-1}\! \tau ) \wh{P} \|_{L_2(\Omega)\! \to\! L_2(\Omega)}
\le \wh{\mathcal C}_4^{(1)} (1\!+\! \eps^{-\!1}|\tau| |\k|^2), \quad \tau \!\in \!\R,\ \eps\!>\!0, \ |\k| \!\le\! \wh{t}_0.
\end{equation}

If $\eps |\tau|^{-1} > \wh{t}_0^{\,3}$, then   \eqref{9.9xx} directly follows from the second inequality in~\eqref{9.9}. Suppose that  $\eps |\tau|^{-1} \le \wh{t}_0^{\,3}$. Then \eqref{9.9bb} yields
\begin{equation*}
\|  \wh{J}_2(\mathbf{k}, \eps^{-1} \tau ) \wh{P} \|_{L_2(\Omega) \to L_2(\Omega)}
\le \wh{\mathcal C}_4^{(1)} (1+ \eps^{-1/3}|\tau|^{1/3}), \quad  |\k| \le \eps^{1/3} |\tau|^{-1/3}.
\end{equation*}
Together with  \eqref{9.9d}, this implies estimate   \eqref{9.9xx}  for $|\k| \le \eps^{1/3} |\tau|^{-1/3}$.

Finally, the required estimate for $|\k| > \eps^{1/3} |\tau|^{-1/3}$ follows from  \eqref{A(k)_nondegenerated_and_c_*}:
\begin{equation*}
\|  \wh{J}_2(\mathbf{k}, \eps^{-1} \tau )  \|_{L_2(\Omega) \to L_2(\Omega)}
\le 2 \wh{c}_*^{-1/2} |\k|^{-1} \le 2 \wh{c}_*^{-1/2} \eps^{-1/3} |\tau|^{1/3}, 
\quad |\k| > \eps^{1/3} |\tau|^{-1/3}.\qedhere
\end{equation*}
\end{proof}

\subsection{Approximation in the operator norm in  $L_2(\Omega;\AC^n)$. The case where  $\wh{N}_0(\boldsymbol{\theta}) =0$\label{sec9.3}}

Now we abandon  the assumption that  $\widehat{N}(\boldsymbol{\theta}) \equiv 0$, but instead we assume that $\widehat{N}_0(\boldsymbol{\theta}) = 0$ for all $\boldsymbol{\theta}$. We would like to apply 
 Theorem~\ref{th3.3}. However, a complication arises  because at some points $\boldsymbol{\theta}$ the multiplicity of the spectrum of the germ $\widehat{S} (\boldsymbol{\theta})$ may change. When approaching such points, the distance between  a pair of different eigenvalues of the germ tends to zero, and we cannot choose  the values  $\widehat{c}^{\circ}_{jl}$, $\widehat{t}^{00}_{jl}$ independent of  $\boldsymbol{\theta}$. Therefore, we are forced to impose an additional condition. It is necessary to take care only about those eigenvalues  for which the corresponding term in representation~(\ref{N*_invar_repr}) is nonzero. Now it is more convenient to use the initial numbering of the eigenvalues of the germ $\widehat{S} (\boldsymbol{\theta})$, agreeing to number them in the nondecreasing order:
$\widehat{\gamma}_1 (\boldsymbol{\theta}) \le  \ldots \le \widehat{\gamma}_n (\boldsymbol{\theta})$. For each $\boldsymbol{\theta}$, by $\widehat{P}^{(k)} (\boldsymbol{\theta})$ we denote the orthogonal projection 
of the space $L_2 (\Omega; \mathbb{C}^n)$ onto the eigenspace of the operator $\widehat{S} (\boldsymbol{\theta})$ corresponding to the eigenvalue $\widehat{\gamma}_k (\boldsymbol{\theta})$. It is clear that for every 
$\boldsymbol{\theta}$ the operator $\widehat{P}^{(k)} (\boldsymbol{\theta})$ coincides with one of the projections  $\widehat{P}_j (\boldsymbol{\theta})$ introduced in Subsection~\ref{eigenval_multipl_section} (but the number $j$ may depend on $\boldsymbol{\theta}$ and changes at points of change in the multiplicity  of the germ spectrum).

\begin{condition}
	\label{cond1}
		$1^\circ$. $\widehat{N}_0(\boldsymbol{\theta})=0$ for all $\boldsymbol{\theta} \in \mathbb{S}^{d-1}$.
		
		\noindent$2^\circ$.  For each pair of indices $(k,r), 1 \le k,r \le n, k \ne r,$ such that 
		 $\widehat{\gamma}_k (\boldsymbol{\theta}_0)\! =\! \widehat{\gamma}_r (\boldsymbol{\theta}_0) $ for some  $\boldsymbol{\theta}_0 \!\in\! \mathbb{S}^{d-1},$ we have ${ {\widehat{P}}^{(k)} (\boldsymbol{\theta}) \widehat{N} (\boldsymbol{\theta}) {\widehat{P}}^{(r)} (\boldsymbol{\theta})\! =\! 0}$ for all $\boldsymbol{\theta} \in \mathbb{S}^{d-1}$.    
	\end{condition}

Assumption $2^\circ$ can be reformulated as follows: we require that, for nonzero (identically)
``blocks''  ${\widehat{P}}^{(k)} (\boldsymbol{\theta}) \widehat{N} (\boldsymbol{\theta}) {\widehat{P}}^{(r)} (\boldsymbol{\theta})$ of the operator $\widehat{N} (\boldsymbol{\theta})$, the branches of eigenvalues   
 $\widehat{\gamma}_k (\boldsymbol{\theta})$ and $\widehat{\gamma}_r (\boldsymbol{\theta})$ do not intersect. 
 Of course, Condition \ref{cond1} is ensured by the following more restrictive condition.

\begin{condition}
	\label{cond2}
	
		$1^\circ$.  $\widehat{N}_0(\boldsymbol{\theta})=0$ for all $\boldsymbol{\theta} \in \mathbb{S}^{d-1}$.
		
		\noindent$2^\circ$. The number $p$ of different eigenvalues of the spectral 
	germ~$\widehat{S}(\boldsymbol{\theta})$ does not depend on $\boldsymbol{\theta} \in \mathbb{S}^{d-1}$.        
\end{condition}

\begin{remark}
	The assumption~$2^\circ$ of Condition~{\ref{cond2}} is a fortiori satisfied  if the spectrum of the germ  $\widehat{S}(\boldsymbol{\theta})$ is simple for all $\boldsymbol{\theta} \in \mathbb{S}^{d-1}$.
\end{remark}

So, we assume that Condition~\ref{cond1} is satisfied. Denote
\begin{align*}
 \widehat{\mathcal{K}} &:= \{ (k,r) \colon 1 \le k,r \le n, \; k \ne r, \;  \widehat{P}^{(k)} (\boldsymbol{\theta}) \widehat{N} (\boldsymbol{\theta}) \widehat{P}^{(r)} (\boldsymbol{\theta}) \not\equiv 0 \},
\\
\widehat{c}^{\circ}_{kr} (\boldsymbol{\theta})&  := \min \{\widehat{c}_*, n^{-1} |\widehat{\gamma}_k (\boldsymbol{\theta}) - \widehat{\gamma}_r (\boldsymbol{\theta})| \}, \quad (k,r) \in \widehat{\mathcal{K}}.
\end{align*}
Since the operator $\widehat{S} (\boldsymbol{\theta})$ depends on $\boldsymbol{\theta} \in \mathbb{S}^{d-1}$ continuously (it is a polynomial of second order), then the perturbation theory of discrete spectrum  shows  that the functions $\widehat{\gamma}_j (\boldsymbol{\theta})$~are continuous on the sphere $\mathbb{S}^{d-1}$. By Condition~\ref{cond1}($2^\circ$), for $(k,r) \in \widehat{\mathcal{K}}$ we have~$|\widehat{\gamma}_k (\boldsymbol{\theta}) - \widehat{\gamma}_r (\boldsymbol{\theta})| > 0$ for all 
$\boldsymbol{\theta} \in \mathbb{S}^{d-1}$, whence 
$\widehat{c}^{\circ}_{kr} := \min_{\boldsymbol{\theta} \in \mathbb{S}^{d-1}} \widehat{c}^{\circ}_{kr} (\boldsymbol{\theta}) > 0$, $(k,r) \in \widehat{\mathcal{K}}$. We put 
\begin{equation}
\label{hatc^circ}
\widehat{c}^{\circ} := \min_{(k,r) \in \widehat{\mathcal{K}}} \widehat{c}^{\circ}_{kr}.
\end{equation}

Clearly, the number~(\ref{hatc^circ})~is a realization of~(\ref{abstr_c^circ}) chosen independent of 
$\boldsymbol{\theta}$.
Under Condition~\ref{cond1}, the number subject to~(\ref{abstr_t^00}) also can be chosen independent of $\boldsymbol{\theta} \in \mathbb{S}^{d-1}$. Taking~(\ref{hatdelta_fixation}) and~(\ref{hatX_1_estmate}) 
into account, we put
\begin{equation*}
\widehat{t}^{00} = (8 \beta_2)^{-1} r_0 \alpha_1^{-3/2} \alpha_0^{1/2} \| g\|_{L_{\infty}}^{-3/2} \| g^{-1}\|_{L_{\infty}}^{-1/2} \widehat{c}^{\circ}.
\end{equation*}
The condition $\widehat{t}^{00} \le \widehat{t}_{0}$ is valid automatically, since $\widehat{c}^{\circ} \le 
\| \widehat{S} (\boldsymbol{\theta}) \| \le \alpha_1 \|g\|_{L_{\infty}}$.

Under Condition~\ref{cond1}, we deduce the following result from Theorem~\ref{th3.3}, by analogy with the proof of Theorem~\ref{th9.2}. 
Now the constants in estimates will depend not only on $\alpha_0$, $\alpha_1$, $\|g\|_{L_\infty}$, $\| g^{-1}\|_{L_\infty}$, and  $r_0$, but also on $\wh{c}^\circ$ and $n$; see Remark~\ref{abstr_constants_remark}.

\begin{theorem} 
\label{th9.4}
Let $\wh{J}_1(\mathbf{k}, \tau )$ and $\wh{J}_2(\mathbf{k}, \tau )$ be the operators defined by \eqref{J1_hat}, \eqref{J2_hat}. 
Suppose that Condition~\emph{\ref{cond1}} \emph{(}or more restrictive Condition~\emph{\ref{cond2}}\emph{)} is satisfied.
 Then for  $\tau \in \mathbb{R},$ $\varepsilon > 0,$ and $\mathbf{k} \in \widetilde{\Omega}$  we have
{\allowdisplaybreaks
\begin{align*}
 \|  \wh{J}_1(\mathbf{k}, \eps^{-1} \tau ) \mathcal{R}(\k,\eps)^{3/4} \|_{L_2(\Omega) \to L_2(\Omega)}
&\le \wh{\mathcal C}_5 (1+ |\tau|)^{1/2} \eps,
\\
 \|  \wh{J}_2(\mathbf{k}, \eps^{-1} \tau ) \mathcal{R}(\k,\eps)^{1/4}\|_{L_2(\Omega) \to L_2(\Omega)}
&\le \wh{\mathcal C}_6 (1+ |\tau|)^{1/2}.
\end{align*}
}
\hspace{-3mm}
The constants $\wh{\mathcal C}_5$ and $\wh{\mathcal C}_6$ depend on $\alpha_0,$ $\alpha_1,$ $\|g\|_{L_\infty},$  $\|g^{-1}\|_{L_\infty},$  $r_0,$  $n,$  and $\widehat{c}^\circ$.
\end{theorem}

We also need the following statement; the proof is similar to the proof of  
Proposition~\ref{prop9.2a}.

\begin{proposition}
\label{prop9.3a}
Under the assumptions of Theorem \emph{\ref{th9.4}}, for $\tau \in \R,$ $\eps >0,$ and $\k \in \wt{\Omega}$ 
we have
\begin{align*}
\|  \wh{J}_2(\mathbf{k}, \eps^{-1} \tau ) \|_{L_2(\Omega) \to L_2(\Omega)}
\le \wh{\mathcal C}_6' (1+ \eps^{-1/3}|\tau|^{1/3}).
\end{align*}
The constant  $\wh{\mathcal C}_6'$ depends on $\alpha_0,$ $\alpha_1,$ $\|g\|_{L_\infty},$  
$\|g^{-1}\|_{L_\infty},$ $r_0,$ $n,$ and $\widehat{c}^\circ$.
\end{proposition}

\subsection{Approximation of the operator $\widehat{\mathcal{A}}(\mathbf{k})^{-1/2} \sin(\varepsilon^{-1} \tau \widehat{\mathcal{A}} (\mathbf{k})^{1/2})$ in the \hbox{``energy''} norm}
Now we apply Theorem~\ref{abstr_sin_general_thrm} to the operator
$\widehat{A}(t, \boldsymbol{\theta}) = \widehat{\mathcal{A}}(\mathbf{k})$ and take 
Remark~\ref{abstr_constants_remark} into account. By \eqref{hat_Z_theta},  
\begin{equation}
\label{tZP_realization}
t \widehat{Z}(\boldsymbol{\theta}) \widehat{P} = \Lambda b(\mathbf{k}) \widehat{P} = \Lambda b(\mathbf{D}+\mathbf{k}) \widehat{P}.
\end{equation}
Denote 
\begin{equation}
\label{hatJ}
\widehat{J}(\mathbf{k}, \tau)
 := \widehat{\mathcal{A}}(\mathbf{k})^{-1/2} \sin ( \tau \widehat{\mathcal{A}}(\mathbf{k})^{1/2}) 
- (I + \Lambda b(\mathbf{D} + \mathbf{k}) \wh{P}) \widehat{\mathcal{A}}^0(\mathbf{k})^{-1/2} \sin ( \tau \widehat{\mathcal{A}}^0(\mathbf{k})^{1/2}).
\end{equation}
Applying Theorem~\ref{abstr_sin_general_thrm}, we have 
\begin{equation}
\label{sin_est_|k|<t^0}
\bigl\| \widehat{\mathcal{A}}(\mathbf{k})^{1/2} \widehat{J}(\mathbf{k}, \varepsilon^{-1} \tau) \mathcal{R}(\mathbf{k}, \varepsilon)  \widehat{P} \bigr\|_{L_2(\Omega) \to L_2 (\Omega) }  \le \widehat{\mathcal{C}}'_7 (1 +  |\tau|) \varepsilon,  \quad \varepsilon > 0, \; \tau \in \mathbb{R}, \; |\mathbf{k} | \le \widehat{t}_0.
\end{equation}
The constant $\widehat{\mathcal{C}}'_7$ depends only on  $\alpha_0$, $\alpha_1$, $\|g\|_{L_\infty}$, $\|g^{-1}\|_{L_\infty}$, and $r_0$.

Estimates for $| \mathbf{k} | > \widehat{t}_0$ are trivial. Obviously, for  
$\varepsilon > 0$, $\tau \in \mathbb{R}$,  and $\mathbf{k}  \in \wt{\Omega}$ we have
\begin{multline}
\label{9.21a}
\bigl\| \widehat{\mathcal{A}}(\mathbf{k})^{1/2} \widehat{J}(\mathbf{k}, \varepsilon^{-1} \tau) \mathcal{R}(\mathbf{k}, \varepsilon)^{1/2}  \widehat{P} \bigr\|_{L_2(\Omega) \to L_2 (\Omega) } 
 \le \bigl\|  \mathcal{R}(\mathbf{k}, \varepsilon)^{1/2}  \widehat{P} \bigr\|
\\
 \times
 \left( 1 + \bigl\| \widehat{\mathcal{A}}(\mathbf{k})^{1/2} \widehat{\mathcal{A}}^0(\mathbf{k})^{-1/2}
 \bigr\|
 + \bigl\| \widehat{\mathcal{A}}(\mathbf{k})^{1/2} \Lambda b(\D+\k) \wh{P} \widehat{\mathcal{A}}^0(\mathbf{k})^{-1/2}\bigr\|\right).
\end{multline}
By~(\ref{hatA(k)}) and~(\ref{g^0_est}), 
\begin{equation}
\label{sin_est_|k|>t^0_1}
\| \widehat{\mathcal{A}}(\mathbf{k})^{1/2}  \widehat{\mathcal{A}}^0(\mathbf{k})^{-1/2} \|
 = \| g^{1/2} b(\mathbf{D} + \mathbf{k}) 
 \widehat{\mathcal{A}}^0(\mathbf{k})^{-1/2} \| \le 
 \|g\|_{L_\infty}^{1/2} \|g^{-1}\|_{L_\infty}^{1/2}, \quad	 \mathbf{k} \in \widetilde{\Omega}.
\end{equation}
Next, we use the estimate 
\begin{equation}
\label{sqrt_hatA_Lambda_P_est}
\| \widehat{\mathcal{A}} (\mathbf{k})^{1/2} \Lambda \widehat{P}_m \|_{L_2 (\Omega) \to L_2 (\Omega)} \le C_\Lambda, \quad \mathbf{k} \in \widetilde{\Omega},
\end{equation}
where $\widehat{P}_m$~is the orthogonal projection of the space $\mathfrak{H}_* = L_2(\Omega; \mathbb{C}^m)$ onto the subspace of constants, and 
$C_\Lambda = \|g\|_{L_\infty}^{1/2} \bigl( 1 + \alpha_1^{1/2} r_1 M_1  \bigr)$.
It is easy to check this estimate using \eqref{rank_alpha_ineq}, \eqref{6.11a}, and \eqref{Lambda_est}. Then 
\begin{equation}
\label{sin_est_|k|>t^0_2}
\begin{split}
&\| \widehat{\mathcal{A}}(\mathbf{k})^{1/2} \Lambda b(\mathbf{D} + \mathbf{k}) \wh{P} \widehat{\mathcal{A}}^0(\mathbf{k})^{-1/2}  \|_{L_2 (\Omega) \to L_2 (\Omega)} 
\\
 &\quad\le C_\Lambda
\|b(\mathbf{D} + \mathbf{k}) \widehat{\mathcal{A}}^0(\mathbf{k})^{-1/2} \|_{L_2 (\Omega) \to L_2 (\Omega)} 
\le C_{\Lambda} \|g^{-1}\|_{L_\infty}^{1/2}, \quad \mathbf{k} \in \widetilde{\Omega}.
\end{split}
\end{equation}
As a result, from~\eqref{R_P_2} with $s=1$, \eqref{9.21a}, \eqref{sin_est_|k|>t^0_1}, and \eqref{sin_est_|k|>t^0_2} 
it follows that
\begin{equation}
\label{sin_est_|k|>t^0}
\bigl\| \widehat{\mathcal{A}}(\mathbf{k})^{1/2} \widehat{J}(\mathbf{k}, \varepsilon^{-1} \tau) 
\mathcal{R}(\mathbf{k}, \varepsilon)^{1/2}  \widehat{P} \bigr\|_{L_2(\Omega) \to L_2 (\Omega) }  \le  
\wh{\mathcal C}''_7 \eps,
 \quad \varepsilon > 0, \ \tau \in \mathbb{R}, \ \mathbf{k} \in \widetilde{\Omega}, \ |\mathbf{k}| > \widehat{t}_0,
\end{equation}
where $\wh{\mathcal C}''_7 =  (\widehat{t}_0)^{-1} (1 +  \|g\|_{L_\infty}^{1/2} \|g^{-1}\|_{L_\infty}^{1/2} + C_{\Lambda} \|g^{-1}\|_{L_\infty}^{1/2})$.

Now we estimate the operator 
\begin{align*}
&\widehat{\mathcal{A}}(\mathbf{k})^{1/2} \widehat{J}(\mathbf{k}, \varepsilon^{-1} \tau) \mathcal{R}(\mathbf{k}, \varepsilon)^{1/2}  (I\!\!-\!\!\widehat{P}) 
\\
&\!=\!
( \sin(\eps^{-1} \tau \widehat{\mathcal{A}}(\mathbf{k})^{1/2})\! -\! \widehat{\mathcal{A}}(\mathbf{k})^{1/2}
\widehat{\mathcal{A}}^0(\mathbf{k})^{-1/2}
\!\sin(\eps^{-1} \tau \widehat{\mathcal{A}}^0(\mathbf{k})^{-1/2})) \mathcal{R}(\mathbf{k}, \varepsilon)^{1/2}  (I\!\!- \!\!\widehat{P}).
\end{align*}
Applying \eqref{R(k,eps)(I-P)_est} with $s=1$ and \eqref{sin_est_|k|>t^0_1}, for 
$\varepsilon > 0$, $\tau \in \mathbb{R}$, and $\mathbf{k} \in \widetilde{\Omega}$ we have 
\begin{equation}
\label{sin_est_wo_hatP}
\| \widehat{\mathcal{A}}(\mathbf{k})^{1/2} \widehat{J}(\mathbf{k}, \varepsilon^{-1} \tau) \mathcal{R}(\mathbf{k}, \varepsilon)^{1/2}  (I-\widehat{P}) \|_{L_2 (\Omega) \to L_2 (\Omega)}  \le 
\wh{\mathcal C}'''_7 \eps,
\end{equation}
where  $\wh{\mathcal C}'''_7 =  r_0^{-1} (1 +  \|g\|_{L_\infty}^{1/2} \|g^{-1}\|_{L_\infty}^{1/2})$.

\smallskip

Relations~(\ref{sin_est_|k|<t^0}), (\ref{sin_est_|k|>t^0}), and~(\ref{sin_est_wo_hatP}) (see~\hbox{\cite[(7.36)]{M}}) 
imply the following result.

\begin{theorem}[see~\cite{M}]
	\label{A(k)_sin_general_thrm}
Suppose that  $\widehat{J}(\mathbf{k}, \tau)$ is the operator defined by~\eqref{hatJ}. For $\tau \in \mathbb{R},$ $\varepsilon > 0,$ and $\mathbf{k} \in \widetilde{\Omega}$ we have
	\begin{equation*}
	\bigl\| \widehat{\mathcal{A}}(\mathbf{k})^{1/2} \widehat{J}(\mathbf{k}, \varepsilon^{-1} \tau) \mathcal{R}(\mathbf{k}, \varepsilon) \bigr\|_{L_2(\Omega) \to L_2 (\Omega) }  \le \widehat{\mathcal{C}}_7 (1 + |\tau|) \varepsilon.
	\end{equation*}
The constant $\widehat{\mathcal{C}}_7$ depends only on $\alpha_0,$ $\alpha_1,$ $\|g\|_{L_\infty},$ $\|g^{-1}\|_{L_\infty},$ $r_0$, and $r_1$.
\end{theorem}

\subsection{Approximation of the operator $\widehat{\mathcal{A}}(\mathbf{k})^{-1/2} \sin(\varepsilon^{-1} \tau \widehat{\mathcal{A}} (\mathbf{k})^{1/2})$ in the \hbox{energy} norm. Improvement of the results}
Under Condition \ref{cond_B}, we apply Theorem~\ref{th3.5}.
Taking~(\ref{hatS_P=hatA^0_P}) and~(\ref{R_P}) into account, for $\tau \in \mathbb{R}$,  $\varepsilon > 0$, and
 $|\mathbf{k}| \le \widehat{t}_0$ we have 
\begin{equation*} 
\bigl\| \widehat{\mathcal{A}}(\mathbf{k})^{1/2} \widehat{J}(\mathbf{k}, \varepsilon^{-1} \tau) \mathcal{R}(\mathbf{k}, \varepsilon)^{3/4}  \widehat{P} \bigr\|_{L_2(\Omega) \to L_2 (\Omega) }   \le  \widehat{\mathcal{C}}'_8(1 + | \tau |)^{1/2} \varepsilon.
\end{equation*} 
Here $\widehat{\mathcal{C}}'_8$ depends on $\alpha_0$, $\alpha_1$, $\|g\|_{L_\infty}$, $\|g^{-1}\|_{L_\infty}$, and $r_0$. Together with~(\ref{sin_est_|k|>t^0}) and~(\ref{sin_est_wo_hatP}) this implies the following result.

\begin{theorem}
	\label{A(k)_sin_enchanced_thrm_1}
 Let $\widehat{J}(\mathbf{k}, \tau)$ be the operator defined by  \eqref{hatJ}.
Suppose that Condition \emph{\ref{cond_B}} is satisfied.
Then for $\tau \in \mathbb{R},$ $\varepsilon > 0,$ and $\mathbf{k} \in \widetilde{\Omega}$ we have 
	\begin{equation*} 
	\bigl\| \widehat{\mathcal{A}}(\mathbf{k})^{1/2} \widehat{J}(\mathbf{k}, \varepsilon^{-1} \tau) \mathcal{R}(\mathbf{k}, \varepsilon)^{3/4} \bigr\|_{L_2(\Omega) \to L_2 (\Omega) } \le  \widehat{\mathcal{C}}_8 (1 + | \tau |)^{1/2} \varepsilon.
	\end{equation*} 
	The constant $\widehat{\mathcal{C}}_8$ depends only on $\alpha_0,$ $\alpha_1,$ $\|g\|_{L_\infty},$ $\|g^{-1}\|_{L_\infty},$ $r_0$, and $r_1$.	
\end{theorem}

Similarly, combining Theorem~\ref{abstr_sin_enchanced_thrm_2},  
\eqref{sin_est_wo_hatP}, and the analog of \eqref{sin_est_|k|>t^0}  (with $\widehat{t}_0$ replaced by  $\widehat{t}^{00}$), we arrive at the following result.

\begin{theorem}
	\label{A(k)_sin_enchanced_thrm_2}
		Let $\widehat{J}(\mathbf{k}, \tau)$ be the operator defined by  \eqref{hatJ}.
Suppose that Condition~\emph{\ref{cond1}} \emph{(}or more restrictive Condition~\emph{\ref{cond2}}\emph{)} is satisfied. Then for $\tau \in \mathbb{R},$ $\varepsilon > 0,$ and $\mathbf{k} \in \widetilde{\Omega}$ we have 
	\begin{equation*} 
	\bigl\| \widehat{\mathcal{A}}(\mathbf{k})^{1/2} \widehat{J}(\mathbf{k}, \varepsilon^{-1} \tau) \mathcal{R}(\mathbf{k}, \varepsilon)^{3/4}  \bigr\|_{L_2(\Omega) \to L_2 (\Omega) } \le \widehat{\mathcal{C}}_9 
	(1 + | \tau|)^{1/2} \varepsilon.
	\end{equation*}
The constant $\widehat{\mathcal{C}}_9$ depends on $\alpha_0,$ $\alpha_1,$ $\|g\|_{L_\infty},$ $\|g^{-1}\|_{L_\infty},$ $r_0,$ $r_1,$  $n,$ and $\widehat{c}^\circ$.
\end{theorem}

\section{Sharpness of the results of \S\ref{sec9} \label{sec10}}

\subsection{Sharpness of the results regarding  the smoothing factor}

In the statements of the present section, we impose one of the following two conditions.

\begin{condition}
\label{cond10.1}
	Let $\widehat{N}_0 (\boldsymbol{\theta})$ be the operator defined by~\emph{(\ref{N0_invar_repr})}. Suppose that  $\widehat{N}_0 (\boldsymbol{\theta}_0) \ne 0$ at least for one point $\boldsymbol{\theta}_0 \in \mathbb{S}^{d-1}$. 
\end{condition}

\begin{condition}
\label{cond10.2}
	Let $\widehat{N}_0 (\boldsymbol{\theta})$ and $\widehat{\mathcal N}^{(q)} (\boldsymbol{\theta})$ be the operators defined by~\eqref{N0_invar_repr} and \eqref{8.32aa}, respectively. Suppose that  
	$\widehat{N}_0 (\boldsymbol{\theta})=0$ for all $\boldsymbol{\theta} \in \mathbb{S}^{d-1}$.
	Suppose that $\widehat{\mathcal N}^{(q)} (\boldsymbol{\theta}_0) \ne 0$ for some $\boldsymbol{\theta}_0 \in \mathbb{S}^{d-1}$ and some $q \in \{ 1,\dots, p(\boldsymbol{\theta}_0)\}$.
\end{condition}

We need the following lemma (see~\cite[Lemma~7.9]{DSu}).

\begin{lemma}[see~\cite{DSu}]
	\label{hatA(k)_Lipschitz_lemma}
 Let $\widehat{\delta}$ and $\widehat{t}_0$ be defined by~\emph{(\ref{hatdelta_fixation})}
	and~\emph{(\ref{hatt0_fixation})}, respectively. Let $\widehat{F} (\mathbf{k})$~be the spectral projection of the operator  $\widehat{\mathcal{A}}(\mathbf{k})$ for the interval $[0, \widehat{\delta}]$. 
	Then for $|\mathbf{k}| \le \widehat{t}_0$ and $|\mathbf{k}_0| \le \widehat{t}_0$ we have
	\begin{align*}
	&\| \widehat{\mathcal{A}}(\mathbf{k})^{1/2} \widehat{F}(\mathbf{k})  - \widehat{\mathcal{A}}(\mathbf{k}_0)^{1/2} \widehat{F}(\mathbf{k}_0)\|_{L_2(\Omega) \to L_2 (\Omega) } \le \widehat{C}' | \mathbf{k} - \mathbf{k}_0|,
	\\
	&\|  \cos ( \tau \widehat{\mathcal{A}}(\mathbf{k})^{1/2}) \widehat{F}(\mathbf{k})  - 
	\cos ( \tau \widehat{\mathcal{A}}(\mathbf{k}_0)^{1/2}) \widehat{F}(\mathbf{k}_0)\|_{L_2(\Omega) \to L_2 (\Omega) } \le \widehat{C}''(\tau) | \mathbf{k} - \mathbf{k}_0|,
	\\
	&\| \widehat{\mathcal{A}}(\mathbf{k})^{-1/2} \sin ( \tau \widehat{\mathcal{A}}(\mathbf{k})^{1/2}) \widehat{F}(\mathbf{k})  - \widehat{\mathcal{A}}(\mathbf{k}_0)^{-1/2} \sin ( \tau \widehat{\mathcal{A}}(\mathbf{k}_0)^{1/2}) \widehat{F}(\mathbf{k}_0)\|_{L_2(\Omega) \to L_2 (\Omega) }
 \le \widehat{C}'''(\tau) | \mathbf{k} - \mathbf{k}_0|.
	\end{align*}
\end{lemma}

  The following theorem proved in  \cite[Theorem 7.8]{DSu} shows that  Theorem~\ref{th9.1} is sharp.
(This result is deduced from Theorem \ref{th3.8} and Lemma~\ref{hatA(k)_Lipschitz_lemma}.)

\begin{theorem}[see \cite{DSu}]
	\label{th10.1}
	Suppose that Condition \emph{\ref{cond10.1}} is satisfied.
		
\noindent $1^\circ.$
Let $0 \ne \tau \in \mathbb{R}$ and $0 \le s < 2$. Then there does not exist a constant $\mathcal{C} (\tau) > 0$
such that the estimate 
\begin{equation}
	\label{10.1a}
	\bigl\| \wh{J}_1 (\mathbf{k},  \varepsilon^{-1} \tau)  \mathcal{R}(\mathbf{k}, \varepsilon)^{s/2} \bigr\|_{L_2(\Omega) \to L_2 (\Omega) }  \le \mathcal{C} (\tau) \varepsilon
	\end{equation}
holds for almost all  $\mathbf{k}  \in \widetilde{\Omega}$ and sufficiently small $\varepsilon > 0$.

\noindent $2^\circ.$
Let $0 \ne \tau \in \mathbb{R}$ and $0 \le r < 1$. Then there does not exist a constant 
 $\mathcal{C} (\tau) > 0$ such that the estimate
\begin{equation}
	\label{10.2a}
	\bigl\| \wh{J}_2 (\mathbf{k},  \varepsilon^{-1} \tau)  \mathcal{R}(\mathbf{k}, \varepsilon)^{r/2} \bigr\|_{L_2(\Omega) \to L_2 (\Omega) }  \le \mathcal{C} (\tau)
	\end{equation}
holds for almost all $\mathbf{k}  \in \widetilde{\Omega}$ and sufficiently small $\varepsilon > 0$.
	\end{theorem}

Now we confirm the sharpness of Theorems \ref{th9.2} and \ref{th9.4}, relying on Theorem~\ref{th3.9}.

\begin{theorem}
	\label{th10.2a}
	Suppose that Condition \emph{\ref{cond10.2}} is satisfied.
	
\noindent $1^\circ.$
Let $0 \ne \tau \in \mathbb{R}$ and $0 \le s < 3/2$. Then there does not exist a constant $\mathcal{C} (\tau) > 0$ such that estimate \eqref{10.1a} holds for almost all $\mathbf{k} \in \widetilde{\Omega}$ and sufficiently small $\varepsilon > 0$.

\noindent $2^\circ.$
 Let $0 \ne \tau \in \mathbb{R}$ and $0 \le r < 1/2$. Then there does not exist a constant  $\mathcal{C} (\tau) > 0$ such that estimate \eqref{10.2a} holds for almost all $\mathbf{k} \in \widetilde{\Omega}$ and sufficiently small 
 $\varepsilon > 0$.
\end{theorem}

\begin{proof}
Let us check statement $1^\circ$. It suffices to assume that \hbox{$1\le s < 3/2$}.
We prove by contradiction. Suppose that for some $\tau \ne 0$ and $1 \le s < 3/2$ there exists a constant  $\mathcal{C}(\tau) > 0$ such that estimate~\eqref{10.1a} holds for almost all $\mathbf{k} \in \widetilde{\Omega}$ 
and sufficiently small $\varepsilon$. 
Multiplying the operator under the norm sign in  \eqref{10.1a} by $\wh{P}$ and using~(\ref{R_P}),
we see that the inequality
\begin{equation}
	\label{10.3a}
	\bigl\|  \bigl( \cos (\varepsilon^{-1} \tau \widehat{\mathcal{A}}(\mathbf{k})^{1/2})  - 
	 \cos (\varepsilon^{-1} \tau \widehat{\mathcal{A}}^0(\mathbf{k})^{1/2}) \bigr) \widehat{P} 
	  \bigr\| \varepsilon^s (|\mathbf{k}|^2 + \varepsilon^2)^{-s/2}  \le {\mathcal{C}}(\tau) \varepsilon
	\end{equation}	
holds for almost all $\mathbf{k} \in \widetilde{\Omega}$ and sufficiently small $\varepsilon$.
	(In the proof, we omit the index of the operator norm in $L_2(\Omega;\AC^n)$.)

Let $|\mathbf{k}| \le \widehat{t}_0$.	By~(\ref{abstr_F(t)_threshold_1}), 
	\begin{equation}
	\label{10.4a}
	\| \widehat{F} (\mathbf{k}) - \widehat{P} \|_{L_2 (\Omega) \to L_2 (\Omega)} \le \widehat{C}_1 |\mathbf{k}|, \quad |\mathbf{k}| \le \widehat{t}_0.
	\end{equation}
From \eqref{10.3a} and \eqref{10.4a} it follows that for some constant $\wt{\mathcal{C}}(\tau)>0$ 
the estimate  
\begin{equation}
	\label{10.5a}
	\|  \cos (\varepsilon^{-1} \tau \widehat{\mathcal{A}}(\mathbf{k})^{1/2})  \widehat{F} (\mathbf{k})\! - \!
	 \cos (\varepsilon^{-1} \tau \widehat{\mathcal{A}}^0(\mathbf{k})^{1/2})  \widehat{P}  
	 \| \varepsilon^s (|\mathbf{k}|^2\!\! +\! \varepsilon^2)^{-s/2} \!\! \le\! \wt{\mathcal{C}}(\tau) \varepsilon
	\end{equation}
holds for almost all $\mathbf{k}$ in the ball $|\mathbf{k}| \le \widehat{t}_0$ and sufficiently small 
$\varepsilon$.

 Note that the projection  $\widehat{P}$ is the spectral projection of the operator $\widehat{\mathcal{A}}^0 (\mathbf{k})$ for the interval $[0, \widehat{\delta}]$. Therefore,  from Lemma~\ref{hatA(k)_Lipschitz_lemma} (applied to $\widehat{\mathcal{A}} (\mathbf{k})$ and $\widehat{\mathcal{A}}^0 (\mathbf{k})$) it follows that, for fixed  $\tau$ and $\varepsilon$, 
the operator under the norm sign in~\eqref{10.5a} is continuous with respect to $\mathbf{k}$ in the ball 
$|\mathbf{k}| \le \widehat{t}_0$. Hence, estimate~\eqref{10.5a} is valid for all values of  
 $\mathbf{k}$ in this ball. In particular, it holds for $\mathbf{k} = t\boldsymbol{\theta}_0$ if $t \le \widehat{t}_0$. Applying~\eqref{10.4a} once again, we see that for some constant $\wh{\mathcal{C}}(\tau)>0$ the estimate
	\begin{equation}
	\label{10.6a}
	\|  \bigl(  \cos (\varepsilon^{-1} \tau \widehat{\mathcal{A}}(t\boldsymbol{\theta}_0)^{1/2}) \! - \! \cos (\varepsilon^{-1} \tau \widehat{\mathcal{A}}^0(t\boldsymbol{\theta}_0)^{1/2}) \bigr) \widehat{P}  
	\| \varepsilon^s (t^2\!\! + \!\varepsilon^2)^{-s/2} \le \widehat{\mathcal{C}}(\tau) \varepsilon
	\end{equation}
holds for all $t \le \widehat{t}_0$ and sufficiently small  $\varepsilon$.
	
 Estimate \eqref{10.6a} corresponds to the abstract estimate \eqref{**.1}. Since, by Condition~\ref{cond10.2}, 
 $\widehat{N}_0 (\boldsymbol{\theta}_0)=0$  and $\widehat{\mathcal N}^{(q)} (\boldsymbol{\theta}_0) \ne 0$, 
 the assumptions of Theorem \ref{th3.9} are satisfied.
	 Applying statement $1^\circ$ of this theorem, we arrive at a contradiction.

We proceed to the proof of statement $2^\circ$. 
 Suppose the opposite. Then for some $\tau \ne 0$ and \hbox{$0 \le r < 1/2$} we have 
\begin{equation}
	\label{10.7n}
	\bigl\|  \bigl(  \widehat{\mathcal{A}}(\mathbf{k})^{-1/2} \sin (\varepsilon^{-1} \tau \widehat{\mathcal{A}}(\mathbf{k})^{1/2})  - \widehat{\mathcal{A}}^0(\mathbf{k})^{-1/2}
	 \sin (\varepsilon^{-1} \tau \widehat{\mathcal{A}}^0(\mathbf{k})^{1/2}) \bigr) \widehat{P}  
	 \bigr\| 
	\varepsilon^r  (|\mathbf{k}|^2 + \varepsilon^2)^{-r/2}  \le {\mathcal{C}}(\tau) 
	\end{equation}
	for almost all $\mathbf{k} \in \widetilde{\Omega}$ and sufficiently small $\varepsilon$. Obviously,
\begin{equation}
	\label{10.7m}
	\bigl\|  \widehat{\mathcal{A}}(\mathbf{k})^{-1/2} \sin (\varepsilon^{-1} \tau \widehat{\mathcal{A}}(\mathbf{k})^{1/2})  \wh{F}(\k)^\perp   \bigr\| \le \wh{\delta}^{-1/2}.  
	\end{equation}
Combining this with  \eqref{10.7n},  we obtain  
(with some constant $\wt{\mathcal{C}}(\tau)>0$)
\begin{equation}
	\label{10.8a}
	\bigl\|  \bigl(  \widehat{\mathcal{A}}(\mathbf{k})^{-1/2} \sin (\varepsilon^{-1} \tau \widehat{\mathcal{A}}(\mathbf{k})^{1/2})  \wh{F}(\k) - \widehat{\mathcal{A}}^0(\mathbf{k})^{-1/2}
	 \sin (\varepsilon^{-1} \tau \widehat{\mathcal{A}}^0(\mathbf{k})^{1/2})\bigr) \widehat{P}  \bigr\|
 \varepsilon^r  (|\mathbf{k}|^2 + \varepsilon^2)^{-r/2}  \le \wt{\mathcal{C}}(\tau) 
	\end{equation}
for almost all $\mathbf{k}\in \wt{\Omega}$ and sufficiently small $\varepsilon$.

Let $|\mathbf{k}| \le \widehat{t}_0$.	
 From Lemma~\ref{hatA(k)_Lipschitz_lemma} (applied to $\widehat{\mathcal{A}} (\mathbf{k})$ and 
 $\widehat{\mathcal{A}}^0 (\mathbf{k})$) it follows that the operator under the norm sign in~\eqref{10.8a} is continuous with respect to $\mathbf{k}$ in the ball $|\mathbf{k}| \le \widehat{t}_0$. Hence, estimate~\eqref{10.8a} holds for all values of $\mathbf{k}$ in this ball. In particular, it is valid for $\mathbf{k} = t\boldsymbol{\theta}_0$ if $t \le \widehat{t}_0$. Applying~\eqref{10.7m} once again, we see that for some constant $\wh{\mathcal{C}}(\tau)>0$ 
 the inequality
\begin{equation}
	\label{10.9a}
	\bigl\|    \bigl(\widehat{\mathcal{A}}(t\boldsymbol{\theta}_0)^{-1/2} \sin (\varepsilon^{-1} \tau \widehat{\mathcal{A}}(t\boldsymbol{\theta}_0 )^{1/2})  -
	 \widehat{\mathcal{A}}^0(t\boldsymbol{\theta}_0 )^{-1/2}
	 \sin (\varepsilon^{-1} \tau \widehat{\mathcal{A}}^0(t\boldsymbol{\theta}_0 )^{1/2}) \bigr) \widehat{P}  \bigr\| 
	 \varepsilon^r  (t^2 + \varepsilon^2)^{-r/2}  \le \wh{\mathcal{C}}(\tau)
\end{equation}
	holds for all $t \le \widehat{t}_0$ and sufficiently small $\varepsilon$.

	 Estimate \eqref{10.9a} corresponds to the abstract estimate \eqref{**.2}. 
 	 Applying statement~$2^\circ$ of Theorem~\ref{th3.9}, we arrive at a contradiction.
\end{proof}

Application of Theorem~\ref{abstr_s<2_general_thrm} allows us to confirm that 
Theorem~\ref{A(k)_sin_general_thrm} is sharp.

\begin{theorem}
	\label{hat_s<2_sinA(k)_thrm}
	Suppose that Condition \emph{\ref{cond10.1}} is satisfied.
	Let $0 \ne \tau \in \mathbb{R}$ and \hbox{$0 \le s < 2$}. Then there does not exist a constant $\mathcal{C} (\tau) > 0$ such that the estimate 
	\begin{equation}
	\label{hat_s<2_sinA(k)_est}
	\bigl\| \widehat{\mathcal{A}}(\mathbf{k})^{1/2} \widehat{J}(\mathbf{k}, \varepsilon^{-1} \tau)  \mathcal{R}(\mathbf{k}, \varepsilon)^{s/2} \bigr\|_{L_2(\Omega) \to L_2 (\Omega) }  \le \mathcal{C} (\tau) \varepsilon
	\end{equation}
	holds for almost all $\mathbf{k} \in \widetilde{\Omega}$ and sufficiently small $\varepsilon > 0$.
\end{theorem}

\begin{proof}
	We prove by contradiction. Suppose that, for some $\tau \ne 0$ and $1 \le s < 2$ there exists a constant  $\mathcal{C}(\tau) > 0$ such that estimate~(\ref{hat_s<2_sinA(k)_est}) holds for almost all $\mathbf{k} \in \widetilde{\Omega}$ and sufficiently small $\varepsilon > 0$. Multiplying the operator in  
	\eqref{hat_s<2_sinA(k)_est} by $\wh{P}$ and using~\eqref{R_P}, we obtain 
	\begin{multline}
	\label{hat_s<2_proof_f1}
	\bigl\| \widehat{\mathcal{A}}(\mathbf{k})^{1/2} \bigl( \widehat{\mathcal{A}}(\mathbf{k})^{-1/2} \sin (\varepsilon^{-1} \tau \widehat{\mathcal{A}}(\mathbf{k})^{1/2})  
	- 	 (I + \Lambda b(\mathbf{D} + \mathbf{k}) \wh{P}) \widehat{\mathcal{A}}^0(\mathbf{k})^{-1/2} \sin (\varepsilon^{-1} \tau \widehat{\mathcal{A}}^0(\mathbf{k})^{1/2}) \bigr) \widehat{P}  \bigr\| 
	\\
	\times\varepsilon^s (|\mathbf{k}|^2 + \varepsilon^2)^{-s/2}  \le {\mathcal{C}}(\tau) \varepsilon
	\end{multline}
	for almost all $\mathbf{k} \in \widetilde{\Omega}$ and sufficiently small $\varepsilon > 0$.
	
	Let $|\mathbf{k}| \le \widehat{t}_0$. By~(\ref{abstr_sqrtA(t)F2(t)_est}),
	\begin{equation}
	\label{hat_s<2_proof_f3}
	\| \widehat{\mathcal{A}}(\mathbf{k})^{1/2} \widehat{F}_2 (\mathbf{k}) \|_{L_2 (\Omega) \to L_2 (\Omega)} \le \widehat{C}_{16} |\mathbf{k}|^2, \quad |\mathbf{k}| \le \widehat{t}_0.  
	\end{equation}
	Combining the formula $\widehat{P} + \Lambda b(\mathbf{D} + \mathbf{k}) \widehat{P} = (\widehat{F}(\mathbf{k}) - \widehat{F}_2(\mathbf{k})) \widehat{P}$ (see~(\ref{abstr_F(t)_threshold_2}), (\ref{abstr_F1_K0_N_invar}),  (\ref{tZP_realization})) and relations~(\ref{S_nondegenerated}), \eqref{10.4a},  \eqref{hat_s<2_proof_f1}, (\ref{hat_s<2_proof_f3}), we see that 
for some $\check{\mathcal{C}}(\tau)>0$ the inequality
	\begin{multline}
	\label{hat_s<2_proof_f4}
	\bigl\| \widehat{\mathcal{A}}(\mathbf{k})^{1/2} \widehat{F}(\mathbf{k}) \bigl( \widehat{\mathcal{A}}(\mathbf{k})^{-1/2} \sin (\varepsilon^{-1} \tau \widehat{\mathcal{A}}(\mathbf{k})^{1/2}) \widehat{F}(\mathbf{k})
	\\
	 -  \widehat{\mathcal{A}}^0(\mathbf{k})^{-1/2} \sin (\varepsilon^{-1} \tau \widehat{\mathcal{A}}^0(\mathbf{k})^{1/2})  \widehat{P}  \bigr) \bigr\|_{L_2(\Omega) \to L_2 (\Omega) }
	\\
	\times\varepsilon^s (|\mathbf{k}|^2 + \varepsilon^2)^{-s/2}  \le \check{\mathcal{C}}(\tau) \varepsilon
	\end{multline}
holds for almost all  $\mathbf{k}$  in the ball $|\mathbf{k}| \le \widehat{t}_0$ and sufficiently small $\varepsilon$.
	 
	From Lemma~\ref{hatA(k)_Lipschitz_lemma} (applied to $\widehat{\mathcal{A}} (\mathbf{k})$ and 
	$\widehat{\mathcal{A}}^0 (\mathbf{k})$) it follows that for fixed $\tau$ and $\varepsilon$ the operator under the norm sign in~(\ref{hat_s<2_proof_f4}) is continuous with respect to $\mathbf{k}$ in the ball $|\mathbf{k}| \le \widehat{t}_0$. Hence, estimate~(\ref{hat_s<2_proof_f4}) holds for all values of $\mathbf{k}$ in this ball. In particular, it is valid for $\mathbf{k} = t\boldsymbol{\theta}_0$ if $t \le \widehat{t}_0$. Applying the
 formula $(\widehat{F}(\mathbf{k}) - \widehat{F}_2(\mathbf{k})) \widehat{P} = \widehat{P} + \Lambda b(\mathbf{D} + \mathbf{k}) \widehat{P}$ and inequalities~(\ref{S_nondegenerated}), \eqref{10.4a}, 	\eqref{hat_s<2_proof_f3} once again, we obtain that  
	\begin{multline}
	\label{hat_s<2_proof_f5}
	\bigl\| \widehat{\mathcal{A}}(t\boldsymbol{\theta}_0)^{1/2} \bigl( \widehat{\mathcal{A}}(t\boldsymbol{\theta}_0)^{-1/2} \sin (\varepsilon^{-1} \tau \widehat{\mathcal{A}}(t\boldsymbol{\theta}_0)^{1/2})
	\\ - (I + \Lambda b(t\boldsymbol{\theta}_0)) \widehat{\mathcal{A}}^0(t\boldsymbol{\theta}_0)^{-1/2} \sin (\varepsilon^{-1} \tau \widehat{\mathcal{A}}^0(t\boldsymbol{\theta}_0)^{1/2}) \bigr) \widehat{P}  \bigr\|
	\\
	\times\varepsilon^s (t^2 + \varepsilon^2)^{-s/2} \le \check{\mathcal{C}}'(\tau) \varepsilon
	\end{multline}
	for all $t \le \widehat{t}_0$ and sufficiently small $\varepsilon$ (with some constant 
	$\check{\mathcal{C}}'(\tau)>0$).
	
In the abstract terms, estimate~(\ref{hat_s<2_proof_f5}) corresponds to estimate~(\ref{abstr_s<2_est_imp}).  Since, by Condition~\ref{cond10.1}, we have $\widehat{N}_0 (\boldsymbol{\theta}_0) \ne 0$, then  application of 
Theorem~\ref{abstr_s<2_general_thrm} leads to a contradiction.
\end{proof}

Similarly to the proof of Theorem~\ref{hat_s<2_sinA(k)_thrm}, from Theorem~\ref{th3.12} we deduce the following statement which confirms the sharpness of Theorems~\ref{A(k)_sin_enchanced_thrm_1} 
and~\ref{A(k)_sin_enchanced_thrm_2}.

\begin{theorem}
	\label{hat_s<3/2_sinA(k)_thrm}
	Suppose that Condition~\emph{\ref{cond10.2}} is satisfied.
	Let $0 \ne \tau \in \mathbb{R}$ and \hbox{$0 \le s < 3/2$}. Then there does not exist a constant 
	$\mathcal{C} (\tau) > 0$ such that estimate  \eqref{hat_s<2_sinA(k)_est} holds for almost all 
	$\mathbf{k}  \in \widetilde{\Omega}$ and sufficiently small $\varepsilon > 0$.
\end{theorem}

\subsection{Sharpness of the results with respect to time}

In the present subsection, we confirm that the results of \S\ref{sec9} are sharp regarding  the dependence of  estimates on $\tau$ (for large $|\tau|$). The following statement shows that Theorem~\ref{th9.1} is sharp.
It easily follows from Theorem \ref{th3.13} by using the same arguments as  in the proof of Theorem  \ref{th10.2a}.

\begin{theorem}
	\label{th10.6}
 Suppose that Condition \emph{\ref{cond10.1}} is satisfied.
	
\noindent $1^\circ.$ Let $s \ge 2$.
  Then there does not exist a positive function $\mathcal{C}(\tau)$ such that  $\lim_{\tau \to \infty} \mathcal{C}(\tau)/|\tau|=0$ and estimate \eqref{10.1a} holds for all  $\tau \in \R,$ almost all $\mathbf{k} \in \wt{\Omega},$ and sufficiently small $\varepsilon > 0$.

\noindent $2^\circ.$ Let $r \ge 1$. Then there does not exist a positive function $\mathcal{C}(\tau)$ such that  $\lim_{\tau \to \infty} \mathcal{C}(\tau)/|\tau|=0$ and estimate \eqref{10.2a} holds for all $\tau \in \R,$ almost all $\mathbf{k} \in \wt{\Omega},$ and sufficiently small  $\varepsilon > 0$.
\end{theorem}

Similarly, Theorem \ref{th4.*2} implies the following statement confirming the sharpness of 
Theorems \ref{th9.2} and \ref{th9.4}.

\begin{theorem}
	\label{th10.7}
Suppose that Condition \emph{\ref{cond10.2}} is satisfied.
	
\noindent $1^\circ.$ Let $s \ge 3/2$.
 There does not exist a positive function $\mathcal{C}(\tau)$ such that 
 $\lim_{\tau \to \infty} \mathcal{C}(\tau)/|\tau|^{1/2}=0$ and
  estimate \eqref{10.1a} holds for all $\tau \in \R,$ almost all $\mathbf{k} \in \wt{\Omega},$ and sufficiently small 
  $\varepsilon > 0$.

\noindent $2^\circ.$ Let $r \ge 1/2$.
 There does not exist a positive function $\mathcal{C}(\tau)$ such that 
 $\lim_{\tau \to \infty} \mathcal{C}(\tau)/|\tau|^{1/2}=0$ and estimate  \eqref{10.2a} holds for all
 $\tau \in \R,$ almost all $\mathbf{k} \in \wt{\Omega},$ and sufficiently small $\varepsilon > 0$.
\end{theorem}

The following result confirms that Theorem~\ref{A(k)_sin_general_thrm} is sharp. It can be deduced from 
Theorem~\ref{th4.6} by the same arguments as in the proof of Theorem~\ref{hat_s<2_sinA(k)_thrm}.

\begin{theorem}
	\label{th10.8}
Suppose that Condition \emph{\ref{cond10.1}} is satisfied. Let $s \ge 2$.
There does not exist a positive function $\mathcal{C}(\tau)$ such that 
$\lim_{\tau \to \infty} \mathcal{C}(\tau)/|\tau|=0$ and estimate  \eqref{hat_s<2_sinA(k)_est} holds for all $\tau \in \R,$ almost all $\mathbf{k} \in \wt{\Omega},$ and sufficiently small $\varepsilon > 0$.
\end{theorem}

Similarly, Theorem \ref{th4.final} implies the following statement demonstrating that Theorems \ref{A(k)_sin_enchanced_thrm_1} and \ref{A(k)_sin_enchanced_thrm_2} are sharp.

\begin{theorem}
	\label{th10.9}
	Suppose that Condition \emph{\ref{cond10.2}} is satisfied. Let $s \ge 3/2$. There does not exist a positive function  $\mathcal{C}(\tau)$ such that $\lim_{\tau \to \infty} \mathcal{C}(\tau)/|\tau|^{1/2}=0$ and estimate 
   \eqref{hat_s<2_sinA(k)_est} holds for all
   $\tau \in \R,$ almost all $\mathbf{k} \in \wt{\Omega},$ and sufficiently small $\varepsilon > 0$.
\end{theorem}

\section{The operator $\mathcal{A} (\mathbf{k})$. Application of the scheme of~\S\ref{abstr_sandwiched_section}}

\subsection{Application of the scheme of~\S\ref{abstr_sandwiched_section} to the operator $\mathcal{A} (\mathbf{k})$}

The operator $\mathcal{A} (\mathbf{k}) = f^* \widehat{\mathcal{A}} (\mathbf{k}) f$ is studied by the method of~\S\ref{abstr_sandwiched_section}. Now we have $\mathfrak{H} = \widehat{\mathfrak{H}} = L_2 (\Omega; \mathbb{C}^n)$ and $\mathfrak{H}_* = L_2 (\Omega; \mathbb{C}^m)$. The role of the operator $A(t)$ is played by  $A(t, \boldsymbol{\theta}) = \mathcal{A}(\mathbf{k})$, the role of $\widehat{A}(t)$ is played by the operator 
$\widehat{A}(t, \boldsymbol{\theta}) = \widehat{\mathcal{A}}(\mathbf{k})$. The isomorphism $M$ is the operator of multiplication by the matrix-valued function $f(\mathbf{x})$. The operator $Q$ is the operator of multiplication by the matrix-valued function $Q(\mathbf{x}) = (f (\mathbf{x}) f (\mathbf{x})^*)^{-1}$.
The block of the operator $Q$ in the subspace $\widehat{\mathfrak{N}}$ (see~(\ref{Ker3})) is the operator of multiplication by the constant matrix $\overline{Q} = (\underline{f f^*})^{-1} = |\Omega|^{-1} \int_{\Omega} (f (\mathbf{x}) f (\mathbf{x})^*)^{-1} d \mathbf{x}$. Next, $M_0$ is the operator of multiplication by the constant matrix 
\begin{equation}
\label{f_0}
f_0 = (\overline{Q})^{-1/2} = (\underline{f f^*})^{1/2}.
\end{equation}
Note that
\begin{equation}
\label{f_0_estimates}
| f_0 | \le \| f \|_{L_{\infty}}, \qquad | f_0^{-1} | \le \| f^{-1} \|_{L_{\infty}}.
\end{equation}

In $L_2 (\mathbb{R}^d; \mathbb{C}^n)$, define the operator
\begin{equation}
\label{A0}
\mathcal{A}^0 :=  f_0 \widehat{\mathcal{A}}^0 f_0 = f_0 b(\mathbf{D})^* g^0 b(\mathbf{D}) f_0.
\end{equation}
Let $\mathcal{A}^0 (\mathbf{k})$~be the corresponding operator family in $L_2 (\Omega; \mathbb{C}^n)$. Then  
\begin{equation}
\label{8.4a}
\mathcal{A}^0 (\mathbf{k}) = f_0 \widehat{\mathcal{A}}^0 (\mathbf{k}) f_0= f_0 b(\mathbf{D} + \mathbf{k})^* g^0 b(\mathbf{D} + \mathbf{k}) f_0
\end{equation}
with periodic boundary conditions.
By~(\ref{hatS_P=hatA^0_P}), 
\begin{equation}
\label{f_0 hatS f_0 P = A^0}
f_0 \widehat{S} (\mathbf{k}) f_0 \widehat{P} = \mathcal{A}^0 (\mathbf{k}) \widehat{P}.
\end{equation}

\subsection{The analytic branches of  eigenvalues and eigenvectors}
\label{sndw_eigenvalues_and_eigenvectors_section}
According to~(\ref{abstr_S_and_S_hat_relation}), the spectral germ $S(\boldsymbol{\theta})$ of the operator  
$A (t, \boldsymbol{\theta})$ acting in the subspace $\mathfrak{N}$ (see (\ref{Ker2})) can be represented as 
$$
S(\boldsymbol{\theta}) = P f^* b(\boldsymbol{\theta})^* g^0 b(\boldsymbol{\theta}) f|_{\mathfrak{N}},
$$
where $P$~is the orthogonal projection of the space $L_2 (\Omega; \mathbb{C}^n)$ onto $\mathfrak{N}$. We put 
$$
S(\k) := t^2 S(\boldsymbol{\theta}) = P f^* b(\k)^* g^0 b(\k) f|_{\mathfrak{N}}.
$$

The analytic (in $t$) branches of the eigenvalues $\lambda_l (t, \boldsymbol{\theta})$ and the analytic branches of the eigenvectors $\varphi_l (t, \boldsymbol{\theta})$ of the operator $A (t, \boldsymbol{\theta})$ 
admit the power series expansions of the form~(\ref{abstr_A(t)_eigenvalues_series}), (\ref{abstr_A(t)_eigenvectors_series}) with the coefficients depending on $\boldsymbol{\theta}$:
\begin{gather}
\label{A_eigenvalues_series}
\lambda_l (t, \boldsymbol{\theta}) = \gamma_l (\boldsymbol{\theta}) t^2 + \mu_l (\boldsymbol{\theta}) t^3 + 
\nu_l (\boldsymbol{\theta}) t^4 + \ldots, \qquad l = 1, \ldots, n,
\\
\label{A_eigenvectors_series}
\varphi_l (t, \boldsymbol{\theta}) = \omega_l (\boldsymbol{\theta}) + t \psi^{(1)}_l (\boldsymbol{\theta}) + \ldots, \qquad l = 1, \ldots, n.
\end{gather}
The vectors $\omega_1 (\boldsymbol{\theta}), \ldots, \omega_n (\boldsymbol{\theta})$ form an orthonormal basis in the subspace $\mathfrak{N}$, and the vectors 
$\zeta_l (\boldsymbol{\theta}) = f \omega_l (\boldsymbol{\theta})$, $l = 1, \ldots, n$, form a basis in 
 $\widehat{\mathfrak{N}}$~(see~(\ref{Ker3}))  orthonormal with the weight:
$(\overline{Q} \zeta_l (\boldsymbol{\theta}), \zeta_j (\boldsymbol{\theta})) = \delta_{jl}, \; j, l = 1, \ldots,n$.

The numbers $\gamma_l (\boldsymbol{\theta})$ and the elements $\omega_l (\boldsymbol{\theta})$ are the eigenvalues and the eigenvectors of the spectral germ $S(\boldsymbol{\theta})$. 
According to~(\ref{abstr_hatS_gener_spec_problem}), the numbers $\gamma_l (\boldsymbol{\theta})$ and the elements $\zeta_l (\boldsymbol{\theta})$ are the eigenvalues and the eigenvectors of the following generalized spectral problem:
\begin{equation}
\label{hatS_gener_spec_problem}
b(\boldsymbol{\theta})^* g^0 b(\boldsymbol{\theta}) \zeta_l (\boldsymbol{\theta}) = \gamma_l (\boldsymbol{\theta}) \overline{Q} \zeta_l (\boldsymbol{\theta}), \qquad l = 1, \ldots, n.
\end{equation}

\subsection{The operators $\widehat{Z}_Q (\boldsymbol{\theta})$ and $\widehat{N}_Q (\boldsymbol{\theta})$}
Now we describe the operators  $\widehat{Z}_Q$ and $\widehat{N}_Q$ (in abstract terms defined in Subsection~\ref{abstr_hatZ_Q_and_hatN_Q_section}). For this, we introduce the $\Gamma$-periodic solution 
$\Lambda_Q(\mathbf{x})$ of the problem 
\begin{equation*}
b(\mathbf{D})^* g(\mathbf{x}) (b(\mathbf{D}) \Lambda_Q(\mathbf{x}) + \mathbf{1}_m) = 0, \quad \int\limits_{\Omega} Q(\mathbf{x}) \Lambda_Q(\mathbf{x}) \, d \mathbf{x} = 0.
\end{equation*}
Clearly, $\Lambda_Q(\mathbf{x})$ differs from the periodic solution $\Lambda(\mathbf{x})$ of the 
problem~(\ref{equation_for_Lambda}) by the constant summand:
\begin{equation}
\label{Lambda_Q=Lambda+Lambda_Q^0}
\Lambda_Q(\mathbf{x}) = \Lambda(\mathbf{x}) + \Lambda_Q^0, \quad \Lambda_Q^0 = -(\overline{Q})^{-1} (\overline{Q \Lambda}).
\end{equation}
As was checked in~\cite[\S5]{BSu3}, now the operators $\widehat{Z}_Q (\boldsymbol{\theta})$ and $\widehat{N}_Q (\boldsymbol{\theta})$ take the form 
\begin{align}
\label{Z_Q(theta)}
\widehat{Z}_Q (\boldsymbol{\theta}) &= \Lambda_Q b(\boldsymbol{\theta}) \widehat{P},
\\
\label{N_Q(theta)}
\widehat{N}_Q (\boldsymbol{\theta}) &= b(\boldsymbol{\theta})^* L_Q (\boldsymbol{\theta}) b(\boldsymbol{\theta}) \widehat{P},
\end{align}
where $L_Q (\boldsymbol{\theta})$~is the ($m \times m$)-matrix given by 
\begin{equation}
\label{L_Q(theta)}
L_Q (\boldsymbol{\theta}) = | \Omega |^{-1} \int\limits_{\Omega} (\Lambda_Q(\mathbf{x})^*b(\boldsymbol{\theta})^* \widetilde{g} (\mathbf{x}) + \widetilde{g} (\mathbf{x})^* b(\boldsymbol{\theta}) \Lambda_Q(\mathbf{x}))\, d \mathbf{x}.
\end{equation}
Obviously,
\begin{equation}
\label{11.14a}
t \wh{Z}_Q(\boldsymbol{\theta}) \wh{P} = t \Lambda_Q b(\boldsymbol{\theta}) \widehat{P} = \Lambda_Q b(\mathbf{D} + \mathbf{k}) \widehat{P}.
\end{equation}

In~\cite[\S5]{BSu3}, some conditions ensuring that $\widehat{N}_Q (\boldsymbol{\theta}) \equiv 0$ were given.

\begin{proposition}[see~\cite{BSu3}]
	\label{N_Q=0_proposit}
	Suppose that at least one of the following assumptions is satisfied{\rm :}
	
\noindent $1^\circ$.	 The operator $\mathcal{A}$ is of the form $\mathcal{A} = f(\mathbf{x})^*\mathbf{D}^* g(\mathbf{x}) \mathbf{D}f(\mathbf{x}),$ where $g(\mathbf{x})$~is a symmetric matrix with real entries.
	
\noindent $2^\circ$.  Relations~\emph{(\ref{g0=overline_g_relat})} are satisfied, i.~e.,  $g^0 = \overline{g}$.
	
\noindent	Then $\widehat{N}_Q (\boldsymbol{\theta}) = 0$ for any $\boldsymbol{\theta} \in \mathbb{S}^{d-1}$.
\end{proposition}

Recall that (see~Subsection~\ref{abstr_hatZ_Q_and_hatN_Q_section})
$$
\widehat{N}_Q (\boldsymbol{\theta}) = \widehat{N}_{0, Q} (\boldsymbol{\theta}) + \widehat{N}_{*,Q} (\boldsymbol{\theta}).
$$
 According to~(\ref{abstr_hatN_0Q_N_*Q}),
$\widehat{N}_{0, Q} (\boldsymbol{\theta}) = \sum_{l=1}^{n} \mu_l (\boldsymbol{\theta}) (\,\cdot\,, \overline{Q} \zeta_l(\boldsymbol{\theta}))_{L_2(\Omega)} \overline{Q} \zeta_l(\boldsymbol{\theta})$.
We have 
\begin{equation*}
(\widehat{N}_Q (\boldsymbol{\theta}) \zeta_l (\boldsymbol{\theta}), \zeta_l (\boldsymbol{\theta}))_{L_2 (\Omega)} = (\widehat{N}_{0,Q} (\boldsymbol{\theta}) \zeta_l (\boldsymbol{\theta}), \zeta_l (\boldsymbol{\theta}))_{L_2 (\Omega)} = \mu_l (\boldsymbol{\theta}), \quad l=1, \ldots, n.
\end{equation*}

The following statement was proved in \cite[\S 5]{BSu3}.

\begin{proposition}
Suppose that  $b(\boldsymbol{\theta}),$ $g (\mathbf{x}),$ and $Q(\mathbf{x})$~are matrices with real entries. Suppose that in the expansions~\emph{(\ref{A_eigenvectors_series})} for the analytic branches of the eigenvectors of the operator $A (t, \boldsymbol{\theta})$ the \textup{``}embryos\/\textup{''} $\omega_l (\boldsymbol{\theta}),$ ${l = 1, \ldots, n,}$ can be chosen so that the vectors  $\zeta_l (\boldsymbol{\theta}) = f \omega_l (\boldsymbol{\theta})$ are real. Then  
	$\mu_l (\boldsymbol{\theta}) = 0, \, l=1, \ldots, n,$ i.~e.\textup, $\widehat{N}_{0,Q} (\boldsymbol{\theta}) = 0$.
\end{proposition}

In the  \textquotedblleft real\textquotedblright \ case under consideration, the germ  
$\widehat{S} (\boldsymbol{\theta})$ is a symmetric matrix with real entries; the matrix $\overline{Q}$ is also symmetric and real. Clearly, in the case of simple eigenvalue~$\gamma_j (\boldsymbol{\theta})$ of  problem~(\ref{hatS_gener_spec_problem}) the eigenvector $\zeta_j (\boldsymbol{\theta}) = f \omega_j (\boldsymbol{\theta})$ is defined uniquely up to a phase factor, and we can always choose it to be real. We obtain the following corollary.

\begin{corollary}
	\label{sndw_simple_spec_N0Q=0_cor}
	Suppose that the matrices  $b(\boldsymbol{\theta}),$ $g (\mathbf{x}),$ and $Q(\mathbf{x})$~have real entries. Suppose that problem~\emph{(\ref{hatS_gener_spec_problem})} has simple spectrum. Then
$\widehat{N}_{0,Q} (\boldsymbol{\theta}) = 0$.
\end{corollary}

\subsection{The operators $\widehat{Z}_{2,Q} (\boldsymbol{\theta})$, $\widehat{R}_{2,Q} (\boldsymbol{\theta}),$ and $\widehat{N}_{1,Q}^0 (\boldsymbol{\theta})$}

Now we describe the operators $\wh{Z}_{2,Q}$, $\wh{R}_{2,Q}$, and $\wh{N}_{1,Q}^0$ (see Subsection \ref{sec5.3}) for the family ${A}(t,\boldsymbol{\theta})$. 
Let $\Lambda_{l,Q}^{(2)}(\x)$ be the $\Gamma$-periodic solution of the problem 
\begin{align*}
b(\D)^* g(\x) \bigl(b(\D) \Lambda_{l,Q}^{(2)}(\x) + b_l \Lambda_Q(\x)\bigr)
&= - b_l^*  \widetilde{g}(\x) + Q(\x) (\overline{Q})^{-1} b_l^* g^0,
\\
 \int\limits_\Omega Q(\x)\Lambda_{l,Q}^{(2)}(\x) \,d\x &=0. 
\end{align*}
We put  
$
\Lambda^{(2)}_Q (\x; \boldsymbol{\theta}) := \sum_{l=1}^d \Lambda_{l,Q}^{(2)}(\x) \theta_l.
$
As was checked in \cite[Subsection 8.4]{VSu2}, we have
$$
\widehat{Z}_{2,Q} (\boldsymbol{\theta}) \!= \!\Lambda^{(2)}_Q(\x; \boldsymbol{\theta}) b(\boldsymbol{\theta}) \widehat{P},
\quad 
\widehat{R}_{2,Q} (\boldsymbol{\theta})\! =\! h(\x) \bigl( b(\D)\Lambda_Q^{(2)}(\x; \boldsymbol{\theta}) 
+ b(\boldsymbol{\theta}) \Lambda_Q(\x) \bigr) b(\boldsymbol{\theta}).
$$
Finally, in \cite[Subsection 8.5]{VSu2}, it was proved that
\begin{align}
\label{N_1Q0_theta}
&\widehat{N}_{1,Q}^0 (\boldsymbol{\theta})\! =\! b(\boldsymbol{\theta})^* L_{2,Q}(\boldsymbol{\theta}) b(\boldsymbol{\theta})
\widehat{P},
\\
\nonumber
\begin{split}
&L_{2,Q}(\boldsymbol{\theta})\! =\! |\Omega|^{-1}\!\! \int\limits_\Omega \!\big( \Lambda_Q^{(2)}(\x; \boldsymbol{\theta})^* b(\boldsymbol{\theta})^*
\widetilde{g}(\x) + \widetilde{g}(\x)^* b(\boldsymbol{\theta}) \Lambda_Q^{(2)}(\x; \boldsymbol{\theta}) \big)\, d\x
\\
&\!+ \! |\Omega|^{-1} \!\!\int\limits_\Omega \!\big( b(\D) \Lambda_Q^{(2)}(\x; \boldsymbol{\theta}) + b(\boldsymbol{\theta}) \Lambda_Q(\x)\big)^*
{g}(\x) \big( b(\D) \Lambda_Q^{(2)}(\x; \boldsymbol{\theta}) + b(\boldsymbol{\theta}) \Lambda_Q(\x)\big)
\, d\x.
\end{split}
\end{align}

\subsection{The multiplicities of the eigenvalues of the germ}
\label{sndw_eigenval_multipl_section}
In the present subsection, it is assumed that \hbox{$n \ge 2$}. We turn to the notation adopted in 
Subsection~\ref{abstr_cluster_section}.
 In general, the number $p(\boldsymbol{\theta})$ of different eigenvalues $\gamma^{\circ}_1 (\boldsymbol{\theta}), \ldots, \gamma^{\circ}_{p(\boldsymbol{\theta})} (\boldsymbol{\theta})$ of the spectral germ $S(\boldsymbol{\theta})$ (or of problem \eqref{hatS_gener_spec_problem}) and their multiplicities  $k_1 (\boldsymbol{\theta}), \ldots, k_{p(\boldsymbol{\theta})} (\boldsymbol{\theta})$ depend on the parameter 
 $\boldsymbol{\theta} \in \mathbb{S}^{d-1}$. 
 By $\mathfrak{N}_j (\boldsymbol{\theta})$ we denote the eigenspace of the germ $S (\boldsymbol{\theta})$ corresponding to the eigenvalue $\gamma^{\circ}_j (\boldsymbol{\theta})$. Then 
 $f \mathfrak{N}_j (\boldsymbol{\theta})=
\operatorname{Ker} \bigl( \wh{S}(\boldsymbol{\theta}) - \gamma_j^\circ(\boldsymbol{\theta}) \overline{Q}\bigr) 
=: \wh{\mathfrak N}_{j,Q}(\boldsymbol{\theta})$~is the eigenspace of the problem~(\ref{hatS_gener_spec_problem}) corresponding to the same eigenvalue $\gamma^{\circ}_j (\boldsymbol{\theta})$. We denote by  $\mathcal{P}_j (\boldsymbol{\theta})$ the \textquotedblleft skew\textquotedblright \ projection of the space $L_2(\Omega; \mathbb{C}^n)$ onto the subspace  
$\wh{\mathfrak N}_{j,Q}(\boldsymbol{\theta})$; $\mathcal{P}_j (\boldsymbol{\theta})$ is orthogonal with respect to the inner product with the weight $\overline{Q}$. According to~(\ref{abstr_hatN_0Q_N_*Q_invar_repr}), we have the following invariant representations for the operators $\widehat{N}_{0,Q} (\boldsymbol{\theta})$  and 
$\widehat{N}_{*,Q} (\boldsymbol{\theta})$:
\begin{equation}
\label{N0Q_invar_repr}
\begin{split}
\widehat{N}_{0,Q} (\boldsymbol{\theta})& = \sum_{j=1}^{p(\boldsymbol{\theta})} \mathcal{P}_j (\boldsymbol{\theta})^* \widehat{N}_Q (\boldsymbol{\theta}) \mathcal{P}_j (\boldsymbol{\theta}), 
\\ 
\widehat{N}_{*,Q} (\boldsymbol{\theta})& = \sum_{\substack{1 \le j, l \le p(\boldsymbol{\theta}):\, j \ne l}} \mathcal{P}_j (\boldsymbol{\theta})^* \widehat{N}_Q (\boldsymbol{\theta}) \mathcal{P}_l (\boldsymbol{\theta}).
\end{split}
\end{equation}

\subsection{The coefficients ${\nu}_l(\boldsymbol{\theta})$}

Applying Proposition~\ref{Prop_nu}, we arrive at the following statement.

\begin{proposition}\label{Prop_nu1_theta_Q}
Suppose that $\wh{N}_{0,Q}(\boldsymbol{\theta})=0$. 
Let ${\gamma}_1^\circ(\boldsymbol{\theta}), \dots, {\gamma}_{p(\boldsymbol{\theta})}^\circ(\boldsymbol{\theta})$ be the different eigenvalues of the problem~\eqref{hatS_gener_spec_problem}\textup,
and let  $k_1(\boldsymbol{\theta}), \dots, k_{p(\boldsymbol{\theta})}(\boldsymbol{\theta})$ be their multiplicities. Let $\wh{P}_{q,Q}(\boldsymbol{\theta})$ be the orthogonal projection of the space $L_2(\Omega;\AC^n)$ onto the subspace  $\wh{\mathfrak{N}}_{q,Q} (\boldsymbol{\theta})  = \operatorname{Ker} (\wh{S}(\boldsymbol{\theta}) 
- {\gamma}_q^\circ(\boldsymbol{\theta}) \overline{Q}),$ $q=1,\dots,p(\boldsymbol{\theta})$.
Suppose that the operators $\wh{Z}_Q(\boldsymbol{\theta}),$ $\wh{N}_Q(\boldsymbol{\theta}),$ and 
$\wh{N}_{1,Q}^0(\boldsymbol{\theta})$ are defined by  \eqref{Z_Q(theta)}, \eqref{N_Q(theta)}, and 
 \eqref{N_1Q0_theta}\textup, respectively.
We introduce the operators $\wh{\mathcal{N}}^{(q)}_Q(\boldsymbol{\theta}),$ $q=1,\dots,p(\boldsymbol{\theta})$\emph{:} the operator $\wh{\mathcal{N}}^{(q)}_Q(\boldsymbol{\theta})$  acts in  
$\wh{\mathfrak{N}}_{q,Q}(\boldsymbol{\theta})$ and is given by   
{\allowdisplaybreaks\begin{align}
\label{8.32a}
&\wh{\mathcal{N}}^{(q)}_Q( \boldsymbol{\theta}) 
:= \wh{P}_{q,Q}(\boldsymbol{\theta})  \wh{N}_{1,Q}^0(\boldsymbol{\theta})\bigl\vert_{\wh{\mathfrak{N}}_{q,Q}(\boldsymbol{\theta})} 
\\
&\!-\! \frac{1}{2} \wh{P}_{q,Q}(\boldsymbol{\theta}) \big(\wh{Z}_Q(\boldsymbol{\theta})^* Q 
\wh{Z}_Q(\boldsymbol{\theta})Q^{-1} \wh{S}(\boldsymbol{\theta}) \wh{P} 
\!+ \!\wh{S}(\boldsymbol{\theta}) \wh{P} Q^{-1}
\wh{Z}_Q(\boldsymbol{\theta})^* Q \wh{Z}_Q(\boldsymbol{\theta})  \big) \Bigl\vert_{\wh{\mathfrak{N}}_{q,Q}(\boldsymbol{\theta})} \notag
\\
&\!+ \!\!\!\sum_{j=1,\dots,p(\boldsymbol{\theta}): j\ne q}\!\!\! ({\gamma}_q^\circ(\boldsymbol{\theta})\! -\!
 {\gamma}_j^\circ(\boldsymbol{\theta}))^{-1} 
\wh{P}_{q,Q}(\boldsymbol{\theta}) \wh{N}_Q(\boldsymbol{\theta}) \wh{P}_{j,Q}(\boldsymbol{\theta}) 
Q^{-1} \wh{P}_{j,Q}(\boldsymbol{\theta})
\wh{N}_Q(\boldsymbol{\theta}) \Bigr \vert_{\wh{\mathfrak{N}}_{q,Q}(\boldsymbol{\theta})}.\notag
\end{align}
}
\hspace{-3mm}
Denote $i(q,\boldsymbol{\theta})= k_1(\boldsymbol{\theta}) + \dots + k_{q-1}(\boldsymbol{\theta}) +1$.
Let ${\nu}_l(\boldsymbol{\theta})$ be the coefficients of $t^4$ from the expansions  
\eqref{A_eigenvalues_series}\textup,  and let ${\omega}_l(\boldsymbol{\theta})$ be the embryos from 
\eqref{A_eigenvectors_series}. Let  
${\zeta}_l(\boldsymbol{\theta})= f {\omega}_l(\boldsymbol{\theta}),$ $l=1,\dots,n$. 
Denote $Q_{\wh{\mathfrak{N}}_{q,Q}(\boldsymbol{\theta}) } = \wh{P}_{q,Q}(\boldsymbol{\theta}) Q\vert_{\wh{\mathfrak{N}}_{q,Q}(\boldsymbol{\theta}) }$. Then 
$$
\wh{\mathcal{N}}^{(q)}_Q(\boldsymbol{\theta}) {\zeta}_l(\boldsymbol{\theta}) 
\!=\! {\nu}_l(\boldsymbol{\theta}) Q_{\wh{\mathfrak{N}}_{q,Q}(\boldsymbol{\theta}) }{\zeta}_l(\boldsymbol{\theta}), \quad 
l\!=\!i(q, \boldsymbol{\theta}), i(q,\boldsymbol{\theta})+1,\dots, i(q,\boldsymbol{\theta}) + k_q(\boldsymbol{\theta}) -1.
$$
\end{proposition}

\section{Approximation for the sandwiched operators  $\cos(\varepsilon^{-1} \tau \mathcal{A}(\mathbf{k})^{1/2})$ 
and $\mathcal{A}(\mathbf{k})^{-1/2}\sin(\varepsilon^{-1} \tau \mathcal{A}(\mathbf{k})^{1/2})$\label{sec12}}

\subsection{Approximation in the operator norm in $L_2(\Omega;\AC^n)$. The general case}

Denote
\begin{align}
\label{J1}
{J}_1(\mathbf{k}, \!\tau )&\!\! :=\! f \cos ( \tau {\mathcal A}(\k)^{1/2}\big) f^{-1}\!\! - \!
f_0 \cos \big(\tau {\mathcal A}^0(\k)^{1/2}) f_0^{-1},
\\
\label{J2}
{J}_2(\mathbf{k}, \!\tau ) &\!\!:=\! \!f {\mathcal A}(\k)^{-\!1/2} \!\sin (\tau {\mathcal A}(\k)^{1/2}) f^{-1}\!\! - \!
f_0 {\mathcal A}^0(\k)^{-1/2} \!\sin (\tau {\mathcal A}^0(\k)^{1/2}) f_0^{-\!1}\!\!\!,
\\
\label{J2tilde}
{J}_3(\mathbf{k}, \!\tau ) &\!\!:= \!f {\mathcal A}(\k)^{-1/2} \sin (\tau {\mathcal A}(\k)^{1/2}) f^{*}\!\! - \!
f_0 {\mathcal A}^0(\k)^{-1/2} \sin (\tau {\mathcal A}^0(\k)^{1/2}) f_0.
\end{align}

We apply theorems of~\S\ref{abstr_sandwiched_section}  to the operator 
${A}(t, \boldsymbol{\theta}) = {\mathcal{A}}(\mathbf{k})$.
 By Remark~\ref{abstr_constants_remark}, we  can track the dependence of the constants in estimates  on the problem data. Note that ${c}_*$, ${\delta}$, and ${t}_0$ do not depend on $\boldsymbol{\theta}$ (see  \eqref{A(k)_nondegenerated_and_c_*},   \eqref{delta_fixation},  \eqref{t0_fixation}). According 
 to~\eqref{X_1_estimate}, the norm $\| {X}_1 (\boldsymbol{\theta}) \|$ can be replaced by 
$\alpha_1^{1/2} \| g \|_{L_{\infty}}^{1/2} \| f \|_{L_{\infty}}$. Therefore, the constants from Theorem~\ref{th5.5}
(applied to the operator ${\mathcal{A}} (\mathbf{k})$) will not depend on~$\boldsymbol{\theta}$. 
They will depend only on  $\alpha_0$, $\alpha_1$, $\|g\|_{L_\infty}$, $\|g^{-1}\|_{L_\infty}$,
$\|f\|_{L_\infty}$, $\|f^{-1}\|_{L_\infty}$, and $r_0$.

\begin{theorem}[see~\cite{BSu5,M,DSu}]
\label{th12.1}
Let ${J}_1(\mathbf{k}, \tau ),$ ${J}_2(\mathbf{k}, \tau ),$ and ${J}_3(\mathbf{k}, \tau )$ 
be defined by \eqref{J1}\textup, \eqref{J2}, and \eqref{J2tilde}, respectively. 
For $\tau \in \R,$ $\eps >0,$ and $\k \in \wt{\Omega}$ we have 
\begin{align}
\label{12.7}
& \|  {J}_1(\mathbf{k}, \eps^{-1} \tau )  \mathcal{R}(\k,\eps) \|_{L_2(\Omega) \to L_2(\Omega)}
\le {\mathcal C}_1 (1+ |\tau|) \eps,
\\
\label{12.8}
& \|  {J}_2(\mathbf{k}, \eps^{-1} \tau ) \mathcal{R}(\k,\eps)^{1/2} \|_{L_2(\Omega) \to L_2(\Omega)}
\le {\mathcal C}_2 (1+ |\tau|),
\\
\label{12.8a}
& \|  {J}_3(\mathbf{k}, \eps^{-1} \tau ) \mathcal{R}(\k,\eps)^{1/2} \|_{L_2(\Omega) \to L_2(\Omega)}
\le \wt{\mathcal C}_2 (1+ |\tau|).
\end{align}
The constants ${\mathcal C}_1,$ ${\mathcal C}_2,$ and $\wt{\mathcal C}_2$ depend only on 
$\alpha_0,$ $\alpha_1,$ $\|g\|_{L_\infty},$  $\|g^{-1}\|_{L_\infty},$ $\|f\|_{L_\infty},$  $\|f^{-1}\|_{L_\infty},$ and $r_0$.
\end{theorem}

Theorem \ref{th12.1} is deduced from Theorem~\ref{th5.5} and relations  \eqref{R_P}--\eqref{R(k,eps)(I-P)_est}.
We should also take into account the obvious estimates
\begin{align} 
\label{12.9_1}
&\|  {J}_1(\mathbf{k}, \eps^{-1} \tau ) \|_{L_2(\Omega) \to L_2(\Omega)} \le 2 \|f\|_{L_\infty} \|f^{-1}\|_{L_\infty},
\\
\nonumber
&\|  {J}_2(\mathbf{k}, \eps^{-1} \tau ) \|_{L_2(\Omega) \to L_2(\Omega)}
\le 2 \|f\|_{L_\infty} \|f^{-1}\|_{L_\infty} \eps^{-1}|\tau|,
\\
\label{12.9_3}
&\|  {J}_3(\mathbf{k}, \eps^{-1} \tau ) \|_{L_2(\Omega) \to L_2(\Omega)}
\le 2 \|f\|^2_{L_\infty}  \eps^{-1}|\tau|.
\end{align}
Earlier, estimate \eqref{12.7} was obtained in \cite[Theorem 9.2]{BSu5}, inequality \eqref{12.8} was proved in  \cite[(7.32)]{M},  and \eqref{12.8a} was found in  \cite[Theorem 9.1]{DSu}.

In what follows, we shall need the following statement.

\begin{proposition}
\label{prop12.1a}
For $\tau \in \R,$ $\eps >0,$ and $\k \in \wt{\Omega}$ we have 
\begin{align}
\label{12.9a}
\|  {J}_3(\mathbf{k}, \eps^{-1} \tau ) \|_{L_2(\Omega) \to L_2(\Omega)}
\le {\mathcal C}_2' (1+ \eps^{-1/2}|\tau|^{1/2}),
\end{align}
where ${\mathcal C}_2'$ depends only on $\alpha_0,$ $\alpha_1,$ $\|g\|_{L_\infty},$  $\|g^{-1}\|_{L_\infty},$
$\|f\|_{L_\infty},$  $\|f^{-1}\|_{L_\infty},$ and $r_0$.
\end{proposition}

\begin{proof}
From  \eqref{2.6} (with $\tau$ replaced by $\eps^{-1} \tau$), \eqref{abstr_tildeJ2_est1}, and \eqref{J3-J3tilde} it follows that
\begin{equation}
\label{12.9b}
\|  {J}_3(\mathbf{k}, \eps^{-\!1} \tau ) \wh{P} \|_{L_2(\Omega) \to L_2(\Omega)}
\le {\mathcal C}_2^{(1)} (1\!+ \!\eps^{-\!1}|\tau| |\k|), \quad \!\tau\! \in \!\R,\ \eps\!>\!0, \ |\k| \!\le \!{t}_0.
\end{equation}

Now, we estimate the operator ${J}_3(\mathbf{k}, \eps^{-1} \tau )  (I-\wh{P})$ for $|\k| \le {t}_0$. 
Obviously,
\begin{equation*}
\begin{aligned}
\|  {J}_3(\mathbf{k}, \eps^{-1} \tau ) (I-\wh{P}) \|_{L_2(\Omega) \to L_2(\Omega)}
&\le \| f \|_{L_\infty}  \| {\mathcal A}(\mathbf k)^{-1/2} f^* (I-\wh{P}) \|_{L_2(\Omega) \to L_2(\Omega)}
\\
&+\| f \|_{L_\infty}^2  \| {\mathcal A}^0(\mathbf k)^{-1/2}  (I-\wh{P}) \|_{L_2(\Omega) \to L_2(\Omega)}.
\end{aligned}
\end{equation*}
The second term is uniformly bounded, which can be easily checked with the help of the discrete Fourier transformation. To estimate the first term, we note that $P f^*  (I-\wh{P}) =0$, by the identity
$Pf^* = f^{-1} (\overline{Q})^{-1} \wh{P}$ (see \eqref{abstr_P_and_P_hat_relation}). Therefore,
$f^* (I - \wh{P}) = (I-P) f^* (I - \wh{P})$, whence  
$$
\| {\mathcal A}(\k)^{-1/2} f^* (I-\wh{P}) \|_{L_2(\Omega) \to L_2(\Omega)}
\le \| f \|_{L_\infty}  \|{\mathcal A}(\k)^{-1/2} (I-P)  \|_{L_2(\Omega) \to L_2(\Omega)}.
$$
This quantity is uniformly bounded due to   \eqref{abstr_F(t)_threshold_1} and \eqref{A(k)_nondegenerated_and_c_*}. As a result, we obtain  
\begin{equation}
\label{12.9e}
\|  {J}_3(\mathbf{k}, \eps^{-1} \tau ) (I-\wh{P}) \|_{L_2(\Omega) \to L_2(\Omega)}
\le {\mathcal C}_2^{(2)}, \quad \tau \in \R,\ \eps >0,\ |\k| \le t_0.
\end{equation}

If $\eps |\tau|^{-1} > {t}_0^{2}$, then  the required estimate  \eqref{12.9a} follows directly from  \eqref{12.9_3}. 
So, we suppose that $\eps |\tau|^{-1} \le {t}_0^{2}$. Then \eqref{12.9b} implies that 
\begin{equation*}
\|  {J}_3(\mathbf{k}, \eps^{-1} \tau ) \wh{P} \|_{L_2(\Omega) \to L_2(\Omega)}
\le {\mathcal C}_2^{(1)} (1+ \eps^{-1/2}|\tau|^{1/2}), \quad  |\k| \le \eps^{1/2} |\tau|^{-1/2}.
\end{equation*}
Combining this with  \eqref{12.9e}, we obtain estimate  \eqref{12.9a} for $|\k| \le \eps^{1/2} |\tau|^{-1/2}$.

Finally,  from \eqref{A(k)_nondegenerated_and_c_*} 
(for the operators ${\mathcal A}(\k)$ and ${\mathcal A}^0(\k)$) it follows that 
$$
\|  {J}_3(\mathbf{k}, \eps^{-1} \tau )  \|_{L_2(\Omega) \to L_2(\Omega)}
\le 2 \|f\|^2_{L_\infty} {c}_*^{-1/2} |\k|^{-1} \le 2 \|f\|^2_{L_\infty} {c}_*^{-1/2} \eps^{-1/2} |\tau|^{1/2}
$$
for $|\k| > \eps^{1/2} |\tau|^{-1/2}$.
\end{proof}

\subsection{Approximation in the operator norm in 
$L_2(\Omega;\AC^n)$. The case where $\wh{N}_Q(\boldsymbol{\theta}) =0$}

Now, we improve the result of Theorem \ref{th12.1} (estimates \eqref{12.7} and \eqref{12.8a}) under some additional assumptions. We impose the following condition.

\begin{condition}
\label{cond_BB}
Let $\widehat{N}_Q(\boldsymbol{\theta})$ be the operator defined by~\eqref{N_Q(theta)}. 
Suppose that $\widehat{N}_Q(\boldsymbol{\theta}) = 0$ for any $\boldsymbol{\theta} \in \mathbb{S}^{d-1}$. 
\end{condition}

\begin{theorem} 
\label{th12.2}
Suppose that Condition \emph{\ref{cond_BB}} is satisfied. Then for 
$\tau \in \mathbb{R},$ $\varepsilon > 0,$ and $\mathbf{k} \in \widetilde{\Omega}$  we have
\begin{align}
\label{12.7a}
& \|  {J}_1(\mathbf{k}, \eps^{-1} \tau ) \mathcal{R}(\k,\eps)^{3/4} \|_{L_2(\Omega) \to L_2(\Omega)}
\le {\mathcal C}_3 (1+ |\tau|)^{1/2} \eps,
\\
\label{12.8b}
& \|  {J}_3(\mathbf{k}, \eps^{-1} \tau ) \mathcal{R}(\k,\eps)^{1/4}\|_{L_2(\Omega) \to L_2(\Omega)}
\le {\mathcal C}_4 (1+ |\tau|)^{1/2}.
\end{align}
The constants ${\mathcal C}_3$ and ${\mathcal C}_4$ depend only on  $\alpha_0,$ $\alpha_1,$ 
$\|g\|_{L_\infty},$  $\|g^{-1}\|_{L_\infty},$ $\|f\|_{L_\infty},$  $\|f^{-1}\|_{L_\infty}$, and $r_0$.
\end{theorem}

\begin{proof}
First, we check  inequality \eqref{12.7a}.
Applying \eqref{abstr_cos_sandwiched_improved_est}  and using  \eqref{R_P} and 
\eqref{f_0 hatS f_0 P = A^0}, we have 
\begin{equation}
\label{12.7b}
\|  {J}_1(\mathbf{k}, \eps^{-1} \tau ) \mathcal{R}(\k,\eps)^{3/4} \wh{P}\|_{L_2(\Omega) \to L_2(\Omega)}
\le {\mathcal C}^\circ_3 (1+ |\tau|)^{1/2} \eps,
 \quad \tau \in \R,\ \eps>0, \ |\k| \le {t}_0.
\end{equation}
 From the analog of \eqref{R_P_2} (with $\wh{t}_0$ replaced by  $t_0$) for $s=1$ 
 and from \eqref{12.9_1} it is seen that the left-hand side of \eqref{12.7b}
does not exceed $2 \|f\|_{L_\infty} \| f^{-1}\|_{L_\infty}{t}_0^{-1} \eps$ for \hbox{$|\k| > {t}_0$}.
Finally, by \eqref{R(k,eps)(I-P)_est} with $s=1$ and \eqref{12.9_1}, the quantity 
$\|  {J}_1(\mathbf{k}, \eps^{-1} \tau ) \mathcal{R}(\k,\eps)^{3/4} (I-\wh{P})\|$ 
does not exceed $2 \|f\|_{L_\infty} \| f^{-1}\|_{L_\infty} r_0^{-1}\eps$ for any $\k \in \wt{\Omega}$.
 As a result, we arrive at inequality \eqref{12.7a}.

We proceed to the proof of estimate \eqref{12.8b}.
By \eqref{abstr_sin_sandwiched_improved_est2}, \eqref{R_P}, and \eqref{f_0 hatS f_0 P = A^0}, 
\begin{equation*}
\|  {J}_3(\mathbf{k}, \eps^{-1} \tau ) \mathcal{R}(\k,\eps)^{1/4} \wh{P}\|_{L_2(\Omega) \to L_2(\Omega)}
\le {\mathcal C}^\circ_4 (1+ |\tau|)^{1/2},\quad \tau \in \R,\ \eps>0, \ |\k| \le {t}_0.
\end{equation*}

Next, by \eqref{12.9e}, the norm of the operator 
${J}_3(\mathbf{k}, \eps^{-1} \tau ) \mathcal{R}(\k,\eps)^{1/4} (I-\wh{P})$ 
does not exceed the constant ${\mathcal C}_2^{(2)}$ for $|\k| \le {t}_0$. 

For $|\k| >t_0$ inequality \eqref{12.8b} follows from  \eqref{A(k)_nondegenerated_and_c_*}  
and the similar  inequality for ${\mathcal A}^0(\k)$.
\end{proof}

We also need the following statement.

\begin{proposition}
\label{prop12.2a}
Under the assumptions of Theorem \emph{\ref{th12.2}}, for $\tau \in \R,$ $\eps >0,$ and $\k \in \wt{\Omega}$
we have  
\begin{align}
\label{12.9x}
\|  {J}_3(\mathbf{k}, \eps^{-1} \tau ) \|_{L_2(\Omega) \to L_2(\Omega)}
\le {\mathcal C}_4' (1+ \eps^{-1/3}|\tau|^{1/3}).
\end{align}
The constant  ${\mathcal C}_4'$ depends on  $\alpha_0,$ $\alpha_1,$ $\|g\|_{L_\infty},$  $\|g^{-1}\|_{L_\infty},$
$\|f\|_{L_\infty},$  $\|f^{-1}\|_{L_\infty},$ and $r_0$.
\end{proposition}

\begin{proof}
From  \eqref{2.8} (with  $\tau$ replaced by $\eps^{-1} \tau$),   \eqref{abstr_tildeJ2_est1}, and \eqref{J3-J3tilde}
it follows that  
\begin{equation}
\label{12.9y}
\|  {J}_3(\mathbf{k}, \eps^{-\!1} \tau ) \wh{P} \|_{L_2(\Omega) \to L_2(\Omega)}
\le {\mathcal C}_4^{(1)}  (1+ \eps^{-\!1}|\tau| |\k|^2), \ \; \tau \!\in\! \R,\ \eps\!>\!0, \ |\k| \!\le \!{t}_0.
\end{equation}

If $\eps |\tau|^{-1} > {t}_0^{3}$, then the required estimate  \eqref{12.9x} follows directly from  \eqref{12.9_3}. 
So, we assume that $\eps |\tau|^{-1} \le {t}_0^{3}$. Then, by \eqref{12.9y},
\begin{equation*}
\|  {J}_3(\mathbf{k}, \eps^{-1} \tau ) \wh{P} \|_{L_2(\Omega) \to L_2(\Omega)}
\le {\mathcal C}_4^{(1)} (1+ \eps^{-1/3}|\tau|^{1/3}), \quad  |\k| \le \eps^{1/3} |\tau|^{-1/3}.
\end{equation*}
Together with \eqref{12.9e}, this leads to estimate \eqref{12.9x}  for $|\k| \le \eps^{1/3} |\tau|^{-1/3}$.

Finally, \eqref{A(k)_nondegenerated_and_c_*} implies that 
$$
\|  {J}_3(\mathbf{k}, \eps^{-1} \tau )  \|_{L_2(\Omega) \to L_2(\Omega)}
\le 2 \|f\|^2_{L_\infty} {c}_*^{-1/2} |\k|^{-1} \le 2 \|f\|^2_{L_\infty} {c}_*^{-1/2} \eps^{-1/3} |\tau|^{1/3}
$$
for $|\k| > \eps^{1/3} |\tau|^{-1/3}$.
\end{proof}

\begin{remark}\label{rem12.1}
$1^\circ$. Under the assumptions of Theorem \ref{th12.2}, we cannot deduce 
the analog of estimate \eqref{12.8b} with $J_3(\k,\eps^{-1} \tau)$ replaced by $J_2(\k,\eps^{-1} \tau)$
from the abstract inequality \eqref{abstr_sin_sandwiched_improved_est1}. The reason is that the operator
\hbox{${J}_2(\mathbf{k}, \eps^{-1} \tau ) \mathcal{R}(\k,\eps)^{1/4} (I-\wh{P})$} does not satisfy the required estimate. 
 For the same reason, under the assumptions of Theorem \ref{th12.3} (see below) there is no analog of estimate \eqref{12.18} for ${J}_2(\mathbf{k}, \eps^{-1} \tau )$.
$2^\circ$. Also, there are no analogs of Propositions \ref{prop12.1a}, \ref{prop12.2a}, and \ref{prop12.3a} 
(see below) for the operator ${J}_2(\mathbf{k}, \eps^{-1} \tau )$, because it is impossible
  to obtain the required estimate for the operator ${J}_2(\mathbf{k}, \eps^{-1} \tau )(I-\wh{P})$.
\end{remark}

\subsection{Approximation in the operator norm in  $L_2(\Omega;\AC^n)$. The case where \hbox{$\wh{N}_{0,Q}(\boldsymbol{\theta}) =0$}}

Now we refuse from Condition~\ref{cond_BB}, but instead assume that  
\hbox{$\widehat{N}_{0,Q}(\boldsymbol{\theta}) = 0$} for all $\boldsymbol{\theta}$. 
As in Subsection~\ref{sec9.3}, in order to apply Theorem~\ref{th5.7}, we need to impose some additional conditions. We use the original numbering of the eigenvalues $\gamma_1 (\boldsymbol{\theta}), \ldots , \gamma_n (\boldsymbol{\theta})$ of the germ $S (\boldsymbol{\theta})$, agreeing to number them
 in the nondecreasing order:
\begin{equation}
\label{gamma(theta)}
\gamma_1 (\boldsymbol{\theta}) \le \gamma_2 (\boldsymbol{\theta}) \le \dots \le \gamma_n (\boldsymbol{\theta}).
\end{equation}
As has been already mentioned, the numbers~(\ref{gamma(theta)}) are simultaneously the eigenvalues of the generalized spectral problem~(\ref{hatS_gener_spec_problem}). For each $\boldsymbol{\theta}$, we denote by 
$\mathcal{P}^{(k)} (\boldsymbol{\theta})$ the \textquotedblleft skew\textquotedblright \ projection (orthogonal with the weight $\overline{Q}$) of the space $L_2 (\Omega; \mathbb{C}^n)$ onto the eigenspace of 
problem~(\ref{hatS_gener_spec_problem}) corresponding to the eigenvalue $\gamma_k (\boldsymbol{\theta})$. Clearly, for each $\boldsymbol{\theta}$ the operator $\mathcal{P}^{(k)} (\boldsymbol{\theta})$ coincides with one of the projections $\mathcal{P}_j (\boldsymbol{\theta})$ introduced in Subsection~\ref{sndw_eigenval_multipl_section} (but the number $j$ may depend on $\boldsymbol{\theta}$ and changes at the points where the multiplicity of  the germ spectrum changes).

\begin{condition}
	\label{sndw_cond1}
	
	$1^\circ$.	 
	$\widehat{N}_{0,Q}(\boldsymbol{\theta})=0$ for any $\boldsymbol{\theta} \in \mathbb{S}^{d-1}$.
	
	\noindent$2^\circ$. For each pair of indices $(k,r), 1 \le k,r \le n, k \ne r,$ such that 
	$\gamma_k (\boldsymbol{\theta}_0) = \gamma_r (\boldsymbol{\theta}_0) $ for some 
	$\boldsymbol{\theta}_0 \in \mathbb{S}^{d-1},$ we have  
	$$
	(\mathcal{P}^{(k)} (\boldsymbol{\theta}))^* \widehat{N}_Q (\boldsymbol{\theta}) \mathcal{P}^{(r)} (\boldsymbol{\theta}) = 0
	$$
	for all  $\boldsymbol{\theta} \in \mathbb{S}^{d-1}$.    
	\end{condition}

Condition~$2^\circ$ can be reformulated as follows: it is assumed that for the nonzero (identically) \textquotedblleft blocks\textquotedblright \ $(\mathcal{P}^{(k)} (\boldsymbol{\theta}))^* \widehat{N}_Q (\boldsymbol{\theta}) \mathcal{P}^{(r)} (\boldsymbol{\theta})$ of the operator $\widehat{N}_Q (\boldsymbol{\theta})$ the branches of the eigenvalues $\gamma_k (\boldsymbol{\theta})$ and  $\gamma_r (\boldsymbol{\theta})$ do not intersect.
Obviously,  Condition~\ref{sndw_cond1} is ensured by the following more restrictive condition.

\begin{condition}
	\label{sndw_cond2}
	
	$1^\circ$.	 
	$\widehat{N}_{0,Q}(\boldsymbol{\theta})=0$ for any $\boldsymbol{\theta} \in \mathbb{S}^{d-1}$.
		
		\noindent$2^\circ$.	Suppose that the number  $p$ of different eigenvalues of the generalized spectral problem~\emph{(\ref{hatS_gener_spec_problem})} does not depend on $\boldsymbol{\theta} \in \mathbb{S}^{d-1}$.       
	\end{condition}

\begin{remark}
	\label{sndw_simple_spec_remark}
	The assumption~$2^\circ$ of Condition~\ref{sndw_cond2} is a fortiori satisfied if the spectrum of the problem~\eqref{hatS_gener_spec_problem} is simple for any  $\boldsymbol{\theta} \in \mathbb{S}^{d-1}$.
\end{remark}

So, we assume that Condition~\ref{sndw_cond1} is satisfied. 
We are interested in the pairs of indices from the set 
\begin{equation*}
\mathcal{K} := \{ (k,r) \colon 1 \le k,r \le n, \; k \ne r, \;  (\mathcal{P}^{(k)} (\boldsymbol{\theta}))^* \widehat{N}_Q (\boldsymbol{\theta}) \mathcal{P}^{(r)} (\boldsymbol{\theta}) \not\equiv 0 \}.
\end{equation*}

Denote 
\begin{equation*}
c^{\circ}_{kr} (\boldsymbol{\theta}) :=  \min \{c_*, n^{-1} |\gamma_k (\boldsymbol{\theta}) - \gamma_r (\boldsymbol{\theta})| \}, \quad (k,r) \in \mathcal{K}.
\end{equation*}
Since the operator $S (\boldsymbol{\theta})$ depends on   $\boldsymbol{\theta}$ continuously, then 
$\gamma_j (\boldsymbol{\theta})$~are continuous functions on the sphere $\mathbb{S}^{d-1}$. 
By Condition~\ref{sndw_cond1}($2^\circ$), we have~$|\gamma_k (\boldsymbol{\theta}) - \gamma_r (\boldsymbol{\theta})| > 0$ for $(k,r) \in \mathcal{K}$ and all $\boldsymbol{\theta} \in \mathbb{S}^{d-1}$, whence 
$c^{\circ}_{kr} :=  \min_{\boldsymbol{\theta} \in \mathbb{S}^{d-1}} c^{\circ}_{kr} (\boldsymbol{\theta}) > 0$, 
$(k,r) \in \mathcal{K}$. We put 
\begin{equation}
\label{c^circ}
c^{\circ} := \min_{(k,r) \in \mathcal{K}} c^{\circ}_{kr}.
\end{equation}

Clearly, the number~(\ref{c^circ})~is a realization of the value~(\ref{abstr_c^circ}) chosen independent of  
$\boldsymbol{\theta}$. Under Condition~\ref{sndw_cond1}, the number $t^{00}$ subject to~(\ref{abstr_t^00}) 
also can be chosen independent of $\boldsymbol{\theta} \in \mathbb{S}^{d-1}$. Taking~(\ref{delta_fixation}) and~(\ref{X_1_estimate}) into account, we put 
\begin{equation*}
t^{00} = (8 \beta_2)^{-1} r_0 \alpha_1^{-3/2} \alpha_0^{1/2} \| g\|_{L_{\infty}}^{-3/2} \| g^{-1}\|_{L_{\infty}}^{-1/2} \|f\|_{L_\infty}^{-3} \|f^{-1}\|_{L_\infty}^{-1} c^{\circ}.
\end{equation*}
(Condition $t^{00} \le t_{0}$ is satisfied because 
$c^{\circ} \le \| S (\boldsymbol{\theta}) \| \le \alpha_1 \|g\|_{L_{\infty}}\|f\|_{L_\infty}^2$.)

Similarly to the proof of Theorem \ref{th12.2}, we deduce the following result from Theorem~\ref{th5.7}.

\begin{theorem}
	\label{th12.3}
Suppose that Condition~\emph{\ref{sndw_cond1}} \emph{(}or more restrictive 
Condition~\emph{\ref{sndw_cond2}}\emph{)} is satisfied. Then for $\tau \in \mathbb{R},$ $\varepsilon > 0,$ and $\mathbf{k} \in \widetilde{\Omega}$ we have 
	\begin{align}
	\nonumber
	&\|  J_1 (\mathbf{k}, \eps^{-1} \tau) \mathcal{R}(\mathbf{k}, \varepsilon)^{3/4} \|_{L_2(\Omega) \to L_2(\Omega)} \le \mathcal{C}_5 (1 + |\tau|)^{1/2} \varepsilon,
	\\
	\label{12.18}
	&\|  J_3 (\mathbf{k}, \eps^{-1} \tau) \mathcal{R}(\mathbf{k}, \varepsilon)^{1/4} \|_{L_2(\Omega) \to L_2(\Omega)} \le \mathcal{C}_6 (1 + |\tau|)^{1/2}.
		\end{align}
The constants $\mathcal{C}_5$ and $\mathcal{C}_6$  depend on $\alpha_0,$ $\alpha_1,$ $\|g\|_{L_\infty},$ $\|g^{-1}\|_{L_\infty},$ $\|f\|_{L_\infty},$ $\|f^{-1}\|_{L_\infty},$ $r_0,$ and also on $n$ and~$c^\circ$.
\end{theorem}

The following statement can be checked by analogy with the proof of Proposition \ref{prop12.2a}.

\begin{proposition}
\label{prop12.3a}
Under the assumptions of Theorem \emph{\ref{th12.3}}, for $\tau \in \R,$ $\eps >0,$ and $\k \in \wt{\Omega}$ 
we have
\begin{align*}
\|  {J}_3(\mathbf{k}, \eps^{-1} \tau ) \|_{L_2(\Omega) \to L_2(\Omega)}
\le {\mathcal C}_6' (1+ \eps^{-1/3}|\tau|^{1/3}).
\end{align*}
The constant  ${\mathcal C}_6'$ depends on $\alpha_0,$ $\alpha_1,$ $\|g\|_{L_\infty},$  $\|g^{-1}\|_{L_\infty},$
$\|f\|_{L_\infty},$  $\|f^{-1}\|_{L_\infty},$  $r_0,$ and also on $n$ and~$c^\circ$.
\end{proposition}

\subsection{Approximation  of the sandwiched operator {$\mathcal{A}(\mathbf{k})^{-1/2} \sin(\varepsilon^{-1} \tau \mathcal{A} (\mathbf{k})^{1/2})$} in the \hbox{``energy''} norm}
Denote 
\begin{align}
\nonumber
\check{J} (\mathbf{k}, \tau) 
&:= f \mathcal{A}(\mathbf{k})^{-1/2} \sin(\tau \mathcal{A}(\mathbf{k})^{1/2}) f^{-1} 
- (I + \Lambda_Q b(\mathbf{D} + \mathbf{k})\wh{P}) f_0 \mathcal{A}^0 (\mathbf{k})^{-1/2} \sin(\tau \mathcal{A}^0(\mathbf{k})^{1/2}) f_0^{-1}, 
\\
\label{9.0}
J (\mathbf{k}, \tau) 
&:= f \mathcal{A}(\mathbf{k})^{-1/2} \sin(\tau \mathcal{A}(\mathbf{k})^{1/2}) f^{-1} 
 - (I + \Lambda b(\mathbf{D} + \mathbf{k}) \wh{P}) f_0 \mathcal{A}^0 (\mathbf{k})^{-1/2} \sin( \tau \mathcal{A}^0(\mathbf{k})^{1/2}) f_0^{-1}.
\end{align}
Applying Theorem~\ref{abstr_sin_sandwiched_general_thrm} and taking~(\ref{R_P}),  (\ref{f_0 hatS f_0 P = A^0}), and
 \eqref{11.14a} into account, we obtain  
\begin{equation}
\label{sndw1_LambdaQ_sin_main_est_wP}
\| \widehat{\mathcal{A}}(\mathbf{k})^{1/2} \check{J} (\mathbf{k}, \varepsilon^{-1} \tau) \mathcal{R}(\mathbf{k}, \varepsilon) \widehat{P}\| \le \mathcal{C}'_7
 (1\! +\! |\tau|) \varepsilon,  \quad \varepsilon\! >\! 0, \; \tau\! \in\! \mathbb{R}, \; |\mathbf{k}| \!\le\! t_0.
\end{equation}
The constant $\mathcal{C}'_7$ depends only on $\alpha_0$, $\alpha_1$, $\|g\|_{L_\infty}$, $\|g^{-1}\|_{L_\infty}$, $\|f\|_{L_\infty}$, $\|f^{-1}\|_{L_\infty}$, and $r_0$.
(For brevity, we omit the index of the operator norm in  $L_2(\Omega;\AC^n)$.) 

We show that, within  the margin of error, $\Lambda_Q$ can be replaced by $\Lambda$ in~(\ref{sndw1_LambdaQ_sin_main_est_wP}). Recall that $\Lambda_Q = \Lambda+ \Lambda_Q^0$. 
Combining \eqref{Lambda_est}, \eqref{Lambda_Q=Lambda+Lambda_Q^0}, and 
 \eqref{f_0_estimates}, we obtain 
\begin{equation}\label{9.2}
| \Lambda_Q^0 | \le  (2 r_0)^{-1} \alpha_0^{-1/2} \|g\|_{L_\infty}^{1/2} \|g^{-1}\|_{L_\infty}^{1/2} \|f\|_{L_\infty}^{2} \|f^{-1}\|_{L_\infty}^{2}.
\end{equation}
By \eqref{rank_alpha_ineq}, 
\begin{align}
\label{hatA(k)_hatP_est}
\| \widehat{\mathcal{A}}(\mathbf{k})^{1/2} \widehat{P} \| 
= \| g^{1/2} b(\mathbf{k}) \widehat{P} \| \le \alpha_1^{1/2} 
\|g\|^{1/2}_{L_\infty} |\mathbf{k}|, \quad \mathbf{k}\in \widetilde{\Omega}.
\end{align}
From~\eqref{R_P}, \eqref{f_0_estimates}, \eqref{8.4a},  \eqref{9.2}, and \eqref{hatA(k)_hatP_est} it follows that 
\begin{multline}
\label{Lambda0_Q_est_1}
\|\widehat{\mathcal{A}}(\mathbf{k})^{1/2} \Lambda_Q^0 b(\mathbf{D} + \mathbf{k}) f_0 \mathcal{A}^0 (\mathbf{k})^{-1/2} \sin(\varepsilon^{-1} \tau \mathcal{A}^0(\mathbf{k})^{1/2}) f_0^{-1} \mathcal{R}(\mathbf{k}, \varepsilon)^{1/2} \widehat{P}\| 
\\  
\le  \alpha^{1/2}_1  \|g \|^{1/2}_{L_\infty} \|g^{-1} \|^{1/2}_{L_\infty}  \|f^{-1}\|_{L_\infty} |\Lambda_Q^0| |\mathbf{k}| \varepsilon (|\mathbf{k}|^2 + \varepsilon^2)^{-1/2} \le {\mathcal C}''_7 \varepsilon,
\end{multline}
  where the constant ${\mathcal C}''_7$ depends on $\alpha_0,\alpha_1,\|g\|_{L_\infty},\|g^{-1}\|_{L_\infty},\|f\|_{L_\infty},\|f^{-1}\|_{L_\infty}$, and~$r_0$.

Relations \eqref{sndw1_LambdaQ_sin_main_est_wP} and \eqref{Lambda0_Q_est_1} imply that  
\begin{equation}
\label{sndw1_sin_main_est_wP}
\| \widehat{\mathcal{A}}(\mathbf{k})^{1/2} J (\mathbf{k}, \varepsilon^{-1} \tau) \mathcal{R}(\mathbf{k}, \varepsilon) \widehat{P}\| \le \mathcal{C}'_7 
 (1 + |\tau|) \varepsilon + {\mathcal C}''_7 \varepsilon,
 \quad \varepsilon > 0, \; \tau \in \mathbb{R}, \; |\mathbf{k}| \le t_0.
\end{equation}

 Estimates for $ | \mathbf{k} | > {t}_0$ are trivial. By~(\ref{R_P}), we have
\begin{equation}
\label{sndw_R_hatP_|k|>t^0_est}
\| \mathcal{R}(\mathbf{k}, \varepsilon)^{1/2}\widehat{P} \|_{L_2 (\Omega) \to L_2 (\Omega)} \le t_0^{-1} \varepsilon, \quad \varepsilon > 0, \; \mathbf{k} \in \widetilde{\Omega}, \; | \mathbf{k} | > t_0.
\end{equation}
Since  ${\mathcal{A}}(\mathbf{k}) = f^*\widehat{\mathcal{A}}(\mathbf{k}) f$, then  
\begin{equation}
\label{sndw_sin_est_|k|>t^0_1}
\| \widehat{\mathcal{A}}(\mathbf{k})^{1/2} f \mathcal{A}(\mathbf{k})^{-1/2} \sin(\varepsilon^{-1} \tau \mathcal{A}(\mathbf{k})^{1/2}) f^{-1} \| 
\!\le \! \|f^{-1}\|_{L_\infty},  \quad \!\varepsilon \!>\! 0, \; \mathbf{k} \!\in \!\widetilde{\Omega}.
\end{equation}
Next, by~(\ref{hatA(k)}), (\ref{g^0_est}), \eqref{f_0_estimates}, and \eqref{8.4a},   
\begin{multline}
\label{sndw_sin_est_|k|>t^0_3}
\| \widehat{\mathcal{A}}(\mathbf{k})^{1/2} f_0 \mathcal{A}^0 (\mathbf{k})^{-1/2} \sin(\varepsilon^{-1} \tau \mathcal{A}^0(\mathbf{k})^{1/2}) f_0^{-1} \| 
 \\
  =  \| g^{1/2} b(\mathbf{D} + \mathbf{k}) f_0 \mathcal{A}^0 (\mathbf{k})^{-1/2} \sin(\varepsilon^{-1} \tau \mathcal{A}^0(\mathbf{k})^{1/2}) f_0^{-1} \| 
  \\
 \le \|g\|_{L_\infty}^{1/2} \|g^{-1}\|_{L_\infty}^{1/2} \|f^{-1}\|_{L_\infty}, \quad \varepsilon > 0, \; \mathbf{k} \in \widetilde{\Omega}.
\end{multline}
Taking (\ref{g^0_est}), \eqref{sqrt_hatA_Lambda_P_est}, \eqref{f_0_estimates}, and \eqref{8.4a} into account, we obtain 
\begin{multline}
\label{sndw_sin_est_|k|>t^0_5}
\|\widehat{\mathcal{A}}(\mathbf{k})^{1/2} \Lambda b(\mathbf{D} + \mathbf{k}) \wh{P} f_0 \mathcal{A}^0 (\mathbf{k})^{-1/2} \sin(\varepsilon^{-1} \tau \mathcal{A}^0(\mathbf{k})^{1/2}) f_0^{-1} \| 
\\ 
\le C_\Lambda
\|b(\mathbf{D} + \mathbf{k}) f_0 \mathcal{A}^0(\mathbf{k})^{-1/2} \sin(\varepsilon^{-1} \tau \mathcal{A}^0(\mathbf{k})^{1/2}) f_0^{-1} \|  
 \\
  \le C_{\Lambda} \|g^{-1}\|_{L_\infty}^{1/2} \|f^{-1}\|_{L_\infty},  \qquad \varepsilon > 0, \; \mathbf{k} \in \widetilde{\Omega}.
\end{multline}
So, from~(\ref{sndw_R_hatP_|k|>t^0_est})--(\ref{sndw_sin_est_|k|>t^0_5}) it follows that 
\begin{align}
\label{sndw_sin_est_|k|>t^0_-1}
\bigl\| \widehat{\mathcal{A}}(\mathbf{k})^{1/2}\! J (\mathbf{k}, \varepsilon^{-1}\! \tau)  \mathcal{R}(\mathbf{k}, \varepsilon)^{1/2}  \widehat{P} \bigr\| \le {\mathcal C}_7''' \varepsilon,\quad \!\varepsilon \!>\! 0, \; \tau\! \in\! \mathbb{R}, \; \mathbf{k} \!\in \!\widetilde{\Omega}, \; |\mathbf{k}|\! >\! t_0,
\end{align}
where ${\mathcal C}_7''' = (1 +  \|g\|_{L_\infty}^{1/2} \|g^{-1}\|_{L_\infty}^{1/2} + C_{\Lambda} \|g^{-1}\|_{L_\infty}^{1/2}) \|f^{-1}\|_{L_\infty}  t_0^{-1}$.
\smallskip

By~\eqref{R(k,eps)(I-P)_est} with $s=1$, \eqref{sndw_sin_est_|k|>t^0_1}, and \eqref{sndw_sin_est_|k|>t^0_3}, 
 \begin{align}
\label{sndw_sin_est_wo_hatP_-1}
\bigl\| \widehat{\mathcal{A}}(\mathbf{k})^{1/2} \!J(\mathbf{k}, \varepsilon^{-1} \tau) \mathcal{R}(\mathbf{k}, \varepsilon)^{1/2}  (I\!-\!\widehat{P}) \bigr\|  \le  \check{\mathcal C}_7 \varepsilon, \!\quad \varepsilon \!> \!0, \; \tau \!\in\! \mathbb{R}, \; \mathbf{k}\! \in \!\widetilde{\Omega},
\end{align}
where $\check{\mathcal C}_7= r_0^{-1}  (1 +  \|g\|_{L_\infty}^{1/2}  \|g^{-1}\|_{L_\infty}^{1/2}) \|f^{-1}\|_{L_\infty}$.

\smallskip
As a result, using~(\ref{sndw1_sin_main_est_wP}), (\ref{sndw_sin_est_|k|>t^0_-1}), and  (\ref{sndw_sin_est_wo_hatP_-1}),  we obtain the following result (proved earlier in~\cite[(7.36)]{M}).

\begin{theorem}[see~\cite{M}]
	\label{sndw_A(k)_sin_general_thrm}
	Suppose that  $J(\mathbf{k},  \tau)$ is the operator defined by \eqref{9.0}. Then 
	\begin{align*}
	\bigl\| \widehat{\mathcal{A}}(\mathbf{k})^{1/2} J(\mathbf{k}, \varepsilon^{-1} \tau)  \mathcal{R}(\mathbf{k}, \varepsilon) 
	\bigr\|_{L_2(\Omega) \to L_2 (\Omega) }  \le \mathcal{C}_7 (1 + |\tau|) \varepsilon
	\end{align*}
	for $\tau \in \mathbb{R},$ $\varepsilon > 0$, and $\mathbf{k} \in \widetilde{\Omega}$.
	The constant $\mathcal{C}_7$ depends only on $\alpha_0,$ $\alpha_1,$ $\|g\|_{L_\infty},$ $\|g^{-1}\|_{L_\infty},$ $\|f\|_{L_\infty},$ $\|f^{-1}\|_{L_\infty},$ $r_0,$ and $r_1$.
\end{theorem}

\subsection{Approximation  of the sandwiched operator  {$\mathcal{A}(\mathbf{k})^{-1/2}\sin(\varepsilon^{-1} \tau \mathcal{A} (\mathbf{k})^{1/2})$} in the energy norm. Improvement of the results}
Now we apply Theorem~\ref{abstr_sin_sandwiched_ench_thrm_1} assuming that Condition
\ref{cond_BB} is satisfied.  Taking~(\ref{R_P}) and~(\ref{f_0 hatS f_0 P = A^0}) into account, we have \begin{equation*}
\|\widehat{\mathcal{A}}(\mathbf{k})^{1/2} \check{J} (\mathbf{k}, \eps^{-1}\tau) \mathcal{R}(\mathbf{k}, \varepsilon)^{3/4} \widehat{P}\|_{L_2(\Omega) \to L_2(\Omega)} \le 
\mathcal{C}'_8  (1 + |\tau|)^{1/2} \varepsilon, 
 \quad \varepsilon > 0, \; \tau \in \mathbb{R}, \; |\mathbf{k}| \le t_0.
\end{equation*}

Together with~(\ref{Lambda0_Q_est_1}),  (\ref{sndw_sin_est_|k|>t^0_-1}), and  (\ref{sndw_sin_est_wo_hatP_-1}), this yields the following result.

\begin{theorem}
	\label{sndw_sin_enchanced_thrm_1}
Suppose that Condition \emph{\ref{cond_BB}} is satisfied.
  Then for $\tau \in \mathbb{R},$ $\varepsilon > 0,$ and $\mathbf{k} \in \widetilde{\Omega}$ we have
  	\begin{align*}
	\bigl\| \widehat{\mathcal{A}}(\mathbf{k})^{1/2} J(\mathbf{k}, \varepsilon^{-1} \tau)  \mathcal{R}(\mathbf{k}, \varepsilon)^{3/4} \bigr\|_{L_2(\Omega) \to L_2 (\Omega) }  \le \mathcal{C}_8 (1 + |\tau|)^{1/2} \varepsilon,
	\end{align*}
where $\mathcal{C}_8$   depends only on  $\alpha_0,\alpha_1,\|g\|_{L_\infty},\|g^{-1}\|_{L_\infty},\|f\|_{L_\infty},\|f^{-1}\|_{L_\infty},r_0,$ and~$r_1$.
\end{theorem}

Similarly, the following result is deduced from Theorem~\ref{abstr_sin_sandwiched_ench_thrm_2}
and relations~\eqref{Lambda0_Q_est_1},  \eqref{sndw_sin_est_|k|>t^0_-1}   (with $t_0$ replaced by $t^{00}$), and \eqref{sndw_sin_est_wo_hatP_-1}.

\begin{theorem}
	\label{sndw_sin_enchanced_thrm_2}
	Suppose that Condition~\emph{\ref{sndw_cond1}} \emph{(}or more restrictive 
	Condition~\emph{\ref{sndw_cond2}}\emph{)} is satisfied. Then for  $\tau \in \mathbb{R},$ $\varepsilon > 0,$ and $\mathbf{k} \in \widetilde{\Omega}$ we have
	\begin{equation*}
	\begin{split}
	\| \widehat{\mathcal{A}}(\mathbf{k})^{1/2} J (\mathbf{k}, \eps^{-1} \tau) \mathcal{R}(\mathbf{k}, \varepsilon)^{3/4} \|_{L_2(\Omega) \to L_2(\Omega)} \le \mathcal{C}_9 (1 + |\tau|)^{1/2} \varepsilon.
	\end{split}
	\end{equation*}
	The constant $\mathcal{C}_9$  depends on $\alpha_0,$ $\alpha_1,$ $\|g\|_{L_\infty},$ $\|g^{-1}\|_{L_\infty},$ $\|f\|_{L_\infty},$ $\|f^{-1}\|_{L_\infty},$ $r_0,$ $r_1,$ and also on $n$ and $c^\circ$.
\end{theorem}

\section{Sharpness of the results of \S\ref{sec12}\label{sec13}}

\subsection{Sharpness of the results regarding  the smoothing factor}

In the statements of the present section  we impose one of the following two conditions.

\begin{condition}
\label{cond13.1}
Let $\widehat{N}_{0,Q} (\boldsymbol{\theta})$ be the operator defined by~\emph{(\ref{N0Q_invar_repr})}. Suppose that $\widehat{N}_{0,Q} (\boldsymbol{\theta}_0) \ne 0$ at some point  $\boldsymbol{\theta}_0 \in \mathbb{S}^{d-1}$.
\end{condition}

\begin{condition}
\label{cond13.2}
	 Let $\widehat{N}_{0,Q} (\boldsymbol{\theta})$ and
	$\widehat{\mathcal N}^{(q)}_{Q} (\boldsymbol{\theta})$ be the operators defined by~\eqref{N0Q_invar_repr} and \eqref{8.32a}, respectively. Suppose that  $\widehat{N}_{0,Q} (\boldsymbol{\theta})=0$ for all  
	 $\boldsymbol{\theta} \in \mathbb{S}^{d-1}$. Suppose that  
	 ${\widehat{\mathcal N}^{(q)}_{Q} (\boldsymbol{\theta}_0) \ne 0}$ for some 
	 $\boldsymbol{\theta}_0 \in \mathbb{S}^{d-1}$ and $q \in \{ 1,\dots, p(\boldsymbol{\theta}_0)\}$.
\end{condition}

We need the following lemma (see \cite[Lemma 9.8]{DSu}).

\begin{lemma}[see~\cite{DSu}]
	\label{A(k)_Lipschitz_lemma}
	Let $\delta$ be defined by~\emph{(\ref{delta_fixation})} and let $t_0$ be given by~\emph{(\ref{t0_fixation})}. Suppose that $F (\mathbf{k})$~is the spectral projection of the operator $\mathcal{A}(\mathbf{k})$ for the interval $[0, \delta]$. Then for $|\mathbf{k}| \le t_0$ and  $|\mathbf{k}_0| \le t_0$ we have 
	\begin{align*}
	\| \mathcal{A}(\mathbf{k})^{1/2} F(\mathbf{k})  - \mathcal{A}(\mathbf{k}_0)^{1/2} F(\mathbf{k}_0)\|_{L_2(\Omega) \to L_2 (\Omega) } &\le C' | \mathbf{k} - \mathbf{k}_0|,
	\\
	\|  \cos ( \tau \mathcal{A}(\mathbf{k})^{1/2}) F(\mathbf{k})  -  \cos ( \tau \mathcal{A}(\mathbf{k}_0)^{1/2}) F(\mathbf{k}_0)\|_{L_2(\Omega) \to L_2 (\Omega) } &\le C''(\tau) | \mathbf{k} - \mathbf{k}_0|,
	\\
	\| \mathcal{A}(\mathbf{k})^{-1/2} \sin ( \tau \mathcal{A}(\mathbf{k})^{1/2}) F(\mathbf{k})  - \mathcal{A}(\mathbf{k}_0)^{-1/2} \sin ( \tau \mathcal{A}(\mathbf{k}_0)^{1/2})& F(\mathbf{k}_0)\|_{L_2(\Omega) \to L_2 (\Omega) } 
	\le C'''(\tau) | \mathbf{k} - \mathbf{k}_0|.
	\end{align*}
\end{lemma}

Applying Theorem~\ref{th6.1}, we confirm that Theorems~\ref{th12.1}
and~\ref{sndw_A(k)_sin_general_thrm} are sharp.

\begin{theorem}
	\label{th13.2}
	Suppose that Condition \emph{\ref{cond13.1}} is satisfied.

\noindent $1^\circ$. Let $0 \ne \tau \in \mathbb{R}$ and $0 \le s < 2$. Then there does not 
exist  a constant $\mathcal{C}(\tau) > 0$ such that  the estimate
		\begin{equation}
		\label{13.1}
		\bigl\|  J_1(\mathbf{k}, \varepsilon^{-1} \tau)  \mathcal{R}(\mathbf{k}, \varepsilon)^{s/2} \bigr\|_{L_2(\Omega) \to L_2 (\Omega) }  \le \mathcal{C} (\tau) \varepsilon
		\end{equation}
holds for almost all $\mathbf{k} \in \widetilde{\Omega}$ and sufficiently small $\varepsilon > 0$.

\noindent
$2^\circ$. Let $0 \ne \tau \in \mathbb{R}$ and $0 \le r < 1$. Then there does not exist a constant 
$\mathcal{C}(\tau) > 0$ such that the estimate
		\begin{equation}
		\label{13.2}
		\bigl\|  J_2(\mathbf{k}, \varepsilon^{-1} \tau)  \mathcal{R}(\mathbf{k}, \varepsilon)^{r/2} \bigr\|_{L_2(\Omega) \to L_2 (\Omega) }  \le \mathcal{C} (\tau)
		\end{equation}
holds for almost all $\mathbf{k} \in \widetilde{\Omega}$ and sufficiently small $\varepsilon > 0$.

\noindent
$3^\circ$. Let $0 \ne \tau \in \mathbb{R}$ and $0 \le r < 1$. Then there does not exist a constant
$\mathcal{C}(\tau) > 0$ such that the estimate
		\begin{equation}
		\label{13.3}
		\bigl\|  J_3(\mathbf{k}, \varepsilon^{-1} \tau)  \mathcal{R}(\mathbf{k}, \varepsilon)^{r/2} \bigr\|_{L_2(\Omega) \to L_2 (\Omega) }  \le \mathcal{C} (\tau) 
		\end{equation}
holds for almost all $\mathbf{k} \in \widetilde{\Omega}$ and sufficiently small $\varepsilon > 0$.

\noindent $4^\circ$. Let $0 \ne \tau \in \mathbb{R}$ and $0 \le s < 2$. Then there does not exist a constant
$\mathcal{C}(\tau) > 0$ such that the estimate
		\begin{equation}
		\label{s<2_sinA(k)_est_1}
		\bigl\| \widehat{\mathcal{A}}(\mathbf{k})^{1/2} J(\mathbf{k}, \varepsilon^{-1} \tau)  \mathcal{R}(\mathbf{k}, \varepsilon)^{s/2} \bigr\|_{L_2(\Omega) \to L_2 (\Omega) }  \le \mathcal{C} (\tau) \varepsilon
		\end{equation}
holds for almost all $\mathbf{k} \in \widetilde{\Omega}$ and sufficiently small $\varepsilon > 0$.
\end{theorem}

\begin{proof}
Statements  $1^\circ$ and $3^\circ$ were proved in  \cite[Theorem 9.7]{DSu}.

Let us check statement $2^\circ$. 
(In the proof, we omit the index of the operator norm in  $L_2(\Omega;\AC^n)$.) We prove by contradiction. 
Suppose the opposite. Then  for some  $\tau \ne 0$ and $0 \le r < 1$ we have 
\begin{equation}
		\label{13.4}
		\bigl\|  J_2(\mathbf{k}, \varepsilon^{-1} \tau) \wh{P}  \bigr\|  
		\eps^r (|\k|^2 + \eps^2)^{-r/2} \le \mathcal{C} (\tau) 
		\end{equation}
for almost all  $\mathbf{k}  \in \widetilde{\Omega}$ and sufficiently small  $\varepsilon$.
 Obviously, 
 		\begin{equation}
		\label{13.4a}
		\bigl\|  f {\mathcal A}(\k)^{-1/2} \sin (\eps^{-1} \tau {\mathcal A}(\k)^{1/2}) F(\k)^\perp  f^{-1}
		\bigr\|   \le \| f\|_{L_\infty} \| f^{-1}\|_{L_\infty} \delta^{-1/2}.  
		\end{equation}
Combining this with  \eqref{13.4}, we see that for some constant $\wt{\mathcal C}(\tau)>0$ 
the estimate
\begin{equation}
	\label{13.7}
\begin{split}
\big\|& \bigl(f {\mathcal{A}}(\mathbf{k})^{-1/2} \sin (\eps^{-1} \tau {\mathcal{A}}(\mathbf{k})^{1/2}) F(\k) f^{-1} 
\\
&-f_0 {\mathcal{A}^0}(\mathbf{k})^{-1/2} \sin (\eps^{-1} \tau {\mathcal{A}}^0(\mathbf{k})^{1/2}) f_0^{-1} \bigr) \wh{P}
\big\|  \eps^r (|\k|^2 + \eps^2)^{-r/2} \le \wt{\mathcal C}(\tau)
\end{split}
\end{equation}
holds for almost all $\mathbf{k}  \in \widetilde{\Omega}$ and sufficiently small $\varepsilon$.

Let $|\k| \le t_0$. From Lemma~\ref{A(k)_Lipschitz_lemma} it follows that the operator under the norm sign in  \eqref{13.7} is continuous with respect to $\k$ in the ball $|\k| \le t_0$. Hence, estimate \eqref{13.7} is valid for any 
 $\k$ in this ball, in particular, for $\k = t \boldsymbol{\theta}_0$ if $t\le t_0$. Applying 
 inequality \eqref{13.4a} once again, we obtain 
\begin{align*}
\big\|& (f {\mathcal{A}}( t \boldsymbol{\theta}_0)^{-1/2} \sin (\eps^{-1} \tau {\mathcal{A}}( t \boldsymbol{\theta}_0)^{1/2})f^{-1}
 \\
&-f_0 {\mathcal{A}^0}( t \boldsymbol{\theta}_0)^{-1/2} \sin (\eps^{-1} \tau 
{\mathcal{A}}^0( t \boldsymbol{\theta}_0)^{1/2}) f_0^{-1} )\wh{P}\big \| 
 \eps^r (t^2 + \eps^2)^{-r/2} \le \wh{\mathcal C}(\tau)
\end{align*}
with some constant $\wh{\mathcal C}(\tau)>0$ for $t\le t_0$ and sufficiently small $\eps$.
In abstract terms, this estimate corresponds to  inequality \eqref{6*.2}. 
By our assumption, we have $\widehat{N}_{0,Q} (\boldsymbol{\theta}_0) \ne 0$. So, the assumption of 
Theorem~\ref{th6.1} is satisfied. 
Applying statement~$2^\circ$ of this theorem, we arrive at a contradiction. 

We proceed to the proof of statement $4^\circ$.
We prove by contradiction. Suppose that for some $\tau \ne 0$ and $1 \le s < 2$ there exists a constant $\mathcal{C}(\tau) > 0$ such that estimate~(\ref{s<2_sinA(k)_est_1}) holds for almost all $\mathbf{k} \in \widetilde{\Omega}$ and sufficiently small  $\varepsilon > 0$. 
	Multiplying the operator under the norm sign in \eqref{s<2_sinA(k)_est_1} by $\wh{P}$ and taking~\eqref{R_P} and  (\ref{Lambda0_Q_est_1}) into account, we see that for some constant $\widetilde{\mathcal{C}}(\tau) > 0$ the estimate  	
	\begin{multline}
	\label{s<2_proof_f1}
	\bigl\| \widehat{\mathcal{A}}(\mathbf{k})^{1/2} \bigl( f \mathcal{A}(\mathbf{k})^{-1/2} \sin (\varepsilon^{-1} \tau \mathcal{A}(\mathbf{k})^{1/2}) f^{-1}  
	\\ - (I + \Lambda_Q b(\mathbf{D} + \mathbf{k})) f_0 \mathcal{A}^0(\mathbf{k})^{-1/2} \sin (\varepsilon^{-1} \tau \mathcal{A}^0(\mathbf{k})^{1/2})f_0^{-1} \bigr) \widehat{P}  \bigr\| 
	\\
\times	\varepsilon^s (|\mathbf{k}|^2 + \varepsilon^2)^{-s/2}  \le \widetilde{\mathcal{C}}(\tau) \varepsilon
	\end{multline}
	holds for almost all $\mathbf{k} \in \wt{\Omega}$ and sufficiently small $\varepsilon$.
	
	Next, we apply~\eqref{abstr_sin_sandwiched_est_3} and the relation 
	$(I + |\mathbf{k}|Z(\boldsymbol{\theta})) P = (F(\mathbf{k}) - F_2(\mathbf{k}))P$ (see  \eqref{abstr_F(t)_threshold_2}, \eqref{abstr_F1_K0_N_invar}). Then from \eqref{s<2_proof_f1} it follows that 
	the estimate
	\begin{multline}
	\label{9.21}
	\bigl\| \mathcal{A}(\mathbf{k})^{1/2}  \bigl( \mathcal{A}(\mathbf{k})^{-1/2} \sin (\varepsilon^{-1} \tau \mathcal{A}(\mathbf{k})^{1/2})
	\\
	- (F(\mathbf{k}) - F_2(\mathbf{k}))   S(\mathbf{k})^{-1/2} \sin ( \varepsilon^{-1} \tau  S(\mathbf{k})^{1/2} ) P \bigr) P  \bigr\|
	\\
	\times\varepsilon^s (|\mathbf{k}|^2 + \varepsilon^2)^{-s/2}  \le \check{\mathcal{C}}(\tau) \varepsilon
	\end{multline}
	holds for almost all  $\mathbf{k} \!\in\! \widetilde{\Omega}$ and sufficiently small  $\varepsilon$ with some constant  $\check{\mathcal{C}}(\tau) \!>\!0$.

	By~\eqref{abstr_F(t)_threshold_1} and \eqref{abstr_sqrtA(t)F2(t)_est},
	\begin{align}
	\label{s<2_proof_f2}
	&\| F (\mathbf{k}) - P \| \le C_1 |\mathbf{k}|, \quad |\mathbf{k}| \le {t}_0,
	\\
	\label{s<2_proof_f3}
	&\| \mathcal{A}(\mathbf{k})^{1/2} F_2 (\mathbf{k}) \| \le C_{16} |\mathbf{k}|^2, \quad |\mathbf{k}| \le {t}_0.  
	\end{align}	
	From~\eqref{S_nondegenerated} and \eqref{9.21}--(\ref{s<2_proof_f3})  it follows that
	\begin{multline}
	\label{s<2_proof_f4}
	\bigl\| \mathcal{A}(\mathbf{k})^{1/2} F(\mathbf{k}) \bigl( \mathcal{A}(\mathbf{k})^{-1/2} \sin (\varepsilon^{-1} \tau \mathcal{A}(\mathbf{k})^{1/2}) F(\mathbf{k})
	\\
	-  S(\mathbf{k})^{-1/2} \sin ( \varepsilon^{-1} \tau  S(\mathbf{k})^{1/2} ) P \bigr) P  \bigr\| \varepsilon^s (|\mathbf{k}|^2 + \varepsilon^2)^{-s/2}  \le \check{\mathcal{C}}'(\tau) \varepsilon
	\end{multline}
	for almost all $\mathbf{k}$ in the ball $|\mathbf{k}| \le t_0$ and sufficiently small 
	$\varepsilon$ with some constant  $\check{\mathcal{C}}'(\tau) >0$.

	From  Lemma~\ref{A(k)_Lipschitz_lemma} it follows that the operator under the norm sign 
	in~(\ref{s<2_proof_f4}) is continuous with respect to $\mathbf{k}$ in the ball $|\mathbf{k}| \le t_0$. Hence, estimate~(\ref{s<2_proof_f4}) is valid for all  $\mathbf{k}$ in this ball. In particular, it holds for $\mathbf{k} = t\boldsymbol{\theta}_0$ if $t \le t_0$. Applying once again the formula $(F(\mathbf{k}) - F_2(\mathbf{k})) P = P + |\mathbf{k}| Z(\boldsymbol{\theta}) P$ and inequalities~ (\ref{S_nondegenerated}), (\ref{s<2_proof_f2}), (\ref{s<2_proof_f3}), and next estimate~(\ref{abstr_sin_sandwiched_est_1}), we obtain  
	\begin{align*}
	\big\| \widehat{\mathcal{A}}(t\boldsymbol{\theta}_0)^{1/2} ( f \mathcal{A}(t\boldsymbol{\theta}_0)^{-1/2}\! \sin (\varepsilon^{-1} \tau \mathcal{A}(t\boldsymbol{\theta}_0)^{1/2})f^{-1} \!
	- \!(I \!+\! \Lambda_Q b(t\boldsymbol{\theta}_0)) f_0 \mathcal{A}^0(t\boldsymbol{\theta}_0)^{-1/2}
	\\
	\times\sin (\varepsilon^{-1} \tau \mathcal{A}^0(t\boldsymbol{\theta}_0)^{1/2})f_0^{-1} ) \widehat{P}  \big\| 
	\varepsilon^s (t^2 + \varepsilon^2)^{-s/2} \le \check{\mathcal{C}}''(\tau) \varepsilon
	\end{align*}
	for all $t \le t_0$ and sufficiently small $\varepsilon$.
	In abstract terms, this estimate corresponds to estimate~(\ref{abstr_sndwchd_s<2_est_imp1}). 
	By our assumption,  $\widehat{N}_{0,Q} (\boldsymbol{\theta}_0) \ne 0$.
	Then, applying statement~$4^\circ$ of Theorem~\ref{th6.1}, we arrive at a contradiction. 
\end{proof}

Now, using  Theorem \ref{th6.2}, we confirm that Theorems \ref{th12.2}, \ref{th12.3},  \ref{sndw_sin_enchanced_thrm_1}, and \ref{sndw_sin_enchanced_thrm_2} are sharp.

\begin{theorem}
	\label{th13.3}
		Suppose that Condition \emph{\ref{cond13.2}} is satisfied.

\noindent $1^\circ$. Let $0 \ne \tau \in \mathbb{R}$ and $0 \le s < 3/2$. Then there does not exist a constant
$\mathcal{C}(\tau) > 0$ such that estimate  \eqref{13.1}  holds for almost all
		$\mathbf{k}  \in \widetilde{\Omega}$ and sufficiently small $\varepsilon > 0$.

\noindent
$2^\circ$.  Let $0 \ne \tau \in \mathbb{R}$ and $0 \le r < 1/2$.
Then there does not exist a constant $\mathcal{C}(\tau) > 0$ such that estimate \eqref{13.3} holds for almost all  
$\mathbf{k} \in \widetilde{\Omega}$ and sufficiently small $\varepsilon > 0$.

\noindent $3^\circ$. Let $0 \ne \tau \in \mathbb{R}$ and $0 \le s < 3/2$. Then there does not exist a constant 
$\mathcal{C}(\tau) > 0$ such that estimate \eqref{s<2_sinA(k)_est_1} holds for almost all $\mathbf{k} \in \widetilde{\Omega}$ and sufficiently small  $\varepsilon > 0$.
\end{theorem}

\begin{proof}
Let us check statement $1^\circ$. 
Suppose the opposite. Then it follows that for some $\tau \ne 0$ and $1 \le s< 3/2$ the estimate
		\begin{equation}
		\label{13.9}
		\bigl\|  J_1(\mathbf{k}, \varepsilon^{-1} \tau) \wh{P}  \bigr\|  
		\eps^s (|\k|^2 + \eps^2)^{-s/2} \le \mathcal{C} (\tau) \varepsilon
		\end{equation}
holds for almost all $\mathbf{k}  \in \widetilde{\Omega}$ and sufficiently small $\varepsilon$.

Let $|\k| \le t_0$.  Using the identity $f^{-1} \wh{P} = P f^* \overline{Q}$ (see \eqref{abstr_P_and_P_hat_relation}) 
and inequality \eqref{s<2_proof_f2}, from  \eqref{13.9} we deduce the estimate (with some constant 
$\widetilde{\mathcal{C}}(\tau) > 0$)
 \begin{equation}
	\label{13.10}
\| f  \cos (\eps^{-1} \tau {\mathcal{A}}(\mathbf{k})^{1/2}) F(\k) f^* \overline{Q}
-f_0  \cos (\eps^{-1} \tau {\mathcal{A}}^0(\mathbf{k})^{1/2}) f_0^{-1} \wh{P}\| 
\eps^s (|\k|^2 + \eps^2)^{-s/2} \le \wt{\mathcal C}(\tau) \eps
\end{equation}
for almost all  $\mathbf{k}  \in \widetilde{\Omega}$ and sufficiently small $\varepsilon$.
 From Lemma~\ref{A(k)_Lipschitz_lemma} it follows that the operator under the norm sign in \eqref{13.10}
 is continuous with resect to  $\k$ in the ball $|\k| \le t_0$. Hence, estimate \eqref{13.10} holds for all
 $\k$ in this ball. In particular, it is valid for $\k = t \boldsymbol{\theta}_0$ if $t\le t_0$. 
Applying  inequality  \eqref{s<2_proof_f2} and the identity
$P f^* \overline{Q}= f^{-1} \wh{P}$ once again, we obtain the estimate
\begin{equation*}
\| \bigl(f  \cos (\eps^{-1} \tau {\mathcal{A}}( t \boldsymbol{\theta}_0)^{1/2})f^{-1} -f_0 
\cos (\eps^{-1} \tau {\mathcal{A}}^0( t \boldsymbol{\theta}_0)^{1/2}) f_0^{-1} \bigr)\wh{P} \| 
\eps^s (t^2 + \eps^2)^{-s/2} \le \wh{\mathcal C}(\tau) \eps
\end{equation*}
with some constant $\wh{\mathcal C}(\tau)>0$ for $t\le t_0$ and sufficiently small  $\eps$.
This contradicts statement $1^\circ$ of Theorem \ref{th6.2}.

We proceed to the proof of statement $2^\circ$.
Suppose the opposite. Then  for some $\tau \ne 0$ and \hbox{$0 \le r < 1/2$} we have
		\begin{equation}
		\label{13.12}
		\bigl\|  J_3(\mathbf{k}, \varepsilon^{-1} \tau) \wh{P}  \bigr\|  
		\eps^r (|\k|^2 + \eps^2)^{-r/2} \le \mathcal{C} (\tau)
		\end{equation}
for almost all $\mathbf{k}  \in \widetilde{\Omega}$ and sufficiently small $\varepsilon$.
Obviously, 
 		\begin{equation}
		\label{13.13}
		\bigl\|  f {\mathcal A}(\k)^{-1/2} \sin (\eps^{-1} \tau {\mathcal A}(\k)^{1/2}) F(\k)^\perp  f^*
		\bigr\|   \le \| f\|_{L_\infty}^2  \delta^{-1/2}.  
		\end{equation}
Combining this with  \eqref{13.12}, we see that for some constant $\wt{\mathcal C}(\tau)>0$  
the estimate 
\begin{equation}
	\label{13.14}
\begin{split}
\bigl\|& \bigl(f {\mathcal{A}}(\mathbf{k})^{-1/2} \sin (\eps^{-1} \tau {\mathcal{A}}(\mathbf{k})^{1/2}) F(\k) f^*
\\
&-f_0 {\mathcal{A}^0}(\mathbf{k})^{-1/2} \sin (\eps^{-1} \tau {\mathcal{A}}^0(\mathbf{k})^{1/2}) f_0 \bigr) \wh{P}
\bigr\| 
 \eps^r (|\k|^2 + \eps^2)^{-r/2} \le \wt{\mathcal C}(\tau)
\end{split}
\end{equation}
holds for almost all  $\mathbf{k}  \in \widetilde{\Omega}$ and sufficiently small $\varepsilon$.
From Lemma~\ref{A(k)_Lipschitz_lemma} it follows that the operator under the norm sign in \eqref{13.14}
 is continuous with respect to $\k$ in the ball $|\k| \le t_0$. Hence, estimate \eqref{13.14} holds for all 
 $\k$ in this ball. In particular, it is valid for $\k = t \boldsymbol{\theta}_0$ if $t\le t_0$. 
Applying inequality \eqref{13.13} once again, we obtain the estimate
\begin{equation*}
	\begin{aligned}
\| \bigl(f  {\mathcal{A}}( t \boldsymbol{\theta}_0)^{-1/2}
&\sin (\eps^{-1} \tau {\mathcal{A}}( t \boldsymbol{\theta}_0)^{1/2})f^{*} -f_0 {\mathcal{A}}^0( t \boldsymbol{\theta}_0)^{-1/2}
\\ &\times
\sin (\eps^{-1} \tau {\mathcal{A}}^0( t \boldsymbol{\theta}_0)^{1/2}) f_0 \bigr)\wh{P}
\|
 \eps^r (t^2 + \eps^2)^{-r/2} \le \wh{\mathcal C}(\tau)
\end{aligned}
\end{equation*}
with some constant $\wh{\mathcal C}(\tau)>0$ for $t\le t_0$ and sufficiently small $\eps$.
This contradicts statement $2^\circ$ of Theorem \ref{th6.2}.

Statement $3^\circ$ is deduced from Theorem \ref{th6.2} (statement $3^\circ$)
similarly to the proof of statement $4^\circ$ of Theorem \ref{th13.2}.
\end{proof}

\subsection{Sharpness of the results with respect to time}

In the present subsection, we confirm the sharpness of the results of \S \ref{sec12} with respect to 
dependence on~$\tau$. The following statement demonstrates that Theorems \ref{th12.1} and~\ref{sndw_A(k)_sin_general_thrm} are sharp. It is easily deduced from Theorem 
\ref{th6.5} with the help of  the same arguments as  in the proof of 
 Theorem \ref{th13.2}.

\begin{theorem}
	\label{th13.6}
			Suppose that Condition \emph{\ref{cond13.1}} is satisfied.
	
\noindent $1^\circ$. Let $s\ge 2$.
There does not exist a positive function $\mathcal{C}(\tau)$ such that  $\lim_{\tau \to \infty} {\mathcal C}(\tau)/
|\tau|=0$ and estimate \eqref{13.1} holds for all   $\tau \in \R,$ almost all $\mathbf{k}  \in \widetilde{\Omega},$ 
and sufficiently small $\varepsilon > 0$.

\noindent
$2^\circ$. Let $r \ge 1$. There does not exist a positive function
$\mathcal{C}(\tau)$ such that  $\lim_{\tau \to \infty} {\mathcal C}(\tau)/ |\tau|=0$ and estimate \eqref{13.2} holds for all $\tau \in \R,$ almost all $\mathbf{k}  \in \widetilde{\Omega},$ and sufficiently small $\varepsilon > 0$.

\noindent
$3^\circ$.  Let $r \ge 1$. 
There does not exist a positive function $\mathcal{C}(\tau)$ such that  $\lim_{\tau \to \infty} {\mathcal C}(\tau)/
|\tau|=0$ and estimate \eqref{13.3} holds for all $\tau \in \R,$ almost all  $\mathbf{k}  \in \widetilde{\Omega},$ 
and sufficiently small $\varepsilon > 0$.

\noindent $4^\circ$. Let $s\ge 2$.
There does not exist a positive function $\mathcal{C}(\tau)$ such that  $\lim_{\tau \to \infty} {\mathcal C}(\tau)/
|\tau|=0$ and estimate  \eqref{s<2_sinA(k)_est_1} holds for all  $\tau \in \R,$ almost all $\mathbf{k}  \in \widetilde{\Omega},$ and sufficiently small $\varepsilon > 0$.
\end{theorem}

 Similarly, from Theorem \ref{th6.7} we deduce the following statement which confirms that Theorems 
 \ref{th12.2}, \ref{th12.3},  \ref{sndw_sin_enchanced_thrm_1}, and \ref{sndw_sin_enchanced_thrm_2} are sharp.
 
\begin{theorem}
	\label{th13.7}
		Suppose that Condition \emph{\ref{cond13.2}} is satisfied.

\noindent $1^\circ$. Let $s\ge 3/2$.
There does not exist a positive function $\mathcal{C}(\tau)$ such that  $\lim_{\tau \to \infty} {\mathcal C}(\tau)/ |\tau|^{1/2} =0$  and estimate \eqref{13.1} holds for all $\tau \in \R,$ almost all $\mathbf{k}  \in \widetilde{\Omega},$ and sufficiently small $\varepsilon > 0$.

\noindent
$2^\circ$. Let $r \ge 1/2$.
There does not exist a positive function $\mathcal{C}(\tau)$ such that  $\lim_{\tau \to \infty} {\mathcal C}(\tau)/
|\tau|^{1/2} =0$ and estimate  \eqref{13.3} holds for all $\tau \in \R,$ almost all $\mathbf{k}  \in \widetilde{\Omega},$
and sufficiently small $\varepsilon > 0$.

\noindent $3^\circ$. Let $s\ge 3/2$.
There does not exist a positive function $\mathcal{C}(\tau)$ such that  $\lim_{\tau \to \infty} {\mathcal C}(\tau)/
|\tau|^{1/2} =0$ and estimate  \eqref{s<2_sinA(k)_est_1} holds for all $\tau \in \R,$ almost all $\mathbf{k}  \in \widetilde{\Omega},$ and sufficiently small $\varepsilon > 0$.
 \end{theorem}

\section{Approximation for the operators $\cos(\varepsilon^{-1} \tau \mathcal{A}^{1/2})$ and
$\mathcal{A}^{-1/2} \sin(\varepsilon^{-1} \tau \mathcal{A}^{1/2})$}

\subsection{Approximation for the operators $\cos(\varepsilon^{-1} \tau \widehat{\mathcal{A}}^{1/2})$ and 
$\widehat{\mathcal{A}}^{-1/2} \sin(\varepsilon^{-1} \tau \widehat{\mathcal{A}}^{1/2})$ of the principal order}
 \label{sec14.1}

In $L_2 (\mathbb{R}^d; \mathbb{C}^n)$, consider the operator 
$$
\widehat{\mathcal{A}} = b(\mathbf{D})^* g(\mathbf{x}) b(\mathbf{D})
$$
 (see~(\ref{hatA})). Let $\widehat{\mathcal{A}}^0$~be the effective operator (see~(\ref{hatA0})). Denote 
\begin{align}
\label{14.1}
\widehat{J}_1(\tau) &:=   \cos( \tau \widehat{\mathcal{A}}^{1/2}) -  \cos(\tau (\widehat{\mathcal{A}}^0)^{1/2}),
\\
\label{14.2}
\widehat{J}_2(\tau) &:=  \widehat{\mathcal{A}}^{-1/2} \sin( \tau \widehat{\mathcal{A}}^{1/2}) - 
(\widehat{\mathcal{A}}^0)^{-1/2} \sin(\tau (\widehat{\mathcal{A}}^0)^{1/2}).
\end{align}
Recall the notation $\mathcal{H}_0 = - \Delta$ and put 
\begin{equation}
\label{R(epsilon)}
\mathcal{R} (\varepsilon) :=  \varepsilon^2 (\mathcal{H}_0 + \varepsilon^2 I)^{-1}.
\end{equation}
The operator $\mathcal{R} (\varepsilon)$ expands in the direct integral of the operators~(\ref{R(k, epsilon)}):
\begin{equation}
\label{14.3a}
\mathcal{R} (\varepsilon) = \mathcal{U}^{-1} \bigg( \int\limits_{\widetilde{\Omega}} \oplus  \mathcal{R} (\mathbf{k}, \varepsilon) \, d \mathbf{k}  \bigg) \mathcal{U}.
\end{equation}

Recall the notation \eqref{J1_hat}, \eqref{J2_hat}. From the expansions of the form~(\ref{decompose}) for 
$\widehat{\mathcal{A}}$ and~$\widehat{\mathcal{A}}^0$ and from \eqref{14.3a} it follows that 
\begin{equation}
\label{14.5}
\|  \widehat{J}_l(\varepsilon^{-1} \tau) 
\mathcal{R}(\varepsilon)^{s/2} \|_{L_2(\mathbb{R}^d) \to L_2(\mathbb{R}^d)}
 = \esssup_{\mathbf{k} \in \widetilde{\Omega}} \|  \widehat{J}_l(\mathbf{k}, \varepsilon^{-1} \tau) \mathcal{R}(\mathbf{k}, \varepsilon)^{s/2} \|_{L_2(\Omega) \to L_2(\Omega)},
\quad l=1,2.
\end{equation}
Therefore, Theorems~\ref{th9.1}, \ref{th9.2}, \ref{th9.4} and Propositions \ref{prop9.1a}, \ref{prop9.2a},
\ref{prop9.3a} directly imply the following statements.
Below we  \textit{combine the formulations} (on improvement of the results), so it is convenient 
\textit{to start a new numbering of the constants}.

\begin{theorem}
	\label{th14.1}
Let $\widehat{J}_1(\tau)$ and $\widehat{J}_2(\tau)$ be the operators defined by  \eqref{14.1}\textup, \eqref{14.2}.
  	For $\tau \in \mathbb{R}$ and $\varepsilon > 0$ we have
  	\begin{align}
	\label{14.6}
	&\bigl\|  \widehat{J}_1(\varepsilon^{-1} \tau) \mathcal{R}(\varepsilon) \bigr\|_{L_2(\mathbb{R}^d) \to L_2(\mathbb{R}^d)} \le \widehat{\mathrm{C}}_1(1 + |\tau|) \varepsilon,
	\\
	\label{14.7}
	&\bigl\| \widehat{J}_2(\varepsilon^{-1} \tau) \mathcal{R}(\varepsilon)^{1/2} \bigr\|_{L_2(\mathbb{R}^d) \to L_2(\mathbb{R}^d)} \le \widehat{\mathrm{C}}_2(1 + |\tau|),
	\\
	\label{14.7a}
	&\bigl\| \widehat{J}_2(\varepsilon^{-1} \tau) \bigr\|_{L_2(\mathbb{R}^d) \to L_2(\mathbb{R}^d)} \le \widehat{\mathrm{C}}'_2(1 + \eps^{-1/2}|\tau|^{1/2}).
	\end{align}
The constants   $\widehat{\mathrm{C}}_1,$ $\widehat{\mathrm{C}}_2,$ and $\widehat{\mathrm{C}}'_2$ depend only on $\alpha_0,$ $\alpha_1,$ $\|g\|_{L_\infty},$ $\|g^{-1}\|_{L_\infty},$ and $r_0$.
\end{theorem}

Earlier, estimate \eqref{14.6} was obtained in  \cite[Theorem 9.2]{BSu5} and inequality \eqref{14.7} was proved in  \cite[Theorem 8.1]{M}.

\begin{theorem}
	\label{th14.2}
  Suppose that Condition \emph{\ref{cond_B}} or Condition~\emph{\ref{cond1}} {\rm (}or more restrictive Condition~\emph{\ref{cond2}}{\rm )} is satisfied.
 Then for  $\tau \in \mathbb{R}$ and $\varepsilon > 0$ we have 
	\begin{align}
	\nonumber
	&\bigl\|  \widehat{J}_1(\varepsilon^{-1} \tau) \mathcal{R}(\varepsilon)^{3/4} \bigr\|_{L_2(\mathbb{R}^d) \to L_2(\mathbb{R}^d)} \le \widehat{\mathrm{C}}_3(1 + |\tau|)^{1/2} \varepsilon,
	\\
	\nonumber
	&\bigl\| \widehat{J}_2(\varepsilon^{-1} \tau) \mathcal{R}(\varepsilon)^{1/4} \bigr\|_{L_2(\mathbb{R}^d) \to L_2(\mathbb{R}^d)} \le \widehat{\mathrm{C}}_4(1 + |\tau|)^{1/2},
	\\
	\label{14.9a}
	&\bigl\| \widehat{J}_2(\varepsilon^{-1} \tau) \bigr\|_{L_2(\mathbb{R}^d) \to L_2(\mathbb{R}^d)} \le \widehat{\mathrm{C}}'_4(1 + \eps^{-1/3}|\tau|^{1/3}).
	\end{align}
	Under Condition \emph{\ref{cond_B}}, the constants  $\widehat{\mathrm{C}}_3,$ $\widehat{\mathrm{C}}_4,$
	and $\widehat{\mathrm{C}}'_4$ depend only on $\alpha_0,$ $\alpha_1,$ $\|g\|_{L_\infty},$ 
	$\|g^{-1}\|_{L_\infty},$ and $r_0$. Under Condition~\emph{\ref{cond1}}, these constants depend on the same parameters and on $n,$ $\widehat{c}^{\circ}$.
\end{theorem}

\subsection{Approximation of the operator  
$\widehat{\mathcal{A}}^{-1/2} \sin(\varepsilon^{-1} \tau \widehat{\mathcal{A}}^{1/2})$ in the energy norm}
\label{sec14.2}

We need the operator $\Pi = \mathcal{U}^{-1} [\widehat{P}] \mathcal{U}$ acting in 
$L_2(\mathbb{R}^d; \mathbb{C}^n)$. Here $[\widehat{P}]$~is the orthogonal projection in $\mathcal{H}= \int_{\widetilde{\Omega}} \oplus L_2 (\Omega; \mathbb{C}^n) \, d \mathbf{k}$, acting on the fibers of the direct integral as the operator $\widehat{P}$ of averaging over the cell. In~\cite[(6.8)]{BSu3}, it was shown that $\Pi$
is given by 
\begin{equation*}
(\Pi \mathbf{u})(\mathbf{x}) = (2 \pi)^{-d/2} \int\limits_{\widetilde{\Omega}} e^{i \left\langle \mathbf{x}, \boldsymbol{\xi} \right\rangle} \widehat{\mathbf{u}} (\boldsymbol{\xi}) \, d \boldsymbol{\xi},
\end{equation*}
 where $\widehat{\mathbf{u}} (\boldsymbol{\xi})$~is the Fourier-image of a function $\mathbf{u} (\mathbf{x})$. 
 I.~e.,   $\Pi$ is the pseudodifferential operator in  $L_2(\mathbb{R}^d; \mathbb{C}^n)$, whose symbol is the characterictic function $\chi_{\widetilde{\Omega}} (\boldsymbol{\xi})$ of the set $\wt{\Omega}$.
Denote 
\begin{equation}\label{10.0}
\widehat{J}(\tau) :=  \widehat{\mathcal{A}}^{-1/2} \sin( \tau \widehat{\mathcal{A}}^{1/2}) - (I + \Lambda b(\mathbf{D}) \Pi) (\widehat{\mathcal{A}}^0)^{-1/2} \sin(\tau (\widehat{\mathcal{A}}^0)^{1/2}).
\end{equation}
Recall  notation (\ref{hatJ}). From the expansions of the form~(\ref{decompose}) for $\widehat{\mathcal{A}}$ and  $\widehat{\mathcal{A}}^0$ and from \eqref{14.3a} it follows that 
\begin{equation}
\label{hat_norms_and_Gelfand_transf}
\| \widehat{\mathcal{A}}^{1/2} \widehat{J}(\varepsilon^{-1} \tau) \mathcal{R}(\varepsilon)^{s/2} \|_{L_2(\mathbb{R}^d) \to L_2(\mathbb{R}^d)} 
= \esssup_{\mathbf{k} \in \widetilde{\Omega}} \| \widehat{\mathcal{A}}(\mathbf{k})^{1/2} \widehat{J}(\mathbf{k}, \varepsilon^{-1} \tau) \mathcal{R}(\mathbf{k}, \varepsilon)^{s/2} \|_{L_2(\Omega) \to L_2(\Omega)}.
\end{equation}
Therefore, Theorems~\ref{A(k)_sin_general_thrm}, \ref{A(k)_sin_enchanced_thrm_1}, and \ref{A(k)_sin_enchanced_thrm_2} directly imply the following statements.

\begin{theorem}[see~\cite{M}]
	\label{A_sin_general_thrm}
	Suppose that $\widehat{J}(\tau)$ is the operator defined by \eqref{10.0}. 
	For $\tau \in \mathbb{R}$ and $\varepsilon > 0$ we have 
	\begin{equation}
	\label{A_sin_general_est}
	\bigl\| \widehat{\mathcal{A}}^{1/2} \widehat{J}(\varepsilon^{-1} \tau) \mathcal{R}(\varepsilon) \bigr\|_{L_2(\mathbb{R}^d) \to L_2(\mathbb{R}^d)} \le \widehat{\mathrm{C}}_5 (1 + |\tau|) \varepsilon. 
	\end{equation}
	The constant $\widehat{\mathrm{C}}_5$ depends only on $\alpha_0,$ $\alpha_1,$ $\|g\|_{L_\infty},$ $\|g^{-1}\|_{L_\infty},$ $r_0,$ and $r_1$.
\end{theorem}

\begin{theorem}
\label{th14.5}
	Suppose that Condition \emph{\ref{cond_B}} or Condition~\emph{\ref{cond1}} {\rm (}or more restrictive Condition~\emph{\ref{cond2}}{\rm )} is satisfied. Then for  $\tau \in \mathbb{R}$ and $\varepsilon > 0$ we have 	\begin{equation*}
	\bigl\| \widehat{\mathcal{A}}^{1/2} \widehat{J}(\varepsilon^{-1} \tau) \mathcal{R}(\varepsilon)^{3/4} \bigr\|_{L_2(\mathbb{R}^d) \to L_2(\mathbb{R}^d)} \le \widehat{\mathrm{C}}_6(1 + |\tau|)^{1/2} \varepsilon. 
	\end{equation*}
Under Condtion \emph{\ref{cond_B}}, the constant   $\widehat{\mathrm{C}}_6$ depends only on  
 $\alpha_0,$ $\alpha_1,$ $\|g\|_{L_\infty},$ $\|g^{-1}\|_{L_\infty},$ $r_0,$ and $r_1$. Under 
 Condition~\emph{\ref{cond1}}, this constant depends on the same parameters and on  $n,$ $\widehat{c}^{\circ}$.
\end{theorem}

Theorem~\ref{A_sin_general_thrm} was known earlier (see~\cite[Theorem 8.1]{M}).

\subsection{Sharpness of the results of Subsections \ref{sec14.1}, \ref{14.2}\label{sec14.3}}

Applying theorems from \S\ref{sec10}, we confirm that the results of Subsections
\ref{sec14.1}, \ref{sec14.2} are sharp. We start with the sharpness regarding the smoothing factor. 
Let us show that Theorems \ref{th14.1} and \ref{A_sin_general_thrm} are sharp.

\begin{theorem}
	\label{th14.7}
	Suppose that Condition \emph{\ref{cond10.1}} is satisfied.
	
\noindent 
$1^\circ$. Let $0 \ne \tau \in \mathbb{R}$ and $0 \le s < 2$.  
Then there does not exist a constant $\mathcal{C}(\tau) > 0$ such that the estimate
	\begin{equation}
	\label{14.17}
	\bigl\|  \widehat{J}_1(\varepsilon^{-1} \tau) \mathcal{R}(\varepsilon)^{s/2} \bigr\|_{L_2(\mathbb{R}^d) \to L_2(\mathbb{R}^d)} \le \mathcal{C}(\tau) \varepsilon
	\end{equation}
	holds for all sufficiently small $\varepsilon > 0$.
	
\noindent 
$2^\circ$. Let $0 \ne \tau \in \mathbb{R}$ and $0 \le r < 1$.  
Then there does not exist a constant $\mathcal{C}(\tau) > 0$ such that the estimate
	\begin{equation}
	\label{14.18}
	\bigl\|  \widehat{J}_2(\varepsilon^{-1} \tau) \mathcal{R}(\varepsilon)^{r/2} \bigr\|_{L_2(\mathbb{R}^d) \to L_2(\mathbb{R}^d)} \le \mathcal{C}(\tau)
	\end{equation}
	holds for all sufficiently small $\varepsilon > 0$.

\noindent 
$3^\circ$. 
Let $0 \ne \tau \in \mathbb{R}$ and $0 \le s < 2$. Then there does not exist a constant $\mathcal{C}(\tau) > 0$ 
such that the estimate 
	\begin{equation}
	\label{hat_s<2_sinA_est}
	\bigl\| \widehat{\mathcal{A}}^{1/2} \widehat{J}(\varepsilon^{-1} \tau) \mathcal{R}(\varepsilon)^{s/2} \bigr\|_{L_2(\mathbb{R}^d) \to L_2(\mathbb{R}^d)} \le \mathcal{C}(\tau) \varepsilon
	\end{equation}
		holds for all sufficiently small $\varepsilon > 0$.
\end{theorem}

\begin{proof}
For instance, let us prove statement $1^\circ$. We prove by contradiction. Suppose that for some $\tau \ne 0$ and  \hbox{$0 \le s < 2$} there exists a constant  $\mathcal{C}(\tau) > 0$ such that~\eqref{14.17} holds for all sufficiently small $\varepsilon > 0$. By~\eqref{14.5}, this means that  estimate~\eqref{10.1a} is valid for almost all $\mathbf{k} \in \widetilde{\Omega}$ and sufficiently small $\varepsilon$. But this contradicts statement $1^\circ$ of Theorem~\ref{th10.1}.

Similarly, statement $2^\circ$ follows from statement $2^\circ$ of Theorem~\ref{th10.1}, and statement $3^\circ$ follows from Theorem \ref{hat_s<2_sinA(k)_thrm}.
\end{proof}

Similarly, applying Theorems \ref{th10.2a} and \ref{hat_s<3/2_sinA(k)_thrm}, we arrive at the following result showing that Theorems \ref{th14.2} and \ref{th14.5} are sharp.

\begin{theorem}
	\label{th14.8}
	Suppose that Condition \emph{\ref{cond10.2}} is satisfied.
	
\noindent $1^\circ.$
 Let $0 \ne \tau \in \mathbb{R}$ and $0 \le s < 3/2$. 
 Then there does not exist a constant $\mathcal{C} (\tau) > 0$ such that  estimate  \eqref{14.17}
holds for all sufficiently small $\varepsilon > 0$.

\noindent $2^\circ.$
 Let $0 \ne \tau \in \mathbb{R}$ and $0 \le r < 1/2$. 
 Then there does not exist a constant $\mathcal{C} (\tau) > 0$ such that  estimate \eqref{14.18} 
 holds for all sufficiently small $\varepsilon > 0$.

\noindent $3^\circ.$ Let $0 \ne \tau \in \mathbb{R}$ and $0 \le s < 3/2$. 
Then there does not exist a constant $\mathcal{C} (\tau) > 0$ such that  estimate  \eqref{hat_s<2_sinA_est} 
holds for all sufficiently small $\varepsilon > 0$.
\end{theorem}

We proceed to the sharpness of the results regarding the dependence of estimates on the parameter $\tau$.
 Theorems \ref{th10.6} and \ref{th10.8} imply the following statement confirming that 
Theorems \ref{th14.1} and \ref{A_sin_general_thrm} are sharp.

\begin{theorem}
	\label{th14.11}
	Suppose that Condition \emph{\ref{cond10.1}} is satisfied.
		
\noindent 
$1^\circ$. Let $s \ge 2$. There does not exist a positive function $\mathcal{C}(\tau)$ such that 
$\lim_{\tau \to \infty} \mathcal{C}(\tau) /|\tau| =0$ and \eqref{14.17} holds for  $\tau \in \R$ and sufficiently small $\varepsilon > 0$.

\noindent 
$2^\circ$. Let $r \ge 1$. There does not exist a positive function $\mathcal{C}(\tau)$ such that 
$\lim_{\tau \to \infty} \mathcal{C}(\tau) /|\tau| =0$ and  \eqref{14.18} holds for  $\tau \in \R$ 
and sufficiently small $\varepsilon > 0$.

\noindent 
$3^\circ$. Let $s \ge 2$. There does not exist a positive function
$\mathcal{C}(\tau)$ such that 
$\lim_{\tau \to \infty} \mathcal{C}(\tau) /|\tau| =0$ and \eqref{hat_s<2_sinA_est} holds for $\tau \in \R$ 
and sufficiently small $\varepsilon > 0$.
\end{theorem}

 Theorems \ref{th10.7} and \ref{th10.9} lead to the following statement confirming that Theorems \ref{th14.2} and \ref{th14.5} are sharp.

\begin{theorem}
	\label{th14.12}
	Suppose that Condition \emph{\ref{cond10.2}} is satisfied.

\noindent 
$1^\circ$. Let $s \ge 3/2$. 
There does not exist a positive function $\mathcal{C}(\tau)$ such that 
$\lim_{\tau \to \infty} \mathcal{C}(\tau) /|\tau|^{1/2} =0$ and  \eqref{14.17} holds
for $\tau \in \R$ and sufficiently small $\varepsilon$.

\noindent 
$2^\circ$. Let $r \ge 1/2$.
There does not exist a positive function $\mathcal{C}(\tau)$ such that 
$\lim_{\tau \to \infty} \mathcal{C}(\tau) /|\tau|^{1/2} =0$ and  \eqref{14.18} holds
for $\tau \in \R$ and sufficiently small $\varepsilon$.

\noindent 
$3^\circ$. Let $s \ge 3/2$.
There does not exist a positive function $\mathcal{C}(\tau)$ such that 
$\lim_{\tau \to \infty} \mathcal{C}(\tau) /|\tau|^{1/2} =0$ and  \eqref{hat_s<2_sinA_est} holds
for $\tau \in \R$ and sufficiently small $\varepsilon$.
\end{theorem}

\subsection{Approximation for the sandwiched operators $\cos(\varepsilon^{-1} \tau {\mathcal{A}}^{1/2})$ and 
${\mathcal{A}}^{-1/2} \sin(\varepsilon^{-1} \tau {\mathcal{A}}^{1/2})$ in the principal order}
 \label{sec14.4}

In $L_2 (\mathbb{R}^d; \mathbb{C}^n)$, we consider the operator~(\ref{A}). Let $f_0$~me the matrix~(\ref{f_0}) and let   $\mathcal{A}^0$~be the operator~(\ref{A0}). Denote 
{\allowdisplaybreaks
\begin{align}
\label{14.26}
{J}_1(\tau) &:= f  \cos( \tau {\mathcal{A}}^{1/2}) f^{-1} - f_0  \cos(\tau ({\mathcal{A}}^0)^{1/2}) f_0^{-1},
\\
\label{14.27}
{J}_2(\tau) &:=  f {\mathcal{A}}^{-1/2} \sin( \tau {\mathcal{A}}^{1/2}) f^{-1} - 
f_0 ({\mathcal{A}}^0)^{-1/2} \sin(\tau ({\mathcal{A}}^0)^{1/2}) f_0^{-1},
\\
\label{14.28}
{J}_3(\tau) &:=  f {\mathcal{A}}^{-1/2} \sin( \tau {\mathcal{A}}^{1/2}) f^{*} - 
f_0 ({\mathcal{A}}^0)^{-1/2} \sin(\tau ({\mathcal{A}}^0)^{1/2}) f_0.
\end{align}
}

We recall  notation \eqref{J1}--\eqref{J2tilde}. From the expansions of the form~(\ref{decompose}) for 
${\mathcal{A}}$ and ${\mathcal{A}}^0$ and from  \eqref{14.3a} it follows that 
\begin{equation*}
\|  {J}_l(\varepsilon^{-1} \tau) \mathcal{R}(\varepsilon)^{s/2} \|_{L_2(\mathbb{R}^d) \to L_2(\mathbb{R}^d)} = \esssup_{\mathbf{k} \in \widetilde{\Omega}} \| {J}_l(\mathbf{k}, \varepsilon^{-1} \tau) \mathcal{R}(\mathbf{k}, \varepsilon)^{s/2} \|_{L_2(\Omega) \to L_2(\Omega)}
\end{equation*}
for $l=1,2,3$. Therefore, Theorems~\ref{th12.1}, \ref{th12.2}, \ref{th12.3} and Propositions  \ref{prop12.1a}, 
\ref{prop12.2a}, \ref{prop12.3a} directly imply the following statements.

\begin{theorem}
	\label{th14.15}
	Let ${J}_1(\tau),$ ${J}_2(\tau),$ and $J_3(\tau)$ be the operators defined by 
	 \eqref{14.26}--\eqref{14.28}.
  	Then for $\tau \in \mathbb{R}$ and $\varepsilon > 0$ we have
  	\begin{align}
	\label{14.30}
	&\bigl\|  {J}_1(\varepsilon^{-1} \tau) \mathcal{R}(\varepsilon) \bigr\|_{L_2(\mathbb{R}^d) \to L_2(\mathbb{R}^d)} \le {\mathrm{C}}_1(1 + |\tau|) \varepsilon,
	\\
	\label{14.31}
	&\bigl\| {J}_2(\varepsilon^{-1} \tau) \mathcal{R}(\varepsilon)^{1/2} \bigr\|_{L_2(\mathbb{R}^d) \to L_2(\mathbb{R}^d)} \le {\mathrm{C}}_2(1 + |\tau|),
	\\
	\label{14.32}
	&\bigl\| {J}_3(\varepsilon^{-1} \tau) \mathcal{R}(\varepsilon)^{1/2} \bigr\|_{L_2(\mathbb{R}^d) \to L_2(\mathbb{R}^d)} \le \wt{\mathrm{C}}_2(1 + |\tau|),
	\\
\label{14.32a}
	&\bigl\| {J}_3(\varepsilon^{-1} \tau) \bigr\|_{L_2(\mathbb{R}^d) \to L_2(\mathbb{R}^d)} 
	\le {\mathrm{C}}_2'(1 + \eps^{-1/2} |\tau|^{1/2}). 	
	\end{align}
	The constants ${\mathrm{C}}_1,$ ${\mathrm{C}}_2,$ $\wt{\mathrm{C}}_2,$  ${\mathrm{C}}_2'$ 
	depend only on $\alpha_0,$ $\alpha_1,$ $\|g\|_{L_\infty},$ $\|g^{-1}\|_{L_\infty},$ $\|f\|_{L_\infty},$ $\|f^{-1}\|_{L_\infty},$ and $r_0$.
\end{theorem}

Earlier, estimate \eqref{14.30} was obtained in  \cite[Theorem 10.2]{BSu5}, inequality \eqref{14.31} was proved in  \cite[Theorem 8.1]{M}, and \eqref{14.32} was proved in  \cite[Theorem 10.5]{DSu}.

\begin{theorem}
	\label{th14.16}
 Suppose that Condition \emph{\ref{cond_BB}} or Condition~\emph{\ref{sndw_cond1}} {\rm (}or more restrictive Condition~\emph{\ref{sndw_cond2}}{\rm )} is satisfied. 
 Then for $\tau \in \mathbb{R}$ and $\varepsilon > 0$ we have
 	\begin{align}
	\nonumber
	&\bigl\|  {J}_1(\varepsilon^{-1} \tau) \mathcal{R}(\varepsilon)^{3/4} \bigr\|_{L_2(\mathbb{R}^d) \to L_2(\mathbb{R}^d)} \le {\mathrm{C}}_3(1 + |\tau|)^{1/2} \varepsilon,
	\\
	\nonumber
	&\bigl\| {J}_3(\varepsilon^{-1} \tau) \mathcal{R}(\varepsilon)^{1/4} \bigr\|_{L_2(\mathbb{R}^d) \to L_2(\mathbb{R}^d)} \le {\mathrm{C}}_4(1 + |\tau|)^{1/2},
	\\
	\label{14.34a}
 &\bigl\| {J}_3(\varepsilon^{-1} \tau) \bigr\|_{L_2(\mathbb{R}^d) \to L_2(\mathbb{R}^d)} 
	\le {\mathrm{C}}_4'(1 + \eps^{-1/3} |\tau|^{1/3}). 	
	\end{align}
	Under Condition \emph{\ref{cond_BB}}, the constants  ${\mathrm{C}}_3,$ ${\mathrm{C}}_4,$ and ${\mathrm{C}}_4'$ depend only on $\alpha_0,$ $\alpha_1,$ 
	$\|g\|_{L_\infty},$ $\|g^{-1}\|_{L_\infty},$ $\|f\|_{L_\infty},$ $\|f^{-1}\|_{L_\infty},$ and $r_0$.
	Under Condition \emph{\ref{sndw_cond1}}, these constants depend on the same parameters and on  $n,$
	${c}^{\circ}$.
\end{theorem}

\subsection{ Approximation for the sandwiched operator $\mathcal{A}^{-1/2} \sin(\varepsilon^{-1} \tau \mathcal{A}^{1/2})$ in the energy norm}\label{sec14.5}
  
Denote
\begin{equation}\label{10.7a}
J(\tau) :=  f \mathcal{A}^{-1/2} \sin( \tau \mathcal{A}^{1/2})f^{-1}
- (I + \Lambda b(\mathbf{D}) \Pi) 
f_0(\mathcal{A}^0)^{-1/2} \sin( \tau (\mathcal{A}^0)^{1/2}) f_0^{-1}.  
\end{equation} 
Similarly to~(\ref{hat_norms_and_Gelfand_transf}), from the direct integral expansion it follows that
\begin{equation*}
\| \widehat{\mathcal{A}}^{1/2} J(\varepsilon^{-1} \tau) \mathcal{R}(\varepsilon)^{s/2} \|_{L_2(\mathbb{R}^d) \to L_2(\mathbb{R}^d)} 
= \esssup_{\mathbf{k} \in \widetilde{\Omega}} \| \widehat{\mathcal{A}}(\mathbf{k})^{1/2} J(\mathbf{k}, \varepsilon^{-1} \tau) \mathcal{R}(\mathbf{k}, \varepsilon)^{s/2} \|_{L_2(\Omega) \to L_2(\Omega)}.
\end{equation*}
Therefore, Theorems~\ref{sndw_A(k)_sin_general_thrm}, \ref{sndw_sin_enchanced_thrm_1}, and \ref{sndw_sin_enchanced_thrm_2} directly imply the following statements.

\begin{theorem}[see~\cite{M}]
	\label{sndw_A_sin_general_thrm}
	 Let $J(\tau)$ be the operator defined by \eqref{10.7a}. For $\tau \in \mathbb{R}$ and $\varepsilon > 0$
	 we have 
	\begin{align*}
	\| \widehat{\mathcal{A}}^{1/2} J(\varepsilon^{-1} \tau) \mathcal{R}(\varepsilon) \|_{L_2(\mathbb{R}^d) \to L_2(\mathbb{R}^d)} \le \mathrm{C}_5(1+ |\tau|) \varepsilon,
	\end{align*}
	where $\mathrm{C}_5$ depends only on $\alpha_0,\alpha_1,\|g\|_{L_\infty},\|g^{-1}\|_{L_\infty},\|f\|_{L_\infty},\|f^{-1}\|_{L_\infty},r_0,$ and~$r_1$.
\end{theorem}

\begin{theorem}
	\label{sndw_A_sin_enchncd_thrm_1}
	 Suppose that Condition \emph{\ref{cond_BB}} or Condition~\emph{\ref{sndw_cond1}} {\rm (}or more restrictive Condition~\emph{\ref{sndw_cond2}}{\rm )} is satisfied. 
Then for $\tau \in \mathbb{R}$ and $\varepsilon > 0$ we have
	\begin{align*}
	\| \widehat{\mathcal{A}}^{1/2} J(\varepsilon^{-1} \tau) \mathcal{R}(\varepsilon)^{3/4} \|_{L_2(\mathbb{R}^d) \to L_2(\mathbb{R}^d)} \le \mathrm{C}_6(1+ |\tau|)^{1/2} \varepsilon.
		\end{align*}
	Under Condition \emph{\ref{cond_BB}}, the constant  ${\mathrm{C}}_6$ depends only on $\alpha_0,$ $\alpha_1,$ $\|g\|_{L_\infty},$ $\|g^{-1}\|_{L_\infty},$ $\|f\|_{L_\infty},$ $\|f^{-1}\|_{L_\infty},$ $r_0,$ and $r_1$.
	Under Condition \emph{\ref{sndw_cond1}}, this constant depends on the same parameters and on $n,$
	${c}^{\circ}$.
\end{theorem}

Theorem~\ref{sndw_A_sin_general_thrm} was known earlier  (see~\cite[Theorem 8.1]{M}).

\subsection{Sharpness of the results of Subsections \ref{sec14.4} and \ref{sec14.5}\label{sec14.6}} 

Theorems of  \S \ref{sec13} imply that the results of Subsections \ref{sec14.4} and \ref{sec14.5} are sharp. 
We start with the sharpness  regarding  the smoothing factor. Applying Theorem~\ref{th13.2}, we confirm that Theorems~\ref{th14.15} and~\ref{sndw_A_sin_general_thrm} are sharp.

\begin{theorem}
	\label{th14.21}
	Suppose that Condition \emph{\ref{cond13.1}} is satisfied.

\noindent $1^\circ$. Let $0 \ne \tau \in \mathbb{R}$ and $0 \le s < 2$. 
Then there does not exist a constant $\mathcal{C}(\tau) > 0$ such that the estimate
	\begin{equation}
	\label{14.39}
	\bigl\|  {J}_1(\varepsilon^{-1} \tau) \mathcal{R}(\varepsilon)^{s/2} \bigr\|_{L_2(\mathbb{R}^d) \to L_2(\mathbb{R}^d)} \le \mathcal{C}(\tau) \varepsilon
	\end{equation}
holds for all sufficiently small $\varepsilon > 0$.
	
\noindent 
$2^\circ$. Let $0 \ne \tau \in \mathbb{R}$ and $0 \le r < 1$. 
Then there does not exist a constant $\mathcal{C}(\tau) > 0$ such that the estimate 
	\begin{equation}
	\label{14.40}
	\bigl\|  {J}_2(\varepsilon^{-1} \tau) \mathcal{R}(\varepsilon)^{r/2} \bigr\|_{L_2(\mathbb{R}^d) \to L_2(\mathbb{R}^d)} \le \mathcal{C}(\tau) 
	\end{equation}
holds for all sufficiently small $\varepsilon > 0$.

\noindent 
$3^\circ$. Let  $0 \ne \tau \in \mathbb{R}$ and $0 \le r < 1$. 
Then there does not exist a constant $\mathcal{C}(\tau) > 0$ such that the estimate 
	\begin{equation}
	\label{14.41}
	\bigl\|  {J}_3(\varepsilon^{-1} \tau) \mathcal{R}(\varepsilon)^{r/2} \bigr\|_{L_2(\mathbb{R}^d) \to L_2(\mathbb{R}^d)} \le \mathcal{C}(\tau)
	\end{equation}
holds for all sufficiently small $\varepsilon > 0$.

\noindent $4^\circ$. Let $0 \ne \tau \in \mathbb{R}$ and $0 \le s < 2$. Then there does not exist a constant $\mathcal{C}(\tau) > 0$ such that the estimate
	\begin{equation}
		\label{s<2_sinA_est_1}
		\bigl\| \widehat{\mathcal{A}}^{1/2} J( \varepsilon^{-1} \tau)  \mathcal{R}( \varepsilon)^{s/2} \bigr\|_{L_2(\Omega) \to L_2 (\Omega) }  \le \mathcal{C} (\tau) \varepsilon
	\end{equation}
	holds for all sufficiently small $\varepsilon > 0$.
	\end{theorem}

Theorem~\ref{th13.3} implies the following statement demonstrating that Theorems~\ref{th14.16} and \ref{sndw_A_sin_enchncd_thrm_1} are sharp.

\begin{theorem}
	\label{th14.22}
	Suppose that Condition \emph{\ref{cond13.2}} is satisfied.

\noindent $1^\circ$. Let $0 \ne \tau \in \mathbb{R}$ and  $0 \le s < 3/2$. Then there does not exist a constant $\mathcal{C}(\tau) > 0$ such that 
 \eqref{14.39} holds for all sufficiently small $\varepsilon > 0$.

\noindent 
$2^\circ$.  Let $0 \ne \tau \in \mathbb{R}$ and $0 \le r < 1/2$. 
Then there does not exist a constant $\mathcal{C}(\tau) > 0$ such that 
 \eqref{14.41} holds for all sufficiently small $\varepsilon > 0$.

\noindent 
$3^\circ$.  Let $0 \ne \tau \in \mathbb{R}$ and $0 \le s < 3/2$. 
Then there does not exist a constant $\mathcal{C}(\tau) > 0$ such that  \eqref{s<2_sinA_est_1} holds for all sufficiently small $\varepsilon > 0$.	
	\end{theorem}

We proceed to the sharpness of the results regarding the dependence of estimates on the parameter $\tau$. 
Applying Theorem \ref{th13.6}, we arrive at the following statement confirming that Theorems~\ref{th14.15} and~\ref{sndw_A_sin_general_thrm} are sharp.

\begin{theorem}
	\label{th14.21a}
	Suppose that Condition \emph{\ref{cond13.1}} is satisfied.
	 
\noindent $1^\circ$.  Let $s \ge 2$. 
There does not exist a positive function
$\mathcal{C}(\tau)$ such that  $\lim_{\tau \to \infty} \mathcal{C}(\tau)/|\tau| =0$ 
	and  \eqref{14.39} holds for $\tau \in \R$ and sufficiently small $\varepsilon > 0$.

\noindent $2^\circ$. Let $r \ge 1$. 
There does not exist a positive function
$\mathcal{C}(\tau)$ such that  $\lim_{\tau \to \infty} \mathcal{C}(\tau)/|\tau| =0$ 
	and \eqref{14.40} holds for $\tau \in \R$ and sufficiently small $\varepsilon > 0$.

\noindent $3^\circ$.  Let $r \ge 1$. 
There does not exist a positive function
$\mathcal{C}(\tau)$ such that  $\lim_{\tau \to \infty} \mathcal{C}(\tau)/|\tau| =0$ 
and \eqref{14.41} holds for $\tau \in \R$ and sufficiently small $\varepsilon > 0$.

\noindent $4^\circ$. Let  $s \ge 2$. There does not exist a positive function
$\mathcal{C}(\tau)$ such that  $\lim_{\tau \to \infty} \mathcal{C}(\tau)/|\tau| =0$ 
 and \eqref{s<2_sinA_est_1} holds for $\tau \in \R$ and sufficiently small $\varepsilon > 0$.
	\end{theorem}

From Theorem \ref{th13.7} we deduce  the following result demonstrating that Theorems \ref{th14.16} and \ref{sndw_A_sin_enchncd_thrm_1} are sharp.

\begin{theorem}
	\label{th14.22a}
	Suppose that Condition \emph{\ref{cond13.2}} is satisfied.

\noindent $1^\circ$.  Let $s \ge 3/2$. 
There does not exist a positive function
$\mathcal{C}(\tau)$ such that  $\lim_{\tau \to \infty} \mathcal{C}(\tau)/|\tau|^{1/2} =0$  
 and \eqref{14.39} holds for $\tau \in \R$ and sufficiently small $\varepsilon > 0$.
 
\noindent $2^\circ$.  Let $r \ge 1/2$. 
There does not exist a positive function
$\mathcal{C}(\tau)$ such that  $\lim_{\tau \to \infty} \mathcal{C}(\tau)/|\tau|^{1/2} =0$ 
and \eqref{14.41} holds for $\tau \in \R$ and sufficiently small $\varepsilon > 0$.

\noindent $3^\circ$. Let $s \ge 3/2$. 
There does not exist a positive function
$\mathcal{C}(\tau)$ such that  $\lim_{\tau \to \infty} \mathcal{C}(\tau)/|\tau|^{1/2} =0$ and \eqref{s<2_sinA_est_1} holds for $\tau \in \R$ and sufficiently small $\varepsilon > 0$.
	\end{theorem}

\subsection{On the possibility to remove the smoothing operator  $\Pi$ in the corrector}
Now, we consider the question about the possibility to remove the operator $\Pi$ in the corrector
(i.~e., to replace $\Pi$ by the identity operator keeping the same order of  errors)
in Theorems  \ref{A_sin_general_thrm}, \ref{th14.5},  \ref{sndw_A_sin_general_thrm}, and \ref{sndw_A_sin_enchncd_thrm_1}. We consider the  more general case of the operator $\A$ (then the results for  $\wh{\A}$ will follow in the case $f = \1$).

\begin{lemma}
\label{lem10.9}
For  $\tau \in \R$ and $\eps>0$ we have  
\begin{align}
\label{*.1}
 \| b(\D) (I\! - \!\Pi)  f_0(\mathcal{A}^0)^{-1/2} \sin( \eps^{-1} \tau (\mathcal{A}^0)^{1/2}) f_0^{-1}  {\mathcal R}(\eps)\|_{L_2(\R^d) \to H^2(\R^d)} &\le {\mathrm C}^{(1)} \eps^2,
\\
\label{*.2}
 \| b(\D) (I\! -\! \Pi)  f_0(\mathcal{A}^0)^{-1/2} \sin( \eps^{-1} \tau (\mathcal{A}^0)^{1/2}) f_0^{-1}  {\mathcal R}(\eps)^{3/4}\|_{L_2(\R^d) \to H^{3/2}(\R^d)} 
 &\le {\mathrm C}^{(2)} \eps^{3/2}.
\end{align}
The constants ${\mathrm C}^{(1)}$ and ${\mathrm C}^{(2)}$ depend on  $\|g^{-1}\|_{L_\infty},$  $\|f^{-1}\|_{L_\infty},$ and  $r_0$.
\end{lemma}

\begin{proof}
Writing the norm in the left-hand side of  \eqref{*.1} in the Fourier-representation and recalling that the symbol of the operator  $\Pi$ is $\chi_{\wt{\Omega}}(\bxi)$ and the symbol of $\A^0$ is  
$f_0 b(\bxi)^* g^0 b(\bxi) f_0$, we obtain:
$$
\begin{aligned}
\|& b(\D) (I \!-\! \Pi)  f_0(\mathcal{A}^0)^{-1/2} \sin( \eps^{-1} \tau (\mathcal{A}^0)^{-1/2}) f_0^{-1}  {\mathcal R}(\eps)\|_{L_2(\R^d) \to H^2(\R^d)}
\\
&\!\!\le\! \sup_{\bxi \in \R^d} (1\!+ \!|\bxi|^2) (1\!-\! \chi_{\wt{\Omega}}(\bxi)) \bigl|b(\bxi)  f_0 
(f_0 b(\bxi)^* g^0 b(\bxi) f_0)^{-1/2}\bigr|
|f_0^{-1}| \eps^2 ( |\bxi|^2\! +\!\eps^2)^{-1}
\\
&\!\!\le \! \| g^{-1}\|^{1/2}_{L_\infty} \| f^{-1}\|_{L_\infty} \eps^2 
\sup_{ |\bxi| \ge r_0} (1\!+ \!|\bxi|^2) ( |\bxi|^2\! +\!\eps^2)^{-1} \!\le\! {\mathrm C}^{(1)} \eps^2,
\end{aligned}
$$
where ${\mathrm C}^{(1)} =  \| g^{-1}\|^{1/2}_{L_\infty} \| f^{-1}\|_{L_\infty} (1+ r_0^{-2})$.
We have used \eqref{g^0_est} and \eqref{f_0_estimates}.
 
 Similarly, one can check estimate  \eqref{*.2} with the constant 
 \begin{equation*}
{\mathrm C}^{(2)} =  \| g^{-1}\|^{1/2}_{L_\infty} \| f^{-1}\|_{L_\infty} (1+ r_0^{-2})^{3/4}.\qedhere
\end{equation*}
\end{proof}

 Let $[\Lambda]$ be the operator of multiplication by the $\Gamma$-periodic solution of  problem~\eqref{equation_for_Lambda}. We formulate the following additional conditions.
 
  \begin{condition}
  \label{cond_Lambda_1}
   The operator $ [\Lambda]$ is continuous from $H^2(\R^d)$ to $H^1(\R^d)$.
  \end{condition}

  \begin{condition}
  \label{cond_Lambda_2}
   The operator $[\Lambda]$ is continuous from $H^{3/2}(\R^d)$ to $H^1(\R^d)$.
  \end{condition}

Denote
\begin{align}
\label{10.17}
&\wh{J}^\circ\!(\tau) :=  \wh{\mathcal{A}}^{-\!1/2} \!\sin( \tau \wh{\mathcal{A}}^{1/2}) - (I + \Lambda b(\mathbf{D}) ) 
(\wh{\mathcal{A}}^0)^{-1/2} \sin( \tau (\wh{\mathcal{A}}^0)^{1/2}), 
\\
\label{10.18}
&J^\circ\!(\tau):=f \!\mathcal{A}^{-\!1/2} \! \sin( \tau \mathcal{A}^{1/2})f^{-\!1} \!\!- \!(I\!\! +\! \Lambda b(\mathbf{D}) ) 
f_0(\mathcal{A}^0)^{-\!1/2}\! \sin( \tau (\mathcal{A}^0)^{1/2}) f_0^{-\!1}\!.
\end{align}

  It is possible to remove the operator $\Pi$ in the estimates from Theorems \ref{A_sin_general_thrm} and \ref{sndw_A_sin_general_thrm} under Condition \ref{cond_Lambda_1}.
  
  \begin{theorem}
   \label{th10.12} 
  Suppose that Condition \emph{\ref{cond_Lambda_1}} is satisfied.  Let $\wh{J}^\circ(\tau)$ and ${J}^\circ(\tau)$ be the operators defined by \eqref{10.17} and  \eqref{10.18}.
  
  \noindent $1^\circ$. 
  For $\tau \in \R$ and $0< \eps \le 1$ we have
  \begin{equation}
	\label{10.19}
	\bigl\| \widehat{\mathcal{A}}^{1/2} \widehat{J}^\circ (\varepsilon^{-1} \tau) \mathcal{R}(\varepsilon) \bigr\|_{L_2(\mathbb{R}^d) \to L_2(\mathbb{R}^d)} \le \widehat{\mathrm{C}}^\circ_5 (1 + |\tau|) \varepsilon. 
	\end{equation}
	The constant  $\widehat{\mathrm{C}}^\circ_5$ depends on $\alpha_0,$ $\alpha_1,$ $\|g\|_{L_\infty},$ $\|g^{-1}\|_{L_\infty},$ $r_0,$ $r_1,$ and also on the norm $\| \D [\Lambda]\|_{H^2 \to L_2}$.
   
  \noindent $2^\circ$. 
  For $\tau \in \R$ and  $0< \eps \le 1$ we have
  \begin{equation}
	\label{10.20}
	\bigl\| \widehat{\mathcal{A}}^{1/2} {J}^\circ (\varepsilon^{-1} \tau) \mathcal{R}(\varepsilon) \bigr\|_{L_2(\mathbb{R}^d) \to L_2(\mathbb{R}^d)} \le {\mathrm{C}}^\circ_5 (1 + |\tau|) \varepsilon. 
	\end{equation}
	The constant  ${\mathrm{C}}^\circ_5$ depends on  $\alpha_0,$ $\alpha_1,$ $\|g\|_{L_\infty},$ $\|g^{-1}\|_{L_\infty},$ $\|f\|_{L_\infty},$ $\|f^{-1} \|_{L_\infty},$ $r_0,$ $r_1,$ and also on the norm 
	$\| \D [\Lambda]\|_{H^2 \to L_2}$.
    \end{theorem}

\begin{proof}
 Let us check statement $2^\circ$.  Statement $1^\circ$ is proved similarly.
  By \eqref{rank_alpha_ineq}, 
  $$
 \| \wh{\mathcal A}^{1/2} [\Lambda] \|_{H^2 \to L_2} = 
 \| g^{1/2} b(\D) [\Lambda] \|_{H^2 \to L_2} \le
 \alpha_1^{1/2}\|g\|_{L_\infty}^{1/2} \| \D [\Lambda] \|_{H^2 \to L_2}. 
 $$
 Combining this with  \eqref{*.1}, we see that the estimate
 $$
\| \wh{\mathcal A}^{1/2} [\Lambda] b(\D) (I \!-\! \Pi)  f_0(\mathcal{A}^0)^{-\!1/2}\! \sin( \eps^{-\!1} \tau (\mathcal{A}^0)^{1/2}) f_0^{-\!1}  {\mathcal R}(\eps)\|_{L_2(\R^d) \to L_2(\R^d)} \!\le \!
{\mathrm C}^{(3)}  \eps
$$
holds for $\tau \in \R$ and $0< \eps \le 1$.
Here ${\mathrm C}^{(3)} = {\mathrm C}^{(1)}  \alpha_1^{1/2}\|g\|_{L_\infty}^{1/2} \| \D [\Lambda] \|_{H^2 \to L_2}$.
 Using this inequality and  Theorem \ref{sndw_A_sin_general_thrm}, we arrive at  \eqref{10.20}.
\end{proof}

  It is possible to remove the operator $\Pi$ in the estimates from Theorems  \ref{th14.5} and    \ref{sndw_A_sin_enchncd_thrm_1}  under Condition \ref{cond_Lambda_2}.
  
  \begin{theorem}
  \label{th10.13}
  Suppose that Condition  \emph{\ref{cond_Lambda_2}} is satisfied.
  Let $\wh{J}^\circ(\tau)$ and ${J}^\circ(\tau)$ be the operators defined by \eqref{10.17} and  \eqref{10.18}.
  
 \noindent  $1^\circ$.  Under the assumptions of Theorem \emph{\ref{th14.5}}, 
 for $\tau \in \R$ and $0< \eps \le 1$ we have
  \begin{equation}
	\label{10.21}
	\bigl\| \widehat{\mathcal{A}}^{1/2} \widehat{J}^\circ (\varepsilon^{-1} \tau) \mathcal{R}(\varepsilon)^{3/4} \bigr\|_{L_2(\mathbb{R}^d) \to L_2(\mathbb{R}^d)} \le \widehat{\mathrm{C}}^\circ_6 (1 + |\tau|)^{1/2} \varepsilon. 
	\end{equation}
Under Condition \emph{\ref{cond_B}}, the constant $\widehat{\mathrm{C}}^\circ_6$ depends on $\alpha_0,$ $\alpha_1,$ $\|g\|_{L_\infty},$ $\|g^{-1}\|_{L_\infty},$ $r_0,$ $r_1,$ and also on the norm $\| \D [\Lambda]\|_{H^{3/2} \to L_2}$. Under Condition  \emph{\ref{cond1}}, this constant depends on the same parameters and on  $n,$
	$\wh{c}^{\circ}$.
   
 \noindent  $2^\circ$. Under the assumptions of Theorem \emph{\ref{sndw_A_sin_enchncd_thrm_1}}, for $\tau \in \R$ and $0< \eps \le 1$ we have
  \begin{equation*}
	\bigl\| \widehat{\mathcal{A}}^{1/2} {J}^\circ (\varepsilon^{-1} \tau) \mathcal{R}(\varepsilon)^{3/4} \bigr\|_{L_2(\mathbb{R}^d) \to L_2(\mathbb{R}^d)} \le {\mathrm{C}}^\circ_6(1 + |\tau|)^{1/2} \varepsilon. 
	\end{equation*}
	Under Condition  \emph{\ref{cond_BB}}, the constant   ${\mathrm{C}}^\circ_6$ depends on  $\alpha_0,$ $\alpha_1,$ $\|g\|_{L_\infty},$ $\|g^{-1}\|_{L_\infty},$ $\|f\|_{L_\infty},$ $\|f^{-1}\|_{L_\infty},$ $r_0,$   $r_1,$ and also on the norm $\| \D [\Lambda]\|_{H^{3/2} \to L_2}$. 
	Under Condition~\emph{\ref{sndw_cond1}}, this constant depends on the same parameters and on  $n,$
	${c}^{\circ}$.
    \end{theorem}

In some cases Condition \ref{cond_Lambda_1} or Condition  \ref{cond_Lambda_2} is satisfied automatically.  
We need the following results, the first one was obtained in  \cite[Proposition 9.3]{Su3},
and the second one was proved in  \cite[Lemma 8.3]{BSu4}.

\begin{proposition}[see~\cite{Su3}]
\label{prop10.10}
Let $\Lambda$ be the $\Gamma$-periodic solution of problem~\eqref{equation_for_Lambda}. 
Let $l=1$ for $d=1,$ $l>1$ for $d=2$, and $l =d/2$ for $d \ge 3$. Then the operator  
$ [\Lambda]$ is continuous from  $H^l(\R^d;\AC^m)$ to $H^1(\R^d;\AC^n),$ and the norm 
$\|  [\Lambda] \|_{H^l \to H^1}$ is controlled in terms of  $d,$  $\alpha_0,$ $\alpha_1,$ $\|g\|_{L_\infty},$ $\|g^{-1}\|_{L_\infty},$ and the parameters of the lattice $\Gamma,$
and for $d=2$ it depends also on  $l$.
\end{proposition}

\begin{proposition}[see~\cite{BSu4}]
\label{prop10.11}
Let $\Lambda$ be the  $\Gamma$-periodic solution of problem~\eqref{equation_for_Lambda}. 
Suppose that $\Lambda \in L_\infty$. Then the operator  
$ [\Lambda]$ is continuous from  $H^1(\R^d;\AC^m)$ to $H^1(\R^d;\AC^n),$ and the norm 
$\|   [\Lambda] \|_{H^1 \to H^1}$ is controlled in terms of 
$d,$  $\alpha_0,$ $\alpha_1,$ $\|g\|_{L_\infty},$ $\|g^{-1}\|_{L_\infty},$ the parameters of the lattice $\Gamma,$ and the norm $\|\Lambda\|_{L_\infty}$.
\end{proposition}

 We indicate  some cases where  Condition \ref{cond_Lambda_1} is satisfied.

\begin{proposition}
\label{prop10.12}
Suppose that at least one of the following assumptions holds\emph{:}

\noindent $1^\circ$. $d \le 4$\emph{;}

\noindent $2^\circ$. 
$\wh{\mathcal A}= \D^* g(\x) \D,$ where the matrix $g(\x)$ has real entries\emph{;}

\noindent $3^\circ$. 
$g^0 = \underline{g}$ \emph{(}i.~e., relations  
\emph{\eqref{g0=underline_g_relat}} are valid\emph{)}.

\noindent Then Condition \emph{\ref{cond_Lambda_1}} is a fortiori satisfied, and the norm  $\| [\Lambda]\|_{H^2 \to H^1}$ is controlled in terms of  $d,$ $\alpha_0,$ $\alpha_1,$ $\|g\|_{L_\infty},$ $\|g^{-1}\|_{L_\infty},$ and the parameters of the lattice~$\Gamma$.
\end{proposition}

\begin{proof}
 For  $d \le 4$, Condition  \ref{cond_Lambda_1} is ensured 
by Proposition \ref{prop10.10}.

 In the case $2^\circ$, it follows from Theorem 13.1 of~\cite[Chapter III]{LaU}
 that  $\Lambda \in L_\infty$ (and the norm $\|\Lambda\|_{L_\infty}$ is estimated in terms of
 $d$, $\|g\|_{L_\infty}$,  $\|g^{-1}\|_{L_\infty},$ and $\Omega$). 
 It remains to apply Proposition \ref{prop10.11}.
 
  In the case where $g^0 = \underline{g}$, the relation $\Lambda \in L_\infty$ (together with a suitable estimate 
  for the norm $\|\Lambda\|_{L_\infty}$) was proved in  \cite[Proposition 6.9]{BSu3}.
   Again, we apply Proposition \ref{prop10.11}.
\end{proof}

Similarly, one can check the following statement which distinguishes  some cases where Condition  \ref{cond_Lambda_2} holds.

\begin{proposition}
\label{prop10.13}
Suppose that at least one of the following assumptions is satisfied\emph{:}

\noindent $1^\circ$. $d \le 3$\emph{;}

\noindent $2^\circ$. 
$\wh{\mathcal A}= \D^* g(\x) \D,$ where the matrix $g(\x)$ has real entries\emph{;}

\noindent $3^\circ$. 
$g^0 = \underline{g}$ \emph{(}i.~e., relations 
\emph{\eqref{g0=underline_g_relat}} are valid\emph{)}.

\noindent Then Condition \emph{\ref{cond_Lambda_2}} is a fortiori satisfied, and the norm $\| [\Lambda]\|_{H^{3/2} \to H^1}$ is controlled in terms of  $d,$ $\alpha_0,$ $\alpha_1,$ $\|g\|_{L_\infty},$ $\|g^{-1}\|_{L_\infty},$ and the parameters of the lattice~$\Gamma$.
\end{proposition}

\begin{remark}
\label{rem_Lambda}
$1^\circ$.  For $d\le 4$ Condition \ref{cond_Lambda_1} is satisfied automatically. As was shown in  \cite[Lemma 8.7]{M}, for $d \ge 5$ condition $\Lambda \in L_d(\Omega)$ ensures Condition \ref{cond_Lambda_1}.

\noindent $2^\circ$. For $d\le 3$ Condition \ref{cond_Lambda_2} is satisfied automatically. By analogy with  \cite[Lemma 8.7]{M}, it is easily seen that for $d \ge 4$ condition $\Lambda \in L_{2d}(\Omega)$ ensures 
Condition~\ref{cond_Lambda_2}.
\end{remark}

\section*{Chapter 3. Homogenization problems for hyperbolic equations}

\section{Approximation for the operators $\cos(\tau \mathcal{A}_\varepsilon^{1/2})$ and $\mathcal{A}_\varepsilon^{-1/2} \sin(\tau \mathcal{A}_\varepsilon^{1/2})$}

\subsection{The operators $\widehat{\mathcal{A}}_\varepsilon$ and $\mathcal{A}_\varepsilon$. Statement of the problem}

If $\psi(\mathbf{x})$~is a measurable $\Gamma$-periodic function in $\mathbb{R}^d$, we denote  
$\psi^{\varepsilon}(\mathbf{x}) := \psi(\varepsilon^{-1} \mathbf{x}), \; \varepsilon > 0$. \emph{Our main objects}~are the operators  $\widehat{\mathcal{A}}_\varepsilon$ and $\mathcal{A}_\varepsilon$ acting in $L_2 (\mathbb{R}^d; \mathbb{C}^n)$ and formally given by 
\begin{align}
\label{Ahat_eps}
\widehat{\mathcal{A}}_\varepsilon &:= b(\mathbf{D})^* g^{\varepsilon}(\mathbf{x}) b(\mathbf{D}), \\
\label{A_eps}
\mathcal{A}_\varepsilon &:= (f^{\varepsilon}(\mathbf{x}))^* b(\mathbf{D})^* g^{\varepsilon}(\mathbf{x}) b(\mathbf{D}) f^{\varepsilon}(\mathbf{x}).
\end{align}
The precise definitions are given in terms of the quadratic forms (cf.~Subsection~\ref{A_oper_subsect}). The coefficients of the operators~(\ref{Ahat_eps}) and~(\ref{A_eps}) oscillate rapidly as $\varepsilon \to 0$.

\emph{Our goal}~is to obtain approximation for the operators $\cos(\tau \mathcal{A}_\varepsilon^{1/2})$ and $\mathcal{A}_\varepsilon^{-1/2} \sin(\tau \mathcal{A}_\varepsilon^{1/2})$ for small $\varepsilon$ and to apply the results to homogenization of the solutions of the Cauchy problem for hyperbolic equations.

\subsection{Scaling transformation}
Let $T_{\varepsilon}$~be a \emph{unitary scaling transformation in $L_2 (\mathbb{R}^d\!; \mathbb{C}^n)$}: 
$(T_{\varepsilon} \mathbf{u})(\mathbf{x}) = \varepsilon^{d/2} \mathbf{u} (\varepsilon \mathbf{x})$, $\varepsilon > 0$.
Then  $\mathcal{A}_\varepsilon = \varepsilon^{-2}T_{\varepsilon}^* \mathcal{A} T_{\varepsilon}$. Hence,
\begin{equation}
\label{sin_and_scale_transform}
\begin{split}
\cos(\tau \mathcal{A}_\varepsilon^{1/2}) &= T_{\varepsilon}^*  \cos(\varepsilon^{-1} \tau \mathcal{A}^{1/2}) T_{\varepsilon}, 
\\
\mathcal{A}_\varepsilon^{-1/2} \sin(\tau \mathcal{A}_\varepsilon^{1/2}) &= \varepsilon T_{\varepsilon}^* \mathcal{A}^{-1/2} \sin(\varepsilon^{-1} \tau \mathcal{A}^{1/2}) T_{\varepsilon}.
\end{split}
\end{equation}
Similar relations are valid also for  $\widehat{\mathcal{A}}_\varepsilon$.
Applying the scaling transformation to the resolvent of the operator $\mathcal{H}_0 = - \Delta$, we obtain 
\begin{equation}
\label{H0_resolv_and_scale_transform}
(\mathcal{H}_0 + I)^{-1} = \varepsilon^2 T_\varepsilon^* (\mathcal{H}_0 + \varepsilon^2 I)^{-1} T_\varepsilon = T_\varepsilon^* \mathcal{R} (\varepsilon) T_\varepsilon.
\end{equation}
Here $\mathcal{R} (\varepsilon)$ is the operator 
\eqref{R(epsilon)}.
 If $\psi(\mathbf{x})$~is a $\Gamma$-periodic function,~then  
\begin{equation}
\label{mult_op_and_scale_transform}
[\psi^{\varepsilon}] = T_\varepsilon^* [\psi] T_\varepsilon.
\end{equation}

\subsection{Approximation for the operators  
$\cos( \tau \widehat{\mathcal{A}}_\varepsilon^{1/2})$ and
$\widehat{\mathcal{A}}_\varepsilon^{-1/2}\sin( \tau \widehat{\mathcal{A}}_\varepsilon^{1/2})$ in the 
principal order\label{sec15.3}}

Denote
\begin{align}
\label{15.6}
	&\wh{J}_{1,\eps}(\tau) :=  \cos(\tau \widehat{\mathcal{A}}_{\varepsilon}^{1/2}) - 
	\cos(\tau (\widehat{\mathcal{A}}^0)^{1/2}),
	\\
\label{15.7}
	&\wh{J}_{2,\eps}(\tau) :=  \widehat{\mathcal{A}}_{\varepsilon}^{-1/2} \sin(\tau \widehat{\mathcal{A}}_{\varepsilon}^{1/2}) - (\widehat{\mathcal{A}}^0)^{-1/2}\sin(\tau (\widehat{\mathcal{A}}^0)^{1/2}).
\end{align}
Applying relations of the form~(\ref{sin_and_scale_transform}) for the operators $\wh{\mathcal{A}}_\varepsilon$ and $\wh{\mathcal{A}}^0$, and also~(\ref{H0_resolv_and_scale_transform}), for $\tau \in \R$ and $\eps>0$ we obtain  \begin{align}
\label{15.8}
 &\wh{J}_{1,\eps}(\tau)(\mathcal{H}_0 + I)^{-s/2} = T_{\varepsilon}^*  \wh{J}_1(\varepsilon^{-1} \tau) \mathcal{R} (\varepsilon)^{s/2} T_{\varepsilon}, 
 \\
\label{15.9}
 &\wh{J}_{2,\eps}(\tau)(\mathcal{H}_0 + I)^{-s/2} = \eps T_{\varepsilon}^*  \wh{J}_2(\varepsilon^{-1} \tau) \mathcal{R} (\varepsilon)^{s/2} T_{\varepsilon}.
\end{align}

Note that the operator $(\mathcal{H}_0 + I)^{s/2}$ is an isometric isomorphism of the Sobolev space  $H^s(\mathbb{R}^d; \mathbb{C}^n)$ onto $L_2(\mathbb{R}^d; \mathbb{C}^n)$.
Taking this into account, applying  Theorems \ref{th14.1}, \ref{th14.2}, and relations \eqref{15.8}, \eqref{15.9},
we directly obtain the following two theorems.

\begin{theorem}[see~\cite{BSu5,M}]
	\label{th15.1}
	Let $\widehat{\mathcal{A}}_{\varepsilon}$~be the operator~\emph{\eqref{Ahat_eps}} and let 
	$\widehat{\mathcal{A}}^0$~be the effective operator~\emph{\eqref{hatA0}}. 
	Let $\wh{J}_{1,\eps}(\tau)$ and  $\wh{J}_{2,\eps}(\tau)$ be the operators defined by  \eqref{15.6}, \eqref{15.7}. 
	Then for  $\tau \in \mathbb{R}$ and $\varepsilon > 0$ we have 
	\begin{align}
	\label{15.10}
	\bigl\| \wh{J}_{1,\eps}(\tau) 
	\bigr\|_{H^2(\mathbb{R}^d) \to L_2(\mathbb{R}^d)} \le \widehat{\mathrm{C}}_1 (1+|\tau|) \varepsilon,
	\\
	\label{15.11}
	\bigl\| \wh{J}_{2,\eps}(\tau) 
	\bigr\|_{H^1(\mathbb{R}^d) \to L_2(\mathbb{R}^d)} \le \widehat{\mathrm{C}}_2 (1+|\tau|) \varepsilon.
	\end{align}
	The constants $\widehat{\mathrm{C}}_1$ and $\widehat{\mathrm{C}}_2$ depend only on $\alpha_0,$ $\alpha_1,$ $\|g\|_{L_\infty},$ $\|g^{-1}\|_{L_\infty},$ and $r_0$.
\end{theorem}

\begin{theorem}
	\label{th15.2}
	Suppose that the assumptions of Theorem \emph{\ref{th15.1}} are satisfied. Suppose that  
	 Condition \emph{\ref{cond_B}} or Condition~\emph{\ref{cond1}} \emph{(}or more restrictive Condition~\emph{\ref{cond2})} is satisfied. 
		Then for $\tau \in \mathbb{R}$ and $\varepsilon > 0$ we have 
	\begin{align}
	\label{15.12}
	\bigl\| \wh{J}_{1,\eps}(\tau) 
	\bigr\|_{H^{3/2}(\mathbb{R}^d) \to L_2(\mathbb{R}^d)} \le \widehat{\mathrm{C}}_3 (1+|\tau|)^{1/2} \varepsilon,
	\\
	\label{15.13}
	\bigl\| \wh{J}_{2,\eps}(\tau) 
	\bigr\|_{H^{1/2}(\mathbb{R}^d) \to L_2(\mathbb{R}^d)} \le \widehat{\mathrm{C}}_4 (1+|\tau|)^{1/2} \varepsilon.
	\end{align}
	Under Condition \emph{\ref{cond_B}} the constants $\widehat{\mathrm{C}}_3$ and $\widehat{\mathrm{C}}_4$ depend only on  $\alpha_0,$ $\alpha_1,$ $\|g\|_{L_\infty},$ $\|g^{-1}\|_{L_\infty},$ and $r_0$.
	Under Condition~\emph{\ref{cond1}}, these constants depend on the same parameters and on $n,$  $\wh{c}^\circ$.
\end{theorem}

Theorem \ref{th15.1} was known earlier: estimate \eqref{15.10} was obtained in  \cite[Theorem 13.1]{BSu5},
and \eqref{15.11} was proved in \cite[Theorem~9.1]{M}.

By using interpolation, we deduce the following corollaries from Theorems \ref{th15.1} and \ref{th15.2}.

\begin{corollary}
	\label{cor15.1a}
	Under the assumptions of Theorem \emph{\ref{th15.1}}, we have
	\begin{align}
	\label{15.16}
	&\bigl\| \wh{J}_{1,\eps}(\tau) 
	\bigr\|_{H^s(\mathbb{R}^d) \to L_2(\mathbb{R}^d)} \le \widehat{\mathfrak{C}}_1(s) (1+|\tau|)^{s/2} \varepsilon^{s/2},
	\quad 0\le s \le 2,\ \tau \in \R, \ \eps>0;
	\\
	\label{15.17}
	&\bigl\| \wh{J}_{2,\eps}(\tau) 
	\bigr\|_{H^r(\mathbb{R}^d) \to L_2(\mathbb{R}^d)} \le \widehat{\mathfrak{C}}_2(r) (1+|\tau|)^{(r+1)/2}
	 \varepsilon^{(r+1)/2},
	\quad 0 \le r \le 1,\ \tau \in \R,\ 0< \eps \le 1;
	 \\
	 \label{hat_J_2epsD*_gen_intrpld}
	&\bigl\| \wh{J}_{2,\eps}(\tau) \mathbf{D}^*\bigr\|_{H^s(\mathbb{R}^d) \to L_2(\mathbb{R}^d)} \le \widehat{\mathfrak{C}}'_2(s) (1+|\tau|)^{s/2} \varepsilon^{s/2},
	\quad 0 \le s \le 2,\ \tau \in \R,\  \eps > 0.
	\end{align}
\end{corollary}

\begin{proof}
Obviously, 
	\begin{equation}
	\label{15.18}
	\bigl\| \wh{J}_{1,\eps}(\tau) \bigr\|_{L_2(\mathbb{R}^d) \to L_2(\mathbb{R}^d)} \le 2,
	\quad \tau \in \R, \ \eps>0.
\end{equation}
Interpolating between \eqref{15.18} and \eqref{15.10}, we arrive at estimate \eqref{15.16} with the constant 
$\widehat{\mathfrak{C}}_1(s) = 2^{1-s/2} \wh{\mathrm{C}}_1^{s/2}$.

By \eqref{14.7a} and \eqref{15.9} (with $s=0$), for $\tau \in \R$ and $0< \eps \le 1$ we have 
	\begin{align}
	\label{15.19}
	\bigl\| \wh{J}_{2,\eps}(\tau) 
	\bigr\|_{L_2(\mathbb{R}^d) \to L_2(\mathbb{R}^d)} \le \widehat{\mathrm{C}}'_2 \eps(1+\eps^{-1/2}|\tau|^{1/2})
	\le 2 \widehat{\mathrm{C}}'_2 \eps^{1/2}(1+|\tau|)^{1/2}.
	\end{align}
Interpolating between  \eqref{15.19} and \eqref{15.11}, we obtain estimate \eqref{15.17} 
with the constant 
$\widehat{\mathfrak{C}}_2(r) = (2 \wh{\mathrm{C}}'_2)^{1-r} \wh{\mathrm{C}}_2^{r}$.

Next, applying the analog of~\eqref{a_form_ineq} for the operator $\widehat{\mathcal{A}}_\varepsilon$, 
we have  
$$
\|\mathbf{D} \widehat{\mathcal{A}}_\varepsilon^{-1/2} \sin(\tau \widehat{\mathcal{A}}_\varepsilon^{1/2})\|_{L_2 \to L_2} \le \widehat{c}_*^{-1/2}.
$$
 Using a similar estimate for the operator $(\widehat{\mathcal{A}}^0)^{-1/2} \sin(\tau (\widehat{\mathcal{A}}^0)^{1/2})$ and passing to  the adjoint operators, we obtain 
\begin{equation}
\label{hat_J_2epsD*_L2_L2}
\bigl\| \wh{J}_{2,\eps}(\tau) \mathbf{D}^*\bigr\|_{L_2(\mathbb{R}^d) \to L_2(\mathbb{R}^d)} \le 2 \widehat{c}_*^{-1/2},
\quad \tau \in \R, \ \eps>0.
\end{equation}
Interpolating between~(\ref{hat_J_2epsD*_L2_L2}) and the estimate $\| \wh{J}_{2,\eps}(\tau) \mathbf{D}^*\|_{H^2 \to L_2} \le \widehat{\mathrm{C}}_2 (1+|\tau|) \varepsilon$ (which obviously follows from~\eqref{15.11}), we obtain~(\ref{hat_J_2epsD*_gen_intrpld}) with the constant $\widehat{\mathfrak{C}}'_2(s) = (2 \widehat{c}_*^{-1/2})^{1-s/2} \wh{\mathrm{C}}_2^{s/2}$.
\end{proof}

\begin{corollary}
	\label{cor15.2a}
	Under the assumptions of Theorem \emph{\ref{th15.2}}, we have
	\begin{align}
	\label{15.20}
	&\bigl\| \wh{J}_{1,\eps}(\tau) 
	\bigr\|_{H^s(\mathbb{R}^d) \to L_2(\mathbb{R}^d)} \le \widehat{\mathfrak{C}}_3(s) (1+|\tau|)^{s/3} \varepsilon^{2 s/3},
	\quad 0\le s \le 3/2,\ \tau \in \R, \ \eps>0;
	\\
	\label{15.21}
	&\bigl\| \wh{J}_{2,\eps}(\tau) 
	\bigr\|_{H^r(\mathbb{R}^d) \to L_2(\mathbb{R}^d)} \le \widehat{\mathfrak{C}}_4(r) (1+|\tau|)^{(r+1)/3}
	 \varepsilon^{2(r+1)/3},
	 \quad 0 \le r \le 1/2,\ \tau \in \R,\ 0< \eps \le 1;
	\\
	\label{hat_J_2epsD*_enchcd_intrpld}
	&\bigl\| \wh{J}_{2,\eps}(\tau) \mathbf{D}^*\bigr\|_{H^s(\mathbb{R}^d) \to L_2(\mathbb{R}^d)} \le \widehat{\mathfrak{C}}'_4(s) (1+|\tau|)^{s/3} \varepsilon^{2s/3},
	\quad 0 \le s \le 3/2,\ \tau \in \R,\  \eps >0.
	\end{align}
\end{corollary}

\begin{proof}
Interpolating between \eqref{15.18} and \eqref{15.12}, we arrive at estimate \eqref{15.20} with the constant 
$\widehat{\mathfrak{C}}_3(s) = 2^{1-2s/3} \wh{\mathrm{C}}_3^{2s/3}$.

By \eqref{14.9a} and \eqref{15.9} (with $s=0$), for $\tau \in \R$ and $0< \eps \le 1$ we have 
	\begin{align}
	\label{15.22a}
	\bigl\| \wh{J}_{2,\eps}(\tau) 
	\bigr\|_{L_2(\mathbb{R}^d) \to L_2(\mathbb{R}^d)} \le \widehat{\mathrm{C}}'_4 \eps(1+\eps^{-1/3}|\tau|^{1/3})
	\le 2 \widehat{\mathrm{C}}'_4 \eps^{2/3}(1+|\tau|)^{1/3}.
	\end{align}
Interpolating between  \eqref{15.22a} and \eqref{15.13}, we obtain estimate \eqref{15.21} with the constant 
$\widehat{\mathfrak{C}}_4(r) = (2 \wh{\mathrm{C}}'_4)^{1-2r} \wh{\mathrm{C}}_4^{2r}$.

Interpolating between~(\ref{hat_J_2epsD*_L2_L2}) and the estimate  
$$
\| \wh{J}_{2,\eps}(\tau) \mathbf{D}^*\|_{H^{3/2} \to L_2} \le \widehat{\mathrm{C}}_4 (1+|\tau|)^{1/2} \varepsilon
$$ 
(which obviously follows from~\eqref{15.13}), we obtain~(\ref{hat_J_2epsD*_enchcd_intrpld}) with the constant $\widehat{\mathfrak{C}}'_4(s) = (2 \widehat{c}_*^{-1/2})^{1-2s/3} \wh{\mathrm{C}}_4^{2s/3}$.
\end{proof}

\begin{remark}
$1^\circ$. Under the assumptions of Theorem \ref{th15.1}, we can consider large values of time 
$\tau = O(\eps^{-\alpha})$, $0< \alpha < 1$, and get the qualified estimates:
$$
\begin{aligned}
\bigl\| \wh{J}_{1,\eps}(\tau) 
	\bigr\|_{H^s(\mathbb{R}^d) \to L_2(\mathbb{R}^d)} &= O(\eps^{s(1-\alpha)/2}),& 0 &\le s \le 2;
	\\
\bigl\| \wh{J}_{2,\eps}(\tau) 
	\bigr\|_{H^r(\mathbb{R}^d) \to L_2(\mathbb{R}^d)} &= O(\eps^{(r+1)(1-\alpha)/2}),& 0 &\le r \le 1;
	\\
	\bigl\| \wh{J}_{2,\eps}(\tau) \mathbf{D}^*\bigr\|_{H^s(\mathbb{R}^d) \to L_2(\mathbb{R}^d)} &= O(\eps^{s(1-\alpha)/2}), & 0 &\le s \le 2.	
\end{aligned}
$$
$2^\circ$. Under the assumptions of Theorem \ref{th15.2}, we can consider large values of time  
$\tau = O(\eps^{-\alpha})$, $0< \alpha < 2$, and get the qualified estimates:
$$
\begin{aligned}
\bigl\| \wh{J}_{1,\eps}(\tau) 
	\bigr\|_{H^s(\mathbb{R}^d) \to L_2(\mathbb{R}^d)} &= O(\eps^{s(2-\alpha)/3}),& 0 &\le s \le 3/2;
	\\
\bigl\| \wh{J}_{2,\eps}(\tau) 
	\bigr\|_{H^r(\mathbb{R}^d) \to L_2(\mathbb{R}^d)} &= O(\eps^{(r+1)(2-\alpha)/3}),& 0 &\le r \le 1/2;
	\\
	\bigl\| \wh{J}_{2,\eps}(\tau) \mathbf{D}^*\bigr\|_{H^s(\mathbb{R}^d) \to L_2(\mathbb{R}^d)} &= O(\eps^{s(2-\alpha)/3}), & 0 &\le s \le 3/2.		
\end{aligned}
$$
\end{remark}

\subsection{Approximation for the operator $\widehat{\mathcal{A}}_\varepsilon^{-1/2}\sin( \tau \widehat{\mathcal{A}}_\varepsilon^{1/2})$  in the energy norm\label{sec15.4}}
We put $\Pi_\varepsilon := T_\varepsilon^* \Pi T_\varepsilon$. Then $\Pi_\varepsilon$~is the pseudodifferential operator in $L_2(\mathbb{R}^d; \mathbb{C}^n)$ with the symbol 
$\chi_{\widetilde{\Omega}/\varepsilon} (\bxi)$:
\begin{equation}
\label{Pi_eps}
(\Pi_\varepsilon \mathbf{u}) (\mathbf{x}) = (2 \pi)^{-d/2} \int\limits_{\widetilde{\Omega}/\varepsilon} e^{i \left\langle\mathbf{x}, \boldsymbol{\xi} \right\rangle} \widehat{\mathbf{u}} (\boldsymbol{\xi}) \, d \boldsymbol{\xi}.
\end{equation}

The following statements were proved in~\cite[Subsection~10.2]{BSu4} and \cite[Proposition 1.4]{PSu}, respectively. 

\begin{proposition}[see~\cite{BSu4}]
	\label{Pi_eps_prop_2}
	Let $\Phi(\mathbf{x})$~be a $\Gamma$-periodic function in $\mathbb{R}^d$ such that $\Phi \in L_2(\Omega)$. Then the operator $[\Phi^\varepsilon] \Pi_\varepsilon$ is bounded in  $L_2(\mathbb{R}^d; \mathbb{C}^n)$ and 
	satisfies the estimate 
	\begin{equation*}
	\| [\Phi^\varepsilon] \Pi_\varepsilon \|_{L_2(\mathbb{R}^d) \to L_2(\mathbb{R}^d)} \le |\Omega|^{-1/2} \|\Phi\|_{L_2(\Omega)}, \quad \varepsilon > 0.
	\end{equation*}
\end{proposition}

\begin{proposition}[see~\cite{PSu}]
	\label{I-Pi_eps}
	 For any function $\u \in H^1(\R^d;\AC^n)$ and any $\eps >0$ we have 
	 \begin{equation*}
	\|  \Pi_\varepsilon \u - \u \|_{L_2(\mathbb{R}^d) \to L_2(\mathbb{R}^d)} \le
	\eps r_0^{-1} \| \D \u \|_{L_2(\R^d)}.
	\end{equation*}
\end{proposition}

Denote  
\begin{equation}
\label{11.7}
	\wh{J}_\eps(\tau) :=  \widehat{\mathcal{A}}_{\varepsilon}^{-1/2} \sin(\tau \widehat{\mathcal{A}}_{\varepsilon}^{1/2}) - (I+\varepsilon \Lambda^\varepsilon b(\mathbf{D}) \Pi_\varepsilon)  (\widehat{\mathcal{A}}^0)^{-1/2} \sin(\tau (\widehat{\mathcal{A}}^0)^{1/2}).
\end{equation}
Applying relations of the form~(\ref{sin_and_scale_transform}) for the operators $\wh{\mathcal{A}}_\varepsilon$ and $\wh{\mathcal{A}}^0$, and also~(\ref{H0_resolv_and_scale_transform}) and~(\ref{mult_op_and_scale_transform}), we obtain 
\begin{equation}
\label{11.*0}
\widehat{\mathcal{A}}_{\varepsilon}^{1/2} \wh{J}_\eps(\tau)
(\mathcal{H}_0 + I)^{-s/2} = T_{\varepsilon}^* \widehat{\mathcal{A}}^{1/2} \wh{J}(\varepsilon^{-1} \tau) \mathcal{R} (\varepsilon)^{s/2} T_{\varepsilon}, \quad \varepsilon>0.
\end{equation}

The following result was proved  in \cite[Theorems 9.5, 10.8]{M} (see also \cite[Theorem 2]{M2}); 
for completeness, we give the proof.

\begin{theorem}[see~\cite{M}]
	\label{A_eps_sin_general_thrm}
	Let $\widehat{\mathcal{A}}_{\varepsilon}$~be the operator~\emph{\eqref{Ahat_eps}}, and let 
	$\widehat{\mathcal{A}}^0$~be the effective operator~\emph{\eqref{hatA0}}. 
	Suppose that $\Lambda(\mathbf{x})$~is the  $\Gamma$-periodic solution of 
	problem~\emph{(\ref{equation_for_Lambda})}. Let $\Pi_\varepsilon$~be the 
	operator~\emph{(\ref{Pi_eps})}.  Let $\wh{J}_\eps(\tau)$ be the operator defined by~\eqref{11.7}. 
	Denote  
\begin{equation}
\label{11.*00}
	\wh{I}_\eps(\tau) :=  g^\eps b(\D) \widehat{\mathcal{A}}_{\varepsilon}^{-1/2} \sin(\tau \widehat{\mathcal{A}}_{\varepsilon}^{1/2}) - \wt{g}^\varepsilon  b(\mathbf{D}) \Pi_\varepsilon (\widehat{\mathcal{A}}^0)^{-1/2} \sin(\tau (\widehat{\mathcal{A}}^0)^{1/2}),
\end{equation}	
where $\wt{g}$ is defined by  \eqref{g_tilde}.
	Then for $\tau \in \mathbb{R}$ and $\varepsilon > 0$ we have 
	\begin{align}
	\label{A_eps_sin_general_est}
	\bigl\| \wh{J}_\eps(\tau) 
	\bigr\|_{H^2(\mathbb{R}^d) \to H^1(\mathbb{R}^d)} \le \widehat{\mathrm{C}}_{7} (1+|\tau|) \varepsilon,
	\\
	\label{A_eps_flux_general_est}
	\bigl\| \wh{I}_\eps(\tau) 
	\bigr\|_{H^2(\mathbb{R}^d) \to L_2(\mathbb{R}^d)} \le \widehat{\mathrm{C}}_{8} (1+|\tau|) \varepsilon.
	\end{align}
The constants $\widehat{\mathrm{C}}_{7}$ and $\widehat{\mathrm{C}}_{8}$ depend only on $\alpha_0,$ $\alpha_1,$ $\|g\|_{L_\infty},$ $\|g^{-1}\|_{L_\infty},$ $r_0,$ and $r_1$.
\end{theorem}

\begin{proof}
	Using \eqref{11.*0},  from~(\ref{A_sin_general_est}) we obtain 
	\begin{equation}
	\label{15.22}
	\bigl\| \widehat{\mathcal{A}}_{\varepsilon}^{1/2} \wh{J}_\eps(\tau)
	 (\mathcal{H}_0 + I)^{-1} \bigr\|_{L_2(\mathbb{R}^d) \to L_2(\mathbb{R}^d)} \le \widehat{\mathrm{C}}_5 (1+|\tau|) \varepsilon. 
	\end{equation}
	Similarly to~(\ref{a_form_ineq}), 
	\begin{equation}
	\label{15.31a}
	\widehat{c}_* \|\mathbf{D}\mathbf{u}\|_{L_2(\mathbb{R}^d)}^2 \le \| \widehat{\mathcal{A}}_\varepsilon^{1/2} \mathbf{u}\|_{L_2(\mathbb{R}^d)}^2, \quad \mathbf{u} \in H^1(\mathbb{R}^d; \mathbb{C}^n).
	\end{equation}
	Hence,
	\begin{equation}
	\label{11.*1}
	\bigl\| \mathbf{D}  \wh{J}_\eps(\tau)
	 (\mathcal{H}_0 + I)^{-1} \bigr\|_{L_2(\mathbb{R}^d) \to L_2(\mathbb{R}^d)} \le 
	 \widehat{c}_*^{-1/2} \widehat{\mathrm{C}}_5 (1+|\tau|) \varepsilon. 
	\end{equation}
	Next, by \eqref{15.11},
	\begin{equation}
	\label{M_th9.1}
	\big\| \big( \widehat{\mathcal{A}}_{\varepsilon}^{-1/2} \sin(\tau \widehat{\mathcal{A}}_{\varepsilon}^{1/2}) - (\widehat{\mathcal{A}}^0)^{-1/2} \sin(\tau (\widehat{\mathcal{A}}^0)^{1/2}) \big) (\mathcal{H}_0 + I)^{-1/2} \big\|_{L_2 \to L_2}
	\le \widehat{\mathrm{C}}_2 (1+|\tau|) \varepsilon.
	\end{equation}
	
	Now, we estimate the norm of the corrector. Let $\Pi_\varepsilon^{(m)}$~be the pseudodifferential operator in  $L_2(\mathbb{R}^d; \mathbb{C}^m)$ with the symbol $\chi_{\widetilde{\Omega}/\varepsilon}(\boldsymbol{\xi})$. According to  Proposition~\ref{Pi_eps_prop_2} and~(\ref{Lambda_est}),
	\begin{equation}
	\label{Lambda_eps_Pi_eps_est}
	\| \Lambda^\varepsilon \Pi_\varepsilon^{(m)} \|_{L_2(\mathbb{R}^d) \to L_2(\mathbb{R}^d)} \le M_1.
	\end{equation}
	Using~(\ref{g^0_est}) and~(\ref{Lambda_eps_Pi_eps_est}), we obtain 
	\begin{equation}
	\label{11.*3}
	\begin{split}
	\| \varepsilon \Lambda^\varepsilon b(\mathbf{D})& \Pi_\varepsilon  (\widehat{\mathcal{A}}^0)^{-1/2} \sin(\tau (\widehat{\mathcal{A}}^0)^{1/2})   \|_{L_2(\mathbb{R}^d) \to L_2(\mathbb{R}^d)} 
	\\ 
	&\le
	\varepsilon \| \Lambda^\varepsilon \Pi_\varepsilon^{(m)} \|_{L_2(\mathbb{R}^d) \to L_2(\mathbb{R}^d)} \|b(\mathbf{D}) (\widehat{\mathcal{A}}^0)^{-1/2} \|_{L_2(\mathbb{R}^d) \to L_2(\mathbb{R}^d)}
\le  	\varepsilon M_1 \|g^{-1}\|_{L_\infty}^{1/2}.
	\end{split}
	\end{equation}
	Together with  \eqref{M_th9.1} this implies 
	\begin{equation}
	\label{11.*2}
	\bigl\| \wh{J}_\eps(\tau)
	 (\mathcal{H}_0 + I)^{-1} \bigr\|_{L_2(\mathbb{R}^d) \to L_2(\mathbb{R}^d)} \le  
	 \bigl( \widehat{\mathrm{C}}_2 + M_1 \|g^{-1}\|_{L_\infty}^{1/2} \bigr) 
	 (1+|\tau|) \varepsilon.
	\end{equation}
Estimates \eqref{11.*1} and \eqref{11.*2} yield inequality  \eqref{A_eps_sin_general_est} with the constant  
$$
\wh{\mathrm C}_{7} = \widehat{c}_*^{-1/2} \widehat{\mathrm{C}}_5 +  \widehat{\mathrm{C}}_2 + M_1 \|g^{-1}\|_{L_\infty}^{1/2}.
$$

Now, we check estimate \eqref{A_eps_flux_general_est}. 
From  \eqref{15.22} it follows that 
\begin{equation}
\label{12.*01}
		\| g^\varepsilon b(\D) \wh{J}_\eps(\tau) \|_{H^2(\R^d) \to L_2 (\mathbb{R}^d)} \le 
		 \| g\|^{1/2}_{L_\infty} \widehat{\mathrm{C}}_5 (1+|\tau|)  \varepsilon.
\end{equation}
  Taking  \eqref{g_tilde} into account, we have 
{\allowdisplaybreaks
\begin{multline}
\label{12.*02}
  g^\eps b(\D) (I+\varepsilon \Lambda^\varepsilon b(\mathbf{D}) \Pi_\varepsilon)  (\widehat{\mathcal{A}}^0)^{-1/2} \sin(\tau (\widehat{\mathcal{A}}^0)^{1/2})
  \\
  =
   \wt{g}^\eps b(\D) \Pi_\eps (\widehat{\mathcal{A}}^0)^{-1/2} \sin(\tau (\widehat{\mathcal{A}}^0)^{1/2})
   + g^\eps b(\D) (I - \Pi_\eps) (\widehat{\mathcal{A}}^0)^{-1/2} \sin(\tau (\widehat{\mathcal{A}}^0)^{1/2})
   \\
   + \eps g^\eps \sum_{l=1}^d b_l \Lambda^\eps D_l b(\D) \Pi_\eps 
   (\widehat{\mathcal{A}}^0)^{-1/2} \sin(\tau (\widehat{\mathcal{A}}^0)^{1/2}).
  \end{multline}
 }
\hspace{-3mm}
  By Proposition~\ref{I-Pi_eps}, 
  \begin{equation}
  \label{12.*03}
  \begin{aligned}
 &\| g^\eps b(\D) (I - \Pi_\eps) (\widehat{\mathcal{A}}^0)^{-1/2} \sin(\tau (\widehat{\mathcal{A}}^0)^{1/2}) \|_{H^2(\R^d) \to L_2(\R^d)}
 \\
 & \le \|g\|_{L_\infty}\|g^{-1} \|^{1/2}_{L_\infty} \| I - \Pi_\eps\|_{H^2(\R^d)\to L_2(\R^d)} 
 \le \eps r_0^{-1} \|g\|_{L_\infty}\|g^{-1} \|^{1/2}_{L_\infty}.
 \end{aligned}
  \end{equation}
  Next, from \eqref{5.7a} and \eqref{Lambda_eps_Pi_eps_est} it follows that  
  \begin{equation}
  \label{12.*04}
  \begin{split}
  \big\| \eps g^\eps \sum_{l=1}^d b_l \Lambda^\eps D_l b(\D) \Pi_\eps 
 (\widehat{\mathcal{A}}^0)^{-1/2} \sin(\tau (\widehat{\mathcal{A}}^0)^{1/2})
  \big\|_{H^2(\R^d) \to L_2(\R^d)}
 \\
 \le \eps \|g\|_{L_\infty} \|g^{-1}\|^{1/2}_{L_\infty} \alpha_1^{1/2} M_1 d^{1/2}.
 \end{split}
  \end{equation}
  As a result, relations \eqref{12.*01}--\eqref{12.*04} together with \eqref{11.7} and \eqref{11.*00} imply  
  \eqref{A_eps_flux_general_est}.
  \end{proof}

Using interpolation, we deduce the following result from Theorem~\ref{A_eps_sin_general_thrm}.

\begin{corollary}
	\label{A_eps_sin_general_interpltd_thrm}
	Suppose that the assumptions of Theorem~\emph{\ref{A_eps_sin_general_thrm}} are satisfied. 
	Then for \hbox{$0 \le s \le 2$}\textup, $\tau \in \R,$ and $\eps >0$ we have
	\begin{align}
	\label{15.43b}
	\| \D \wh{J}_\eps(\tau)
	\|_{H^{s}(\mathbb{R}^d) \to L_2(\mathbb{R}^d)} \le \widehat{\mathfrak{C}}_5(s) 
	(1 + |\tau|)^{s/2} \varepsilon^{s/2},
	\\
	\label{15.44b}
	\| \wh{I}_\eps(\tau)
	\|_{H^{s}(\mathbb{R}^d) \to L_2(\mathbb{R}^d)} \le \widehat{\mathfrak{C}}_6(s) (1 + |\tau|)^{s/2} 
	\varepsilon^{s/2}.
	\end{align}
\end{corollary}

\begin{proof}
	We rewrite estimate  \eqref{11.*1} in the form 
	\begin{equation}
	\label{15.44}
	\bigl\| \mathbf{D}  \wh{J}_\eps(\tau) \bigr\|_{H^2(\mathbb{R}^d) \to L_2(\mathbb{R}^d)} \le 
	 \widehat{c}_*^{-1/2} \widehat{\mathrm{C}}_5 (1+|\tau|) \varepsilon. 
	\end{equation}
	
Now, we estimate the quantity $\bigl\| \mathbf{D}  \wh{J}_\eps(\tau) \bigr\|_{L_2(\mathbb{R}^d) \to L_2(\mathbb{R}^d)}$. From \eqref{15.31a} and the similar estimate  for the operator $\widehat{\mathcal{A}}^0$ it follows that
	\begin{equation}
	\label{15.45}	
\| \D ( \widehat{\mathcal{A}}_{\varepsilon}^{-\!1/2} \!\sin ( \tau \widehat{\mathcal{A}}_{\varepsilon}^{1/2})\!	- \!( \widehat{\mathcal{A}}^0 )^{-\!1/2} \!\sin ( \tau (\widehat{\mathcal{A}}^0)^{1/2}))\|_{L_2(\R^d) \to L_2(\R^d)}\!\! \le\! 2 \wh{c}_*^{-1/2}\!\!.
	\end{equation}
	Next, 
	\begin{equation}
	\label{15.46}	
	\begin{split}
 D_l &\bigl( \eps \Lambda^\eps b(\D) \Pi_\eps 
( \widehat{\mathcal{A}}^0 )^{-1/2} \sin ( \tau (\widehat{\mathcal{A}}^0)^{1/2}) \bigr)
\\
&=
( D_l \Lambda)^\eps \Pi_\eps^{(m)} b(\D)  
(\widehat{\mathcal{A}}^0 )^{-1/2} \sin ( \tau (\widehat{\mathcal{A}}^0)^{1/2}) 
\\
&\quad+
 \eps \Lambda^\eps \Pi_\eps^{(m)} b(\D) 
(\widehat{\mathcal{A}}^0 )^{-1/2} \sin ( \tau (\widehat{\mathcal{A}}^0)^{1/2}) D_l \Pi_\eps,  
\quad  l=1,\dots,d.
	\end{split}
	\end{equation}
According to Proposition~\ref{Pi_eps_prop_2} and~\eqref{D_Lambda_est},
	\begin{equation}
	\label{15.48}
	\| (\D \Lambda)^\varepsilon \Pi_\varepsilon^{(m)} \|_{L_2(\mathbb{R}^d) \to L_2(\mathbb{R}^d)} \le M_2.
	\end{equation}
Hence,
	\begin{equation}
	\label{15.49}
	\big \| ( \D \Lambda)^\eps \Pi_\eps^{(m)} b(\D)  
(\widehat{\mathcal{A}}^0 )^{-\!1/2}\! \sin ( \tau (\widehat{\mathcal{A}}^0)^{1/2}) \big\|_{L_2(\R^d) \to L_2(\R^d)}
\le M_2 \| g^{-\!1}\|_{L_\infty}^{1/2}\!.
	\end{equation}
	Next, we have 
\begin{equation}
	\label{15.50}
\begin{aligned}
	\bigl \| & \eps \Lambda^\eps \Pi_\eps^{(m)} b(\D)  
(\widehat{\mathcal{A}}^0 )^{-1/2} \sin ( \tau (\widehat{\mathcal{A}}^0)^{1/2}) \D \Pi_\eps \bigr\|_{L_2(\R^d) \to L_2(\R^d)} 
\\
&\le \eps 
	\bigl \| \Lambda^\eps \Pi_\eps^{(m)} \bigr\|_{L_2 \to L_2} \bigl \| b(\D)  (\widehat{\mathcal{A}}^0 )^{-1/2} \bigr \|_{L_2 \to L_2}  \bigl \| \D \Pi_\eps \bigr\|_{L_2 \to L_2}.
\end{aligned}	
\end{equation}
By	\eqref{Pi_eps}, 
	\begin{equation}
	\label{15.51}
	\bigl \|  \D \Pi_\eps \bigr\|_{L_2(\R^d) \to L_2(\R^d)} = \sup_{\boldsymbol{\xi} \in \wt{\Omega}/\eps} |\boldsymbol{\xi}| \le \eps^{-1} r_1.
	\end{equation}
Relations  \eqref{Lambda_eps_Pi_eps_est}, \eqref{15.50}, and \eqref{15.51} imply that
\begin{equation}
	\label{15.52}
	\big \|  \eps \Lambda^\eps \Pi_\eps^{(m)} b(\D)  
(\widehat{\mathcal{A}}^0 )^{-1/2} \sin ( \tau (\widehat{\mathcal{A}}^0)^{1/2}) \D \Pi_\eps \big\|_{L_2(\R^d) \to L_2(\R^d)} 
\le M_1 \| g^{-1}\|_{L_\infty}^{1/2} r_1. 
\end{equation}
	
	As a result, from \eqref{15.46}, \eqref{15.49}, and \eqref{15.52} it follows that  
\begin{equation}
	\label{15.53}
	\big \| \D  \eps \Lambda^\eps b(\D) \Pi_\eps 
( \widehat{\mathcal{A}}^0 )^{-1/2} \sin ( \tau (\widehat{\mathcal{A}}^0)^{1/2}) 
 \big\|_{L_2(\R^d) \to L_2(\R^d)}
 \le (M_1 r_1 + M_2) \| g^{-1}\|_{L_\infty}^{1/2}.
	\end{equation}
	Combining \eqref{15.45} and \eqref{15.53}, we obtain 
\begin{equation}
	\label{15.54}
	\bigl \| \D \wh{J}_\eps(\tau)   \bigr\|_{L_2(\R^d) \to L_2(\R^d)} \le \wh{\mathrm C}_{7}'
	= 2 \wh{c}_*^{-1/2} +(M_1 r_1+ M_2) \| g^{-1}\|_{L_\infty}^{1/2}.
	\end{equation}
	
	Interpolating between  \eqref{15.54} and \eqref{15.44}, we arrive at estimate \eqref{15.43b} with the constant  $\wh{\mathfrak{C}}_5(s)= (\wh{\mathrm C}_{7}')^{1-s/2} ( \widehat{c}_*^{-1/2} \widehat{\mathrm{C}}_5)^{s/2}$.	
	
	We proceed to the proof of estimate \eqref{15.44b}. Let us estimate the norm  
	$\bigl \| \wh{I}_\eps(\tau)   \bigr\|_{L_2 \to L_2}$. Obviously,
\begin{equation}
	\label{15.55}
	\bigl \| g^\eps b(\D) \widehat{\mathcal{A}}_{\varepsilon}^{-1/2} \sin ( \tau \widehat{\mathcal{A}}_{\varepsilon}^{1/2})    \bigr\|_{L_2(\R^d) \to L_2(\R^d)} \le \| g\|_{L_\infty}^{1/2}. 
	\end{equation}
	Next, from \eqref{g_tilde}, \eqref{6.11a}, and Proposition~\ref{Pi_eps_prop_2} it follows that 
\begin{equation}
	\label{15.55a}
	\bigl \| \wt{g}^\eps \Pi_\eps^{(m)}   \bigr\|_{L_2(\R^d) \to L_2(\R^d)} \le 2 \| g\|_{L_\infty}. 
\end{equation}
	Therefore,  
\begin{equation}
	\label{15.56}
	\bigl \| \wt{g}^\eps b(\D) \Pi_\eps (\widehat{\mathcal{A}}^0)^{-1/2} 
	\sin ( \tau (\widehat{\mathcal{A}}^0)^{1/2})    \bigr\|_{L_2(\R^d) \to L_2(\R^d)} \le 2
	\| g \|_{L_\infty} \| g^{-1}\|_{L_\infty}^{1/2}. 
	\end{equation}
	Combining  \eqref{15.55} and \eqref{15.56}, we obtain 
\begin{equation}
	\label{15.57}
	\bigl \| \wh{I}_\eps(\tau)   \bigr\|_{L_2(\R^d) \to L_2(\R^d)} \le \wh{\mathrm{C}}_{8}'
	= \| g\|_{L_\infty}^{1/2} +  2 \| g \|_{L_\infty} \| g^{-1}\|_{L_\infty}^{1/2}.
	\end{equation}

Interpolating between  \eqref{15.57} and \eqref{A_eps_flux_general_est}, we arrive at estimate \eqref{15.44b} with the constant $\wh{\mathfrak{C}}_6(s)= (\wh{\mathrm C}_{8}')^{1-s/2}  \wh{\mathrm C}_{8}^{s/2}$.	
\end{proof}

\begin{remark}
	\label{rem15.11a}
	From \eqref{15.19}, \eqref{11.*3}, and \eqref{15.54} it follows that 
	\begin{equation}
	\label{15.*}
	\|  \wh{J}_\eps(\tau)
	\|_{L_2(\mathbb{R}^d) \to H^1(\mathbb{R}^d)} \le \widehat{\mathrm{C}}_{7}''
	\bigl( 1+ (1+|\tau|)^{1/2}\eps^{1/2}\bigr), \quad \tau \in \R, \ 0< \eps \le 1.
	\end{equation}
	Interpolating between \eqref{15.*} and \eqref{A_eps_sin_general_est},  for $\tau \in \R$ and  $0< \eps \le 1$
	we obtain 
	\begin{equation}
	\label{15.**}
	\|  \wh{J}_\eps(\tau)\|_{H^s(\mathbb{R}^d) \to H^1(\mathbb{R}^d)}
	\le
	\wh{\mathfrak{C}}_5'(s) (1 + |\tau|)^{s/2} \varepsilon^{s/2} \bigl( 1+ (1+|\tau|)^{1/2}\eps^{1/2}\bigr)^{1-s/2},
	\quad 0 \le s \le 2.
	\end{equation}
	This estimate is interesting for bounded values of  $(1+|\tau|)\eps$, in this case 
	the right-hand side of  \eqref{15.**} does not exceed   
	$C(1 + |\tau|)^{s/2} \varepsilon^{s/2}$, i.~e., has the same order as estimate \eqref{15.43b}.
\end{remark}

By analogy with the proof of Theorem~\ref{A_eps_sin_general_thrm}, we deduce the following statement from 
Theorem \ref{th14.5}.

\begin{theorem}
	\label{A_eps_sin_enchcd_thrm_1}
	Suppose that the assumptions of Theorem~\emph{\ref{A_eps_sin_general_thrm}} are satisfied. Suppose that  Condition \emph{\ref{cond_B}} or Condition~\emph{\ref{cond1}} \emph{(}or more restrictive Condition~\emph{\ref{cond2})} is satisfied.
	Then for $\tau \in \mathbb{R}$ and $\varepsilon > 0$ we have 
	\begin{align}
	\label{11.*5}
	\bigl\| \wh{J}_\eps(\tau)  
	\bigr\|_{H^{3/2}(\mathbb{R}^d) \to H^1(\mathbb{R}^d)} \le \widehat{\mathrm{C}}_{9}(1+|\tau|)^{1/2} \varepsilon,
	\\
	\nonumber
	\bigl\| \wh{I}_\eps(\tau)  
	\bigr\|_{H^{3/2}(\mathbb{R}^d) \to L_2(\mathbb{R}^d)} \le \widehat{\mathrm{C}}_{10} (1+|\tau|)^{1/2} \varepsilon.
	\end{align}
	Under Condition \emph{\ref{cond_B}}, the constants $\widehat{\mathrm{C}}_{9}$ and 
	$\widehat{\mathrm{C}}_{10}$ depend only on $\alpha_0,$ $\alpha_1,$ $\|g\|_{L_\infty},$ $\|g^{-1}\|_{L_\infty},$ $r_0,$ and $r_1$.
	Under Condition~\emph{\ref{cond1}}, these constants depend on the same parameters and on 
	$n,$ $\widehat{c}^{\circ}$.
\end{theorem}

By interpolation, we deduce the following corollary from Theorem \ref{A_eps_sin_enchcd_thrm_1} and relations~\eqref{15.54},~\eqref{15.57}. 

\begin{corollary}
	\label{A_eps_sin_enchcd_interpltd_thrm_1}
	Under the assumptions of Theorem~\emph{\ref{A_eps_sin_enchcd_thrm_1}}, for $0\le s \le 3/2,$ $\tau \in \R,$ and $\eps >0$ we have
		\begin{align}
		\label{15.333}
	\bigl\| \D \wh{J}_\eps(\tau)
	\bigr\|_{H^{s}(\mathbb{R}^d) \to L_2(\mathbb{R}^d)} \le \widehat{\mathfrak{C}}_7(s) 
	(1 + |\tau|)^{s/3} \varepsilon^{2s/3}, 
	\\ 
	\nonumber
	\bigl\| \wh{I}_\eps(\tau)
	\bigr\|_{H^{s}(\mathbb{R}^d) \to L_2(\mathbb{R}^d)} \le \widehat{\mathfrak{C}}_{8}(s) 
	(1 + |\tau|)^{s/3} \varepsilon^{2s/3}.
	\end{align}
	\end{corollary}

\begin{remark}
	\label{rem15.15a}
	Under the assumptions of Theorem~\ref{A_eps_sin_enchcd_thrm_1}, from \eqref{15.22a}, \eqref{11.*3}, 
	and \eqref{15.54} it follows that 
	\begin{equation}
	\label{15.***}
	\|  \wh{J}_\eps(\tau)
	\|_{L_2(\mathbb{R}^d) \to H^1(\mathbb{R}^d)} \le \widehat{\mathrm{C}}_{9}'
	\bigl( 1+ (1+|\tau|)^{1/3}\eps^{2/3}\bigr), \quad \tau \in \R, \ 0< \eps \le 1.
	\end{equation}
	Interpolating between \eqref{15.***} and \eqref{11.*5},  for $0\le s \le 3/2$, $\tau \in \R$, 
	and $0< \eps \le 1$ we obtain 
	\begin{equation}
	\label{15.****}
	\|  \wh{J}_\eps(\tau)\|_{H^s(\mathbb{R}^d) \to H^1(\mathbb{R}^d)} \!\le\!
	\wh{\mathfrak{C}}_7'(s) (1\! + \!|\tau|)^{s/3} \varepsilon^{2s/3} ( 1\!\!+\! (1\!+\!|\tau|)^{1/3}\eps^{2/3})^{1\!-2s/3}\!.
	\end{equation}
	For bounded values of  $(1\!+\!|\tau|)^{1/2}\eps$,
	the right-hand side of  \!\eqref{15.****} does not exceed   
	$C(1\! + \!|\tau|)^{s/3} \!\varepsilon^{2s/3}$\!, i.~e., has the same order as estimate~\!\eqref{15.333}.
\end{remark}

\begin{remark}
$1^\circ$.  Under the assumptions of Theorem \ref{A_eps_sin_general_thrm}, for  $\tau = O(\eps^{-\alpha})$,
$0< \alpha < 1$,  we get the qualified estimates\emph{:}
$$
\begin{aligned}
\bigl\| \D \wh{J}_{\eps}(\tau) 
	\bigr\|_{H^s(\mathbb{R}^d) \to L_2(\mathbb{R}^d)} &= O(\eps^{s(1-\alpha)/2}),\quad 
	0 \le s \le 2;
	\\
\bigl\| \wh{I}_{\eps}(\tau) 
	\bigr\|_{H^s(\mathbb{R}^d) \to L_2(\mathbb{R}^d)} &= O(\eps^{s(1-\alpha)/2}),\quad 0 \le s \le 2.	
\end{aligned}
$$
$2^\circ$. Under the assumptions of Theorem \ref{A_eps_sin_enchcd_thrm_1}, for
 $\tau = O(\eps^{-\alpha})$, $0< \alpha < 2$, we get the qualified estimates:
$$
\begin{aligned}
\bigl\| \D \wh{J}_{\eps}(\tau) 
	\bigr\|_{H^s(\mathbb{R}^d) \to L_2(\mathbb{R}^d)} &= O(\eps^{s(2-\alpha)/3}),\quad 0 \le s \le 3/2;
	\\
\bigl\| \wh{I}_{\eps}(\tau) 
	\bigr\|_{H^s(\mathbb{R}^d) \to L_2(\mathbb{R}^d)} &= O(\eps^{s(2-\alpha)/3}),\quad 0 \le s \le 3/2.	
\end{aligned}
$$
\end{remark}

\subsection{Sharpness of the results of Subsections \ref{sec15.3} and \ref{sec15.4}}

Applying theorems of Subsection \ref{sec14.3}, we confirm that the results of Subsections
 \ref{sec15.3} and \ref{sec15.4} are sharp. First, we discuss the sharpness of the results regarding  the type of the operator norm.
 The following statement, confirming that Theorems \ref{th15.1} and \ref{A_eps_sin_general_thrm} are sharp, is deduced from Theorem \ref{th14.7}.
 
\begin{theorem}
	\label{th15.17}
	Suppose that Condition~\emph{\ref{cond10.1}} is satisfied. 
	
	\noindent $1^\circ$. Let $0 \ne \tau \in \mathbb{R}$ and $0 \le s < 2$. Then there does not exist a constant $\mathcal{C}(\tau) > 0$ such that the estimate 
	\begin{equation}
	\label{15.61}
	\bigl\| \wh{J}_{1,\eps}(\tau)
	\bigr\|_{H^s(\mathbb{R}^d) \to L_2(\mathbb{R}^d)} \le \mathcal{C}(\tau) \varepsilon
	\end{equation}
	holds for all sufficiently small $\varepsilon > 0$.
	
	\noindent $2^\circ$. Let $0 \ne \tau \in \mathbb{R}$ and  $0 \le r < 1$. Then there does not exist a constant  $\mathcal{C}(\tau) > 0$ such that the estimate
	\begin{equation}
	\label{15.62}
	\bigl\| \wh{J}_{2,\eps}(\tau)
	\bigr\|_{H^r(\mathbb{R}^d) \to L_2(\mathbb{R}^d)} \le \mathcal{C}(\tau) \varepsilon
	\end{equation}
holds for all sufficiently small $\varepsilon > 0$.

	\noindent $3^\circ$. Let $0 \ne \tau \in \mathbb{R}$ and  $0 \le s < 2$. Then there does not exist a constant $\mathcal{C}(\tau) > 0$ such that the estimate
	\begin{equation}
	\label{s<2_sin_est1}
	\bigl\| \wh{J}_\eps(\tau)
	\bigr\|_{H^s(\mathbb{R}^d) \to H^1(\mathbb{R}^d)} \le \mathcal{C}(\tau) \varepsilon
	\end{equation}
	holds for all sufficiently small $\varepsilon > 0$.	
	\end{theorem}

\begin{proof}
 Let us check statement $1^\circ$. Suppose that for some $\tau \ne 0$ and $0\le s <2$
estimate \eqref{15.61} holds for sufficiently small~$\eps$. Applying the scaling transformation (see \eqref{15.8}), we see that estimate \eqref{14.17} is satisfied. But this contradicts statement $1^\circ$ of Theorem  \ref{th14.7}.

Statement $2^\circ$ follows from   \eqref{15.9} and statement  $2^\circ$ of Theorem  \ref{th14.7}.

 We proceed to the proof of  statement $3^\circ$. Suppose that for some $\tau \ne 0$ and $0 \le s < 2$ 
 estimate  \eqref{s<2_sin_est1} is satisfied. Then  
\begin{equation*}
\big\| \mathbf{D} \wh{J}_\eps(\tau)
(\mathcal{H}_0 + I)^{-s/2} \big\|_{L_2(\mathbb{R}^d) \to L_2(\mathbb{R}^d)} \le {\mathcal{C}}(\tau) \varepsilon
\end{equation*}
for sufficiently small $\varepsilon$.  Hence, estimate
\begin{equation*}
\bigl\| \widehat{\mathcal{A}}_\varepsilon^{1/2} \wh{J}_\eps(\tau)
(\mathcal{H}_0 + I)^{-s/2}
\bigr\|_{L_2(\mathbb{R}^d) \to L_2(\mathbb{R}^d)} \le \wh{\mathcal{C}}(\tau) \varepsilon
\end{equation*}
is also satisfied for sufficiently small $\varepsilon$ (with some constant $\wh{\mathcal{C}}(\tau)$).
Applying the scaling transformation, we see that  estimate~(\ref{hat_s<2_sinA_est}) holds for sufficiently 
small $\varepsilon$. But this contradicts statement $3^\circ$ of Theorem~\ref{th14.7}.
\end{proof}

Next, Theorem \ref{th14.8} allows us to confirm that 
Theorems \ref{th15.2} and \ref{A_eps_sin_enchcd_thrm_1} are sharp.

\begin{theorem}
	\label{th15.18}
	Suppose that Condition~\emph{\ref{cond10.2}} is satisfied. 
	
	\noindent $1^\circ$. Let $0 \ne \tau \in \mathbb{R}$ and $0 \le s < 3/2$. Then there does not exist a constant $\mathcal{C}(\tau) > 0$ such that estimate 
	\eqref{15.61} holds for sufficiently small $\varepsilon > 0$.
	
	\noindent $2^\circ$. Let $0 \ne \tau \in \mathbb{R}$ and $0 \le r < 1/2$. Then there does not exist a constant  $\mathcal{C}(\tau) > 0$ such that estimate 
	\eqref{15.62} holds for sufficiently small $\varepsilon > 0$.

	\noindent $3^\circ$. Let $0 \ne \tau \in \mathbb{R}$ and $0 \le s < 3/2$. Then there does not exist a constant  $\mathcal{C}(\tau) > 0$ such that estimate   \eqref{s<2_sin_est1} holds for sufficiently small $\varepsilon > 0$.
	\end{theorem}

 Now we discuss the sharpness of the results  regarding  the dependence of estimates on the parameter $\tau$. 
 Theorem \ref{th14.11} implies the following statement  which shows that Theorems \ref{th15.1} 
 and  \ref{A_eps_sin_general_thrm} are sharp.

\begin{theorem}
	\label{th15.21}
	Suppose that Condition~\emph{\ref{cond10.1}} is satisfied. 
	
	\noindent $1^\circ$. Let $s \ge 2$. There does not exist a positive function 
	$\mathcal{C}(\tau)$ such that  $\lim_{\tau \to \infty} {\mathcal C}(\tau)/|\tau|=0$ and estimate \eqref{15.61}
	holds for $\tau \in \R$ and sufficiently small $\varepsilon > 0$.

	\noindent $2^\circ$. Let $r \ge 1$. There does not exist a positive function
	$\mathcal{C}(\tau)$ such that 
	$\lim_{\tau \to \infty} {\mathcal C}(\tau)/|\tau|=0$ and estimate \eqref{15.62}
	holds for $\tau \in \R$ and sufficiently small $\varepsilon > 0$.

	\noindent $3^\circ$. Let $s \ge 2$. There does not exist a positive function $\mathcal{C}(\tau)$ such that  $\lim_{\tau \to \infty} {\mathcal C}(\tau)/|\tau|=0$ and estimate  \eqref{s<2_sin_est1}
	holds for $\tau \in \R$ and sufficiently small $\varepsilon > 0$.
\end{theorem}

 Theorem \ref{th14.12} shows that Theorems \ref{th15.2} and \ref{A_eps_sin_enchcd_thrm_1} are sharp.

\begin{theorem}
	\label{th15.22}
	Suppose that Condition~\emph{\ref{cond10.2}} is satisfied. 
	
	\noindent $1^\circ$. Let $s \ge 3/2$. There does not exist a positive function $\mathcal{C}(\tau)$ such that 
	$\lim_{\tau \to \infty} {\mathcal C}(\tau)/|\tau|^{1/2}=0$ and estimate \eqref{15.61}
holds for $\tau \in \R$ and sufficiently small $\varepsilon > 0$.

	\noindent $2^\circ$. 
	Let $r \ge 1/2$. There does not exist a positive function $\mathcal{C}(\tau)$ such that 
	$\lim_{\tau \to \infty} {\mathcal C}(\tau)/|\tau|^{1/2}=0$ and estimate \eqref{15.62}	
	holds for $\tau \in \R$ and sufficiently small $\varepsilon > 0$.

	\noindent $3^\circ$. Let $s \ge 3/2$. There does not exist a positive function $\mathcal{C}(\tau)$ such that 
	$\lim_{\tau \to \infty} {\mathcal C}(\tau)/|\tau|^{1/2}=0$ and estimate \eqref{s<2_sin_est1}
	holds for $\tau \in \R$ and sufficiently small $\varepsilon > 0$.
\end{theorem}

\subsection{Approximation for the sandwiched operators $\cos(\tau \mathcal{A}_\varepsilon^{1/2})$ and $\mathcal{A}_\varepsilon^{-1/2} \sin(\tau \mathcal{A}_\varepsilon^{1/2})$ in the principal order\label{sec15.5}}
Now we proceed to consideration of the operator $\mathcal{A}_\varepsilon$ (see~(\ref{A_eps})). 
Let $\mathcal{A}^0$~be the operator~(\ref{A0}).  Denote 
\begin{align}
\label{15.70}
	{J}_{1,\eps}(\tau) &:=  
	f^\varepsilon  \cos ( \tau \mathcal{A}_\varepsilon^{1/2}) (f^\varepsilon)^{-1} - 
	f_0  \cos( \tau (\mathcal{A}^0)^{1/2}) f_0^{-1},
	\\
	\label{15.71}
	{J}_{2,\eps}(\tau) &:=  
	f^\varepsilon \mathcal{A}_\varepsilon^{-1/2} \sin ( \tau \mathcal{A}_\varepsilon^{1/2}) (f^\varepsilon)^{-1} - 
	f_0 (\mathcal{A}^0)^{-1/2} \sin( \tau (\mathcal{A}^0)^{1/2}) f_0^{-1},
	\\
	\label{15.72}
	{J}_{3,\eps}(\tau) &:=  
	f^\varepsilon \mathcal{A}_\varepsilon^{-1/2} \sin ( \tau \mathcal{A}_\varepsilon^{1/2}) (f^\varepsilon)^{*} - 
	f_0 (\mathcal{A}^0)^{-1/2} \sin( \tau (\mathcal{A}^0)^{1/2}) f_0.
\end{align}
Relations  \eqref{sin_and_scale_transform} and \eqref{H0_resolv_and_scale_transform}  imply that 
\begin{align}
\label{15.73}
{J}_{1,\eps}(\tau) (\mathcal{H}_0+I)^{-s/2}=  &\,
T_\eps^* J_1(\eps^{-1} \tau) \mathcal{R}(\eps)^{s/2} T_\eps,
\\
\label{15.74}
{J}_{l,\eps}(\tau) (\mathcal{H}_0+I)^{-s/2}= & \,\eps
T_\eps^* J_l(\eps^{-1} \tau) \mathcal{R}(\eps)^{s/2} T_\eps,\quad l=2,3.
\end{align}

Applying Theorems \ref{th14.15} and \ref{th14.16} and taking \eqref{15.73}, \eqref{15.74} into account, we obtain the following two theorems.

\begin{theorem}[see~\cite{BSu5,M,DSu}]
	\label{th15.25}
Let $\mathcal{A}_{\varepsilon}$~be the operator~\emph{(\ref{A_eps})}, and let $\mathcal{A}^0$~be the 
operator~\emph{(\ref{A0})}. Let $J_{1,\eps}(\tau),$ $J_{2,\eps}(\tau),$  and
$J_{3,\eps}(\tau)$ be the operators defined by  \eqref{15.70}--\eqref{15.72}. 
Then for $\tau \in \mathbb{R}$ and $\varepsilon > 0$ we have 
\begin{align}
\label{15.75}
\bigl\| J_{1,\eps}(\tau) \bigr\|_{H^{2}(\mathbb{R}^d) \to L_2(\mathbb{R}^d)} 
\le \mathrm{C}_1 (1+|\tau|) \varepsilon,
\\
\label{15.76}
\bigl\| J_{2,\eps}(\tau) \bigr\|_{H^{1}(\mathbb{R}^d) \to L_2(\mathbb{R}^d)} 
\le \mathrm{C}_2 (1+|\tau|) \varepsilon,
\\
\label{15.77}
\bigl\| J_{3,\eps}(\tau) \bigr\|_{H^{1}(\mathbb{R}^d) \to L_2(\mathbb{R}^d)} 
\le \wt{\mathrm{C}}_2 (1+|\tau|) \varepsilon,
\end{align}
where  ${\mathrm{C}}_1,$ ${\mathrm{C}}_2,$ $\wt{\mathrm{C}}_2$ depend on $\alpha_0,$ $\alpha_1,$ 
$\|g\|_{L_\infty},$ $\|g^{-1}\|_{L_\infty},$ $\|f\|_{L_\infty},$ $\|f^{-1}\|_{L_\infty},$ and $r_0$.
\end{theorem}

\begin{theorem}
	\label{th15.26}
Suppose that the assumptions of Theorem \emph{\ref{th15.25}} are satisfied.
Suppose that  Condition \emph{\ref{cond_BB}} or Condition~\emph{\ref{sndw_cond1}} {\rm (}or more  restrictive Condition~\emph{\ref{sndw_cond2}}{\rm )} is satisfied. Then for $\tau \in \mathbb{R}$ and $\varepsilon > 0$ we have 
\begin{align}
\label{15.78}
\bigl\| J_{1,\eps}(\tau) \bigr\|_{H^{3/2}(\mathbb{R}^d) \to L_2(\mathbb{R}^d)} 
\le \mathrm{C}_3 (1+|\tau|)^{1/2} \varepsilon,
\\
\label{15.79}
\bigl\| J_{3,\eps}(\tau) \bigr\|_{H^{1/2}(\mathbb{R}^d) \to L_2(\mathbb{R}^d)} 
\le {\mathrm{C}}_4 (1+|\tau|)^{1/2} \varepsilon.
\end{align}
	Under Condition  \emph{\ref{cond_BB}}, the constants   ${\mathrm{C}}_3$ and ${\mathrm{C}}_4$ depend on  $\alpha_0,$ $\alpha_1,$ $\|g\|_{L_\infty},$ $\|g^{-1}\|_{L_\infty},$ $\|f\|_{L_\infty},$ $\|f^{-1}\|_{L_\infty},$ and $r_0$.
	Under Condition \emph{\ref{sndw_cond1}}, these constants depend on the same parameters and on 
	$n,$  $c^\circ$.
\end{theorem}

Theorem \ref{th15.25} was known earlier: estimate \eqref{15.75} was obtained in  \cite[Theorem~13.3]{BSu5},
inequality \eqref{15.76} was proved in \cite[Theorem 9.1]{M}, and \eqref{15.77} was proved in 
 \cite[Theorem 11.6]{DSu}.

By interpolation, we  deduce the following corollaries from Theorems  \ref{th15.25} and \ref{th15.26}.

\begin{corollary}
	\label{cor15.28}
	Under the assumptions of Theorem  \emph{\ref{th15.25}}, we have  
	\begin{align}
	\label{15.82}
	&\bigl\| {J}_{1,\eps}(\tau) 
	\bigr\|_{H^s(\mathbb{R}^d) \to L_2(\mathbb{R}^d)} \le {\mathfrak{C}}_1(s) (1+|\tau|)^{s/2} \varepsilon^{s/2},
	\quad 0\le s \le 2,\ \tau \in \R, \ \eps>0;
	\\
	\label{15.83}
	&\bigl\| {J}_{3,\eps}(\tau) 
	\bigr\|_{H^r(\mathbb{R}^d) \to L_2(\mathbb{R}^d)} \le {\mathfrak{C}}_2(r) (1+|\tau|)^{(r+1)/2}
	 \varepsilon^{(r+1)/2},
	 \quad 0 \le r \le 1,\ \tau \in \R,\ 0< \eps \le 1;
	\\
	\label{J_3epsD*_gen_intrpld}
	&\bigl\| {J}_{3,\eps}(\tau) \mathbf{D}^*
	\bigr\|_{H^s(\mathbb{R}^d) \to L_2(\mathbb{R}^d)} \le {\mathfrak{C}}'_2(s) (1+|\tau|)^{s/2} \varepsilon^{s/2},
	\quad 0\le s \le 2,\ \tau \in \R, \ \eps>0.
	\end{align}
\end{corollary}

\begin{proof}
By  \eqref{f_0_estimates}, 
	\begin{equation}
	\label{15.84}
	\bigl\| {J}_{1,\eps}(\tau) \bigr\|_{L_2(\mathbb{R}^d) \to L_2(\mathbb{R}^d)} \le 2 \|f\|_{L_\infty} \|f^{-1}\|_{L_\infty},
	\quad \tau \in \R, \ \eps>0.
\end{equation}
Interpolating between \eqref{15.84} and \eqref{15.75}, we arrive at estimate  \eqref{15.82} with the constant  
${\mathfrak{C}}_1(s) = (2 \|f\|_{L_\infty} \|f^{-1}\|_{L_\infty})^{1-s/2} {\mathrm{C}}_1^{s/2}$.

By \eqref{14.32a} and \eqref{15.74} (with $s=0$), for  $\tau \in \R$ and $0< \eps \le 1$ we have  
	\begin{align}
	\label{15.85}
	\bigl\| {J}_{3,\eps}(\tau) 
	\bigr\|_{L_2(\mathbb{R}^d) \to L_2(\mathbb{R}^d)} \le {\mathrm{C}}'_2 \eps(1+\eps^{-1/2}|\tau|^{1/2})
	\le 2 {\mathrm{C}}'_2 \eps^{1/2}(1+|\tau|)^{1/2}.
	\end{align}
Interpolating between  \eqref{15.85} and \eqref{15.77}, we obtain estimate  \eqref{15.83} with the constant  
${\mathfrak{C}}_2(r) = (2 {\mathrm{C}}'_2)^{1-r} \wt{\mathrm{C}}_2^{r}$.

Next, using the analog of~\eqref{a_form_ineq} for $\mathcal{A}_\varepsilon$, we obtain   
$$
\bigl\| \mathbf{D} f^\varepsilon \mathcal{A}_\varepsilon^{-1/2} \sin(\tau \mathcal{A}_\varepsilon^{1/2}) (f^\varepsilon)^* \bigr\|_{L_2(\mathbb{R}^d) \to L_2(\mathbb{R}^d)} \le \widehat{c}_*^{-1/2} \| f\|_{L_\infty}.
$$
Applying a similar estimate  for the operator  $\D f_0(\mathcal{A}^0)^{-1/2} \sin(\tau (\mathcal{A}^0)^{1/2})f_0$ and 
 passing to the adjoint operators, we get 
\begin{equation}
\label{J_3epsD*_gen_L2_L2}
\bigl\| {J}_{3,\eps}(\tau) \mathbf{D}^*
\bigr\|_{L_2(\mathbb{R}^d) \to L_2(\mathbb{R}^d)} \le 2\widehat{c}_*^{-1/2} \| f \|_{L_\infty}.
\end{equation}
Interpolating between~(\ref{J_3epsD*_gen_L2_L2}) and the estimate 
$\| J_{3,\eps}(\tau) \mathbf{D}^* \|_{H^{2} \to L_2} 
\le \wt{\mathrm{C}}_2 (1+|\tau|) \varepsilon$ (which obviously follows from~\eqref{15.77}), we obtain~(\ref{J_3epsD*_gen_intrpld}) with the constant ${\mathfrak{C}}'_2(s) = (2 \widehat{c}_*^{-1/2} \|f\|_{L_\infty})^{1-s/2} \wt{\mathrm{C}}_2^{s/2}$.
\end{proof}

\begin{remark}
\label{rem15.28a}
Under the assumptions of Theorem \ref{th15.25}, it is possible to obtain the result for the operator 
${J}_{2,\eps}(\tau)$, interpolating between the obvious estimate 
$\| {J}_{2,\eps}(\tau) \|_{L_2 \to L_2} \le 2 |\tau| \|f\|_{L_\infty} \|f^{-1}\|_{L_\infty}$ and \eqref{15.76}. 
This yields 
$$
\| {J}_{2,\eps}(\tau) \|_{H^r(\R^d) \to L_2(\R^d)} \le \widetilde{\mathfrak C}_2(r) (1+ |\tau|) \eps^r, 
\quad 0\le r \le 1, \ \tau \in \R,\ \eps>0.
$$
It is impossible to obtain an analog of estimate \eqref{15.83} for ${J}_{2,\eps}(\tau)$. See Remark~\ref{rem12.1}.
\end{remark}

\begin{corollary}
	\label{cor15.29}
	Under the assumptions of Theorem \emph{\ref{th15.26}}, we have 
	\begin{align}
	\label{15.86}
	&\bigl\| {J}_{1,\eps}(\tau) 
	\bigr\|_{H^s(\mathbb{R}^d) \to L_2(\mathbb{R}^d)} \le {\mathfrak{C}}_3(s) (1+|\tau|)^{s/3} \varepsilon^{2 s/3},
	\quad 0\le s \le 3/2,\ \tau \in \R, \ \eps>0;
	\\
	\label{15.87}
	&\bigl\| {J}_{3,\eps}(\tau) 
	\bigr\|_{H^r(\mathbb{R}^d) \to L_2(\mathbb{R}^d)} \le {\mathfrak{C}}_4(r) (1+|\tau|)^{(r+1)/3}
	 \varepsilon^{2(r+1)/3},
	 \quad 0 \le r \le 1/2,\ \tau \in \R,\ 0< \eps \le 1;
	\\
	\label{J_3epsD*_enchcd_intrpltd}
	&\bigl\| {J}_{3,\eps}(\tau) \mathbf{D}^*
	\bigr\|_{H^s(\mathbb{R}^d) \to L_2(\mathbb{R}^d)} \le {\mathfrak{C}}'_4(s) (1+|\tau|)^{s/3} \varepsilon^{2s/3},
	\quad 0\le s \le 3/2,\ \tau \in \R, \ \eps>0.
	\end{align}
\end{corollary}

\begin{proof}
Interpolating between \eqref{15.84} and \eqref{15.78}, we arrive at estimate \eqref{15.86} with the constant
${\mathfrak{C}}_3(s) = (2 \| f\|_{L_\infty} \| f^{-1}\|_{L_\infty})^{1-2s/3} {\mathrm{C}}_3^{2s/3}$.

By \eqref{14.34a} and \eqref{15.74} (with $s=0$), for $\tau \in \R$ and $0< \eps \le 1$ we have 
	\begin{align}
	\label{15.88}
	\bigl\| {J}_{3,\eps}(\tau) 
	\bigr\|_{L_2(\mathbb{R}^d) \to L_2(\mathbb{R}^d)} \le {\mathrm{C}}'_4 \eps(1+\eps^{-1/3}|\tau|^{1/3})
	\le 2 {\mathrm{C}}'_4 \eps^{2/3}(1+|\tau|)^{1/3}.
	\end{align}
Interpolating between \eqref{15.88} and \eqref{15.79}, we obtain estimate \eqref{15.87} with the constant 
${\mathfrak{C}}_4(r) = (2 {\mathrm{C}}'_4)^{1-2r} {\mathrm{C}}_4^{2r}$.

Finally, interpolating between~(\ref{J_3epsD*_gen_L2_L2})  and the estimate  
$$
\| J_{3,\eps}(\tau) \mathbf{D}^* \|_{H^{3/2} \to L_2} \le {\mathrm{C}}_4 (1+|\tau|)^{1/2} \varepsilon
$$ 
(which obviously follows from~(\ref{15.79})), we obtain~(\ref{J_3epsD*_enchcd_intrpltd}) with the constant ${\mathfrak{C}}'_4(s) = (2 \widehat{c}_*^{-1/2} \|f\|_{L_\infty})^{1-2s/3} {\mathrm{C}}_4^{2s/3}$.
\end{proof}

\begin{remark}
$1^\circ$.  Under the assumptions of Theorem \ref{th15.25}, for $\tau\! =\! O(\eps^{-\alpha})$, ${0\!< \!\alpha\! <\! 1}$, we obtain the qualified estimates 
$$
\begin{aligned}
\bigl\| {J}_{1,\eps}(\tau) 
	\bigr\|_{H^s(\mathbb{R}^d) \to L_2(\mathbb{R}^d)} &= O(\eps^{s(1-\alpha)/2}),& 0 &\le s \le 2;
	\\
\bigl\| {J}_{3,\eps}(\tau) 
	\bigr\|_{H^r(\mathbb{R}^d) \to L_2(\mathbb{R}^d)} &= O(\eps^{(r+1)(1-\alpha)/2}),& 0 &\le r \le 1;
	\\
	\bigl\| {J}_{3,\eps}(\tau) \mathbf{D}^* 
	\bigr\|_{H^s(\mathbb{R}^d) \to L_2(\mathbb{R}^d)} &= O(\eps^{s(1-\alpha)/2}),& 0 &\le s \le 2.	
\end{aligned}
$$
$2^\circ$. Under the assumptions of Theorem \ref{th15.26}, for $\tau = O(\eps^{-\alpha})$, $0< \alpha < 2$,  we obtain the qualified estimates
$$
\begin{aligned}
\bigl\| {J}_{1,\eps}(\tau) 
	\bigr\|_{H^s(\mathbb{R}^d) \to L_2(\mathbb{R}^d)} &= O(\eps^{s(2-\alpha)/3}),& 0 &\le s \le 3/2;
	\\
\bigl\| {J}_{3,\eps}(\tau) 
	\bigr\|_{H^r(\mathbb{R}^d) \to L_2(\mathbb{R}^d)} &= O(\eps^{(r+1)(2-\alpha)/3}),& 0 &\le r \le 1/2;
	\\
	\bigl\| {J}_{3,\eps}(\tau) \mathbf{D}^*
	\bigr\|_{H^s(\mathbb{R}^d) \to L_2(\mathbb{R}^d)} &= O(\eps^{s(2-\alpha)/3}),& 0 &\le s \le 3/2.
\end{aligned}
$$
\end{remark}

\subsection{Approximation for the sandwiched operator $\mathcal{A}_\varepsilon^{-1/2} \sin(\tau \mathcal{A}_\varepsilon^{1/2})$ in the energy norm\label{sec15.6}}

Denote 
\begin{equation}
\label{11.*7}
	{J}_\eps(\tau) :=  
	f^\varepsilon \mathcal{A}_\varepsilon^{-1/2} \sin ( \tau \mathcal{A}_\varepsilon^{1/2}) (f^\varepsilon)^{-1}
	- (I+\varepsilon \Lambda^\varepsilon b(\mathbf{D}) \Pi_\varepsilon) f_0 (\mathcal{A}^0)^{-1/2} \sin( \tau (\mathcal{A}^0)^{1/2}) f_0^{-1}.
\end{equation}
Applying relations of the form~(\ref{sin_and_scale_transform}) for the operators $\mathcal{A}_\varepsilon$, $\mathcal{A}^0$, and also~(\ref{H0_resolv_and_scale_transform}) and~(\ref{mult_op_and_scale_transform}), we obtain 
\begin{align*}
\begin{split}
\widehat{\mathcal{A}}_{\varepsilon}^{1/2} J_\eps(\tau)
(\mathcal{H}_0 + I)^{-s/2} = T_{\varepsilon}^* \widehat{\mathcal{A}}^{1/2} J(\varepsilon^{-1} \tau) \mathcal{R} (\varepsilon)^{s/2} T_{\varepsilon}, \quad \varepsilon>0.
\end{split}
\end{align*}
By analogy with the proof of Theorem~\ref{A_eps_sin_general_thrm},  using this identity, we deduce 
the following result from Theorem~\ref{sndw_A_sin_general_thrm}  (see~\cite[Theorems~9.5, 10.8]{M}).

\begin{theorem}[see~\cite{M}]
	\label{sndw_A_eps_sin_general_thrm}
Let $\mathcal{A}_{\varepsilon}$~be the operator~\emph{(\ref{A_eps})}, and let $\mathcal{A}^0$~be the operator~\emph{(\ref{A0})}. Suppose that the operator $J_\eps(\tau)$ is given by \eqref{11.*7}. Denote
\begin{align*}
I_\eps(\tau) :=  
g^\eps b(\D) f^\varepsilon \mathcal{A}_\varepsilon^{-1/2} \sin ( \tau \mathcal{A}_\varepsilon^{1/2}) (f^\varepsilon)^{-1} 
- \wt{g}^\eps b(\D) \Pi_\varepsilon f_0 (\mathcal{A}^0)^{-1/2} \sin( \tau (\mathcal{A}^0)^{1/2}) f_0^{-1}.
\end{align*}
Then for $\tau \in \mathbb{R}$ and $\varepsilon > 0$ we have 
\begin{align}
\label{15.92}
\bigl\| J_\eps(\tau) \bigr\|_{H^{2}(\mathbb{R}^d) \to H^1(\mathbb{R}^d)} 
\le \mathrm{C}_{7} (1+|\tau|) \varepsilon,
\\
\label{15.93}
\bigl\| I_\eps(\tau) \bigr\|_{H^{2}(\mathbb{R}^d) \to L_2(\mathbb{R}^d)} 
\le \mathrm{C}_{8} (1+|\tau|) \varepsilon.
\end{align}
	The constants  ${\mathrm{C}}_{7}$ and $\mathrm{C}_{8}$ depend on  $\alpha_0,\alpha_1,\|g\|_{L_\infty},\|g^{-1}\|_{L_\infty},\|f\|_{L_\infty},\|f^{-1}\|_{L_\infty},$ $r_0,$ and~$r_1$.
\end{theorem}

With the help of interpolation, we deduce the following result from Theorem  \ref{sndw_A_eps_sin_general_thrm}.

\begin{corollary}
	\label{cor15.34}
	Under the assumptions of Theorem~\emph{\ref{sndw_A_eps_sin_general_thrm}}, for 
	$0 \le s \le 2,$ $\tau \in \R,$ and $\eps >0$ we have 
	\begin{align}
	\label{15.94}
	\| \D {J}_\eps(\tau)
	\|_{H^{s}(\mathbb{R}^d) \to L_2(\mathbb{R}^d)} & \le {\mathfrak{C}}_5(s) (1 + |\tau|)^{s/2} \varepsilon^{s/2},
	\\
	\label{15.95}
	\| {I}_\eps(\tau)
	\|_{H^{s}(\mathbb{R}^d) \to L_2(\mathbb{R}^d)} & \le {\mathfrak{C}}_6(s) (1 + |\tau|)^{s/2} \varepsilon^{s/2}.
	\end{align}
\end{corollary}

\begin{proof}
By \eqref{15.92}, 
\begin{equation}
\label{15.96}
\bigl\| \D J_\eps(\tau) \bigr\|_{H^{2}(\mathbb{R}^d) \to L_2(\mathbb{R}^d)} 
\le \mathrm{C}_{7} (1+|\tau|) \varepsilon, \quad \tau \in \R,\ \eps>0.
\end{equation}

By analogy with \eqref{15.45}--\eqref{15.53}, it is easy to check that  
\begin{equation}
\label{15.97}
\bigl\| \D J_\eps(\tau) \bigr\|_{L_2(\mathbb{R}^d) \to L_2(\mathbb{R}^d)} 
\le \mathrm{C}_{7}', \quad \tau \in \R,\ \eps>0,
\end{equation}
where $\mathrm{C}_{7}'\!\! = \! 2 c_*^{-\!1/2}\!\!+\! (M_1 r_1\! +\! M_2)  \| g^{-\!1}\|^{1/2}_{L_\infty} \| f^{-\!1}\|_{L_\infty}$\!. Interpolating between \eqref{15.97} and~\eqref{15.96}, we arrive at estimate \eqref{15.94} with the constant
${{\mathfrak C}_5(s)\!=\! ({\mathrm C}'_{7})^{1\!-\!s/2} {\mathrm C}_{7}^{s/2}}$.

\smallskip

Let us check \eqref{15.95}. Similarly to \eqref{15.55}--\eqref{15.57}, it is easily seen that 
\begin{equation}
\label{15.98}
\bigl\| I_\eps(\tau) \bigr\|_{L_2(\mathbb{R}^d) \to L_2(\mathbb{R}^d)} 
\le \mathrm{C}_{8}', \quad \tau \in \R,\ \eps>0,
\end{equation}
where $\mathrm{C}_{8}'\! = \! ( \| g\|_{L_\infty}^{1/2}\!\! +\! 2 \| g\|_{L_\infty} \| g^{-\!1}\|^{1/2}_{L_\infty}) 
\| f^{-1}\|_{L_\infty}$. Interpolating between  \eqref{15.98} and~\eqref{15.93}, we arrive at estimate  \eqref{15.95} 
with the constant  
${\mathfrak C}_6(s)\!= \!({\mathrm C}'_{8})^{1\!-\!s/2} {\mathrm C}_{8}^{s/2}$\!\!.
\end{proof}

\begin{remark}
	\label{rem15.36a}
	Taking \eqref{Lambda_eps_Pi_eps_est} into account, we have 
$$
	\|  {J}_\eps(\tau)\|_{L_2(\mathbb{R}^d) \to L_2(\mathbb{R}^d)} \le 
	2 |\tau| \|f\|_{L_\infty} \| f^{-1}\|_{L_\infty} + \eps M_1 \|g^{-1}\|^{1/2}_{L_\infty} \| f^{-1}\|_{L_\infty}.
$$
	Together with  \eqref{15.97}, this implies  
	\begin{equation}
	\label{15.98a}
	\|  {J}_\eps(\tau)\|_{L_2(\mathbb{R}^d) \to H^1(\mathbb{R}^d)} \le {\mathrm{C}}_{7}''(1+|\tau|), \quad \tau \in \R, \ 0< \eps \le 1.
	\end{equation}
	Interpolating between  \eqref{15.98a} and \eqref{15.92}, for $\tau \in \R$ and  $0< \eps \le 1$ we obtain 
	\begin{equation*}
	\|  {J}_\eps(\tau)\|_{H^s(\mathbb{R}^d) \to H^1(\mathbb{R}^d)} \le
	{\mathfrak{C}}_5'(s) (1 + |\tau|) \varepsilon^{s/2},	\quad 0 \le s \le 2.
	\end{equation*}
It is impossible to obtain estimate for  $\|  {J}_\eps(\tau)\|_{H^s \to H^1}$ of the same order as in \eqref{15.94},
because there is no analog of inequality \eqref{15.19} for the operator  
$J_{2,\eps}(\tau)$; cf. Remark \ref{rem15.11a}.
\end{remark}

By analogy with the proof of Theorem~\ref{A_eps_sin_general_thrm}, we deduce the following statement from Theorem \ref{sndw_A_sin_enchncd_thrm_1}.

\begin{theorem}
	\label{th15.35}
Suppose that the assumptions of Theorem~\emph{\ref{sndw_A_eps_sin_general_thrm}} are satisfied.
Suppose that Condition \emph{\ref{cond_BB}} or Condition~\emph{\ref{sndw_cond1}} \emph{(}or more restrictive Condition~\emph{\ref{sndw_cond2}}\emph{)} is satisfied. Then for  $\tau \in \mathbb{R}$ and
 $\varepsilon > 0$ we have 
\begin{align*}
\bigl\| J_\eps(\tau) \bigr\|_{H^{3/2}(\mathbb{R}^d) \to H^1(\mathbb{R}^d)} 
&\le \mathrm{C}_{9} (1+|\tau|)^{1/2} \varepsilon,
\\
\bigl\| I_\eps(\tau) \bigr\|_{H^{3/2}(\mathbb{R}^d) \to L_2(\mathbb{R}^d)} 
&\le \mathrm{C}_{10} (1+|\tau|)^{1/2} \varepsilon.
\end{align*}
	Under Condition \emph{\ref{cond_BB}}, the constants  $\mathrm{C}_{9}$ and $\mathrm{C}_{10}$ depend only on  $\alpha_0,$ $\alpha_1,$ $\|g\|_{L_\infty},$ $\|g^{-1}\|_{L_\infty},$ $\|f\|_{L_\infty},$ $\|f^{-1}\|_{L_\infty},$ $r_0,$ and $r_1$.
 Under Condition~\emph{\ref{sndw_cond1}}, these constants depend on the same parameters and on $n,$ $c^\circ$.
\end{theorem}

By interpolation, we deduce the following corollary from Theorem \ref{th15.35} and estimates \eqref{15.97}, \eqref{15.98}.

\begin{corollary}
	\label{cor15.37}
	Under the assumptions of Theorem~\emph{\ref{th15.35}}, we have
	\begin{align*}
	\| \D {J}_\eps(\tau)
	\|_{H^{s}(\mathbb{R}^d) \to L_2(\mathbb{R}^d)} & \le {\mathfrak{C}}_7(s) (1 + |\tau|)^{s/3} \varepsilon^{2s/3},
	\quad 0 \le s \le 3/2,\ \tau \in \R,\ \eps >0,
	\\
	\| {I}_\eps(\tau)
	\|_{H^{s}(\mathbb{R}^d) \to L_2(\mathbb{R}^d)} & \le {\mathfrak{C}}_{8}(s) (1 + |\tau|)^{s/3} \varepsilon^{2s/3},
	\quad 0 \le s \le 3/2,\ \tau \in \R,\ \eps >0.
	\end{align*}
	\end{corollary}

\begin{remark}
1) Under the assumptions of Theorem \ref{sndw_A_eps_sin_general_thrm}, for 
 $\tau \!= \!O(\eps^{-\alpha})$, ${0\!<\! \alpha\! <\! 1}$, we obtain the qualified estimates
 {\allowdisplaybreaks
\begin{align*}
\bigl\| \D {J}_{\eps}(\tau) 
	\bigr\|_{H^s(\mathbb{R}^d) \to L_2(\mathbb{R}^d)} &= O(\eps^{s(1-\alpha)/2}),\quad 0 \le s \le 2;
	\\
\bigl\| {I}_{\eps}(\tau) 
	\bigr\|_{H^s(\mathbb{R}^d) \to L_2(\mathbb{R}^d)} &= O(\eps^{s(1-\alpha)/2}),\quad 0 \le s \le 2.	
\end{align*}
}
2) Under the assumptions of Theorem \ref{th15.35}, for $\tau = O(\eps^{-\alpha})$,
$0< \alpha < 2$, we obtain the qualified estimates
$$
\begin{aligned}
\bigl\| \D {J}_{\eps}(\tau) 
	\bigr\|_{H^s(\mathbb{R}^d) \to L_2(\mathbb{R}^d)} &= O(\eps^{s(2-\alpha)/3}),\quad 0 \le s \le 3/2;
	\\
\bigl\| {I}_{\eps}(\tau) 
	\bigr\|_{H^s(\mathbb{R}^d) \to L_2(\mathbb{R}^d)} &= O(\eps^{s(2-\alpha)/3}),\quad 0 \le s \le 3/2.	
\end{aligned}
$$
\end{remark}

\subsection{Sharpness of the results of Subsections \ref{sec15.5} and \ref{sec15.6}}

 Applying theorems from Subsection \ref{sec14.6}, we confirm that the results of Subsections
 \ref{sec15.5} and \ref{sec15.6} are sharp. First, we discuss the sharpness of the results  regarding  the type of the operator norm. The following statement confirming the sharpness of Theorems
 \ref{th15.25} and \ref{sndw_A_eps_sin_general_thrm} is deduced from Theorem  \ref{th14.21} by the scaling transformation.
 
\begin{theorem}
	\label{th15.40}
 Suppose that Condition~\emph{\ref{cond13.1}} is satisfied. 
	
	\noindent $1^\circ$. Let $0 \ne \tau \in \mathbb{R}$ and $0 \le s < 2$. Then there does not exist a constant 
	$\mathcal{C}(\tau) > 0$ such that the estimate 
	\begin{equation}
	\label{15.108}
	\left\| {J}_{1,\eps}(\tau)
	\right\|_{H^s(\mathbb{R}^d) \to L_2(\mathbb{R}^d)} \le \mathcal{C}(\tau) \varepsilon
	\end{equation}
	holds for all sufficiently small $\varepsilon > 0$.

	\noindent $2^\circ$. Let $0 \ne \tau \in \mathbb{R}$ and $0 \le r < 1$. 
	Then there does not exist a constant $\mathcal{C}(\tau) > 0$ such that the estimate
	\begin{equation}
	\label{15.109}
	\left\| {J}_{2,\eps}(\tau)
	\right\|_{H^r(\mathbb{R}^d) \to L_2(\mathbb{R}^d)} \le \mathcal{C}(\tau) \varepsilon
	\end{equation}
	holds for all sufficiently small $\varepsilon > 0$.

	\noindent $3^\circ$. Let $0 \ne \tau \in \mathbb{R}$ and  $0 \le r < 1$. Then there does not exist a constant  $\mathcal{C}(\tau) > 0$ such that the estimate
	\begin{equation}
	\label{15.110}
	\left\| {J}_{3,\eps}(\tau)
	\right\|_{H^r(\mathbb{R}^d) \to L_2(\mathbb{R}^d)} \le \mathcal{C}(\tau) \varepsilon
	\end{equation}
holds for all sufficiently small $\varepsilon > 0$.

	\noindent $4^\circ$. Let $0 \ne \tau \in \mathbb{R}$ and $0 \le s < 2$. Then there does not exist a constant  $\mathcal{C}(\tau) > 0$ such that the estimate 
\begin{equation}
	\label{15.111}
	\left\| {J}_{\eps}(\tau)
	\right\|_{H^s(\mathbb{R}^d) \to H^1(\mathbb{R}^d)} \le \mathcal{C}(\tau) \varepsilon
	\end{equation}
	holds for all sufficiently small $\varepsilon > 0$.
	\end{theorem}

Next, Theorem \ref{th14.22} confirms that Theorems \ref{th15.26} and \ref{th15.35} are sharp.

\begin{theorem}
	\label{th15.41}
	Suppose that Condition~\emph{\ref{cond13.2}} is satisfied. 
	
	\noindent $1^\circ$. Let $0 \ne \tau \in \mathbb{R}$ and $0 \le s < 3/2$. Then there does not exist a constant $\mathcal{C}(\tau) > 0$ such that estimate \eqref{15.108} holds for all sufficiently small $\varepsilon$.

	\noindent $2^\circ$. Let $0 \ne \tau \in \mathbb{R}$ and $0 \le r < 1/2$. 
	Then there does not exist a constant  $\mathcal{C}(\tau) > 0$ such that estimate 
	\eqref{15.110} holds for all sufficiently small  $\varepsilon$.

\noindent $3^\circ$. Let $0 \ne \tau \in \mathbb{R}$ and $0 \le s < 3/2$. Then there does not exist a constant  $\mathcal{C}(\tau) > 0$ such that estimate  \eqref{15.111} holds for all sufficiently small $\varepsilon$.
\end{theorem}

Now we discuss the sharpness of the results  regarding the  dependence of estimates on the parameter $\tau$. 
Theorem \ref{th14.21a} implies the following statement demonstrating that Theorems 
\ref{th15.25} and \ref{sndw_A_eps_sin_general_thrm} are sharp.
 
\begin{theorem}
	\label{th15.43}
	Suppose that Condition~\emph{\ref{cond13.1}} is satisfied. 
	
	\noindent $1^\circ$. Let $s \ge 2$. There does not exist a positive function $\mathcal{C}(\tau)$ such that  $\lim_{\tau \to \infty} {\mathcal C}(\tau)/|\tau| =0$ 
	 and estimate \eqref{15.108} holds for  $\tau\in \R$ and sufficiently small $\varepsilon > 0$.
		
	\noindent $2^\circ$. 
	Let $r \ge 1$. There does not exist a positive function $\mathcal{C}(\tau)$ such that 
	$\lim_{\tau \to \infty} {\mathcal C}(\tau)/|\tau| =0$ 
	  and estimate \eqref{15.109} holds for  $\tau\in \R$ and sufficiently small $\varepsilon > 0$.
		
	\noindent $3^\circ$. Let $r \ge 1$. There does not exist a positive function $\mathcal{C}(\tau)$ such that 
	$\lim_{\tau \to \infty} {\mathcal C}(\tau)/|\tau| =0$ 
	  and estimate \eqref{15.110} holds for  $\tau\in \R$ and sufficiently small $\varepsilon > 0$.
		
	\noindent $4^\circ$. Let $s \ge 2$. There does not exist a positive function $\mathcal{C}(\tau)$ such that  $\lim_{\tau \to \infty} {\mathcal C}(\tau)/|\tau| =0$ 
	 and estimate \eqref{15.111} holds for  $\tau\in \R$ and sufficiently small $\varepsilon > 0$.
\end{theorem}

Theorem \ref{th14.22a} demonstrates that Theorems \ref{th15.26} and \ref{th15.35} are sharp.

\begin{theorem}
	\label{th15.44}
	Suppose that Condition~\emph{\ref{cond13.2}} is satisfied. 
	
	\noindent $1^\circ$. Let $s \ge 3/2$. There does not exist a positive function $\mathcal{C}(\tau)$ such that 
	$\lim_{\tau \to \infty} {\mathcal C}(\tau)/|\tau|^{1/2} =0$  and estimate \eqref{15.108} holds for  $\tau\in \R$ and sufficiently small $\varepsilon > 0$.
	
	\noindent $2^\circ$. Let $r \ge 1/2$. There does not exist a positive function $\mathcal{C}(\tau)$ such that 
	$\lim_{\tau \to \infty} {\mathcal C}(\tau)/|\tau|^{1/2} =0$  and estimate \eqref{15.110} holds for  $\tau\in \R$ and sufficiently small $\varepsilon > 0$.
	
	\noindent $3^\circ$. Let $s \ge 3/2$. There does not exist a positive function $\mathcal{C}(\tau)$ such that 
	$\lim_{\tau \to \infty} {\mathcal C}(\tau)/|\tau|^{1/2} =0$ 
	 and estimate \eqref{15.111} holds for  $\tau\in \R$ and sufficiently small $\varepsilon > 0$.
\end{theorem}

\subsection{On the possibility to remove the smoothing operator $\Pi_\eps$ in the corrector}
Now we consider the question about possibility to remove the operator $\Pi_\eps$ from the corrector
in Theorems  \ref{A_eps_sin_general_thrm}, \ref{A_eps_sin_enchcd_thrm_1},  \ref{sndw_A_eps_sin_general_thrm}, \ref{th15.35}. 

Denote 
{\allowdisplaybreaks
\begin{align}
\label{15.119}
	\wh{J}^\circ_\eps(\tau) &:=  
	 \wh{\mathcal{A}}_\varepsilon^{-1/2} \sin ( \tau \wh{\mathcal{A}}_\varepsilon^{1/2})  - 
	(I+\varepsilon \Lambda^\varepsilon b(\mathbf{D})) (\wh{\mathcal{A}}^0)^{-1/2} 
	\sin( \tau (\wh{\mathcal{A}}^0)^{1/2}),
\\
\label{15.120}
	\wh{I}^\circ_\eps(\tau) &:=  
	 g^\eps b(\D) \wh{\mathcal{A}}_\varepsilon^{-1/2} \sin ( \tau \wh{\mathcal{A}}_\varepsilon^{1/2})  - 
	\wt{g}^\eps b(\mathbf{D}) (\wh{\mathcal{A}}^0)^{-1/2} \sin( \tau (\wh{\mathcal{A}}^0)^{1/2}),
\\
\label{15.121}
	{J}^\circ_\eps(\tau)  &:= 
	f^\varepsilon \mathcal{A}_\varepsilon^{-1/2} \sin ( \tau \mathcal{A}_\varepsilon^{1/2}) (f^\varepsilon)^{-1}
	- 	(I+\varepsilon \Lambda^\varepsilon b(\mathbf{D})) f_0 (\mathcal{A}^0)^{-1/2} \sin( \tau (\mathcal{A}^0)^{1/2}) f_0^{-1},
\\	
\label{15.122}
	{I}^\circ_\eps(\tau) &:=  
g^\eps b(\D) f^\varepsilon \mathcal{A}_\varepsilon^{-1/2} \sin ( \tau \mathcal{A}_\varepsilon^{1/2}) (f^\varepsilon)^{-1}- \wt{g}^\eps  b(\mathbf{D}) f_0 (\mathcal{A}^0)^{-1/2} \sin( \tau (\mathcal{A}^0)^{1/2}) f_0^{-1}.
\end{align}
}

 From Theorem \ref{th10.12} we deduce the following result.

  \begin{theorem}\label{th15.48}
  Suppose that Condition \emph{\ref{cond_Lambda_1}} is satisfied.
  
   \noindent $1^\circ$. Under the assumptions of Theorem \emph{\ref{A_eps_sin_general_thrm}}, the operators
   \eqref{15.119} and \eqref{15.120} satisfy the following estimates for $\tau \in \R$ and $0< \eps \le 1$\emph{:}
  \begin{align}
	\label{15.125}
	\bigl\|  \widehat{J}^\circ_\varepsilon( \tau)  \bigr\|_{H^2(\mathbb{R}^d) \to H^1(\mathbb{R}^d)} \le \widehat{\mathrm{C}}^\circ_{7}(1 + |\tau|) \varepsilon,
\\
\label{15.126}
	\bigl\|  \widehat{I}^\circ_\varepsilon( \tau)  \bigr\|_{H^2(\mathbb{R}^d) \to L_2(\mathbb{R}^d)} \le \widehat{\mathrm{C}}^\circ_{8}(1 + |\tau|) \varepsilon.	
	\end{align}
	The constants $\widehat{\mathrm{C}}^\circ_{7}$ and $\widehat{\mathrm{C}}^\circ_{8}$ 
	depend on  $\alpha_0,$ $\alpha_1,$ $\|g\|_{L_\infty},$ $\|g^{-1}\|_{L_\infty},$ $r_0,$ $r_1,$ and also on the norm  $\|  [\Lambda]\|_{H^2 \to H^1}$.
   
   \noindent $2^\circ$. Under the assumptions of Theorem \emph{\ref{sndw_A_eps_sin_general_thrm}}, the operators 
   \eqref{15.121} and  \eqref{15.122} satisfy the following estimates for $\tau \in \R$ and $0< \eps \le 1$\emph{:}   
  \begin{align}
	\label{15.127}
	\bigl\|  {J}^\circ_\varepsilon( \tau)  \bigr\|_{H^2(\mathbb{R}^d) \to H^1(\mathbb{R}^d)} \le {\mathrm{C}}^\circ_{7}(1 + |\tau|) \varepsilon,
\\
\label{15.128}
	\bigl\|  {I}^\circ_\varepsilon( \tau)  \bigr\|_{H^2(\mathbb{R}^d) \to L_2(\mathbb{R}^d)} \le {\mathrm{C}}^\circ_{8}(1 + |\tau|) \varepsilon.	
	\end{align}
	The constants ${\mathrm{C}}^\circ_{7}$ and ${\mathrm{C}}^\circ_{8}$  depend on 
	$\alpha_0,$ $\alpha_1,$ $\|g\|_{L_\infty},$ $\|g^{-1}\|_{L_\infty},$ $\|f\|_{L_\infty},$ $\|f^{-1}\|_{L_\infty},$ $r_0,$ $r_1,$ and also on the norm $\|  [\Lambda]\|_{H^2 \to H^1}$.
     \end{theorem}

\begin{proof}
Let us check statement $1^\circ$. Statement $2^\circ$ is proved similarly.

From \eqref{10.19} and \eqref{15.31a} it follows that
  \begin{equation}
	\label{15.129}
	\bigl\|  \D \widehat{J}^\circ_\varepsilon( \tau)  \bigr\|_{H^2(\mathbb{R}^d) \to L_2(\mathbb{R}^d)} \le 
	\wh{c}_*^{-1/2} \widehat{\mathrm C}^\circ_{5}(1 + |\tau|) \varepsilon,
	\quad \tau \in \R,\ 0< \eps \le 1.
\end{equation}
	
	Now we estimate the norm  $\bigl\|  \widehat{J}^\circ_\varepsilon( \tau)  \bigr\|_{H^2 \to L_2}$. 
	For the operator $\wh{\mathcal A}_\eps^{-1/2} \sin (\tau \wh{\mathcal A}_\eps^{1/2})
	-  (\wh{\mathcal A}^0)^{-1/2} \sin (\tau (\wh{\mathcal A}^0)^{1/2})$, we apply estimate \eqref{15.11}.
	In order to estimate the corrector, we use the scaling transformation:
 {\allowdisplaybreaks
\begin{multline}
	\label{15.131}
	\bigl\|   \eps \Lambda^\eps b(\D)  (\wh{\mathcal A}^0)^{-1/2} \sin (\tau (\wh{\mathcal A}^0)^{1/2})
	\bigr\|_{H^2(\mathbb{R}^d) \to L_2(\mathbb{R}^d)} 
	\\
	 =
	\eps \bigl\|  
	\Lambda^\eps b(\D)  (\wh{\mathcal A}^0)^{-1/2} \sin (\tau (\wh{\mathcal A}^0)^{1/2})
	( {\mathcal H}_0 +I)^{-1}\bigr\|_{L_2(\mathbb{R}^d) \to L_2(\mathbb{R}^d)} 
	\\
	 =\eps \bigl\|  
	\Lambda b(\D)  (\wh{\mathcal A}^0)^{-1/2} \sin ( \eps^{-1}\tau (\wh{\mathcal A}^0)^{1/2}) \mathcal{R}(\eps)
	\bigr\|_{L_2(\mathbb{R}^d) \to L_2(\mathbb{R}^d)} 
	\\
	 \le \eps \| \Lambda \mathcal{R}(\eps)\|_{L_2(\mathbb{R}^d) \to L_2(\mathbb{R}^d)} \bigl\| b(\D)  (\wh{\mathcal A}^0)^{-1/2}\bigr\|_{L_2(\mathbb{R}^d) \to L_2(\mathbb{R}^d)}
	\\
	\le \eps \| [\Lambda]\|_{H^2(\mathbb{R}^d) \to L_2(\mathbb{R}^d)} 
	\|\mathcal{R}(\eps)\|_{L_2(\mathbb{R}^d) \to H^2(\mathbb{R}^d)} 
	\| g^{-1}\|^{1/2}_{L_\infty}.
	\end{multline}
	}
	\hspace{-3mm}
We have taken into account that the operator $\mathcal{R}(\eps)$ commutes with differentiation, and then also with the functions of  $\wh{\mathcal A}^0$. Next,
\begin{equation}
\label{15.132}
\|\mathcal{R}(\eps) \|_{L_2(\mathbb{R}^d) \to H^2(\mathbb{R}^d)}
= \sup_{\bxi\in \R^d}
(1+ |\bxi|^2) \eps^2 (|\bxi|^2 + \eps^2)^{-1} \le 1+\eps^2 \le 2, \quad 0< \eps \le 1.
\end{equation}
As a result, relations  \eqref{15.11}, \eqref{15.131}, and \eqref{15.132} imply that  
$$
\bigl\|  \widehat{J}^\circ_\varepsilon( \tau)  \bigr\|_{H^2 (\R^d)\to L_2(\R^d)}\!\! \le \!
\bigl(  \widehat{\mathrm C}_{2} + 2 \| g^{-1}\|^{1/2}_{L_\infty}  \| [\Lambda]\|_{H^2 \to L_2} \bigr)
(1+|\tau|) \eps,\ \; \tau \!\in\! \R,\ 0\!< \!\eps \!\le \!1.
$$
Combining this with	\eqref{15.129}, we arrive at the required estimate \eqref{15.125}.

Now we check  \eqref{15.126}. From  \eqref{10.19} it follows that  
  \begin{equation}
	\label{15.133}
	\bigl\|  g^\eps b(\D) \widehat{J}^\circ_{\varepsilon}( \tau)  \bigr\|_{H^2(\mathbb{R}^d) \to L_2(\mathbb{R}^d)}\!\! \le\! 
	\| g\|_{L_\infty}^{1/2} \widehat{\mathrm{C}}^\circ_{5}(1\! +\! |\tau|) \varepsilon,
	\quad\! \tau \!\in \!\R,\ 0\!<\! \eps\! \le\! 1.
\end{equation}
By \eqref{g_tilde}, 
  \begin{equation}
	\label{15.134}
	\begin{split}
	 g^\eps b(\D) (I &+ \eps \Lambda^\eps b(\D))(\wh{\mathcal A}^0)^{-1/2} \sin (\tau (\wh{\mathcal A}^0)^{1/2})
	\\
	& = \wt{g}^\eps b(\D)( \wh{\mathcal A}^0)^{-1/2} \sin (\tau (\wh{\mathcal A}^0)^{1/2})
	\\
	&\quad+ \eps g^\eps \sum_{l=1}^d b_l \Lambda^\eps D_l b(\D)
	 ( \wh{\mathcal A}^0)^{-1/2} \sin (\tau (\wh{\mathcal A}^0)^{1/2}).
\end{split}
\end{equation}
Let us estimate the  $(H^2 \!\to\! L_2)$-norm of the second summand. Similarly to \eqref{15.131}, we have 
 \begin{equation}
	\label{15.135}
	\begin{aligned}
	&\eps \bigl\|    \Lambda^\eps D_l b(\D)  (\wh{\mathcal A}^0)^{-1/2} \sin (\tau (\wh{\mathcal A}^0)^{1/2})
	\bigr\|_{H^2(\mathbb{R}^d) \to L_2(\mathbb{R}^d)} 
	\\
	& = \bigl\|  
	\Lambda D_l b(\D)  (\wh{\mathcal A}^0)^{-1/2} \sin ( \eps^{-1}\tau (\wh{\mathcal A}^0)^{1/2}) \mathcal{R}(\eps)
	\bigr\|_{L_2(\mathbb{R}^d) \to L_2(\mathbb{R}^d)} 
	\\
	& \le \| \Lambda D_l \mathcal{R}(\eps)\|_{L_2(\mathbb{R}^d) \to L_2(\mathbb{R}^d)} \bigl\| b(\D)  (\wh{\mathcal A}^0)^{-1/2}\bigr\|_{L_2(\mathbb{R}^d) \to L_2(\mathbb{R}^d)}
	\\
	& \le \| [\Lambda]\|_{H^1(\mathbb{R}^d) \to L_2(\mathbb{R}^d)} 
	\|D_l \mathcal{R}(\eps)\|_{L_2(\mathbb{R}^d) \to H^1(\mathbb{R}^d)} 
	\| g^{-1}\|^{1/2}_{L_\infty}.
	\end{aligned}
\end{equation}
Note that Condition \ref{cond_Lambda_1} ensures that the operator  
$[\Lambda]$ is bounded from $H^1({\mathbb{R}}^d; {\mathbb{C}}^m)$ to $L_2({\mathbb{R}}^d; {\mathbb{C}}^n)$.
The norm $\|[\Lambda]\|_{H^1 \to L_2}$ is controlled in terms of $\|[\Lambda]\|_{H^2 \to H^1}$;
see \cite[Subsection 1.3.2]{MSh}. Obviously, for $0< \eps \le 1$ we have 
\begin{equation}
\label{15.136}
\| D_l \mathcal{R}(\eps)\|_{L_2(\mathbb{R}^d) \to H^1(\mathbb{R}^d)}\!\! = \!\sup_{\bxi\in \R^d}
(1\!+\! |\bxi|^2)^{1/2} |\xi_l| \eps^2 (|\bxi|^2\!\! + \!\eps^2)^{-1} \!\!\le \!\eps\!+\!\eps^2\! \le\! 2\eps.
\end{equation}
 From \eqref{15.135} and \eqref{15.136} it is seen that the $(H^2 \to L_2)$-norm of the second term in \eqref{15.134}
 does not exceed $C \eps$. Together with  \eqref{15.133} this implies   \eqref{15.126}.
\end{proof}

Similarly, from Theorem \ref{th10.13} we deduce the following statement.

  \begin{theorem}
  \label{th15.49}
  Suppose that Condition \emph{\ref{cond_Lambda_2}} is satisfied.
  
  \noindent  $1^\circ$. Under the assumptions of Theorem \emph{\ref{A_eps_sin_enchcd_thrm_1}}, for $\tau \in \R$ 
  and $0< \eps \le 1$ we have
  \begin{align}
	\label{15.135a}
	\bigl\|  \widehat{J}^\circ_\varepsilon(\tau) \bigr\|_{H^{3/2}(\mathbb{R}^d) \to H^1(\mathbb{R}^d)} &\le \widehat{\mathrm{C}}^\circ_{9} (1 + |\tau|)^{1/2} \varepsilon,
	\\
	\label{15.136a}
	\bigl\|  \widehat{I}^\circ_\varepsilon(\tau) \bigr\|_{H^{3/2}(\mathbb{R}^d) \to L_2(\mathbb{R}^d)} &\le \widehat{\mathrm{C}}^\circ_{10} (1 + |\tau|)^{1/2} \varepsilon.
	\end{align}
	Under Condition \emph{\ref{cond_B}}, the constants $\widehat{\mathrm{C}}^\circ_{9}$ and $\widehat{\mathrm{C}}^\circ_{10}$ depend on  $\alpha_0,$ $\alpha_1,$ $\|g\|_{L_\infty},$ $\|g^{-1}\|_{L_\infty},$ $r_0,$ $r_1,$ and also on the norm  $\| [\Lambda]\|_{H^{3/2} \to H^1}$. Under Condition \emph{\ref{cond1}}, these constants depend on the same parameters and on $n,$ $\wh{c}^\circ$.
   
  \noindent $2^\circ$. Under the assumptions of Theorem \emph{\ref{th15.35}}, for $\tau \in \R$ and $0< \eps \le 1$
   we have 
  \begin{align}
	\label{15.137}
	\bigl\|  {J}^\circ_\varepsilon(\tau) \bigr\|_{H^{3/2}(\mathbb{R}^d) \to H^1(\mathbb{R}^d)}
	 &\le {\mathrm{C}}^\circ_{9} (1 + |\tau|)^{1/2} \varepsilon,
	\\
	\nonumber
	\bigl\|  {I}^\circ_\varepsilon(\tau) \bigr\|_{H^{3/2}(\mathbb{R}^d) \to L_2(\mathbb{R}^d)} 
	&\le {\mathrm{C}}^\circ_{10} (1 + |\tau|)^{1/2} \varepsilon.
	\end{align}
	Under Condition \emph{\ref{cond_BB}}, the constants ${\mathrm{C}}^\circ_{9}$ and ${\mathrm{C}}^\circ_{10}$ depend on  $\alpha_0,$ $\alpha_1,$ 
	$\|g\|_{L_\infty},$ $\|g^{-1}\|_{L_\infty},$ $\|f\|_{L_\infty},$ $\|f^{-1}\|_{L_\infty},$ $r_0,$ $r_1,$ and also on the norm $\| [\Lambda]\|_{H^{3/2} \to H^1}$. Under Condition~\emph{\ref{sndw_cond1}}, these constants depend on the same parameters and on $n,$~$c^\circ$.
   \end{theorem}

\begin{proof}
Let us check statement $1^\circ$. Statement $2^\circ$ is proved similarly.

From \eqref{10.21} and \eqref{15.31a} it follows that 
  \begin{equation}
	\label{15.143}
	\bigl\|  \D \widehat{J}^\circ_\varepsilon( \tau)  \bigr\|_{H^{3/2}(\mathbb{R}^d) \to L_2(\mathbb{R}^d)}\! \le 
	\wh{c}_*^{-1/2} \widehat{\mathrm{C}}^\circ_{6}(1 \!+\! |\tau|)^{1/2} \varepsilon,
	\quad \!\tau \!\in\! \R,\ 0\!<\! \eps \!\le\! 1.
\end{equation}
	
	Now we estimate the norm $\bigl\|  \widehat{J}^\circ_\varepsilon( \tau)  \bigr\|_{H^{3/2} \to L_2}$. 
For the operator $\wh{\mathcal A}_\eps^{-1/2} \sin (\tau \wh{\mathcal A}_\eps^{1/2})
	-  (\wh{\mathcal A}^0)^{-1/2} \sin (\tau (\wh{\mathcal A}^0)^{1/2})$, we apply \eqref{15.13}.
	Similarly to \eqref{15.131}, 
	 \begin{equation}
	\label{15.145}
	\begin{aligned}
	&\bigl\|   \eps \Lambda^\eps b(\D)  (\wh{\mathcal A}^0)^{-1/2} \sin (\tau (\wh{\mathcal A}^0)^{1/2})
	\bigr\|_{H^{3/2}(\mathbb{R}^d) \to L_2(\mathbb{R}^d)} 
	\\
	& =\eps \bigl\|  
	\Lambda b(\D)  (\wh{\mathcal A}^0)^{-1/2} \sin ( \eps^{-1}\tau (\wh{\mathcal A}^0)^{1/2}) \mathcal{R}(\eps)^{3/4}
	\bigr\|_{L_2(\mathbb{R}^d) \to L_2(\mathbb{R}^d)} 
	\\
	& \le \eps \| g^{-1}\|^{1/2}_{L_\infty} \| [\Lambda]\|_{H^{3/2}(\mathbb{R}^d) \to L_2(\mathbb{R}^d)} 
	\|\mathcal{R}(\eps)^{3/4}\|_{L_2(\mathbb{R}^d) \to H^{3/2}(\mathbb{R}^d)}. 
	\end{aligned}
\end{equation}
Obviously, for $0< \eps \le 1$ we have 
\begin{equation}
\label{15.146}
\|\mathcal{R}(\eps)^{3/4} \|_{L_2(\mathbb{R}^d) \to H^{3/2}(\mathbb{R}^d)} 
= \sup_{\bxi\in \R^d} (1+ |\bxi|^2)^{3/4} \eps^{3/2} (|\bxi|^2 + \eps^2)^{-3/4}
 \le (1+\eps^2)^{3/4} \le 2^{3/4}.
\end{equation}
As a result, relations  \eqref{15.13}, \eqref{15.145}, and \eqref{15.146} imply that  
$$
\bigl\|  \widehat{J}^\circ_\varepsilon( \tau)  \bigr\|_{H^{3/2} (\R^d)\to L_2(\R^d)} \le 
\bigl(  \widehat{\mathrm{C}}_{4} + 2^{3/4} \| g^{-1}\|^{1/2}_{L_\infty}   \| [\Lambda]\|_{H^{3/2} \to L_2} \bigr)
(1+|\tau|)^{1/2} \eps
$$
for $\tau \in \R$ and $0< \eps \le 1$.
Together with \eqref{15.143}, this yields the required estimate  \eqref{15.135a}.

Now, we check \eqref{15.136a}. From  \eqref{10.21} we deduce 
  \begin{equation}
	\label{15.147}
	\bigl\|  g^\eps b(\D) \widehat{J}^\circ_{\varepsilon}( \tau)  \bigr\|_{H^{3/2}(\mathbb{R}^d) \to L_2(\mathbb{R}^d)} \le 
	\| g\|_{L_\infty}^{1/2} \widehat{\mathrm{C}}^\circ_{6}(1 + |\tau|)^{1/2} \varepsilon,
\quad	\tau \in \R,\ 0< \eps \le 1.
\end{equation}
We use  \eqref{15.134} and estimate the  $(H^{3/2} \to L_2)$-norm of the second term. Similarly to \eqref{15.135}, we have 
 \begin{equation}
	\label{15.148}
	\begin{aligned}
	&\eps \bigl\|    \Lambda^\eps D_l b(\D)  (\wh{\mathcal A}^0)^{-1/2} \sin (\tau (\wh{\mathcal A}^0)^{1/2})
	\bigr\|_{H^{3/2}(\mathbb{R}^d) \to L_2(\mathbb{R}^d)} 
	\\
	& = \bigl\|  
	\Lambda D_l b(\D)  (\wh{\mathcal A}^0)^{-1/2} \sin ( \eps^{-1}\tau (\wh{\mathcal A}^0)^{1/2}) \mathcal{R}(\eps)^{3/4}
	\bigr\|_{L_2(\mathbb{R}^d) \to L_2(\mathbb{R}^d)} 
	\\
	& \le \| g^{-1}\|^{1/2}_{L_\infty} \| [\Lambda]\|_{H^{1/2}(\mathbb{R}^d) \to L_2(\mathbb{R}^d)} 
	\|D_l \mathcal{R}(\eps)^{3/4}\|_{L_2(\mathbb{R}^d) \to H^{1/2}(\mathbb{R}^d)}. 
	\end{aligned}
\end{equation}
Note that Condition \ref{cond_Lambda_2} ensures that the operator $[\Lambda]$ is bounded from $H^{1/2}(\R^d;\AC^n)$ to $L_2(\R^d;\AC^n)$, and the norm $\| [\Lambda] \|_{H^{1/2} \to L_2}$ is controlled by $\| [\Lambda] \|_{H^{3/2} \to H^1}$; see \cite[Subsection 2.2.2]{MSh}. Obviously, for $0< \eps \le 1$ we have 
\begin{equation*}
\| D_l \mathcal{R}(\eps)^{3/4} \|_{L_2(\mathbb{R}^d) \to H^{1/2}(\mathbb{R}^d)} = \sup_{\bxi\in \R^d}
(1+ |\bxi|^2)^{1/4} |\xi_l| \eps^{3/2} (|\bxi|^2 + \eps^2)^{-3/4} \le  2^{1/4} \eps.
\end{equation*}
 Together with  \eqref{15.148} this implies that the  $(H^{3/2} \to L_2)$-norm of the second term in~\eqref{15.134}
does not exceed $C \eps$. Combining this with  \eqref{15.147}, we arrive at  \eqref{15.136a}.
\end{proof}

\subsection{Interpolational results without smoothing}
Interpolational results without smoothing operator differ from the results of Corollaries \ref{A_eps_sin_general_interpltd_thrm}, \ref{A_eps_sin_enchcd_interpltd_thrm_1},
  \ref{cor15.28}, \ref{cor15.29}.
 The reason is that  the operators  $\eps \Lambda^\eps b(\D) (\wh{\mathcal A}^0)^{-1/2} \sin (\tau \wh{\mathcal A}^0)^{1/2})$ and $\eps \Lambda^\eps b(\D) f_0 ({\mathcal A}^0)^{-1/2} \sin (\tau {\mathcal A}^0)^{1/2}) f_0^{-1}$ 
 are not bounded from  $L_2(\R^d; \AC^n)$ to $H^1(\R^d; \AC^n)$. 
 
 We impose an additional condition.

  \begin{condition}
  \label{cond_Lambda_infty}
   Suppose that the   $\Gamma$-periodic solution $\Lambda$ of problem \eqref{equation_for_Lambda}
   is bounded\textup, i.~e.,  $\Lambda \in L_\infty$.
   \end{condition}

We need the following statement; see  \cite[Corollary 2.4]{PSu}.

\begin{proposition} [see~\cite{PSu}]
\label{prop_PSu}
Suppose that Condition  \emph{\ref{cond_Lambda_infty}} is satisfied. Then for any function $u \in H^1(\R^d)$ 
and $\eps>0$ we have
\begin{equation*}
\int\limits_{\R^d} | (\D \Lambda)^\eps(\x)|^2 |u(\x)|^2 \,d\x \le \beta_1 \| u\|^2_{L_2(\R^d)}
+ \beta_2 \eps^2 \| \Lambda\|^2_{L_\infty} \| \D u\|^2_{L_2(\R^d)}.
\end{equation*}
The constants $\beta_1$ and $\beta_2$ depend on  $m,$ $d,$ $\alpha_0,$ $\alpha_1,$ 
$\| g\|_{L_\infty},$ and $\| g^{-1}\|_{L_\infty}$.
\end{proposition}
 
We rely on the following statement.
 
 \begin{proposition}
 \label{prop11.19}
 Suppose that Condition \emph{\ref{cond_Lambda_infty}} is satisfied. 
 Then for  ${0\!<\! \eps \!\le\! 1}$ and $\tau \in \R$ we have 
 \begin{align}
 \label{11.*20}
\bigl\| \D \widehat{J}^\circ_\varepsilon( \tau)  \bigr\|_{H^{1}(\mathbb{R}^d) \to L_2(\mathbb{R}^d)} \le 
\wh{\mathrm{C}}_{11},
\\
\label{11.*21}
\bigl\|  \wh{I}^\circ_\varepsilon( \tau)  \bigr\|_{H^{1}(\mathbb{R}^d) \to L_2(\mathbb{R}^d)} 
\le \wh{\mathrm{C}}_{12}.
\end{align}
 The constants $\wh{\mathrm{C}}_{11}$ and $\wh{\mathrm{C}}_{12}$ depend on  $m,$ $d,$ $\alpha_0,$ $\alpha_1,$ 
$\| g\|_{L_\infty},$  $\| g^{-1}\|_{L_\infty},$ and $\| \Lambda\|_{L_\infty}$.
 \end{proposition}

\begin{proof}
Let us check \eqref{11.*20}.
We estimate the norm of the corrector. By Proposition \ref{prop_PSu}, we have 
\begin{equation*}
\begin{aligned}
\| \D \eps \Lambda^\eps b&(\D) (\wh{\mathcal A}^0)^{-1/2} \sin (\tau \wh{\mathcal A}^0)^{1/2}) 
\|_{H^1(\R^d) \to L_2(\R^d)}
\\
&\le \| (\D\Lambda)^\eps b(\D) (\wh{\mathcal A}^0)^{-1/2} \sin (\tau (\wh{\mathcal A}^0)^{1/2}) \|_{H^1 \to L_2}
\\
&\quad+ \eps \|\Lambda \|_{L_\infty} \|\D b(\D) (\wh{\mathcal A}^0)^{-1/2} \sin (\tau (\wh{\mathcal A}^0)^{1/2}) \|_{H^1\to L_2}
\\
& \le  \sqrt{\beta_1} 
\| b(\D) (\wh{\mathcal A}^0)^{-1/2} \sin (\tau (\wh{\mathcal A}^0)^{1/2}) \|_{H^1\to L_2}
\\
&\quad+
 \bigl(1 + \sqrt{\beta_2} \bigr)\eps \|\Lambda \|_{L_\infty} \|\D b(\D) 
 (\wh{\mathcal A}^0)^{-1/2} \sin (\tau (\wh{\mathcal A}^0)^{1/2}) \|_{H^1\to L_2}
 \\
 &\le \sqrt{\beta_1} \| g^{-1}\|_{L_\infty}^{1/2} +\bigl(1 + \sqrt{\beta_2} \bigr)\eps \|\Lambda \|_{L_\infty} 
 \| g^{-1}\|_{L_\infty}^{1/2}.
\end{aligned}
\end{equation*}
Together with  \eqref{15.45}, this implies \eqref{11.*20}.

Now we check estimate  \eqref{11.*21}. By \eqref{g_tilde} and \eqref{15.120},
$$
\begin{aligned}
\wh{I}^\circ_\varepsilon( \tau) &= g^\eps b(\D) \bigl( \wh{\mathcal A}_\eps^{-1/2} \sin (\tau \wh{\mathcal A}_\eps^{1/2}) - (\wh{\mathcal A}^0)^{-1/2} \sin (\tau (\wh{\mathcal A}^0)^{1/2}) \bigr)
\\
&- g^\eps (b(\D) \Lambda)^\eps b(\D) (\wh{\mathcal A}^0)^{-1/2} \sin (\tau (\wh{\mathcal A}^0)^{1/2}). 
\end{aligned}
$$
Denote the terms on the right by  $\wh{I}^\circ_{1,\varepsilon}( \tau)$ and $\wh{I}^\circ_{2,\varepsilon}( \tau)$. 
Obviously,
$$
\| \wh{I}^\circ_{1,\varepsilon}( \tau)\|_{L_2(\R^d) \to L_2(\R^d)} 
 \le \| g\|_{L_\infty}^{1/2} + \| g\|_{L_\infty} \| g^{-1} \|_{L_\infty}^{1/2}.
$$
Using the relation $b(\D) \Lambda = \sum_{l=1}^d b_l D_l \Lambda$ and Proposition \ref{prop_PSu} 
and taking \eqref{5.7a} into account, we obtain 
{\allowdisplaybreaks
\begin{align*}
&\| \wh{I}^\circ_{2,\varepsilon}( \tau)\|_{H^1(\R^d) \to L_2(\R^d)} 
\\
& \le \| g\|_{L_\infty} (d \alpha_1)^{1/2}
\sqrt{\beta_1} \|  b(\D)(\wh{\mathcal A}^0)^{-1/2} \sin (\tau (\wh{\mathcal A}^0)^{1/2}) \|_{H^1 \to L_2}
\\
&\quad
+  \| g\|_{L_\infty} (d \alpha_1)^{1/2} \sqrt{\beta_2}\eps \| \Lambda \|_{L_\infty} 
\| \D  b(\D)(\wh{\mathcal A}^0)^{-1/2} \sin (\tau (\wh{\mathcal A}^0)^{1/2}) \|_{H^1 \to L_2}
\\
 &\le \| g\|_{L_\infty} (d\alpha_1)^{1/2} \bigl( \sqrt{\beta_1}\| g^{-1}\|^{1/2}_{L_\infty}  
 + \sqrt{\beta_2} \eps \| \Lambda \|_{L_\infty} \| g^{-1}\|^{1/2}_{L_\infty} \bigr).
\end{align*}
}
\hspace{-3mm}
As a result, we arrive at estimate \eqref{11.*21}.
\end{proof}

 According to Remark \ref{rem_Lambda}, Condition \ref{cond_Lambda_infty} ensures that 
 Conditions  \ref{cond_Lambda_1} and \ref{cond_Lambda_2} are satisfied.
Using interpolation, we deduce the following corollary from  Theorems \ref{th15.48}$(1^\circ)$, \ref{th15.49}$(1^\circ)$ and Proposition \ref{prop11.19}.

\begin{corollary}
\label{cor15.53}
 Suppose that Condition  \emph{\ref{cond_Lambda_infty}} is satisfied.
  
  \noindent $1^\circ$. Under the assumptions of Theorem \emph{\ref{A_eps_sin_general_thrm}}, for   
   $\tau \in \R$ and $0< \eps \le 1$ we have  
  \begin{align}
	\label{15.153}
	&\bigl\|  \D \widehat{J}^\circ_\varepsilon( \tau)  \bigr\|_{H^{1+r}(\mathbb{R}^d) \to L_2(\mathbb{R}^d)} \le \widehat{\mathfrak{C}}^\circ_5(r)(1 + |\tau|)^r \varepsilon^r,\quad 0\le r \le 1, 
	\\
	\nonumber
	&\bigl\|   \widehat{I}^\circ_\varepsilon( \tau)  \bigr\|_{H^{1+r}(\mathbb{R}^d) \to L_2(\mathbb{R}^d)} \le \widehat{\mathfrak{C}}^\circ_6 (r)(1 + |\tau|)^r \varepsilon^r,\quad 0\le r \le 1.
	\end{align}
	
  \noindent $2^\circ$. Under the assumptions of Theorem \emph{\ref{A_eps_sin_enchcd_thrm_1}}, for   
   $\tau \in \R$ and $0< \eps \le 1$ we have  
  \begin{align}
	\label{15.155}
	&\bigl\|  \D \widehat{J}^\circ_\varepsilon( \tau)  \bigr\|_{H^{1+r}(\mathbb{R}^d) \to L_2(\mathbb{R}^d)} \le \widehat{\mathfrak{C}}^\circ_7(r)(1 + |\tau|)^r \varepsilon^{2r},\quad 0\le r \le 1/2, 
	\\
	\nonumber
	&\bigl\|   \widehat{I}^\circ_\varepsilon( \tau)  \bigr\|_{H^{1+r}(\mathbb{R}^d) \to L_2(\mathbb{R}^d)} \le \widehat{\mathfrak{C}}^\circ_{8} (r)(1 + |\tau|)^r \varepsilon^{2r},\quad 0\le r \le 1/2.
	\end{align}
 \end{corollary}

\begin{remark}
\label{rem15.56}
Suppose that Condition \ref{cond_Lambda_infty} is satisfied.
 
 \noindent $1^\circ$.  Under the assumptions of Theorem  \ref{A_eps_sin_general_thrm}, from  \eqref{15.11}, \eqref{11.*20}, and the obvious estimate 
 \begin{equation}
 \label{15.158aa}
 \| \eps \Lambda^\eps b(\D) (\wh{\mathcal A}^0)^{-1/2} \sin (\tau (\wh{\mathcal A}^0)^{1/2})
 \|_{L_2(\R^d) \to L_2(\R^d)}  \le \eps \| \Lambda \|_{L_\infty} \| g^{-1}\|_{L_\infty}^{1/2}
\end{equation} 
 it follows that 
 \begin{equation}
 \label{15.158a}
 \bigl\|  \widehat{J}^\circ_\varepsilon( \tau)  \bigr\|_{H^1(\mathbb{R}^d) \to H^1(\mathbb{R}^d)} \le
 \wh{\mathrm{C}}_{13} \bigl(1+ (1+|\tau|) \eps \bigr),
 \quad \tau \in \R,\ 0< \eps \le 1.
\end{equation}
Interpolating between \eqref{15.158a} and \eqref{15.125}, for $\tau \in \R$ and $0< \eps \le 1$ we obtain  
 \begin{equation*}
 \bigl\|  \widehat{J}^\circ_\varepsilon( \tau)  \bigr\|_{H^{1+r}(\mathbb{R}^d) \to H^1(\mathbb{R}^d)} \le
 \wh{\mathfrak{C}}_{9}(r) (1+|\tau|)^{r} \eps^{r} \bigl(1+ (1+|\tau|) \eps \bigr)^{1-r},
 \quad 0 \le r \le 1.
\end{equation*}
For bounded values of $(1+|\tau|)\eps$ the right-hand side does not exceed 
$C (1+|\tau|)^{r} \eps^{r}$, i.~e., has the same order as estimate \eqref{15.153}.

\noindent $2^\circ$. Under the assumptions of Theorem  \ref{A_eps_sin_enchcd_thrm_1},
 from \eqref{15.13}, \eqref{11.*20}, and \eqref{15.158aa} it follows that 
 \begin{equation}
 \label{15.158c}
 \bigl\|  \widehat{J}^\circ_\varepsilon( \tau)  \bigr\|_{H^1(\mathbb{R}^d) \to H^1(\mathbb{R}^d)} \le
 \wh{\mathrm{C}}_{14} \bigl(1+ (1+|\tau|)^{1/2} \eps \bigr),
 \quad \tau \in \R,\ 0< \eps \le 1.
\end{equation}
Interpolating between  \eqref{15.158c} and \eqref{15.135a}, for $\tau \in \R$ and $0< \eps \le 1$ we have  
 \begin{equation}
 \label{15.158d}
 \bigl\|  \widehat{J}^\circ_\varepsilon( \tau)  \bigr\|_{H^{1+r}(\mathbb{R}^d) \to H^1(\mathbb{R}^d)} \le
 \wh{\mathfrak{C}}_{10}(r) (1\!+\!|\tau|)^{r} \eps^{2r} \bigl(1\!+\! (1\!+\!|\tau|)^{1/2} \eps \bigr)^{1\!-2r},
 \quad 0 \le r \le 1/2.
\end{equation}
For bounded values of $(1+|\tau|)^{1/2} \eps$ the right-hand side does not exceed 
\hbox{$C (1+|\tau|)^{r} \eps^{2r}$}, i.~e., has the same order as estimate~\eqref{15.155}.
\end{remark}

It is easy to check the analog of  Proposition   \ref{prop11.19}  for the operators 
${J}^\circ_\varepsilon( \tau)$ and~${I}^\circ_\varepsilon( \tau)$.

 \begin{proposition}
 \label{prop15.56a}
 Suppose that Condition  \emph{\ref{cond_Lambda_infty}} is satisfied. Then for  ${0\!<\! \eps \!\le\! 1}$ and 
 $\tau \in \R$ we have 
 \begin{align}
 \label{15.174a}
\bigl\| \D {J}^\circ_\varepsilon( \tau)  \bigr\|_{H^{1}(\mathbb{R}^d) \to L_2(\mathbb{R}^d)} \le 
{\mathrm{C}}_{11},
\\
\nonumber
\bigl\|  {I}^\circ_\varepsilon( \tau)  \bigr\|_{H^{1}(\mathbb{R}^d) \to L_2(\mathbb{R}^d)} 
\le {\mathrm{C}}_{12}.
\end{align}
 The constants ${\mathrm{C}}_{11}$ and ${\mathrm{C}}_{12}$ depend on 
  $m,$ $d,$ $\alpha_0,$ $\alpha_1,$  $\| g\|_{L_\infty},$ $\| g^{-1}\|_{L_\infty},$ $\| f\|_{L_\infty},$ $\| f^{-1}\|_{L_\infty},$ and also on $\| \Lambda\|_{L_\infty}$.
 \end{proposition}

With the help of interpolation, Theorems \ref{th15.48}$(2^\circ)$, \ref{th15.49}$(2^\circ)$ 
and Proposition~\ref{prop15.56a} imply the following corollary.

\begin{corollary}
\label{cor15.54b}
  Suppose that Condition \emph{\ref{cond_Lambda_infty}} is satisfied.
  
\noindent   $1^\circ$. Under the assumptions of Theorem \emph{\ref{sndw_A_eps_sin_general_thrm}}, for   
   $\tau \in \R$ and $0< \eps \le 1$ we have  
  \begin{align}
	\label{15.176}
	&\bigl\|  \D {J}^\circ_\varepsilon( \tau)  \bigr\|_{H^{1+r}(\mathbb{R}^d) \to L_2(\mathbb{R}^d)} \le {\mathfrak{C}}^\circ_5(r)(1 + |\tau|)^r \varepsilon^r,\quad 0\le r \le 1, 
	\\
	\nonumber
	&\bigl\|   {I}^\circ_\varepsilon( \tau)  \bigr\|_{H^{1+r}(\mathbb{R}^d) \to L_2(\mathbb{R}^d)} \le {\mathfrak{C}}^\circ_6 (r)(1 + |\tau|)^r \varepsilon^r,\quad 0\le r \le 1.
	\end{align}
	
   \noindent $2^\circ$. Under the assumptions of Theorem \emph{\ref{th15.35}}, for  
   $\tau \in \R$ and $0< \eps \le 1$ we have  
  \begin{align*}
	\bigl\|  \D {J}^\circ_\varepsilon( \tau)  \bigr\|_{H^{1+r}(\mathbb{R}^d) \to L_2(\mathbb{R}^d)} &\le {\mathfrak{C}}^\circ_7(r)(1 + |\tau|)^r \varepsilon^{2r},\quad 0\le r \le 1/2, 
	\\
	\bigl\|   {I}^\circ_\varepsilon( \tau)  \bigr\|_{H^{1+r}(\mathbb{R}^d) \to L_2(\mathbb{R}^d)} &\le {\mathfrak{C}}^\circ_{8} (r)(1 + |\tau|)^r \varepsilon^{2r},\quad 0\le r \le 1/2.
	\end{align*}
 \end{corollary}

\begin{remark}
\label{rem15.59}
Suppose that Condition \ref{cond_Lambda_infty} is satisfied.
 
\noindent $1^\circ$.  Under the assumptions of Theorem   \ref{sndw_A_eps_sin_general_thrm},
 from  \eqref{15.76}, \eqref{15.174a}, and the  obvious estimate 
 \begin{equation*}
 \| \eps \Lambda^\eps b(\D) f_0 ({\mathcal A}^0)^{-1/2} 
 \sin (\tau ({\mathcal A}^0)^{1/2}) f_0^{-1} \|_{L_2(\R^d) \to L_2(\R^d)} 
\le \eps \| \Lambda \|_{L_\infty} \| g^{-1}\|_{L_\infty}^{1/2}
 \| f^{-1}\|_{L_\infty}
\end{equation*} 
 it follows that 
 \begin{equation}
 \label{15.183}
 \bigl\|  {J}^\circ_\varepsilon( \tau)  \bigr\|_{H^1(\mathbb{R}^d) \to H^1(\mathbb{R}^d)} \le
 {\mathrm{C}}_{13} \bigl(1+ (1+|\tau|) \eps \bigr),
 \quad \tau \in \R,\ 0< \eps \le 1.
\end{equation}
Interpolating between \eqref{15.183} and \eqref{15.127}, for $\tau \in \R$ and $0< \eps \le 1$ we obtain 
 \begin{equation*}
 \bigl\|  {J}^\circ_\varepsilon( \tau)  \bigr\|_{H^{1+r}(\mathbb{R}^d) \to H^1(\mathbb{R}^d)} \le
 {\mathfrak{C}}_{9}(r) (1+|\tau|)^{r} \eps^{r} \bigl(1+ (1+|\tau|) \eps \bigr)^{1-r},
 \quad 0 \le r \le 1.
\end{equation*}
For bounded values of $(1+|\tau|)\eps$ the right-hand side does not exceed 
$C (1+|\tau|)^{r} \eps^{r}$, i.~e., has the same order as estimate \eqref{15.176}.

\noindent $2^\circ$.  Under the assumptions of Theorem  {\ref{th15.35}}, interpolating between
\eqref{15.183} and \eqref{15.137}, for $\tau \in \R$ and $0< \eps \le 1$ we have 
 \begin{equation*}
 \bigl\|  {J}^\circ_\varepsilon( \tau)  \bigr\|_{H^{1+r}(\mathbb{R}^d) \to H^1(\mathbb{R}^d)} \le
 {\mathfrak{C}}_{10}(r) (1+|\tau|)^{r} \eps^{2r} \bigl(1+ (1+|\tau|) \eps \bigr)^{1-2r},
 \quad 0 \le r \le 1/2.
\end{equation*}
The order of this estimate is worse than the order of  \eqref{15.158d}. The reason is that there is no analog of estimate  \eqref{15.79} for the operator $J_{2,\eps}(\tau)$.
\end{remark}

Some cases where Condition  \ref{cond_Lambda_infty} is a fortiori satisfied were given in  \cite[Lemma 8.7]{BSu4}.

\begin{proposition}
\label{prop11.22}
Suppose that at least one of the following assumptions holds\emph{:}

\noindent $1^\circ$. $d \le 2$\emph{;}

\noindent $2^\circ$.  $\wh{\mathcal A}= \D^* g(\x) \D,$ where the matrix $g(\x)$ has real entries\emph{;}

\noindent $3^\circ$.  $g^0 = \underline{g}$ \emph{(}i.~e.\textup, relations 
\emph{\eqref{g0=underline_g_relat}} are valid\emph{)}.

\noindent Then Condition \emph{\ref{cond_Lambda_infty}} is a fortiori satisfied\textup, and the norm 
$\| \Lambda\|_{L_\infty}$ is controlled in terms of $d,$ $\alpha_0,$ $\alpha_1,$ $\|g\|_{L_\infty},$ 
$\|g^{-1}\|_{L_\infty},$ and the parameters of the lattice~$\Gamma$.
\end{proposition}

\subsection{ Special cases}

Suppose that $g^0 = \overline{g}$, i.~e.,   relations \eqref{g0=overline_g_relat} are satisfied.
Then the $\Gamma$-periodic solution of problem \eqref{equation_for_Lambda} is equal to zero: $\Lambda =0$.
In this case, the corrector is equal to zero and the operator~\eqref{11.7} takes the form  
\hbox{$\wh{J}_\eps(\tau) = \wh{J}_{2,\eps}(\tau)$}.
According to  \eqref{N(theta)} and \eqref{L(theta)}, we also have 
$\widehat{N} (\boldsymbol{\theta})=0$ for any $\boldsymbol{\theta} \in \mathbb{S}^{d-1}$. 
Thus, the assumptions of Corollary  \ref{A_eps_sin_enchcd_interpltd_thrm_1} are satisfied.
We arrive at the following statement.

\begin{proposition}
 Let $g^0 = \overline{g},$ i.~e., relations~\eqref{g0=overline_g_relat} are valid.
 Then for $\tau \in \R$ and $0< \eps \le 1$ we have 
  $$
  \begin{aligned}
 \|  \D \bigl(\wh{\mathcal{A}}_\varepsilon^{-1/2} \sin ( \tau \wh{\mathcal{A}}_\varepsilon^{1/2})  - 
	 (\wh{\mathcal{A}}^0)^{-1/2} \sin( \tau (\wh{\mathcal{A}}^0)^{1/2}) \bigr)\|_{H^s(\R^d) \to L_2(\R^d)}
	 \\
	 \le \wh{\mathfrak{C}}_{7}(s) (1+ |\tau|)^{s/3} \eps^{2s/3},\quad 0\le s \le 3/2.
  \end{aligned}
  $$
\end{proposition}

 Similarly, if $g^0 = \overline{g}$, then
 the operator \eqref{11.*7}  takes the form ${J}_\eps(\tau) = J_{2,\eps}(\tau)$.
 According to  \eqref{Lambda_Q=Lambda+Lambda_Q^0}, we have $\Lambda_Q(\mathbf{x}) = 0$, whence 
$\widehat{N}_Q (\boldsymbol{\theta}) =0$ for any 
$\boldsymbol{\theta} \in \mathbb{S}^{d-1}$; see  \eqref{N_Q(theta)}, \eqref{L_Q(theta)}.
 Thus, the assumptions of  Corollary \ref{cor15.37} are satisfied.
 We obtain the following statement.
 
\begin{proposition}
 Let $g^0 = \overline{g},$ i.~e., relations~\eqref{g0=overline_g_relat} are valid.
 Then for \hbox{$0\le s \le 3/2,$} $\tau \in \R,$ and $0< \eps \le 1$ we have 
  $$
  \begin{aligned}
 \|  \D \bigl(
  f^\varepsilon \mathcal{A}_\varepsilon^{-1/2} \sin ( \tau \mathcal{A}_\varepsilon^{1/2}) (f^\varepsilon)^{-1} - 
	 f_0 (\mathcal{A}^0)^{-1/2} \sin( \tau (\mathcal{A}^0)^{1/2}) f_0^{-1}
 \bigr)\|_{H^s(\R^d) \to L_2(\R^d)}
	\\
	 \le {\mathfrak{C}}_{7}(s) (1+ |\tau|)^{s/3} \eps^{2s/3}.
  \end{aligned}
  $$
\end{proposition}

Now, we consider the case where  $g^0= \underline{g}$, i.~e., relations  
\eqref{g0=underline_g_relat} are satisfied. According to  \cite[Remark 3.5]{BSu3}, in this case we have  
$\wt{g}(\x) = g^0 = \underline{g}$.
Then the operator  \eqref{15.120} obviously satisfies the estimate
\begin{equation}
\label{15.178}
\| \wh{I}^\circ_\eps(\tau) \|_{L_2(\R^d) \to L_2(\R^d)} \le 2 \| g\|_{L_\infty}^{1/2}.
\end{equation}
From Proposition~\ref{N=0_proposit}($3^\circ$) it follows that  $\wh{N}(\boldsymbol{\theta})=0$ for all 
$\boldsymbol{\theta} \in \mathbb{S}^{d-1}.$ Moreover, by Proposition \ref{prop11.22}$(3^\circ)$, 
 Condition \ref{cond_Lambda_infty} is satisfied. 
By Theorem \ref{th15.49}($1^\circ$), estimate  \eqref{15.136a} holds.
Interpolating between  \eqref{15.178} and \eqref{15.136a}, we arrive at the following statement.

\begin{proposition}
 Suppose that $g^0 = \underline{g},$ i.~e., relations~\eqref{g0=underline_g_relat} are satisfied.
 Then for \hbox{$0\le s \le 3/2,$} $\tau \in \R,$ and $0< \eps \le 1$ we have 
  \begin{multline*}
 \|  g^\eps b(\D) \wh{\mathcal{A}}_\varepsilon^{-1/2} \sin ( \tau \wh{\mathcal{A}}_\varepsilon^{1/2}) - 
	 g^0 b(\D) (\wh{\mathcal{A}}^0)^{-1/2} \sin( \tau (\wh{\mathcal{A}}^0)^{1/2}) \|_{H^{s}(\R^d) \to L_2(\R^d)}
	 \\
	 \le \wh{\mathfrak{C}}_{11}(s) (1+ |\tau|)^{s/3} \eps^{2s/3}.
  \end{multline*}
 \end{proposition}

Similarly, for $g^0 = \underline{g}$ the operator  \eqref{15.122} admits the estimate
\begin{equation}
\label{15.178*}
\| {I}^\circ_\eps(\tau) \|_{L_2(\R^d) \to L_2(\R^d)} \le 2 \| g\|_{L_\infty}^{1/2} \| f^{-1} \|_{L_\infty}.
\end{equation}
Note that the operator $\widehat{N}_Q (\boldsymbol{\theta})$ can be nonzero for some 
$\boldsymbol{\theta}$ (there is no analog of Proposition \ref{N=0_proposit}($3^\circ$)). 
Therefore, we apply Theorem~\ref{th15.48}$(2^\circ)$.
Interpolating between  \eqref{15.178*} and \eqref{15.128}, we arrive at the following statement.

\begin{proposition}
 Suppose that $g^0 = \underline{g},$ i.~e., relations~\eqref{g0=underline_g_relat} are satisfied.
 Then for  $0\le s \le 2,$ $\tau \in \R,$ and $0< \eps \le 1$ we have
 \begin{multline*}
 \bigl \| g^\eps b(\D) f^\eps {\mathcal{A}}_\varepsilon^{-1/2} \sin ( \tau {\mathcal{A}}_\varepsilon^{1/2}) (f^\eps)^{-1}  
  \\
  -  g^0 b(\D) f_0 ({\mathcal{A}}^0)^{-1/2} \sin( \tau ({\mathcal{A}}^0)^{1/2}) f_0^{-1} \bigr\|_{H^{s}(\R^d) \to L_2(\R^d)}
	 \\
	\le {\mathfrak{C}}_{11}(s) (1+ |\tau|)^{s/2} \eps^{s/2}.
 \end{multline*}
\end{proposition}

\section{Homogenization of the Cauchy problem \\for a hyperbolic equation}

\subsection{The Cauchy problem with the operator $\widehat{\mathcal{A}}_\varepsilon$}
 Let $\mathbf{u}_\varepsilon (\mathbf{x}, \tau)$~be the solution of the following Cauchy problem:
\begin{equation}
\label{nonhomog_Cauchy_hatA_eps}
\left\{
\begin{aligned}
&\frac{\partial^2 \mathbf{u}_\varepsilon (\mathbf{x}, \tau)}{\partial \tau^2} = - b(\mathbf{D})^* g^\varepsilon (\mathbf{x}) b(\mathbf{D}) \mathbf{u}_\varepsilon (\mathbf{x}, \tau) + \mathbf{F} (\mathbf{x}, \tau)
+\D^* \mathbf{G} (\mathbf{x}, \tau), \\
& \mathbf{u}_\varepsilon (\mathbf{x}, 0) =\boldsymbol{\phi}(\x), \quad \frac{\partial \mathbf{u}_\varepsilon }{\partial \tau} (\mathbf{x}, 0) = \boldsymbol{\psi}(\mathbf{x}) + \D^* \boldsymbol{\rho}(\mathbf{x}),
\end{aligned}
\right.
\end{equation}
where $\boldsymbol{\rho} = \operatorname{col} \{\boldsymbol{\rho}_1,\dots, \boldsymbol{\rho}_d\}$, 
$\mathbf{G} = \operatorname{col} \{ \mathbf{G}_1, \dots, \mathbf{G}_d\}$; 
$\boldsymbol{\phi}, \boldsymbol{\psi}, \boldsymbol{\rho}_j \in L_2 (\mathbb{R}^d; \mathbb{C}^n)$, 
 $\mathbf{F}, \mathbf{G}_j \in L_{1, \mathrm{loc}} (\mathbb{R}; L_2 (\mathbb{R}^d; \mathbb{C}^n) )$~are given functions. The solution of this problem admits the following representation 
\begin{equation}
\label{11.*32}
\begin{split}
\mathbf{u}_\varepsilon (\,\cdot\,, \tau) &= \cos(\tau \widehat{\mathcal{A}}_\varepsilon^{1/2}) \boldsymbol{\phi} +
\widehat{\mathcal{A}}_\varepsilon^{-1/2} \sin(\tau \widehat{\mathcal{A}}_\varepsilon^{1/2}) 
(\boldsymbol{\psi}+ \D^* \boldsymbol{\rho})
\\
 &+ \int\limits_{0}^{\tau} \widehat{\mathcal{A}}_\varepsilon^{-1/2} \sin((\tau - \widetilde{\tau}) \widehat{\mathcal{A}}_\varepsilon^{1/2}) (\mathbf{F} (\,\cdot\,, \widetilde{\tau}) + \D^* {\mathbf G}(\,\cdot\,, \widetilde{\tau})) \, d \widetilde{\tau} .
\end{split}
\end{equation}
Let $\mathbf{u}_0 (\mathbf{x}, \tau)$~be the solution of the \textquotedblleft homogenized\textquotedblright \ problem:
\begin{equation}
\label{nonhomog_Cauchy_hatA_0}
\left\{
\begin{aligned}
&\frac{\partial^2 \mathbf{u}_0 (\mathbf{x}, \tau)}{\partial \tau^2} = - b(\mathbf{D})^* g^0 b(\mathbf{D}) \mathbf{u}_0 (\mathbf{x}, \tau) + \mathbf{F} (\mathbf{x}, \tau) +\D^* \mathbf{G} (\mathbf{x}, \tau), \\
& \mathbf{u}_0 (\mathbf{x}, 0) = \boldsymbol{\phi}(\x), \quad \frac{\partial \mathbf{u}_0 }{\partial \tau} (\mathbf{x}, 0) = \boldsymbol{\psi}(\mathbf{x}) +\D^* \boldsymbol{\rho}(\mathbf{x}).
\end{aligned}
\right.
\end{equation}
Then 
\begin{equation}
\label{11.*33}
\begin{split}
\mathbf{u}_0 (\,\cdot\,, \tau) &= 
\cos(\tau (\widehat{\mathcal{A}}^0)^{1/2}) \boldsymbol{\phi} +
(\widehat{\mathcal{A}}^0)^{-1/2} \sin(\tau (\widehat{\mathcal{A}}^0)^{1/2}) (\boldsymbol{\psi} + \D^* \boldsymbol{\rho})
\\
&+ \int\limits_{0}^{\tau} (\widehat{\mathcal{A}}^0)^{-1/2} \sin((\tau - \widetilde{\tau}) (\widehat{\mathcal{A}}^0)^{1/2})( \mathbf{F} (\,\cdot\,, \widetilde{\tau}) + \D^* {\mathbf G}(\,\cdot\,, \widetilde{\tau})) \, d \widetilde{\tau}.
\end{split}
\end{equation}

\begin{theorem}
	\label{th16.1}
	Let $\mathbf{u}_\varepsilon$~be the solution of problem~\eqref{nonhomog_Cauchy_hatA_eps}, and let  
	$\mathbf{u}_0$~be the solution of the homogenized problem~\eqref{nonhomog_Cauchy_hatA_0}.

	\noindent	$1^\circ$.  If  $\boldsymbol{\phi} \in H^{s}(\mathbb{R}^d; \mathbb{C}^n),$ 
	$\boldsymbol{\psi} \in H^{r}(\mathbb{R}^d; \mathbb{C}^n),$
	$\boldsymbol{\rho} \in H^{s}(\mathbb{R}^d; \mathbb{C}^{dn}),$
	$\mathbf{F} \in L_{1, \mathrm{loc}}(\mathbb{R}; H^{r}(\mathbb{R}^d; \mathbb{C}^n)),$ and
	$\mathbf{G} \in L_{1, \mathrm{loc}}(\mathbb{R}; H^{s}(\mathbb{R}^d; \mathbb{C}^{dn})),$ 
	where $0 \le s \le 2,$ $0\le r \le 1,$
		then for $\tau \in \mathbb{R}$ and $0< \varepsilon \le 1$ we have
		\begin{equation}
		\label{16.5}
		\begin{aligned}
		&\| \mathbf{u}_\varepsilon(\,\cdot\,, \tau) - \mathbf{u}_0 (\,\cdot\,, \tau) \|_{L_2 (\mathbb{R}^d)} \le 
		\widehat{\mathfrak{C}}_1 (s) (1+|\tau|)^{s/2}  \varepsilon^{s/2}  
		\| \boldsymbol{\phi} \|_{H^{s}(\mathbb{R}^d)} 
		\\
		&+ \widehat{\mathfrak{C}}_2 (r) (1+|\tau|)^{(r+1)/2}  \varepsilon^{(r+1)/2} \big( 
		\| \boldsymbol{\psi} \|_{H^{r}(\mathbb{R}^d)} 
		+ \|\mathbf{F} \|_{L_1((0,\tau);H^{r}(\mathbb{R}^d))}  \big)
		\\
		&+ \widehat{\mathfrak{C}}_2'(s) (1+|\tau|)^{s/2}  \varepsilon^{s/2} \big( 
		\| \boldsymbol{\rho} \|_{H^{s}(\mathbb{R}^d)} 
		+ \|\mathbf{G} \|_{L_1((0,\tau);H^{s}(\mathbb{R}^d))}  \big).
		\end{aligned}
		\end{equation}
		
\noindent		$2^\circ$. If $\boldsymbol{\phi}, \boldsymbol{\psi} \in L_2 (\mathbb{R}^d; \mathbb{C}^n),$
$\boldsymbol{\rho} \in L_2 (\mathbb{R}^d; \mathbb{C}^{dn}),$ 
$\mathbf{F} \in L_{1, \mathrm{loc}} (\mathbb{R}; L_2 (\mathbb{R}^d; \mathbb{C}^n) ),$  and
$\mathbf{G} \in L_{1, \mathrm{loc}} (\mathbb{R}; L_2 (\mathbb{R}^d; \mathbb{C}^{dn}) ),$ then for  $\tau \in \R$
we have 
$$
\lim_{\eps \to 0}\| \mathbf{u}_\varepsilon (\,\cdot\,, \tau) - \mathbf{u}_0 (\,\cdot\,, \tau) \|_{L_2(\R^d)} =0.
$$
	\end{theorem}
	
	\begin{proof}
	Estimate \eqref{16.5} directly follows from Corollary \ref{cor15.1a} and representations \eqref{11.*32}, \eqref{11.*33}. 
Statement~$2^\circ$ follows from~$1^\circ$, by the Banach--Steinhaus theorem.
	\end{proof}

Statement $1^\circ$ of Theorem \ref{th16.1} can be improved under some additional assumptions.
 Corollary  \ref{cor15.2a} implies the following result.

\begin{theorem}
	\label{th16.2}
	Suppose that  $\mathbf{u}_\varepsilon$~is the solution of problem~\eqref{nonhomog_Cauchy_hatA_eps} and 
	$\mathbf{u}_0$~is the solution of the homogenized problem~\eqref{nonhomog_Cauchy_hatA_0}.
	Suppose that Condition \emph{\ref{cond_B}} or Condition~\emph{\ref{cond1}} \emph{(}or more restrictive Condition~\emph{\ref{cond2})} is satisfied. 
If  $\boldsymbol{\phi} \in H^{s}(\mathbb{R}^d; \mathbb{C}^n),$
		$\boldsymbol{\psi} \in H^{r}(\mathbb{R}^d; \mathbb{C}^n),$
		$\boldsymbol{\rho} \in H^{s}(\mathbb{R}^d; \mathbb{C}^{dn}),$
	$\mathbf{F} \in L_{1, \mathrm{loc}}(\mathbb{R}; H^{r}(\mathbb{R}^d; \mathbb{C}^n)),$ and
	$\mathbf{G} \in L_{1, \mathrm{loc}}(\mathbb{R}; H^{s}(\mathbb{R}^d; \mathbb{C}^{dn})),$ 
	where $0 \le s \le 3/2,$ $0\le r \le 1/2,$
		then for  $\tau \in \mathbb{R}$ and $0< \varepsilon \le 1$ we have 
		\begin{equation*}
		\begin{aligned}
		&\| \mathbf{u}_\varepsilon(\,\cdot\,, \tau) - \mathbf{u}_0 (\,\cdot\,, \tau) \|_{L_2 (\mathbb{R}^d)} \le 
		\widehat{\mathfrak{C}}_3 (s) (1+|\tau|)^{s/3}  \varepsilon^{2s/3}  
		\| \boldsymbol{\phi} \|_{H^{s}(\mathbb{R}^d)} 
		\\
		&+ \widehat{\mathfrak{C}}_4 (r) (1+|\tau|)^{(r+1)/3}  \varepsilon^{2(r+1)/3} \big( 
		\| \boldsymbol{\psi} \|_{H^{r}(\mathbb{R}^d)} 
		+ \|\mathbf{F} \|_{L_1((0,\tau);H^{r}(\mathbb{R}^d))}  \big)
		\\
		&+ \widehat{\mathfrak{C}}'_4 (s) (1+|\tau|)^{s/3}  \varepsilon^{2s /3} \big( 
		\| \boldsymbol{\rho} \|_{H^{s}(\mathbb{R}^d)} 
		+ \|\mathbf{G} \|_{L_1((0,\tau);H^{s}(\mathbb{R}^d))}  \big).
		\end{aligned}
		\end{equation*}
	\end{theorem}

Now, suppose that  $\boldsymbol{\phi}=0$,  $\boldsymbol{\rho}=0$, and $\mathbf{G}=0$. Denote by  
$\mathbf{v}_\varepsilon$ the first order approximation to the  solution of problem  
\eqref{nonhomog_Cauchy_hatA_eps}:
\begin{equation}
\label{v_eps_first_approx}
\mathbf{v}_\varepsilon(\mathbf{x}, \tau) := \mathbf{u}_0 (\mathbf{x}, \tau) + \varepsilon \Lambda^\varepsilon(\x) b(\mathbf{D}) (\Pi_\varepsilon \mathbf{u}_0 ) (\mathbf{x}, \tau).
\end{equation}
We also introduce  notation for the ``flux'':
\begin{equation}
\label{p_eps}
\mathbf{p}_\varepsilon(\mathbf{x}, \tau) := g^\varepsilon (\x) b(\mathbf{D}) \mathbf{u}_\eps (\mathbf{x}, \tau).
\end{equation}

\begin{theorem}
	\label{th16.4}
	Suppose that $\mathbf{u}_\varepsilon$~is the solution of problem~\emph{(\ref{nonhomog_Cauchy_hatA_eps})}  with $\boldsymbol{\phi}=0,$ $\boldsymbol{\rho}=0,$ and $\mathbf{G}=0$. Let $\mathbf{v}_\varepsilon$ and $\mathbf{p}_\varepsilon$ be defined by  \eqref{v_eps_first_approx} and \eqref{p_eps}. Denote  
	$\q_\eps(\mathbf{x}, \tau) := \wt{g}^\eps (\x)b(\D) (\Pi_\eps \mathbf{u}_0)(\x, \tau)$.

\noindent		$1^\circ$.  If  $\boldsymbol{\psi} \in H^{2}(\mathbb{R}^d; \mathbb{C}^n)$ and $\mathbf{F} \in L_{1, \mathrm{loc}}(\mathbb{R}; H^{2}(\mathbb{R}^d; \mathbb{C}^n)),$  
		then for  $\tau \in \mathbb{R}$ and  $0 < \varepsilon \le 1$ we have 
		{\allowdisplaybreaks
		\begin{align*}
		\| \mathbf{u}_\varepsilon(\,\cdot\,, \tau) - \mathbf{v}_\varepsilon (\,\cdot\,, \tau) \|_{H^1 (\mathbb{R}^d)} \le 
		\widehat{\mathrm{C}}_{7} (1+|\tau|)  \varepsilon \bigl( \| \boldsymbol{\psi} \|_{H^{2}(\mathbb{R}^d)} + \|\mathbf{F} \|_{L_1((0,\tau);H^{2}(\mathbb{R}^d))}  \bigr),
		\\
		\| \mathbf{p}_\varepsilon(\,\cdot\,, \tau) - \q_\eps(\,\cdot\,, \tau) \|_{L_2 (\mathbb{R}^d)} \le 
		\widehat{\mathrm{C}}_{8} (1+|\tau|)  \varepsilon \bigl( \| \boldsymbol{\psi} \|_{H^{2}(\mathbb{R}^d)} + \|\mathbf{F} \|_{L_1((0,\tau);H^{2}(\mathbb{R}^d))}  \bigr).
		\end{align*}
	
	\noindent	$2^\circ$.  If  
		$\boldsymbol{\psi} \in H^{s}(\mathbb{R}^d; \mathbb{C}^n)$ and $\mathbf{F} \in L_{1, \mathrm{loc}}(\mathbb{R}; H^{s}(\mathbb{R}^d; \mathbb{C}^n)),$ $0 \le s \le 2,$ 
		then for $\tau \in \mathbb{R}$ and $0 < \varepsilon \le 1$ we have
		\begin{align*}
		\| \D \mathbf{u}_\varepsilon(\,\cdot\,, \tau) &-\D \mathbf{v}_\varepsilon (\,\cdot\,, \tau) \|_{L_2 (\mathbb{R}^d)} 
		\le 
		\widehat{\mathfrak{C}}_5 (s) (1+|\tau|)^{s/2}  \varepsilon^{s/2} \big( \| \boldsymbol{\psi} \|_{H^{s}(\mathbb{R}^d)} + \|\mathbf{F} \|_{L_1((0,\tau);H^{s}(\mathbb{R}^d))}  \big),
		\\
		\|  \mathbf{p}_\varepsilon(\,\cdot\,, \tau)& - \q_\eps (\,\cdot\,, \tau) 
		 \|_{L_2 (\mathbb{R}^d)} 
		\le
		\widehat{\mathfrak{C}}_6 (s) (1+|\tau|)^{s/2}  \varepsilon^{s/2} \big( \| \boldsymbol{\psi} \|_{H^{s}(\mathbb{R}^d)} + \|\mathbf{F} \|_{L_1((0,\tau);H^{s}(\mathbb{R}^d))}  \big).
		\end{align*}
		}
		
	\noindent	$3^\circ$. If $\boldsymbol{\psi} \in L_2 (\mathbb{R}^d; \mathbb{C}^n)$  and $\mathbf{F} \in L_{1, \mathrm{loc}} (\mathbb{R}; L_2 (\mathbb{R}^d; \mathbb{C}^n) ),$ then for  $\tau \in \R$ we have
	$$
	\lim\limits_{\varepsilon \to 0} \| \mathbf{u}_\varepsilon (\,\cdot\,, \tau) - \mathbf{v}_\varepsilon (\,\cdot\,, \tau) \|_{H^1(\mathbb{R}^d)} = 0
	$$
	and 
	$$
	\lim\limits_{\varepsilon \to 0} \| 
		\mathbf{p}_\varepsilon(\,\cdot\,, \tau) - \q_\eps(\,\cdot\,, \tau) \|_{L_2 (\mathbb{R}^d)} = 0.
		$$
	\end{theorem}

\begin{proof}
Statement $1^\circ$ follows from Theorem \ref{A_eps_sin_general_thrm} and representations 
\eqref{11.*32}, \eqref{11.*33}. 
Similarly, statement $2^\circ$ follows from Corollary~\ref{A_eps_sin_general_interpltd_thrm}.

Taking Remark \ref{rem15.11a} into account, we deduce statement $3^\circ$ from  $1^\circ$ by the Banach--Steinhaus theorem.
\end{proof}

\begin{remark}
\label{rem16.5}
By Remark \ref{rem15.11a}, under the assumptions of Theorem \ref{th16.4}$(2^\circ)$, for 
 $0\le s \le 2$, $\tau \in \R,$ and $0< \eps \le 1$ we have 
	\begin{equation*}
		\begin{split}
		\|  \mathbf{u}_\varepsilon(\,\cdot\,, \tau) - \mathbf{v}_\varepsilon (\,\cdot\,, \tau) \|_{H^1 (\mathbb{R}^d)} &\le 
		\widehat{\mathfrak{C}}_5' (s) (1+|\tau|)^{s/2}  \varepsilon^{s/2}
		\bigl( 1+ (1+|\tau|)^{1/2}  \varepsilon^{1/2}\bigr)^{1-s/2}
		\\
		 &\times \bigl( \| \boldsymbol{\psi} \|_{H^{s}(\mathbb{R}^d)} + 
		 \|\mathbf{F} \|_{L_1((0,\tau);H^{s}(\mathbb{R}^d))}  \bigr).
		\end{split}
		\end{equation*}
		For bounded values of $(1+|\tau|) \varepsilon$ the right-hand side is of order
		$(1+|\tau|)^{s/2}  \varepsilon^{s/2}$.
\end{remark}

Statements $1^\circ$ and $2^\circ$ of Theorem \ref{th16.4} can be improved under some additional assumptions. 
The following result is deduced from Theorem \ref{A_eps_sin_enchcd_thrm_1} and Corollary 
\ref{A_eps_sin_enchcd_interpltd_thrm_1}.

\begin{theorem}
	\label{th16.6}
Suppose that the assumptions of Theorem \emph{\ref{th16.4}} are satisfied.
Suppose that Condition \emph{\ref{cond_B}} or Condition~\emph{\ref{cond1}} \emph{(}or more restrictive Condition~\emph{\ref{cond2})} is satisfied.	
	
\noindent	$1^\circ$.  If  $\boldsymbol{\psi} \in H^{3/2}(\mathbb{R}^d; \mathbb{C}^n)$ and $\mathbf{F} \in L_{1, \mathrm{loc}}(\mathbb{R}; H^{3/2}(\mathbb{R}^d; \mathbb{C}^n)),$  
		then for  $\tau \in \mathbb{R}$ and  $0 < \varepsilon \le 1$ we have
		\begin{align*}
		\| \mathbf{u}_\varepsilon(\cdot, \tau)\! -\! \mathbf{v}_\varepsilon (\cdot, \tau) \|_{H^1 (\mathbb{R}^d)}\!\! \le \!
		\widehat{\mathrm{C}}_{9} (1\!+\!|\tau|)^{1/2}  \varepsilon \big( \| \boldsymbol{\psi} \|_{H^{3/2}(\mathbb{R}^d)} \!\!+\! \|\mathbf{F} \|_{L_1((0,\tau);H^{3/2}(\mathbb{R}^d))}  \big),
		\\
		\| \mathbf{p}_\varepsilon(\cdot, \tau) \!-\! \q_\eps(\cdot, \tau) \|_{L_2 (\mathbb{R}^d)}\!\! \le \!
		\widehat{\mathrm{C}}_{10} (1\!+\!|\tau|)^{1/2}  \varepsilon \big( \| \boldsymbol{\psi} \|_{H^{3/2}(\mathbb{R}^d)}\!\! +\! \|\mathbf{F} \|_{L_1((0,\tau);H^{3/2}(\mathbb{R}^d))}  \big).
		\end{align*}
	
\noindent	$2^\circ$.  If   
		$\boldsymbol{\psi} \in H^{s}(\mathbb{R}^d; \mathbb{C}^n)$ and $\mathbf{F} \in L_{1, \mathrm{loc}}(\mathbb{R}; H^{s}(\mathbb{R}^d; \mathbb{C}^n)),$ $0 \le s \le 3/2,$ 
		then for  $\tau \in \mathbb{R}$ and  $0 < \varepsilon \le 1$  we have 
		{\allowdisplaybreaks
		\begin{align*}
		\| \D \mathbf{u}_\varepsilon(\,\cdot\,, \tau)& -\D \mathbf{v}_\varepsilon (\,\cdot\,, \tau) \|_{L_2 (\mathbb{R}^d)} 
	\le 	\widehat{\mathfrak{C}}_7 (s) (1+|\tau|)^{s/3}  \varepsilon^{2 s/3} \big( \| \boldsymbol{\psi} \|_{H^{s}(\mathbb{R}^d)} + \|\mathbf{F} \|_{L_1((0,\tau);H^{s}(\mathbb{R}^d))}  \big),
		\\
		\|  \mathbf{p}_\varepsilon(\,\cdot\,, \tau) &- \q_\eps (\,\cdot\,, \tau) 
		 \|_{L_2 (\mathbb{R}^d)} 
	\le 	\widehat{\mathfrak{C}}_{8} (s) (1+|\tau|)^{s/3}  \varepsilon^{2s/3} \big( \| \boldsymbol{\psi} \|_{H^{s}(\mathbb{R}^d)} + \|\mathbf{F} \|_{L_1((0,\tau);H^{s}(\mathbb{R}^d))}  \big).
		\end{align*}
		}
	\end{theorem}

\begin{remark}
By Remark \ref{rem15.15a}, under the assumptions of Theorem  \ref{th16.6}$(2^\circ)$,
for \hbox{$0\le s \le 3/2$}, $\tau \in \R,$ and $0< \eps \le 1$ we have 
	\begin{equation*}
		\begin{split}
		\|  \mathbf{u}_\varepsilon(\,\cdot\,, \tau) - \mathbf{v}_\varepsilon (\,\cdot\,, \tau) \|_{H^1 (\mathbb{R}^d)} &\le 
		\widehat{\mathfrak{C}}_7' (s) (1+|\tau|)^{s/3}  \varepsilon^{2 s/3}
		\bigl( 1+ (1+|\tau|)^{1/3}  \varepsilon^{2/3}\bigr)^{1-2s/3}
		\\
		 &\times \bigl( \| \boldsymbol{\psi} \|_{H^{s}(\mathbb{R}^d)} + 
		 \|\mathbf{F} \|_{L_1((0,\tau);H^{s}(\mathbb{R}^d))}  \bigr).
		\end{split}
		\end{equation*}
		For bounded values of  $(1+|\tau|)^{1/2} \varepsilon$ the right-hand side is of order
		$(1+|\tau|)^{s/3}  \varepsilon^{2s/3}$.
\end{remark}

 Now, we discuss the possibility to replace the first order approximation \eqref{v_eps_first_approx}  by 
 \begin{equation}
\label{16.27}
\mathbf{v}^0_\varepsilon(\mathbf{x}, \tau) := \mathbf{u}_0 (\mathbf{x}, \tau) + \varepsilon \Lambda^\varepsilon(\x) b(\mathbf{D}) \mathbf{u}_0 (\mathbf{x}, \tau).
\end{equation}
The following result is deduced from Theorem \ref{th15.48}$(1^\circ)$, Corollary \ref{cor15.53}($1^\circ$) 
and Remark \ref{rem15.56}$(1^\circ)$.

\begin{theorem}
	\label{th16.9}
	Suppose that  $\mathbf{u}_\varepsilon$~is the solution of problem~\emph{(\ref{nonhomog_Cauchy_hatA_eps})} with $\boldsymbol{\phi}=0,$ $\boldsymbol{\rho}=0,$ and $\mathbf{G}=0$. Let $\mathbf{v}^0_\varepsilon$ and $\mathbf{p}_\varepsilon$ be defined by  \eqref{16.27} and \eqref{p_eps}. Denote $\q_\eps^0(\x,\tau):=\wt{g}^\eps (\x) b(\D) \mathbf{u}_0(\x, \tau)$.

	\noindent	$1^\circ$. 
		Suppose that Condition \emph{\ref{cond_Lambda_1}} is satisfied.
		 If  $\boldsymbol{\psi} \in H^{2}(\mathbb{R}^d; \mathbb{C}^n)$ and $\mathbf{F} \in L_{1, \mathrm{loc}}(\mathbb{R}; H^{2}(\mathbb{R}^d; \mathbb{C}^n)),$  
		then for $\tau \in \mathbb{R}$ and  $0 < \varepsilon \le 1$ we have 
				\begin{align*}
		\| \mathbf{u}_\varepsilon(\,\cdot\,, \tau) - \mathbf{v}^0_\varepsilon (\,\cdot\,, \tau) \|_{H^1 (\mathbb{R}^d)} \le 
		\widehat{\mathrm{C}}^\circ_{7} (1+|\tau|)  \varepsilon \bigl( \| \boldsymbol{\psi} \|_{H^{2}(\mathbb{R}^d)} + \|\mathbf{F} \|_{L_1((0,\tau);H^{2}(\mathbb{R}^d))}  \bigr),
		\\
		\| \mathbf{p}_\varepsilon(\,\cdot\,, \tau) - \q^0_\eps (\,\cdot\,, \tau) \|_{L_2 (\mathbb{R}^d)} \le 
		\widehat{\mathrm{C}}^\circ_{8} (1+|\tau|)  \varepsilon \bigl( \| \boldsymbol{\psi} \|_{H^{2}(\mathbb{R}^d)} + \|\mathbf{F} \|_{L_1((0,\tau);H^{2}(\mathbb{R}^d))}  \bigr).
		\end{align*}
	
	\noindent	$2^\circ$.  
		Suppose that Condition \emph{\ref{cond_Lambda_infty}} is satisfied.
		If   $\boldsymbol{\psi} \in H^{1+r}(\mathbb{R}^d; \mathbb{C}^n)$ and $\mathbf{F} \in L_{1, \mathrm{loc}}(\mathbb{R}; H^{1+r}(\mathbb{R}^d; \mathbb{C}^n)),$ $0 \le r \le 1,$ 
		then for $\tau \in \mathbb{R}$ and $0 < \varepsilon \le 1$  we have
		\begin{align*}
		\| \D \mathbf{u}_\varepsilon(\,\cdot\,, \tau) &-\D \mathbf{v}^0_\varepsilon (\,\cdot\,, \tau) \|_{L_2 (\mathbb{R}^d)} 
	\le 	\widehat{\mathfrak{C}}^\circ_5 (r) (1+|\tau|)^{r}  \varepsilon^{r} \big( \| \boldsymbol{\psi} \|_{H^{1+r}(\mathbb{R}^d)} + \|\mathbf{F} \|_{L_1((0,\tau);H^{1+r}(\mathbb{R}^d))}  \big),
		\\
		\|  \mathbf{p}_\varepsilon(\,\cdot\,, \tau)& - \q_\eps^0(\,\cdot\,, \tau) 
		 \|_{L_2 (\mathbb{R}^d)} 
	\le 	\widehat{\mathfrak{C}}_6^\circ (r) (1+|\tau|)^{r}  \varepsilon^{r} \big( \| \boldsymbol{\psi} \|_{H^{1+r}(\mathbb{R}^d)} + \|\mathbf{F} \|_{L_1((0,\tau);H^{1+r}(\mathbb{R}^d))}  \big).
		\end{align*}
		
\noindent		$3^\circ$. 
		Suppose that Condition  \emph{\ref{cond_Lambda_infty}} is satisfied.
		If $\boldsymbol{\psi} \in H^1 (\mathbb{R}^d; \mathbb{C}^n)$ and  $\mathbf{F} \in L_{1, \mathrm{loc}} (\mathbb{R}; H^1 (\mathbb{R}^d; \mathbb{C}^n) ),$ then for  $\tau \in \R$ we have
		$$
		\lim\limits_{\varepsilon \to 0} \| \mathbf{u}_\varepsilon (\,\cdot\,, \tau) - \mathbf{v}^0_\varepsilon (\,\cdot\,, \tau) \|_{H^1(\mathbb{R}^d)} = 0 \ \mbox{ and }\ \lim\limits_{\varepsilon \to 0} \mathbf{p}_\varepsilon(\,\cdot\,, \tau) - \q^0_\eps(\,\cdot\,, \tau) 
		 \|_{L_2 (\mathbb{R}^d)} = 0.
		$$
	\end{theorem}

\begin{remark}
By Remark \ref{rem15.56}$(1^\circ)$, 
under the assumptions of Theorem \ref{th16.9}$(2^\circ)$, for \hbox{$0\le r \le 1$}, $\tau \in \R,$ and $0< \eps \le 1$ we have
	\begin{align*}
		\|& \mathbf{u}_\varepsilon(\,\cdot\,, \tau) - \mathbf{v}^0_\varepsilon (\,\cdot\,, \tau) 
		\|_{H^1 (\mathbb{R}^d)} 
		\\
		&\le 
		\widehat{\mathfrak{C}}_{9} (r) (1+|\tau|)^{r}  \varepsilon^{r} 
		\big( 1+ (1+|\tau|)\eps\big)^{1-r}
		\big( \| \boldsymbol{\psi} \|_{H^{1+r}(\mathbb{R}^d)} + \|\mathbf{F} \|_{L_1((0,\tau);H^{1+r}(\mathbb{R}^d))}  \big).
		\end{align*}
For bounded values of  $(1+|\tau|)\eps$, the right-hand side is of order $ (1+|\tau|)^{r}  \varepsilon^{r}$.
\end{remark}

Statements $1^\circ$ and $2^\circ$ of Theorem \ref{th16.9} can be improved under some additional assumptions.
 Theorem \ref{th15.49}$(1^\circ)$ and Corollary \ref{cor15.53}($2^\circ$) imply the following result.

\begin{theorem}
  \label{th16.10}
  Suppose that $\mathbf{u}_\varepsilon$~is the solution of problem~\emph{(\ref{nonhomog_Cauchy_hatA_eps})} with
	$\boldsymbol{\phi}=0,$ $\boldsymbol{\rho}=0,$ and $\mathbf{G}=0$. Let $\mathbf{v}^0_\varepsilon$ and $\mathbf{p}_\varepsilon$ be defined by  \eqref{16.27} and \eqref{p_eps}\textup, and let $\q_\eps^0(\x,\tau):=\wt{g}^\eps (\x) b(\D) \mathbf{u}_0(\x, \tau)$. Suppose that Condition \emph{\ref{cond_B}} or Condition~\emph{\ref{cond1}} \emph{(}or more restrictive Condition~\emph{\ref{cond2})} is satisfied.

\noindent		$1^\circ$. 
		Suppose that Condition \emph{\ref{cond_Lambda_2}} is satisfied.
		 If  $\boldsymbol{\psi} \in H^{3/2}(\mathbb{R}^d; \mathbb{C}^n)$ and $\mathbf{F} \in L_{1, \mathrm{loc}}(\mathbb{R}; H^{3/2}(\mathbb{R}^d; \mathbb{C}^n)),$  
		then for $\tau \in \mathbb{R}$ and  $0 < \varepsilon \le 1$ we have 
		\begin{align*}
		\| \mathbf{u}_\varepsilon(\,\cdot\,, \tau)\! -\! \mathbf{v}^0_\varepsilon (\,\cdot\,, \tau) \|_{H^1 (\mathbb{R}^d)} \le 
		\widehat{\mathrm{C}}^\circ_{9} (1\!+\!|\tau|)^{1/2}  \varepsilon \big( \| \boldsymbol{\psi} \|_{H^{3/2}(\mathbb{R}^d)}\!\! +\! \|\mathbf{F} \|_{L_1((0,\tau);H^{3/2}(\mathbb{R}^d))}  \big),
		\\
		\| \mathbf{p}_\varepsilon(\,\cdot\,, \tau)\! - \!\q^0_\eps(\,\cdot\,, \tau) \|_{L_2 (\mathbb{R}^d)} \le 
		\widehat{\mathrm{C}}^\circ_{10} (1\!+\!|\tau|)^{1/2}  \varepsilon \big( \| \boldsymbol{\psi} \|_{H^{3/2}(\mathbb{R}^d)}\!\! +\! \|\mathbf{F} \|_{L_1((0,\tau);H^{3/2}(\mathbb{R}^d))}  \big).
		\end{align*}
	
	\noindent	$2^\circ$.  
		Suppose that Condition \emph{\ref{cond_Lambda_infty}} is satisfied.
	If  $\boldsymbol{\psi} \in H^{1+r}(\mathbb{R}^d; \mathbb{C}^n)$ and $\mathbf{F} \in L_{1, \mathrm{loc}}(\mathbb{R}; H^{1+r}(\mathbb{R}^d; \mathbb{C}^n)),$ $0 \le r \le 1/2,$ 
		then for  $\tau \in \mathbb{R}$ and  $0 < \varepsilon \le 1$  we have
		\begin{align*}
		\|  \D \mathbf{u}_\varepsilon(\,\cdot\,, \tau) &-\D \mathbf{v}^0_\varepsilon (\,\cdot\,, \tau) 
		\|_{L_2 (\mathbb{R}^d)}
		\le  \widehat{\mathfrak{C}}^\circ_7 (r) (1+|\tau|)^{r}  \varepsilon^{2r} \bigl( \| \boldsymbol{\psi} \|_{H^{1+r}(\mathbb{R}^d)} + \|\mathbf{F} \|_{L_1((0,\tau);H^{1+r}(\mathbb{R}^d))}  \bigr),
		\\
		\|  \mathbf{p}_\varepsilon(\,\cdot\,, \tau) &- \q^0_\eps (\,\cdot\,, \tau) 
		 \|_{L_2 (\mathbb{R}^d)} 
	\le  \widehat{\mathfrak{C}}_{8}^\circ (r) (1+|\tau|)^{r}  \varepsilon^{2r} \bigl( \| \boldsymbol{\psi} \|_{H^{1+r}(\mathbb{R}^d)} + \|\mathbf{F} \|_{L_1((0,\tau);H^{1+r}(\mathbb{R}^d))}  \bigr).
		\end{align*}
  \end{theorem}

\begin{remark}
\label{rem16.12}
By Remark \ref{rem15.56}$(2^\circ)$, under the assumptions of Theorem \ref{th16.10}$(2^\circ)$, for \hbox{$0\le r \le 1/2$}, $\tau \in \R,$ and $0< \eps \le 1$ we have
		\begin{equation*}
		\begin{aligned}
		\| \mathbf{u}_\varepsilon(\,\cdot\,, \tau) - \mathbf{v}^0_\varepsilon (\,\cdot\,, \tau) 
		\|_{H^1 (\mathbb{R}^d)} &\le 
		\widehat{\mathfrak{C}}_{10} (r) (1+|\tau|)^{r}  \varepsilon^{2r} 
		\bigl( 1+ (1+|\tau|)^{1/2}\eps\bigr)^{1-2r}
		\\
		&\times \bigl( \| \boldsymbol{\psi} \|_{H^{1+r}(\mathbb{R}^d)} + \|\mathbf{F} \|_{L_1((0,\tau);H^{1+r}(\mathbb{R}^d))}  \bigr).
		\end{aligned}
		\end{equation*}
For bounded values of $(1+|\tau|)^{1/2}\eps$  the right-hand side is of order $ (1+|\tau|)^{r}  \varepsilon^{2r}$.
\end{remark}

\subsection{The Cauchy problem with the operator $\mathcal{A}_\varepsilon$}
Various statements of the Cauchy problem are possible. 
We consider a single statement of the problem:
\begin{equation}
\label{nonhomog_Cauchy_wQ}
\begin{cases}
 Q^\varepsilon (\mathbf{x}) \frac{\partial^2 \mathbf{u}_\varepsilon (\mathbf{x}, \tau)}{\partial \tau^2} = - b(\mathbf{D})^* g^\varepsilon (\mathbf{x}) b(\mathbf{D}) \mathbf{u}_\varepsilon (\mathbf{x}, \tau) 
\\ \qquad\qquad\qquad\qquad
+  Q^\eps(\x)\mathbf{F}_1 (\mathbf{x}, \tau) + \mathbf{F}_2 (\mathbf{x}, \tau) + \D^* \mathbf{G} (\mathbf{x}, \tau), 
\\
 \mathbf{u}_\varepsilon (\mathbf{x}, 0)\! = \!\boldsymbol{\phi}(\mathbf{x}), \quad
 \!\!\frac{\partial \mathbf{u}_\varepsilon }{\partial \tau} (\mathbf{x}, 0) \!= \!\boldsymbol{\psi}_1(\mathbf{x})
 \!+ \!(Q^\eps(\x))^{-1} (\boldsymbol{\psi}_2(\mathbf{x}) \!+\! \D^* \boldsymbol{\rho}(\mathbf{x})).
\end{cases}
\end{equation}
Here  
$$
\boldsymbol{\rho}\! = \!\operatorname{col} \{ \boldsymbol{\rho}_1,\dots, \boldsymbol{\rho}_d\},
\quad 
\mathbf{G} \!= \!\operatorname{col} \{ \mathbf{G}_1,\dots, \mathbf{G}_d \},
$$ 
${\boldsymbol{\phi}, \boldsymbol{\psi}_1, \boldsymbol{\psi}_2, \boldsymbol{\rho}_j \!\in\! L_2 (\mathbb{R}^d; \mathbb{C}^n), \, \mathbf{F}_1, \mathbf{F}_2, \mathbf{G}_j \in L_{1, \mathrm{loc}} (\mathbb{R}; L_2 (\mathbb{R}^d; \mathbb{C}^n) )}$ are given functions, 
$Q(\mathbf{x})$~is a $\Gamma$-periodic Hermitian $(n \times n)$-matrix-valued function such that 
\hbox{$Q(\mathbf{x}) > 0$} and $Q, Q^{-1} \in L_\infty$.
We factorize the matrix $Q(\x)^{-1}$:
$
Q(\x)^{-1} = f(\x) f(\x)^*.
$
Without loss of generality, assume that the $(n\times n)$-matrix-valued function  $f(\x)$ is periodic. 
Automatically, we have  $f, f^{-1} \in L_\infty$.
Let $\mathcal{A}_\varepsilon$~be the operator~(\ref{A_eps}). 

 By substitution $\mathbf{z}_\varepsilon (\,\cdot\,, \tau) := (f^\varepsilon)^{-1} \mathbf{u}_\varepsilon (\,\cdot\,, \tau)$,    problem~(\ref{nonhomog_Cauchy_wQ}) can be rewritten as follows:
\begin{equation*}
\left\{
\begin{aligned}
\frac{\partial^2 \mathbf{z}_\varepsilon (\mathbf{x}, \tau)}{\partial \tau^2} =& - 
({\mathcal A}_\eps
\mathbf{z}_\varepsilon) (\mathbf{x}, \tau) 
+  (f^\varepsilon (\mathbf{x}))^{-1} \mathbf{F}_1 (\mathbf{x}, \tau) +
(f^\varepsilon (\mathbf{x}))^*( \mathbf{F}_2 (\mathbf{x}, \tau) + \D^* \mathbf{G} (\mathbf{x}, \tau)), 
\\
 \mathbf{z}_\varepsilon (\mathbf{x}, 0) = & (f^\eps(\x))^{-1}\boldsymbol{\phi}(\mathbf{x}), 
\\
\frac{\partial \mathbf{z}_\varepsilon }{\partial \tau} (\mathbf{x}, 0) = &
(f^\varepsilon (\mathbf{x}))^{-1} \boldsymbol{\psi}_1(\mathbf{x})
+ (f^\varepsilon (\mathbf{x}))^{*} (\boldsymbol{\psi}_2(\mathbf{x})+ \D^* \boldsymbol{\rho}(\mathbf{x})).
\end{aligned}
\right.
\end{equation*}
Writing down representation for the solution $\mathbf{z}_\varepsilon$ of this problem, we arrive at the following representation for  $\mathbf{u}_\varepsilon = f^\eps \mathbf{z}_\varepsilon$:
\begin{multline}
\label{16.47}
\mathbf{u}_\varepsilon (\,\cdot\,, \tau) 
=
f^\eps \cos(\tau \mathcal{A}_\varepsilon^{1/2}) (f^\varepsilon)^{-1} \boldsymbol{\phi} +
f^\eps \mathcal{A}_\varepsilon^{-1/2} \sin(\tau \mathcal{A}_\varepsilon^{1/2}) (f^\varepsilon)^{-1} \boldsymbol{\psi}_1
 \\
+ f^\eps \mathcal{A}_\varepsilon^{-1/2} \sin(\tau \mathcal{A}_\varepsilon^{1/2}) (f^\varepsilon)^* (\boldsymbol{\psi}_2
 +\D^* \boldsymbol{\rho})
 \\
+  \int\limits_{0}^{\tau} f^\eps \mathcal{A}_\varepsilon^{-1/2} \sin((\tau - \widetilde{\tau}) \mathcal{A}_\varepsilon^{1/2}) (f^\varepsilon)^{-1} \mathbf{F}_1 (\,\cdot\,, \widetilde{\tau}) \, d \widetilde{\tau}
  \\
+\int\limits_{0}^{\tau} f^\eps \mathcal{A}_\varepsilon^{-1/2} \sin((\tau - \widetilde{\tau}) \mathcal{A}_\varepsilon^{1/2}) (f^\varepsilon)^{*} (\mathbf{F}_2 (\,\cdot\,, \widetilde{\tau}) + \D^* {\mathbf G} (\,\cdot\,, \widetilde{\tau}) )\, d \widetilde{\tau}.
\end{multline}

Let $\mathbf{u}_0 (\mathbf{x}, \tau)$~be the solution of the ``homogenized\/'' problem
\begin{equation}
\begin{cases}
 \overline{Q} \frac{\partial^2 \mathbf{u}_0 (\mathbf{x}, \tau)}{\partial \tau^2}
 = - b(\mathbf{D})^* g^0 b(\mathbf{D}) \mathbf{u}_0 (\mathbf{x}, \tau) 
+ \overline{Q}\mathbf{F}_1 (\mathbf{x}, \tau) + \mathbf{F}_2 (\mathbf{x}, \tau)+\D^* \mathbf{G} (\mathbf{x}, \tau), 
\\
 \mathbf{u}_0 (\mathbf{x}, 0)\! = \!\boldsymbol{\phi}(\mathbf{x}), 
\quad \frac{\partial \mathbf{u}_0 }{\partial \tau} (\mathbf{x}, 0)\! =\! \boldsymbol{\psi}_1(\mathbf{x}) \!+\!
  (\overline{Q})^{-1} (\boldsymbol{\psi}_2(\mathbf{x})\! + \! \D^* \boldsymbol{\rho}(\mathbf{x})),\label{nonhomog_Cauchy_wQ_eff}
\end{cases}
\end{equation}
where $\overline{Q}$~is the mean value of the matrix $Q(\mathbf{x})$ over $\Omega$. 
Putting $f_0 = (\overline{Q})^{-1/2}$ and substituting $\mathbf{z}_0 (\,\cdot\,, \tau) :=
f_0^{-1} \mathbf{u}_0 (\,\cdot\,, \tau)$,  we obtain the representation 
{\allowdisplaybreaks
\begin{multline}
\label{16.49}
 \mathbf{u}_0 (\,\cdot\,, \tau)
=  f_0  \cos(\tau (\mathcal{A}^0)^{1/2}) f_0^{-1} \boldsymbol{\phi} +  
 f_0  (\mathcal{A}^0)^{-1/2} \sin(\tau (\mathcal{A}^0)^{1/2}) f_0^{-1} \boldsymbol{\psi}_1
 \\
 + f_0  (\mathcal{A}^0)^{-1/2} \sin(\tau (\mathcal{A}^0)^{1/2}) f_0 (\boldsymbol{\psi}_2
 + \D^* \boldsymbol{\rho})
 \\
 + \int\limits_{0}^{\tau} f_0  (\mathcal{A}^0)^{-1/2} \sin((\tau- \widetilde{\tau}) (\mathcal{A}^0)^{1/2}) f_0^{-1} 
\mathbf{F}_1 (\,\cdot\,, \widetilde{\tau}) \, d \widetilde{\tau}
\\
 +   \int\limits_{0}^{\tau} f_0 (\mathcal{A}^0)^{-1/2} \sin((\tau - \widetilde{\tau}) (\mathcal{A}^0)^{1/2}) f_0 
 (\mathbf{F}_2 (\,\cdot\,, \widetilde{\tau})  + \D^* \mathbf{G}(\,\cdot\,,\widetilde{\tau} )) \, d \widetilde{\tau}.
\end{multline}
}

Applying Theorem \ref{th15.25}, Corollary \ref{cor15.28}, Remark \ref{rem15.28a},
and using representations \eqref{16.47}, \eqref{16.49}, we arrive at the following result.

\begin{theorem}
  \label{th16.13}
 Suppose that $\mathbf{u}_\varepsilon$~is the solution of problem~\eqref{nonhomog_Cauchy_wQ} 
  and $\mathbf{u}_0$ is the solution of the homogenized problem \eqref{nonhomog_Cauchy_wQ_eff}.
  
	\noindent	$1^\circ$. 
		 If  $\boldsymbol{\rho}=0,$ ${\mathbf G} =0,$ 
		 $\boldsymbol{\phi} \in H^{2}(\mathbb{R}^d; \mathbb{C}^n),$
		 $\boldsymbol{\psi}_1,  \boldsymbol{\psi}_2 \in H^{1}(\mathbb{R}^d; \mathbb{C}^n),$ and 
		 $\mathbf{F}_1, \mathbf{F}_2 \in L_{1, \mathrm{loc}}(\mathbb{R}; H^{1}(\mathbb{R}^d; \mathbb{C}^n)),$  
		then for  $\tau \in \mathbb{R}$ and  $\varepsilon >0$ we have 
		\begin{multline*}
		\| \mathbf{u}_\varepsilon(\,\cdot\,, \tau) - \mathbf{u}_0 (\,\cdot\,, \tau) \|_{L_2 (\mathbb{R}^d)} \le 
		{\mathrm{C}}_1 (1+|\tau|)  \varepsilon  \| \boldsymbol{\phi} \|_{H^{2}(\mathbb{R}^d)}
		\\
		+ {\mathrm{C}}_2(1+|\tau|)  \varepsilon
		\bigl( \| \boldsymbol{\psi}_1 \|_{H^{1}(\mathbb{R}^d)} + \|\mathbf{F}_1 \|_{L_1((0,\tau);H^{1}(\mathbb{R}^d))}  \bigr) 
		\\
		+ \wt{\mathrm{C}}_2(1+|\tau|)  \varepsilon
		\bigl( \| \boldsymbol{\psi}_2 \|_{H^{1}(\mathbb{R}^d)} + \|\mathbf{F}_2 \|_{L_1((0,\tau);H^{1}(\mathbb{R}^d))}  \bigr). 
		\end{multline*}

\noindent		$2^\circ$. 
		 If   $\boldsymbol{\phi} \in H^{s}(\mathbb{R}^d; \mathbb{C}^n),$
		 $\boldsymbol{\psi}_1,  \boldsymbol{\psi}_2 \in H^{r}(\mathbb{R}^d; \mathbb{C}^n),$
		  $\boldsymbol{\rho} \in H^{s}(\mathbb{R}^d; \mathbb{C}^{dn}),$  
		 $\mathbf{F}_1, \mathbf{F}_2 \in L_{1, \mathrm{loc}}(\mathbb{R}; H^{r}(\mathbb{R}^d; \mathbb{C}^n)),$  
		  $\mathbf{G} \in L_{1, \mathrm{loc}}(\mathbb{R}; H^{s}(\mathbb{R}^d; \mathbb{C}^n)),$
		where $0\le s\le 2,$ $0 \le r \le 1,$ then for  $\tau \in \mathbb{R}$ and  $0< \varepsilon \le 1$ we have 
		\begin{multline*}
		\| \mathbf{u}_\varepsilon(\,\cdot\,, \tau) - \mathbf{u}_0 (\,\cdot\,, \tau) \|_{L_2 (\mathbb{R}^d)} \le 
		{\mathfrak{C}}_1(s) (1+|\tau|)^{s/2}  \varepsilon^{s/2}  \| \boldsymbol{\phi} \|_{H^{s}(\mathbb{R}^d)}
		\\
		+ {\mathfrak{C}}_2(r) (1+|\tau|)^{(r+1)/2}  \varepsilon^{(r+1)/2}
		\bigl( \| \boldsymbol{\psi}_2 \|_{H^{r}(\mathbb{R}^d)} + \|\mathbf{F}_2 \|_{L_1((0,\tau);H^{r}(\mathbb{R}^d))}  \bigr) 
		\\
		+ {\mathfrak{C}}'_2(s) (1+|\tau|)^{s/2}  \varepsilon^{s/2}
		\bigl( \| \boldsymbol{\rho} \|_{H^s(\mathbb{R}^d)} + \|\mathbf{G} \|_{L_1((0,\tau);H^s(\mathbb{R}^d))}  \bigr) 
		\\
		+ \wt{\mathfrak{C}}_2(r)(1+|\tau|)  \varepsilon^r
		\bigl( \| \boldsymbol{\psi}_1 \|_{H^{r}(\mathbb{R}^d)} + \|\mathbf{F}_1 \|_{L_1((0,\tau);H^{r}(\mathbb{R}^d))}  \bigr). 
		\end{multline*}

\noindent	$3^\circ$.
	If  
	$
	\boldsymbol{\phi}, \boldsymbol{\psi}_1,  \boldsymbol{\psi}_2 \in L_2(\mathbb{R}^d; \mathbb{C}^n),\quad \boldsymbol{\rho} \in L_2(\mathbb{R}^d; \mathbb{C}^{dn}),$ 	 $\mathbf{F}_1, \mathbf{F}_2 \in L_{1, \mathrm{loc}}(\mathbb{R}; L_2(\mathbb{R}^d; \mathbb{C}^n)),$ and $\mathbf{G} \in L_{1, \mathrm{loc}}(\mathbb{R}; L_2(\mathbb{R}^d; \mathbb{C}^{dn})),$
	then  
		 $$
		 \lim_{\eps \to 0}  \| \mathbf{u}_\varepsilon(\,\cdot\,, \tau) - \mathbf{u}_0 (\,\cdot\,, \tau) \|_{L_2 (\mathbb{R}^d)}
		 =0, \quad \tau \in \R.
		 $$
		 \end{theorem}

In the case where $\boldsymbol{\psi}_1 =0$ and $\mathbf{F}_1 =0$, it is possible to improve statements~$1^\circ$
and~$2^\circ$ of Theorem \ref{th16.13} under some additional assumptions. 
Corollary~\ref{cor15.29} leads to the following result.

\begin{theorem}
  \label{th16.14}
  Suppose that $\mathbf{u}_\varepsilon$~is the solution of problem~\eqref{nonhomog_Cauchy_wQ} 
  and $\mathbf{u}_0$ is the solution of the homogenized problem \eqref{nonhomog_Cauchy_wQ_eff} with  
  $\boldsymbol{\psi}_1 =0$ and $\mathbf{F}_1 =0$.
   Suppose that Condition \emph{\ref{cond_BB}} or Condition~\emph{\ref{sndw_cond1}} {\rm (}or more restrictive Condition~\emph{\ref{sndw_cond2}}{\rm )} is satisfied.
		 If  $\boldsymbol{\phi} \in H^{s}(\mathbb{R}^d; \mathbb{C}^n),$
		 $ \boldsymbol{\psi}_2 \in H^{r}(\mathbb{R}^d; \mathbb{C}^n),$
		 $\boldsymbol{\rho} \in H^{s}(\mathbb{R}^d; \mathbb{C}^n),$
		 $\mathbf{F}_2 \in L_{1, \mathrm{loc}}(\mathbb{R}; H^{r}(\mathbb{R}^d; \mathbb{C}^n)),$  and
		 $\mathbf{G} \in L_{1, \mathrm{loc}}(\mathbb{R}; H^{s}(\mathbb{R}^d; \mathbb{C}^n)),$  
		where $0\le s\le 3/2,$ $0 \le r \le 1/2,$ then for $\tau \in \mathbb{R}$ and  $0< \varepsilon \le 1$ we have
		{\allowdisplaybreaks\begin{multline*}
		\| \mathbf{u}_\varepsilon(\,\cdot\,, \tau) - \mathbf{u}_0 (\,\cdot\,, \tau)  \|_{L_2 (\mathbb{R}^d)} \le 
		{\mathfrak{C}}_3(s) (1+|\tau|)^{s/3}  \varepsilon^{2s/3}  \| \boldsymbol{\phi} \|_{H^{s}(\mathbb{R}^d)}
		\\
		+ {\mathfrak{C}}_4(r) (1+|\tau|)^{(r+1)/3}  \varepsilon^{2(r+1)/3}
		\bigl( \| \boldsymbol{\psi}_2 \|_{H^{r}(\mathbb{R}^d)} + \|\mathbf{F}_2 \|_{L_1((0,\tau);H^{r}(\mathbb{R}^d))}  \bigr)
		\\
		+ {\mathfrak{C}}'_4(s) (1+|\tau|)^{s/3}  \varepsilon^{2s/3}
		\bigl( \| \boldsymbol{\rho} \|_{H^{s}(\mathbb{R}^d)} + \|\mathbf{G} \|_{L_1((0,\tau);H^{s}(\mathbb{R}^d))}  \bigr). 
		\end{multline*}
		}
		 \end{theorem}

Now, we assume that  $ \boldsymbol{\phi}=0$,  $\boldsymbol{\psi}_2=0$, $\boldsymbol{\rho}=0$,
 $\mathbf{F}_2=0$, and ${\mathbf G}=0$.
In this case it is possible to approximate the solution of problem~\eqref{nonhomog_Cauchy_wQ} in the energy norm.
Applying Theorem   \ref{sndw_A_eps_sin_general_thrm}, Corollary \ref{cor15.34}, and Remark \ref{rem15.36a}, we arrive at the following result.

\begin{theorem}
\label{th16.16}
  Suppose that $\mathbf{u}_\varepsilon$~is the solution of problem~\eqref{nonhomog_Cauchy_wQ} 
  with $ \boldsymbol{\phi}=0,$  $\boldsymbol{\psi}_2=0,$ $\boldsymbol{\rho}=0,$  $\mathbf{F}_2=0,$ and
  ${\mathbf G} =0$. Let $\mathbf{u}_0$ be the solution of the homogenized problem  \eqref{nonhomog_Cauchy_wQ_eff}.
  We put 
  $\mathbf{v}_\varepsilon:= \mathbf{u}_0 + \varepsilon \Lambda^\varepsilon  b(\mathbf{D}) \Pi_\eps \mathbf{u}_0,$
$\mathbf{p}_\varepsilon:= g^\eps b(\D) \mathbf{u}_\varepsilon,$ and
$\q_\eps:= \wt{g}^\eps b(\D)\Pi_\eps \mathbf{u}_0$.
  
 \noindent $1^\circ$.
   If  		 $\boldsymbol{\psi}_1 \in H^{2}(\mathbb{R}^d; \mathbb{C}^n)$ and
		 $\mathbf{F}_1 \in L_{1, \mathrm{loc}}(\mathbb{R}; H^{2}(\mathbb{R}^d; \mathbb{C}^n)),$  
		then for $\tau \in \mathbb{R}$ and  $\varepsilon >0$ we have  
\begin{align*}
\| \mathbf{u}_\varepsilon(\,\cdot\,, \tau) - \mathbf{v}_\eps (\,\cdot\,, \tau) \|_{H^1 (\mathbb{R}^d)}\le
{\mathrm{C}}_{7}(1+|\tau|)  \varepsilon
\bigl( \| \boldsymbol{\psi}_1 \|_{H^2(\mathbb{R}^d)} + \|\mathbf{F}_1 \|_{L_1((0,\tau);H^{2}(\mathbb{R}^d))}  \bigr),
\\
		\| \mathbf{p}_\varepsilon(\,\cdot\,, \tau) - \q_\eps (\,\cdot\,, \tau) \|_{L_2 (\mathbb{R}^d)} \le {\mathrm{C}}_{8}(1+|\tau|)  \varepsilon
\bigl( \| \boldsymbol{\psi}_1 \|_{H^2(\mathbb{R}^d)} + \|\mathbf{F}_1 \|_{L_1((0,\tau);H^{2}(\mathbb{R}^d))}  \bigr).
\end{align*}

\noindent		$2^\circ$.
		 If  
		 $\boldsymbol{\psi}_1 \in H^{s}(\mathbb{R}^d; \mathbb{C}^n)$ and  
		 $\mathbf{F}_1 \in L_{1, \mathrm{loc}}(\mathbb{R}; H^{s}(\mathbb{R}^d; \mathbb{C}^n)),$  
		where $0\le s\le 2,$  then for $\tau \in \mathbb{R}$ and  $\varepsilon >0$ we have 
\begin{align*}
\|  \D \mathbf{u}_\varepsilon(\,\cdot\,, \tau)& -\D \mathbf{v}_\eps (\,\cdot\,, \tau) \|_{L_2(\mathbb{R}^d)}
 \le {\mathfrak{C}}_{5}(s) (1+|\tau|)^{s/2}  \varepsilon^{s/2}
\bigl( \| \boldsymbol{\psi}_1 \|_{H^s(\mathbb{R}^d)} + \|\mathbf{F}_1 \|_{L_1((0,\tau);H^{s}(\mathbb{R}^d))}  \bigr),
\\
	\| \mathbf{p}_\varepsilon(\,\cdot\,, \tau)& - \mathbf{q}_\eps (\,\cdot\,, \tau) \|_{L_2 (\mathbb{R}^d)} 
\le {\mathfrak{C}}_{6}(s)(1+|\tau|)^{s/2}  \varepsilon^{s/2}
\bigl( \| \boldsymbol{\psi}_1 \|_{H^s(\mathbb{R}^d)} + \|\mathbf{F}_1 \|_{L_1((0,\tau);H^{s}(\mathbb{R}^d))}  \bigr).
\end{align*}
		
\noindent	$3^\circ$.
	 If 
	 	 $\boldsymbol{\psi}_1 \in L_2(\mathbb{R}^d; \mathbb{C}^n)$ and
		 $\mathbf{F}_1 \in L_{1, \mathrm{loc}}(\mathbb{R}; L_2(\mathbb{R}^d; \mathbb{C}^n)),$  
		 then for $\tau \in \R$ we have 
		 $$
		\lim_{\eps \to 0}  \| \mathbf{u}_\varepsilon(\,\cdot\,, \tau) - \mathbf{v}_\eps (\,\cdot\,, \tau) 
		 \|_{H^1 (\mathbb{R}^d)} =0\quad\mbox{and}\quad \lim_{\eps \to 0} 
		 \| \mathbf{p}_\varepsilon(\,\cdot\,, \tau) -  \mathbf{q}_\eps (\,\cdot\,, \tau)
		 \|_{L_2 (\mathbb{R}^d)} =0.
				$$
\end{theorem}
 
 \begin{remark}
 By Remark  \ref{rem15.36a}, under the assumptions of Theorem \ref{th16.16}$(2^\circ)$, for
  $0\le s \le 2$, $\tau \in \R$, and $0< \eps \le 1$ we have
\begin{equation*}
\| \mathbf{u}_\varepsilon(\,\cdot\,, \tau) -\mathbf{v}_\eps (\,\cdot\,, \tau) \|_{H^1 (\mathbb{R}^d)}\!\le\!
{\mathfrak{C}}'_{5}(s) (1\!+\!|\tau|)  \varepsilon^{s/2}
\big( \| \boldsymbol{\psi}_1 \|_{H^s(\mathbb{R}^d)}\! + \!\|\mathbf{F}_1 \|_{L_1((0,\tau);H^{s}(\mathbb{R}^d))}  \big).
\end{equation*}
 \end{remark}
 
 Statements $1^\circ$ and $2^\circ$ of Theorem \ref{th16.16} can be improved under some additional assumptions.
 Theorem \ref{th15.35} and Corollary \ref{cor15.37} imply the following result.

\begin{theorem}
\label{th16.17}
  Suppose that the assumptions of Theorem \emph{\ref{th16.16}} are satisfied.
  Suppose that Condition \emph{\ref{cond_BB}} or Condition~\emph{\ref{sndw_cond1}} \emph{(}or more restrictive Condition~\emph{\ref{sndw_cond2}}\emph{)} is satisfied.

  \noindent $1^\circ$.
   If   $\boldsymbol{\psi}_1 \in H^{3/2}(\mathbb{R}^d; \mathbb{C}^n)$ and
		 $\mathbf{F}_1 \in L_{1, \mathrm{loc}}(\mathbb{R}; H^{3/2}(\mathbb{R}^d; \mathbb{C}^n)),$  
		then for $\tau \in \mathbb{R}$ and   $\varepsilon >0$ we have  
\begin{align*}
&\!\| \mathbf{u}_\varepsilon(\,\cdot\,, \tau) \!-\! \mathbf{v}_\eps (\,\cdot\,, \tau) \|_{H^1(\mathbb{R}^d)} \le
{\mathrm{C}}_{9}(1\!+\!|\tau|)^{1/2}  \varepsilon
\big( \| \boldsymbol{\psi}_1 \|_{H^{3/2}(\mathbb{R}^d)} \!\!+\! \|\mathbf{F}_1 \|_{L_1((0,\tau);H^{3/2}(\mathbb{R}^d))}  \big),
\\
	&	\!\| \mathbf{p}_\varepsilon(\,\cdot\,, \tau)\! - \!\q_\eps (\,\cdot\,, \tau) \|_{L_2 (\mathbb{R}^d)} \le {\mathrm{C}}_{10}(1\!\!+\!|\tau|)^{1/2}  \varepsilon
\big( \| \boldsymbol{\psi}_1 \|_{H^{3/2}(\mathbb{R}^d)} \!\!+\! \!\|\mathbf{F}_1 \|_{L_1((0,\tau);H^{3/2}(\mathbb{R}^d))}  \big).
\end{align*}

	\noindent	$2^\circ$.
		 If   $\boldsymbol{\psi}_1 \in H^{s}(\mathbb{R}^d; \mathbb{C}^n)$ and
		 $\mathbf{F}_1 \in L_{1, \mathrm{loc}}(\mathbb{R}; H^{s}(\mathbb{R}^d; \mathbb{C}^n)),$  
		where $0\le s\le 3/2,$  then for $\tau \in \mathbb{R}$ and  $0< \varepsilon \le 1$ we have 
\begin{align*}
\| \D \mathbf{u}_\varepsilon(\,\cdot\,, \tau) &-\D \mathbf{v}_\eps (\,\cdot\,, \tau) \|_{L_2 (\mathbb{R}^d)}
\le {\mathfrak{C}}_{7}(s) (1+|\tau|)^{s/3}  \varepsilon^{2s/3}
\bigl( \| \boldsymbol{\psi}_1 \|_{H^s(\mathbb{R}^d)} + \|\mathbf{F}_1 \|_{L_1((0,\tau);H^{s}(\mathbb{R}^d))}  \bigr),
\\
		\| \mathbf{p}_\varepsilon(\,\cdot\,, \tau) &-  \mathbf{q}_\eps (\,\cdot\,, \tau) \|_{L_2 (\mathbb{R}^d)} 
	\le {\mathfrak{C}}_{8}(s) (1+|\tau|)^{s/3}  \varepsilon^{2s/3}
\bigl( \| \boldsymbol{\psi}_1 \|_{H^s(\mathbb{R}^d)} + \|\mathbf{F}_1 \|_{L_1((0,\tau);H^{s}(\mathbb{R}^d))}  \bigr).
\end{align*}
	\end{theorem}
 
 Now, we discuss the possibility to remove the smoothing operator from the corrector. 
  Theorem \ref{th15.48}$(2^\circ)$, Corollary \ref{cor15.54b}$(1^\circ)$, and Remark \ref{rem15.59}$(1^\circ)$ imply the following result.
  
 \begin{theorem}
 \label{th16.19}
 Suppose that  
 $\mathbf{u}_\varepsilon$~is the solution of problem~\eqref{nonhomog_Cauchy_wQ} 
  with $ \boldsymbol{\phi}=0,$  $\boldsymbol{\psi}_2=0,$ $\boldsymbol{\rho} =0,$  $\mathbf{F}_2=0,$
  and $\mathbf{G}=0$. Let
  $\mathbf{u}_0$ be the solution of the homogenized problem \eqref{nonhomog_Cauchy_wQ_eff}.
  We put 
$\mathbf{v}^0_\varepsilon:= \mathbf{u}_0 + \varepsilon \Lambda^\varepsilon  b(\mathbf{D}) \mathbf{u}_0,$
$\mathbf{p}_\varepsilon:= g^\eps b(\D) \mathbf{u}_\varepsilon,$ $\q_\eps^0:=\wt{g}^\eps b(\D) \mathbf{u}_0$.

	\noindent	$1^\circ$. 
		Suppose that Condition \emph{\ref{cond_Lambda_1}} is satisfied.
		If  $\boldsymbol{\psi}_1 \in H^{2}(\mathbb{R}^d; \mathbb{C}^n)$ and $\mathbf{F}_1 \in L_{1, \mathrm{loc}}(\mathbb{R}; H^{2}(\mathbb{R}^d; \mathbb{C}^n)),$  
		then for  $\tau \in \mathbb{R}$ and  $0 < \varepsilon \le 1$ we have 
		\begin{align*}
		\| \mathbf{u}_\varepsilon(\,\cdot\,, \tau) - \mathbf{v}^0_\varepsilon (\,\cdot\,, \tau) \|_{H^1 (\mathbb{R}^d)} \le 
		{\mathrm{C}}^\circ_{7} (1+|\tau|)  \varepsilon \bigl( \| \boldsymbol{\psi}_1 \|_{H^{2}(\mathbb{R}^d)} + \|\mathbf{F}_1 \|_{L_1((0,\tau);H^{2}(\mathbb{R}^d))}  \bigr),
		\\
		\| \mathbf{p}_\varepsilon(\,\cdot\,, \tau) - \q_\eps^0 (\,\cdot\,, \tau) \|_{L_2 (\mathbb{R}^d)} \le 
		{\mathrm{C}}^\circ_{8} (1+|\tau|)  \varepsilon \bigl( \| \boldsymbol{\psi} _1\|_{H^{2}(\mathbb{R}^d)} + \|\mathbf{F}_1 \|_{L_1((0,\tau);H^{2}(\mathbb{R}^d))}  \bigr).
		\end{align*}
	
	\noindent	$2^\circ$.  
		Suppose that Condition \emph{\ref{cond_Lambda_infty}} is satisfied.
		If  
		$\boldsymbol{\psi}_1 \in H^{1+r}(\mathbb{R}^d; \mathbb{C}^n)$ and $\mathbf{F}_1 \in L_{1, \mathrm{loc}}(\mathbb{R}; H^{1+r}(\mathbb{R}^d; \mathbb{C}^n)),$ $0 \le r \le 1,$ then for $\tau \in \mathbb{R}$ and $0 < \varepsilon \le 1$ we have 
		\begin{align*}
		\| \D \mathbf{u}_\varepsilon(\,\cdot\,, \tau)& -\D \mathbf{v}^0_\varepsilon (\,\cdot\,, \tau) \|_{L_2 (\mathbb{R}^d)}
	\le 
		{\mathfrak{C}}^\circ_5 (r) (1+|\tau|)^{r}  \varepsilon^{r} \bigl( \| \boldsymbol{\psi}_1 \|_{H^{1+r}(\mathbb{R}^d)} + \|\mathbf{F}_1 \|_{L_1((0,\tau);H^{1+r}(\mathbb{R}^d))}  \bigr),
		\\
		\|  \mathbf{p}_\varepsilon(\,\cdot\,, \tau)& - \q^0_\eps(\,\cdot\,, \tau) \|_{L_2 (\mathbb{R}^d)} 
	\le 
		{\mathfrak{C}}_6^\circ (r) (1+|\tau|)^{r}  \varepsilon^{r} \bigl( \| \boldsymbol{\psi}_1 \|_{H^{1+r}(\mathbb{R}^d)} + \|\mathbf{F}_1 \|_{L_1((0,\tau);H^{1+r}(\mathbb{R}^d))}  \bigr).
		\end{align*}
		
	\noindent	$3^\circ$. 
		Suppose that Condition \emph{\ref{cond_Lambda_infty}} is satisfied.
		If 
		$$
		\boldsymbol{\psi}_1 \in H^1 (\mathbb{R}^d, \mathbb{C}^n)\quad\mbox{and} \quad \mathbf{F}_1 \in L_{1, \mathrm{loc}} (\mathbb{R}; H^1 (\mathbb{R}^d, \mathbb{C}^n) ),
		$$
	then for  $\tau \in \R$ we have 
		$$
		\lim\limits_{\varepsilon \to 0} \| \mathbf{u}_\varepsilon (\,\cdot\,, \tau) - \mathbf{v}^0_\varepsilon (\,\cdot\,, \tau) \|_{H^1(\mathbb{R}^d)} = 0\quad\mbox{and}\quad \lim\limits_{\varepsilon \to 0} \| 
		\mathbf{p}_\varepsilon(\,\cdot\,, \tau) - \q^0_\eps (\,\cdot\,, \tau) \|_{L_2 (\mathbb{R}^d)} = 0.
		$$
 \end{theorem}
 
 \begin{remark}
By Remark  \ref{rem15.59}$(1^\circ)$, under the assumptions of Theorem \ref{th16.19}$(2^\circ)$, for 
\hbox{$0\le r \le 1$}, $\tau \in \R,$ and $0< \eps \le 1$ we have
		\begin{equation*}
		\begin{aligned}
		&\| \mathbf{u}_\varepsilon(\,\cdot\,, \tau) - \mathbf{v}^0_\varepsilon (\,\cdot\,, \tau) 
		\|_{H^1 (\mathbb{R}^d)} 
		\\
		& \le 
		{\mathfrak{C}}_{9} (r) (1+|\tau|)^{r}  \varepsilon^{r} 
		\bigl( 1+ (1+|\tau|)\eps\bigr)^{1-r}
		\bigl( \| \boldsymbol{\psi}_1 \|_{H^{1+r}(\mathbb{R}^d)} + \|\mathbf{F}_1 \|_{L_1((0,\tau);H^{1+r}(\mathbb{R}^d))}  \bigr).
		\end{aligned}
		\end{equation*}
For bounded values of  $(1+|\tau|)\eps$ the right-hand side is of order $ (1+|\tau|)^{r}  \varepsilon^{r}$.
\end{remark}

Statements $1^\circ$ and $2^\circ$ of Theorem \ref{th16.19} can be improved under some additional assumptions.
 Theorem \ref{th15.49}$(2^\circ)$  and Corollary \ref{cor15.54b}$(2^\circ)$ imply the following result.

 \begin{theorem}
 \label{th16.20}
  Suppose that the assumptions of Theorem \emph{\ref{th16.19}} are satisfied. 
  Suppose that Condition \emph{\ref{cond_BB}} or Condition~\emph{\ref{sndw_cond1}} \emph{(}or more restrictive Condition~\emph{\ref{sndw_cond2}}\emph{)} is satisfied.

\noindent	$1^\circ$. 
		Suppose that Condition  \emph{\ref{cond_Lambda_2}} is satisfied.
		If  $\boldsymbol{\psi}_1 \in H^{3/2}(\mathbb{R}^d; \mathbb{C}^n)$ and $\mathbf{F}_1 \in L_{1, \mathrm{loc}}(\mathbb{R}; H^{3/2}(\mathbb{R}^d; \mathbb{C}^n)),$  then for $\tau \in \mathbb{R}$ and 
		 $0 < \varepsilon \le 1$ we have 
		\begin{align*}
		\| \mathbf{u}_\varepsilon(\,\cdot\,, \tau) &- \mathbf{v}^0_\varepsilon (\,\cdot\,, \tau) \|_{H^1 (\mathbb{R}^d)} 
	\le 		{\mathrm{C}}^\circ_{9} (1+|\tau|)^{1/2}  \varepsilon \bigl( \| \boldsymbol{\psi}_1 \|_{H^{3/2}(\mathbb{R}^d)} + \|\mathbf{F}_1 \|_{L_1((0,\tau);H^{3/2}(\mathbb{R}^d))}  \bigr),
		\\
		\| \mathbf{p}_\varepsilon(\,\cdot\,, \tau) &- 
		\q_\eps^0 (\,\cdot\,, \tau) \|_{L_2 (\mathbb{R}^d)}
	\le 		{\mathrm{C}}^\circ_{10} (1+|\tau|)^{1/2}  \varepsilon \bigl( \| \boldsymbol{\psi}_1 \|_{H^{3/2}(\mathbb{R}^d)} + \|\mathbf{F}_1 \|_{L_1((0,\tau);H^{3/2}(\mathbb{R}^d))}  \bigr).
		\end{align*}
	
	\noindent	$2^\circ$.  
		Suppose that Condition \emph{\ref{cond_Lambda_infty}} is satisfied.
		If $\boldsymbol{\psi}_1 \in H^{1+r}(\mathbb{R}^d; \mathbb{C}^n)$ and $\mathbf{F}_1 \in L_{1, \mathrm{loc}}(\mathbb{R}; H^{1+r}(\mathbb{R}^d; \mathbb{C}^n)),$ $0 \le r \le 1/2,$ 
		then for  $\tau \in \mathbb{R}$ and   $0 < \varepsilon \le 1$ we have 
		\begin{equation}
		\label{16.76}
		\| 
		\D \mathbf{u}_\varepsilon(\,\cdot\,, \tau) -\D \mathbf{v}^0_\varepsilon (\,\cdot\,, \tau) \|_{L_2 (\mathbb{R}^d)}
		 \le {\mathfrak{C}}^\circ_7 (r) (1+|\tau|)^{r}  \varepsilon^{2r} \bigl( \| \boldsymbol{\psi}_1 \|_{H^{1+r}(\mathbb{R}^d)} + \|\mathbf{F}_1 \|_{L_1((0,\tau);H^{1+r}(\mathbb{R}^d))}  \bigr),
		\end{equation}
	\begin{equation*}
		\| 
		 \mathbf{p}_\varepsilon(\,\cdot\,, \tau) - \q_\eps^0 (\,\cdot\,, \tau) 
		 \|_{L_2 (\mathbb{R}^d)} 
	\le 
		{\mathfrak{C}}_{8}^\circ (r) (1+|\tau|)^{r}  \varepsilon^{2 r} \bigl( \| \boldsymbol{\psi}_1 \|_{H^{1+r}(\mathbb{R}^d)} + \|\mathbf{F}_1 \|_{L_1((0,\tau);H^{1+r}(\mathbb{R}^d))}  \bigr).
		\end{equation*}
 \end{theorem}

 \begin{remark}
By Remark \ref{rem15.59}$(2^\circ)$, under the assumptions of Theorem \ref{th16.20}$(2^\circ)$, for \hbox{$0\le r \le 1/2$}, $\tau \in \R,$ and $0< \eps \le 1$ we have
		\begin{equation*}
		\begin{aligned}
		\| \mathbf{u}_\varepsilon(\,\cdot\,, \tau) - \mathbf{v}^0_\varepsilon (\,\cdot\,, \tau) 
		\|_{H^1 (\mathbb{R}^d)} &\le 
		{\mathfrak{C}}_{10} (r) (1+|\tau|)^{r}  \varepsilon^{2r} 
		\bigl( 1+ (1+|\tau|)\eps\bigr)^{1-2r}
		\\
		&\times \bigl( \| \boldsymbol{\psi}_1 \|_{H^{1+r}(\mathbb{R}^d)} + \|\mathbf{F}_1 \|_{L_1((0,\tau);H^{1+r}(\mathbb{R}^d))}  \bigr).
		\end{aligned}
		\end{equation*}
The order of this estimate is worse than in  \eqref{16.76}.
\end{remark}

\section{Application of the general results: the acoustics equation}

\subsection{The model operator\label{sec17.1}}
In $L_2(\R^d)$, consider the operator
\begin{equation}
\label{17.1}
\wh{\mathcal A} = \D^* g(\x)\D = - \operatorname{div} g(\x) \nabla.
\end{equation}
Here $g(\x)$ is a $\Gamma$-periodic Hermitian $(d \times d)$-matrix-valued function such that 
$g(\x) >0$ and $g,g^{-1} \in L_\infty$. The operator \eqref{17.1} is a particular case of the operator
\eqref{hatA}. In this case, we have $n=1$, $m=d$, and $b(\D)=\D$. Obviously, condition~\eqref{rank_alpha_ineq}
is valid with $\alpha_0 = \alpha_1=1$. According to \eqref{hatA0}, the effective operator for the operator 
\eqref{17.1} is given by  
$$
\wh{\mathcal A}^0 = \D^* g^0 \D = - \operatorname{div} g^0 \nabla.
$$
According to  \eqref{g_tilde}, \eqref{g0}, the effective matrix $g^0$ is defined as follows. 
Let $\e_1,\dots, \e_d$ be the standard orthonormal basis in $\R^d$.
Let $\Phi_j \in \wt{H}^1(\Omega)$ be the weak $\Gamma$-periodic solution of the problem
$$
 \operatorname{div} g(\x) \left(\nabla \Phi_j(\x) + \e_j \right) =0,\quad \int\limits_\Omega \Phi_j(\x) \,d\x=0.
$$
Then $\Lambda(\x)$ is the row  $\Lambda(\x)= i(\Phi_1(\x),\dots,\Phi_d(\x))$, and $\wt{g}(\x)$ is the  $(d\times d)$-matrix with the columns 
$\wt{\mathbf g}_j(\x) = g(\x) \left(\nabla \Phi_j(\x) + \e_j \right)$, $j=1,\dots,d$. The effective matrix is given by  
$$
g^0 = |\Omega|^{-1} \int\limits_\Omega \wt{g}(\x) \, d\x.
$$ 
In the case where $d=1$, we have $m=n=1$, whence $g^0=\underline{g}$.

If $g(\x)$ is a symmetric matrix with real entries, then Proposition 
\ref{N=0_proposit}$(1^\circ)$ implies that $\wh{N}(\boldsymbol{\theta})=0$ for all
$\boldsymbol{\theta} \in \mathbb{S}^{d-1}$. If $g(\x)$ is a Hermitian matrix with complex entries, then in general the operator $\wh{N}(\boldsymbol{\theta})$ is not equal to zero.
Since $n=1$, then the operator $\wh{N}(\boldsymbol{\theta})= \wh{N}_0(\boldsymbol{\theta})$ 
is the operator of multiplication by $\wh{\mu}(\boldsymbol{\theta})$, where 
$\wh{\mu}(\boldsymbol{\theta})$ is the coefficient of  $t^3$ in the expansion for the first eigenvalue 
$$
\wh{\lambda}_1(t, \boldsymbol{\theta}) =  \wh{\gamma} (\boldsymbol{\theta}) t^2 
+ \wh{\mu} (\boldsymbol{\theta}) t^3 + \wh{\nu} (\boldsymbol{\theta}) t^4+\dots
$$
of the operator $\wh{\mathcal A}(\k) = \wh{A}(t,\boldsymbol{\theta})$. A calculation (see \cite[Subsection~10.3]{BSu3}) shows that 
$$
\begin{aligned}
&\wh{N}(\boldsymbol{\theta}) = \wh{\mu}(\boldsymbol{\theta}) = -i \sum_{j,l,k=1}^d
\left( a_{jlk} - a^*_{jlk}\right) \theta_j \theta_l \theta_k,\quad \boldsymbol{\theta} \in \mathbb{S}^{d-1},
\\
&a_{jlk} = |\Omega|^{-1} \int\limits_\Omega \Phi_j(\x)^* \langle g(\x) (\nabla \Phi_l(\x) + \e_l), \e_k\rangle \, d\x,
\quad j,l,k = 1,\dots,d.
\end{aligned}
$$
In \cite[Subsection 10.4]{BSu3}, there is an  example of the operator \eqref{17.1} with the Hermitian matrix  
$g(\x)$ with complex entries such that  $\wh{N}(\boldsymbol{\theta}) = \wh{\mu}(\boldsymbol{\theta}) \not\equiv 0$.

Now, we describe the operator $\wh{\mathcal N}^{(1)}(\boldsymbol{\theta})$ which is the operator of multiplication by $\wh{\nu}(\boldsymbol{\theta})$. Let $\Psi_{jl}(\x)$ be the  $\Gamma$-periodic solution of the problem
\begin{equation*}
- \operatorname{div} g(\x) \left(\nabla \Psi_{jl}(\x) - \Phi_j(\x) \e_l \right) =
g^0_{lj} - \wt{g}_{lj}(\x),\quad \int\limits_\Omega \Psi_{jl}(\x) \,d\x=0.
\end{equation*}
As was checked in \cite[Subsection 14.5]{VSu2},  we have
{\allowdisplaybreaks
\begin{align*}
\wh{\mathcal N}^{(1)}(\boldsymbol{\theta}) = \wh{\nu}(\boldsymbol{\theta})& =
\!\!\sum_{p,q,l,k=1}^d \!\!\left( \alpha_{pqlk} - (\overline{\Phi_p^* \Phi_q}) g^0_{lk} \right) \theta_p \theta_q \theta_l \theta_k,
\\
\alpha_{pqlk} &= |\Omega|^{-1}\! \int\limits_\Omega\!\! ( \wt{g}_{lp}(\x) \Psi_{qk}(\x) + \wt{g}_{kq}(\x) \Psi_{pl}(\x))\, d\x
\\
&\quad+ |\Omega|^{-1} \!\!\int\limits_\Omega \!\langle g(\x) ( \nabla \Psi_{qk}(\x) - \Phi_q(\x)\e_k, \nabla \Psi_{pl}(\x) - \Phi_p(\x)\e_l
\rangle\, d\x.
\end{align*}
} 
\begin{remark}
In \cite[Lemma 12.2]{D} it was shown that for $d=1$ and $g(x) \ne \operatorname{const}$ 
we always have  $\wh{\nu}(1) = \wh{\nu}(-1) \ne 0$. 
Therefore, the authors beleive that in the multidimensional case, as a rule, \hbox{$\wh{\nu}(\boldsymbol{\theta})
 \ne 0$}.   
\end{remark}

Consider the Cauchy problem
\begin{equation}
\label{17.4}
\left\{
\begin{aligned}
&\frac{\partial^2 {u}_\varepsilon (\mathbf{x}, \tau)}{\partial \tau^2} = - \mathbf{D}^* g^\varepsilon (\mathbf{x}) \mathbf{D} {u}_\varepsilon (\mathbf{x}, \tau), \\
& {u}_\varepsilon (\mathbf{x}, 0) =\phi(\x), \quad \frac{\partial {u}_\varepsilon }{\partial \tau} (\mathbf{x}, 0) = 
{\psi}(\mathbf{x}) + \D^* \boldsymbol{\rho}(\x),
\end{aligned}
\right.
\end{equation}
where ${\phi}, {\psi} \in L_2 (\mathbb{R}^d)$, $\boldsymbol{\rho} \in L_2 (\mathbb{R}^d; \AC^d)$. (For simplicity, we consider the homogeneous equation.)
Let $u_0$ be the solution of the homogenized problem
\begin{equation}
\label{17.5}
\left\{
\begin{aligned}
&\frac{\partial^2 {u}_0 (\mathbf{x}, \tau)}{\partial \tau^2} = - \mathbf{D}^* g^0 (\mathbf{x}) \mathbf{D} {u}_0 (\mathbf{x}, \tau), \\
& {u}_0 (\mathbf{x}, 0) =\phi(\x), \quad \frac{\partial {u}_0}{\partial \tau} (\mathbf{x}, 0) = {\psi}(\mathbf{x})
+\D^* \boldsymbol{\rho}(\x).
\end{aligned}
\right.
\end{equation}

Applying Theorem \ref{th16.1} in the general case and Theorem \ref{th16.2} in the ``real'' case, we obtain the following result.

\begin{proposition}
	\label{th17.1}
	Suppose that  ${u}_\varepsilon$~is the solution of problem~\eqref{17.4} and  
	${u}_0$~is the solution of the homogenized problem~\eqref{17.5}.

	\noindent	$1^\circ$.  If  ${\phi} \in H^{s}(\mathbb{R}^d),$
		${\psi} \in H^{r}(\mathbb{R}^d),$ and $\boldsymbol{\rho} \in H^s(\R^d;\AC^d),$ where $0 \le s \le 2,$ $0\le r \le 1,$
		then for  $\tau \in \mathbb{R}$ and $0< \varepsilon \le 1$ we have 
		\begin{equation*}
		\begin{aligned}
		&\| {u}_\varepsilon(\,\cdot\,, \tau) - {u}_0 (\,\cdot\,, \tau) \|_{L_2 (\mathbb{R}^d)} \le 
		\widehat{\mathfrak{C}}_1 (s) (1+|\tau|)^{s/2}  \varepsilon^{s/2}  
		\| {\phi} \|_{H^{s}(\mathbb{R}^d)} 
		\\
		&+ \widehat{\mathfrak{C}}_2 (r) (1+|\tau|)^{(r+1)/2}  \varepsilon^{(r+1)/2} 
		\| {\psi} \|_{H^{r}(\mathbb{R}^d)}
		+ \widehat{\mathfrak{C}}'_2 (s) (1+|\tau|)^{s/2}  \varepsilon^{s/2} 
		\| \boldsymbol{\rho} \|_{H^{s}(\mathbb{R}^d)}.
		\end{aligned}
		\end{equation*}
		If ${\phi}, {\psi} \in L_2 (\mathbb{R}^d)$  and  
		$\boldsymbol{\rho} \in L_2 (\mathbb{R}^d; \AC^d),$ then for $\tau \in \R$ we have 
	$$
	\lim\limits_{\varepsilon \to 0} \| {u}_\varepsilon (\,\cdot\,, \tau) - {u}_0 (\,\cdot\,, \tau) \|_{L_2(\mathbb{R}^d)} = 0.
	$$
		
\noindent	$2^\circ$. Let $g(\x)$ be a symmetric matrix with real entries. 
	If  ${\phi} \in H^{s}(\mathbb{R}^d),$ ${\psi} \in H^{r}(\mathbb{R}^d),$ and  
	$\boldsymbol{\rho} \in H^s(\R^d;\AC^d),$ 
	where $0 \le s \le 3/2,$ $0\le r \le 1/2,$ then for $\tau \in \mathbb{R}$ and $0< \varepsilon \le 1$ we have
		\begin{equation*}
		\begin{aligned}
		&\| {u}_\varepsilon(\,\cdot\,, \tau) - {u}_0 (\,\cdot\,, \tau) \|_{L_2 (\mathbb{R}^d)} \le 
		\widehat{\mathfrak{C}}_3 (s) (1+|\tau|)^{s/3}  \varepsilon^{2s/3}  
		\| {\phi} \|_{H^{s}(\mathbb{R}^d)} 
		\\
		&+ \widehat{\mathfrak{C}}_4 (r) (1+|\tau|)^{(r+1)/3}  \varepsilon^{2(r+1)/3}  
		\| {\psi} \|_{H^{r}(\mathbb{R}^d)}
		+ \widehat{\mathfrak{C}}'_4 (s) (1+|\tau|)^{s/3}  \varepsilon^{2s/3}  
		\| \boldsymbol{\rho} \|_{H^{s}(\mathbb{R}^d)}.
		\end{aligned}
		\end{equation*}
	\end{proposition}

 Now, we consider the case where $\phi=0$ and $\boldsymbol{\rho} =0$, 
 and approximate the solution in the energy norm.
According to  \eqref{v_eps_first_approx}, the first order approximation takes the form 
		\begin{equation}
		\label{17.8}
v_\eps(\x,\tau) = u_0(\x,\tau) + \eps \sum_{j=1}^d \Phi_j^\eps(\x) (\Pi_\eps 
\partial_j u_0)(\x,\tau).
	\end{equation}
By Proposition \ref{prop11.22}$(2^\circ)$, in the ``real'' case we have $\Lambda \in L_\infty$, 
and then it is possible to use the first order approximation without smoothing: 
	\begin{equation}
		\label{17.9}
v_\eps^0(\x,\tau) = u_0(\x,\tau) + \eps \sum_{j=1}^d \Phi_j^\eps(\x)  \partial_j u_0(\x,\tau).
	\end{equation}

In the general case we apply Theorem \ref{th16.4} and Remark \ref{rem16.5}, and in the  ``real'' case we apply Theorem \ref{th16.10} and Remark \ref{rem16.12}.

\begin{proposition}
	\label{prop17.2}
	Suppose that ${u}_\varepsilon$ is the solution of problem~\eqref{17.4} and 
	${u}_0$ is the solution of problem~\eqref{17.5} with $\phi\!=\!0$ and $\boldsymbol{\rho}\! =\!0$.
	Let $v_\eps$ be given by \eqref{17.8} and $v_\eps^0$ by~\eqref{17.9}.
	
\noindent		$1^\circ$. 
		 If  ${\psi} \in H^{2}(\mathbb{R}^d),$ then for  $\tau \in \mathbb{R}$ and $0< \varepsilon \le 1$ we have
		\begin{equation*}
		\| {u}_\varepsilon(\,\cdot\,, \tau) - {v}_\eps (\,\cdot\,, \tau) \|_{H^1(\mathbb{R}^d)} \le 
		\widehat{\mathrm{C}}_{7} (1+|\tau|)  \varepsilon  \| {\psi} \|_{H^{2}(\mathbb{R}^d)}. 
		\end{equation*}
	If  ${\psi} \in H^{s}(\mathbb{R}^d),$ where $0 \le s \le 2,$ then for
		$\tau \in \mathbb{R}$ and $0< \varepsilon \le 1$ we have 
		\begin{align*}
		\|& \nabla {u}_\varepsilon(\,\cdot\,, \tau) - \nabla {v}_\eps (\,\cdot\,, \tau) \|_{L_2 (\mathbb{R}^d)} 
		\le \widehat{\mathfrak{C}}_5 (s) (1+|\tau|)^{s/2}  \varepsilon^{s/2}
		 \| {\psi} \|_{H^{s}(\mathbb{R}^d)},
		\\
		\|& g^\eps \nabla {u}_\varepsilon(\,\cdot\,, \tau) - \wt{g}^\eps  \Pi_\eps \nabla {u}_0 (\,\cdot\,, \tau) \|_{L_2 (\mathbb{R}^d)} 
		\le \widehat{\mathfrak{C}}_6 (s) (1+|\tau|)^{s/2}  \varepsilon^{s/2}
		  \| {\psi} \|_{H^{s}(\mathbb{R}^d)}.
		\end{align*}
		If ${\psi} \in L_2 (\mathbb{R}^d),$ then for $\tau \in \R$ we have
		\begin{equation*}
		\begin{split}
		&\lim\limits_{\varepsilon \to 0} \| {u}_\varepsilon (\,\cdot\,, \tau) - {v}_\eps (\,\cdot\,, \tau) \|_{H^1(\mathbb{R}^d)} = 0, 
		\\
		&\lim\limits_{\varepsilon \to 0} 
		\| g^\eps \nabla {u}_\varepsilon(\,\cdot\,, \tau) - \wt{g}^\eps  \Pi_\eps \nabla {u}_0 (\,\cdot\,, \tau) \|_{L_2 (\mathbb{R}^d)} = 0.
		\end{split}
		\end{equation*} 
	
\noindent	$2^\circ$. Let $g(\x)$ be a symmetric matrix with real entries. If 
		  ${\psi} \in H^{3/2}(\mathbb{R}^d),$ then for $\tau \in \mathbb{R}$ and $0< \varepsilon \le 1$ we have 
		\begin{equation*}
		\| {u}_\varepsilon(\,\cdot\,, \tau) - {v}^0_\eps (\,\cdot\,, \tau) \|_{H^1 (\mathbb{R}^d)} \le 
		\widehat{\mathrm{C}}_{9}^\circ (1+|\tau|)^{1/2}  \varepsilon  
		 \| {\psi} \|_{H^{3/2}(\mathbb{R}^d)}.
		\end{equation*}
	If ${\psi} \in H^{1+r}(\mathbb{R}^d),$ where $0 \le r \le 1/2,$ 
		then for $\tau \in \mathbb{R}$ and $0< \varepsilon \le 1$ we have 
		\begin{align*}
		\| & \nabla {u}_\varepsilon(\,\cdot\,, \tau) - \nabla {v}^0_\eps (\,\cdot\,, \tau) \|_{L_2 (\mathbb{R}^d)}
		 \le 
		\widehat{\mathfrak{C}}_7^\circ (r) (1+|\tau|)^{r}  \varepsilon^{2r}
		  \| {\psi} \|_{H^{1+r}(\mathbb{R}^d)},
		\\
		\| & g^\eps \nabla {u}_\varepsilon(\,\cdot\,, \tau) - \wt{g}^\eps  \nabla {u}_0 (\,\cdot\,, \tau) \|_{L_2 (\mathbb{R}^d)} 
		\le 
		\widehat{\mathfrak{C}}_{8}^\circ (r) (1+|\tau|)^{r}  \varepsilon^{2r}
		 \| {\psi} \|_{H^{1+r}(\mathbb{R}^d)}.
		\end{align*}
		If ${\psi} \in H^1 (\mathbb{R}^d),$ then for $\tau \in \R$ we have 
		\begin{equation*}
		\lim\limits_{\varepsilon \to 0} \| {u}_\varepsilon (\,\cdot\,, \tau) - {v}^0_\eps (\,\cdot\,, \tau) \|_{H^1(\mathbb{R}^d)}\!\! =\! 0, \quad \lim\limits_{\varepsilon \to 0} 
		\| g^\eps \nabla {u}_\varepsilon(\,\cdot\,, \tau) - \wt{g}^\eps  \nabla {u}_0 (\,\cdot\,, \tau) \|_{L_2 (\mathbb{R}^d)}\!\!=\!0.
		\end{equation*} 
	\end{proposition}

	\subsection{The acoustics equation}
	Under the assumptions of Subsection \ref{sec17.1}, suppose in addition that $g(\x)$ is a symmetric matrix with real entries. The matrix $g(\x)$ characterizes the parameters of the acoustical  (in general, anisotropic) medium. Let 
	$Q(\x)$ be a $\Gamma$-periodic function in~$\R^d$ such that $Q(\x)>0$ and $Q, Q^{-1} \in L_\infty$. 
	This function  plays the role of the medium density. We put $f(\x) = Q(\x)^{-1/2}$.
	
	We consider the Cauchy problem for the acoustics equation in the medium with rapidly oscillating characteristics:
\begin{equation}
\label{17.15}
\left\{
\begin{aligned}
&Q^\eps(\x) \frac{\partial^2 {u}_\varepsilon (\mathbf{x}, \tau)}{\partial \tau^2} = - \mathbf{D}^* g^\varepsilon (\mathbf{x}) \mathbf{D} {u}_\varepsilon (\mathbf{x}, \tau), \\
& {u}_\varepsilon (\mathbf{x}, 0) =\phi(\x), \quad \frac{\partial {u}_\varepsilon }{\partial \tau} (\mathbf{x}, 0) = 
{\psi}_1(\mathbf{x}) + (Q^\eps)^{-1}({\psi}_2(\mathbf{x})+ \D^* \boldsymbol{\rho}(\x)),
\end{aligned}
\right.
\end{equation}
where ${\phi}, {\psi}_1, \psi_2 \in L_2 (\mathbb{R}^d)$, $\boldsymbol{\rho} \in L_2(\R^d;\AC^d)$.
(For simplicity, we consider the homogeneous equation.)
Suppose that  $u_0$ is the solution of the homogenized problem
\begin{equation}
\label{17.16}
\left\{
\begin{aligned}
& \overline{Q} \frac{\partial^2 {u}_0 (\mathbf{x}, \tau)}{\partial \tau^2} = - \mathbf{D}^* g^0 (\mathbf{x}) \mathbf{D} {u}_0 (\mathbf{x}, \tau), \\
& {u}_0 (\mathbf{x}, 0) =\phi(\x), \quad \frac{\partial {u}_0}{\partial \tau} (\mathbf{x}, 0) = {\psi}_1(\mathbf{x})+
(\overline{Q})^{-1}( {\psi}_2(\mathbf{x}) + \D^* \boldsymbol{\rho}(\x)).
\end{aligned}
\right.
\end{equation}

By Proposition \ref{N_Q=0_proposit}$(1^\circ)$, we have $\wh{N}_Q(\boldsymbol{\theta})=0$
for any $\boldsymbol{\theta} \in {\mathbb S}^{d-1}$.
In the general case we apply Theorem \ref{th16.13}, and in the case where $\psi_1=0$ we apply 
Theorem~\ref{th16.14}.

It is possible to approximate the solution in the energy norm if $\phi=0$,  ${\psi_2=0}$, and ${\boldsymbol{\rho} =0}$. As has been already mentioned, we have $\Lambda \in L_\infty$, and therefore Theorem \ref{th16.20} can be applied. Let us formulate the results.

\begin{proposition}
	\label{prop17.3}
	Let ${u}_\varepsilon$~be the solution of problem~\eqref{17.15}, and let  
	${u}_0$~be the solution of the homogenized problem~\eqref{17.16}.

	\noindent	$1^\circ$. 
		If ${\phi} \in H^{s}(\mathbb{R}^d),$ $\psi_1 \in H^\sigma(\R^d),$ $\psi_2 \in H^r(\R^d),$ and
		 $\boldsymbol{\rho} \in H^{s}(\mathbb{R}^d; \AC^d),$ where
		 \hbox{$0\le s \le 3/2$}, $0\le \sigma \le 1,$ $0\le r \le 1/2,$ 		
		 then for $\tau \in \mathbb{R}$ and $0< \varepsilon \le 1$ we have 
		{\allowdisplaybreaks
		\begin{align*}
		&\| {u}_\varepsilon(\,\cdot\,, \tau) - {u}_0 (\,\cdot\,, \tau) \|_{L_2 (\mathbb{R}^d)} \le 
		{\mathfrak{C}}_{3}(s) (1+|\tau|)^{s/3}  \varepsilon^{2s/3} \| \phi\|_{H^s(\R^d)}
		\\
		&+  
		\widetilde{\mathfrak{C}}_{2}(\sigma) (1+|\tau|)  \varepsilon^{\sigma} \| \psi_1 \|_{H^\sigma(\R^d)}+  
		{\mathfrak{C}}_{4}(r) (1+|\tau|)^{(1+r)/3}  \varepsilon^{2(1+r)/3} \| \psi_2\|_{H^r(\R^d)} 
		\\
		&+ {\mathfrak{C}}_{4}'(s) (1+|\tau|)^{s/3}  \varepsilon^{2s/3} \| \boldsymbol{\rho}\|_{H^s(\R^d)}.  
		\end{align*}
		}
		\hspace{-3mm}
		If $\phi, {\psi}_1, \psi_2 \in L_2 (\mathbb{R}^d)$ and $\boldsymbol{\rho} \in L_2(\R^d;\AC^d),$ then 
		$$
		\lim\limits_{\varepsilon \to 0} \| {u}_\varepsilon (\,\cdot\,, \tau) - {u}_0 (\,\cdot\,, \tau) 
		\|_{L_2(\mathbb{R}^d)} = 0
		$$
		for $\tau \in \mathbb{R}$.

	\noindent	$2^\circ$. Let $\phi=0,$ $\psi_2 =0,$ and $\boldsymbol{\rho}=0$. We put  
		$
		v_\eps^0 = u_0 + \eps \sum_{j=1}^d \Phi^\eps_j \partial_j u_0.
		$
		If  $\psi_1 \in H^{3/2}(\R^d),$ then for $\tau \in \mathbb{R}$ and $0< \varepsilon \le 1$ we have
		\begin{equation*}
		\| {u}_\varepsilon(\,\cdot\,, \tau) -  {v}^0_\eps (\,\cdot\,, \tau) \|_{H^1 (\mathbb{R}^d)} \le 
		{\mathrm{C}}_{9}^\circ (1+|\tau|)^{1/2}  \varepsilon
		 \| {\psi}_1 \|_{H^{3/2}(\mathbb{R}^d)}.
		 \end{equation*}
		 If $\psi_1 \in H^{1+r}(\R^d),$ where $0\le r \le 1/2,$ then for
		 $\tau \in \mathbb{R}$ and $0< \varepsilon \le 1$ we have
		 \begin{align*}
		\| \nabla {u}_\varepsilon(\,\cdot\,, \tau) - \nabla {v}^0_\eps (\,\cdot\,, \tau) \|_{L_2 (\mathbb{R}^d)} &\le 
		{\mathfrak{C}}_{7}^\circ (r) (1+|\tau|)^{r}  \varepsilon^{2r}
		 \| {\psi}_1 \|_{H^{1+r}(\mathbb{R}^d)} ,
		\\
	\| g^\eps \nabla {u}_\varepsilon(\,\cdot\,, \tau) - \wt{g}^\eps  \nabla {u}_0 (\,\cdot\,, \tau) \|_{L_2 (\mathbb{R}^d)} 	&\le 
		{\mathfrak{C}}^\circ_{8} (r) (1+|\tau|)^{r}  \varepsilon^{2r}
		  \| {\psi}_1 \|_{H^{1+r}(\mathbb{R}^d)}.
		\end{align*}
		If ${\psi}_1 \in H^1 (\mathbb{R}^d),$  then for $\tau \in \R$ we have 
		\begin{equation*}
		\lim\limits_{\varepsilon \to 0} \| {u}_\varepsilon (\,\cdot\,, \tau) - {v}^0_\eps (\,\cdot\,, \tau) \|_{H^1(\mathbb{R}^d)} \!\!= \!0, \quad \lim\limits_{\varepsilon \to 0} 
		\| g^\eps \nabla {u}_\varepsilon(\,\cdot\,, \tau) - \wt{g}^\eps  \nabla {u}_0 (\,\cdot\,, \tau) \|_{L_2 (\mathbb{R}^d)}\!\!=\!0.
		\end{equation*} 
	\end{proposition}

\section{Application of the general results: the system of elasticity}

\subsection{The operator of elasticity theory}
Let $d \ge 2$. We represent the elasticity operator  as in  \cite[Chapter 5, \S 2]{BSu1}.
 Let $\zeta$ be an  orthogonal second rank tensor in $\R^d$.
In the standard orthonormal basis in $\R^d$, it is represented by a matrix 
$\zeta = \{\zeta_{jl}\}_{j,l=1}^d$. We consider  symmetric tensors $\zeta$ and identify them with vectors $\zeta_* \in \AC^m$, $2m= d(d+1)$, by the following rule. The vector $\zeta_*$ consists of all components  
$\zeta_{jl}$, $j \le l$, ordered in a fixed way. 

For the displacement vector $\u \in H^1(\R^d;\AC^d)$, we introduce the deformation tensor  
$e(\u) = \frac{1}{2} \{ \partial_l u_j + \partial_j u_l\}$. Let $e_*(\u)$ be the vector corresponding to the tensor 
$e(\u)$ in accordance with the rule described above. The relation $b(\D) \u = -i e_*(\u)$ determines   
an $(m\times d)$-matrix DO  $b(\D)$ uniquely (the symbol $b(\boldsymbol{\xi})$ of this DO  is a matrix with real entries). For instance, with an appropriate ordering, we have 
$$
b(\D) = \begin{pmatrix} \xi_1 & 0 \\  \frac{1}{2} \xi_2  & \frac{1}{2} \xi_1 \\ 0 & \xi_2 \end{pmatrix}, \quad d=2.
$$ 
In the case under consideration, $n=d$ and $m=d(d+1)/2$. It is easily seen that condition \eqref{rank_alpha_ineq}
is satisfied, and  $\alpha_0$, $\alpha_1$ depend only on $d$.

Let $\sigma(\u)$ be the stress tensor, and let $\sigma_*(\u)$ be the corresponding vector.
Then the Hooke law on  proportionality of stresses and deformations can be expressed by the relation
$\sigma_*(\u) = g(\x) e_*(\u)$, where $g(\x)$ is a symmetric  $(m\times m)$-matrix with real entries. The matrix $g$ characterizes the parameters of the elastic (in general, anisotropic) medium.
We assume that the matrix-valued function $g(\x)$ is periodic and such that $g(\x)>0$ and $g,g^{-1} \in L_\infty$.

The energy of elastic deformations is given by the quadratic form
\begin{equation}
\label{forma_w}
\begin{split}
w[\u,\u]& = \frac{1}{2} \int\limits_{\R^d} \langle \sigma_*(\u), e_*(\u) \rangle \,d\x
\\
&= \frac{1}{2} \int\limits_{\R^d} \langle g(\x) b(\D)\u, b(\D)\u \rangle \,d\x, \ \u \in H^1(\R^d;\AC^d).
\end{split}
\end{equation}
The operator $\mathcal W$ generated by this form  in the space $L_2(\R^d;\AC^d)$ is called the elasticity operator. Thus, we have $2 {\mathcal W} = \wh{\mathcal A} = b(\D)^* g(\x) b(\D)$.

In the case of isotropic medium, the matrix $g(\x)$ is expressed in terms of two functional parameters 
$\lambda(\x)$ and $\mu(\x)$ (the Lame parameters).  Here $\mu$ is the shear modulus. Often, another parameter
$K(\x)$ is introduced instead of $\lambda$;  $K(\x)$ is called the modulus of volume compression. We need yet another modulus $\beta(\x)$. Here are the relations:
$K(\x) = \lambda(\x) + \frac{2\mu(\x)}{d}$, $\beta(\x) = \mu(\x) + \frac{\lambda(\x)}{2}.$
In the isotropic case, the conditions that ensure the positive definiteness of the matrix $g(\x)$ are as follows:
$\mu(\x) \ge \mu_0 >0$ and $K(\x) \ge K_0 >0$.
As an example, we write down the matrix $g(\x)$ in the isotropic case for $d=2$:
$$
g(\x) = \begin{pmatrix} K(\x) + \mu(\x) & 0 & K(\x) - \mu(\x) \\
0 & 4 \mu(\x) & 0 \\ K(\x) - \mu(\x) & 0 & K(\x) + \mu(\x)\end{pmatrix}.
$$

\subsection{Homogenization of the elasticity system}
Now, we consider the elasticity operator 
${\mathcal W}_\eps = \frac{1}{2} \wh{\mathcal A}_\eps = \frac{1}{2} b(\D)^* g^\eps(\x) b(\D)$ with rapidly oscillating coefficients. The effective matrix $g^0$ and the effective operator ${\mathcal W}^0 = \frac{1}{2} \wh{\mathcal A}^0 = \frac{1}{2} b(\D)^* g^0 b(\D)$
are constructed by the general rules  (see Subsections \ref{sec8.2}, \ref{sec8.3}).
In the isotropic case, the effective medium is in general anisotropic.

In general, the operator $\wh{N}(\boldsymbol{\theta})$ is not equal to zero. 
Moreover, there are examples where  $\wh{N}_0(\boldsymbol{\theta}) \ne 0$ at some points
$\boldsymbol{\theta} \in {\mathbb S}^{d-1}$ (even in the isotropic case). See \cite[Example 8.7]{Su4}, \cite[Subsection 14.3]{DSu}.

Let $Q(\x)$ be a $\Gamma$-periodic symmetric $(d\times d)$-matrix-valued function with real entries and such that $Q(\x)>0$; $Q, Q^{-1} \in L_\infty$. (Usually,  $Q$ is a scalar function having the sense of  the density of the medium). Denote $f(\x)= Q(\x)^{-1/2}$.
Consider the Cauchy problem  for the elasticity system with rapidly oscillating coefficients:
\begin{equation}
\label{17.20}
\begin{cases}
Q^\eps(\x) \frac{\partial^2 {\u}_\varepsilon (\mathbf{x}, \tau)}{\partial \tau^2}\! = \!- {\mathcal W}_\eps {\u}_\varepsilon (\mathbf{x}, \tau), \\
 {\u}_\varepsilon (\mathbf{x}, 0) \!=\!\boldsymbol{\phi}(\x), \quad 
\frac{\partial {\u}_\varepsilon }{\partial \tau} (\mathbf{x}, 0)\! =\! 
\boldsymbol{\psi}_1(\mathbf{x})\! +\! (Q^\eps)^{-1} (\boldsymbol{\psi}_2(\mathbf{x})\!+\! \D^* \boldsymbol{\rho}(\x)),
\end{cases}
\end{equation}
where $\boldsymbol{\phi}, \boldsymbol{\psi}_1, \boldsymbol{\psi}_2 \in L_2 (\mathbb{R}^d; \mathbb{C}^d)$ and
$\boldsymbol{\rho} \in L_2 (\mathbb{R}^d; \mathbb{C}^{d^2})$.
(For simplicity, we consider the homogeneous equation.)
Let $\u_0$ be the solution of the homogenized problem
\begin{equation}
\label{17.21}
\left\{
\begin{aligned}
& \overline{Q} \frac{\partial^2 {\u}_0 (\mathbf{x}, \tau)}{\partial \tau^2} = - {\mathcal W}^0 
{\u}_0 (\mathbf{x}, \tau), \\
& {\u}_0 (\mathbf{x}, 0) =\boldsymbol{\phi}(\x), \quad \frac{\partial {\u}_0}{\partial \tau} (\mathbf{x}, 0) = 
\boldsymbol{\psi}_1(\mathbf{x})+
(\overline{Q})^{-1}( \boldsymbol{\psi}_2(\mathbf{x}) + \D^* \boldsymbol{\rho}(\x)).
\end{aligned}
\right.
\end{equation}

Theorem \ref{th16.13} can be applied.
It is possible to approximate the solution in the energy norm in the case where $\boldsymbol{\phi} =0$, 
$\boldsymbol{\psi}_2=0$, and $\boldsymbol{\rho}=0$.
We can apply Theorem \ref{th16.16}. Let us formulate the results.

\begin{proposition}
	\label{prop17.4}
	Let  ${\u}_\varepsilon$~be the solution of problem~\eqref{17.20}, and let  
	${\u}_0$~be the solution of the homogenized problem~\eqref{17.21}.

	\noindent	$1^\circ$. 
		 If  $\boldsymbol{\phi} \in H^{s}(\mathbb{R}^d;\AC^d),$ $\boldsymbol{\psi}_1, \boldsymbol{\psi}_2 \in 
		 H^r (\R^d; \AC^d),$ $\boldsymbol{\rho} \in H^{s}(\mathbb{R}^d;\AC^{d^2}),$
		 where ${0\le s \le 2,}$ ${0\le r \le 1,}$ 		
		 then for $\tau \in \mathbb{R}$ and $0< \varepsilon \le 1$ we have 
		\begin{equation*}
		\begin{aligned}
		&\| {\u}_\varepsilon(\,\cdot\,, \tau) - {\u}_0 (\,\cdot\,, \tau) \|_{L_2 (\mathbb{R}^d)} \le 
		{\mathfrak{C}}_{1}(s) (1+|\tau|)^{s/2}  \varepsilon^{s/2} \| \boldsymbol{\phi} \|_{H^s(\R^d)}
		\\
		&+  
		\widetilde{\mathfrak{C}}_{2}(r) (1+|\tau|)  \varepsilon^{r} \| \boldsymbol{\psi}_1 \|_{H^r(\R^d)}+  
		{\mathfrak{C}}_{2}(r) (1+|\tau|)^{(1+r)/2}  \varepsilon^{(1+r)/2} \| \boldsymbol{\psi}_2\|_{H^r (\R^d)} 
		\\
	&+{\mathfrak{C}}'_{2}(s) (1+|\tau|)^{s/2}  \varepsilon^{s/2} \| \boldsymbol{\rho} \|_{H^s (\R^d)}. 
		\end{aligned}
		\end{equation*}
		If $\boldsymbol{\phi}, \boldsymbol{\psi}_1, \boldsymbol{\psi}_2 \in 
		L_2 (\mathbb{R}^d; \mathbb{C}^d)$ and 
		$\boldsymbol{\rho} \in L_2 (\mathbb{R}^d; \mathbb{C}^{d^2}),$ then 
	$$
	\lim\limits_{\varepsilon \to 0} \| {\u}_\varepsilon (\,\cdot\,, \tau) - {\u}_0 (\,\cdot\,, \tau) \|_{L_2} = 0
	$$
	for $\tau \in \mathbb{R}$.

\noindent		$2^\circ$. Let  $\boldsymbol{\phi} =0,$  $\boldsymbol{\psi}_2 =0,$ and $\boldsymbol{\rho}=0$.  We put $\v_\eps = \u_0 + \eps \Lambda^\eps b(\D)\Pi_\eps \u_0.$
		If  
		 $\boldsymbol{\psi}_1 \in H^{2}(\R^d;\AC^d),$ then for $\tau \in \mathbb{R}$ and $0< \varepsilon \le 1$ we have
		\begin{equation*}
		\| {\u}_\varepsilon(\,\cdot\,, \tau) -  {\v}_\eps (\,\cdot\,, \tau) \|_{H^1 (\mathbb{R}^d)} \le 
		{\mathrm{C}}_{7} (1+|\tau|)  \varepsilon
		 \| \boldsymbol{\psi}_1 \|_{H^{2}(\mathbb{R}^d)}.
		 \end{equation*}
		  If $\boldsymbol{\psi}_1 \in H^{s}(\R^d;\AC^d),$ where $0\le s \le 2,$
		 then for  $\tau \in \mathbb{R}$ and $0< \varepsilon \le 1$ we have 
		\begin{align*}
		&\| \D {\u}_\varepsilon(\,\cdot\,, \tau)\! - \!\D {\v}_\eps (\,\cdot\,, \tau) \|_{L_2 (\mathbb{R}^d)}\! \le\! 
		{\mathfrak{C}}_{5} (s) (1\!+\!|\tau|)^{s/2}  \varepsilon^{s/2}
		 \| \boldsymbol{\psi}_1 \|_{H^{s}(\mathbb{R}^d)} ,
		\\
		&\| g^\eps b(\D) {\u}_\varepsilon(\,\cdot\,, \tau)\! - \!\wt{g}^\eps  b(\D) (\Pi_\eps {\u}_0) (\,\cdot\,, \tau) \|_{L_2 (\mathbb{R}^d)} \!\le \!
		{\mathfrak{C}}_{6} (s) (1\!+\!|\tau|)^{s/2}  \varepsilon^{s/2}
		  \| \boldsymbol{\psi}_1 \|_{H^{s}(\mathbb{R}^d)}.
		\end{align*}
		If $\boldsymbol{\psi}_1 \in L_2 (\mathbb{R}^d),$ then for  $\tau \in \R$ we have 
		{\allowdisplaybreaks
		\begin{align*}
		& \lim\limits_{\varepsilon \to 0} \| {\u}_\varepsilon (\,\cdot\,, \tau) - {\v}_\eps (\,\cdot\,, \tau) \|_{H^1(\mathbb{R}^d)} = 0, 
		\\ 
		& \lim\limits_{\varepsilon \to 0} 
		\| g^\eps b(\D) {\u}_\varepsilon(\,\cdot\,, \tau) - \wt{g}^\eps  b(\D) (\Pi_\eps {\u}_0) (\,\cdot\,, \tau)
		 \|_{L_2 (\mathbb{R}^d)} =0.
		\end{align*} 
		}
	\end{proposition}

\subsection{The Hill body}
In mechanics  (see, e.~g., \cite{ZhKO}), an elastic isotropic medium with $\mu(\x)= \mu_0 = \operatorname{const}$ is called the Hill body. In this case, a simpler factorization for the operator $\mathcal W$ is possible; see \cite[Chapter 5, Subsection 2.3]{BSu1}. The form~\eqref{forma_w} can be represented as 
$$
w[\u,\u] = \int\limits_{\R^d} \langle g_{\wedge}(\x) b_\wedge(\D) \u,  b_\wedge(\D) \u \rangle\, d\x,
\quad \u \in H^1(\R^d; \AC^d).
$$
We have $m_\wedge = 1+ d(d-1)/2$. The symbol of the operator $b_\wedge(\D)$ is the matrix 
$b_\wedge(\boldsymbol{\xi})$ of size \hbox{$m_\wedge \times d$} defined as follows.
The first row is  $(\xi_1, \xi_2, \dots, \xi_d)$. The other rows correspond to pairs of indices $(j,l)$,
$1\le j<l\le d$. The entry in the  $(j,l)$th row and the $j$th column is $\xi_l$; the entry in the $(j,l)$th row and the $l$th column is $-\xi_j$;  all other entries of the $(j,l)$th row are equal to zero.
The matrix $g_\wedge(\x)$ is the diagonal matrix given by 
$$
g_\wedge(\x) = \operatorname{diag} \{ {\beta}(\x), \mu_0/2, \dots, \mu_0/2\}.
$$
The effective operator is given by 
${\mathcal W}^0 = b_\wedge(\D)^* g^0_\wedge b_\wedge(\D),$
where the effective matrix $g^0_\wedge$ coincides with  $\underline{g_\wedge}$:
$$
g^0_\wedge = \underline{g_\wedge} =
\operatorname{diag} \{ \underline{\beta}, \mu_0/2, \dots, \mu_0/2\}.
$$

By Proposition \ref{prop11.22}$(3^\circ)$, Condition $\Lambda \!\in\! L_\infty$ is satisfied.
For problem~\eqref{17.20}, Theorem  \ref{th16.13} is applicable;  in the case where $\boldsymbol{\phi}\!=\!0$ and
$\boldsymbol{\psi}_2\!=\!0$, we can apply Theorem~\ref{th16.19}.

Let us discuss the case where $Q(\x)=\mathbf{1}$. By Proposition \ref{N=0_proposit}$(3^\circ)$,
we have  $\wh{N}(\boldsymbol{\theta})=0$ for any  $\boldsymbol{\theta}\in {\mathbb S}^{d-1}$.
Therefore, Theorem \ref{th16.2} can be applied; in the case where $\boldsymbol{\phi}=0$, 
$\boldsymbol{\psi}_2=0$, and $\boldsymbol{\rho}=0$, we can apply Theorem \ref{th16.10}.

\end{document}